\documentclass[12pt]{amsart}
\usepackage[normalem]{ulem} 
\usepackage{amsmath, amssymb, fullpage}
\usepackage[all]{xy}
\usepackage{mathrsfs} 
\usepackage{mathtools} 
\usepackage[colorlinks=true,citecolor=black,linkcolor=black,urlcolor=blue]{hyperref}
\usepackage{silence}
\usepackage{ulem} 

\usepackage[symbol]{footmisc} 
\usepackage{enumitem} 

\usepackage{thmtools}
\usepackage{thm-restate}

\usepackage{tikz,scalerel}
\usetikzlibrary{shapes}
\usetikzlibrary{arrows.meta} 
\usetikzlibrary{calc} 
\usepackage{pgfplots} 
\usepackage{tikz-3dplot} 

\usepackage{comment}
\usepackage{colortbl}
\usepackage[most]{tcolorbox} 
\usepackage{subcaption}
\usepackage{graphicx} 

\usepackage{oldgerm} 

\usepackage{scalerel}

\newcommand{\sym}{\mathfrak{S}}
\newcommand{\perm}{\mathfrak{P}}

\tikzstyle{vertex}=[circle, draw, inner sep=0pt, minimum size=4pt]
\newcommand{\vertex}{\node[vertex]}
\tikzstyle{vinenode}=[circle, draw, inner sep=0pt, minimum size=3pt, fill]
\newcommand{\vnode}{\node[vinenode]}

\usepackage[]{xcolor}
\definecolor{RawSienna}{cmyk}{0,0.72,1,0.45}
\definecolor{lavgray}{rgb}{0.67, 0.66, 0.72}
\definecolor{darkred}{rgb}{0.5, 0.30, 0.30}
\definecolor{cadmiumgreen}{rgb}{0.0, 0.42, 0.24}
\definecolor{darkgreen}{RGB}{65,117,5}
\definecolor{darkyellow}{RGB}{245,166,35}

\newtheorem{theorem}{Theorem}[section]
\newtheorem{lemma}[theorem]{Lemma}
\newtheorem{proposition}[theorem]{Proposition}
\newtheorem{corollary}[theorem]{Corollary}

\newtheorem{question}[theorem]{Question}

\newtheorem{conjecture}[theorem]{Conjecture}

\theoremstyle{definition}
\newtheorem{remark}[theorem]{Remark}
\newtheorem{definition}[theorem]{Definition}
\newtheorem{example}[theorem]{Example}
\numberwithin{equation}{section}

\newcommand{\defn}[1]{\emph{\color{purple} #1}} 

\newcommand\Cliques{\mathrm{Cliques}}
\newcommand\SatCliques{\mathrm{SatCliques}}
\newcommand\Groves{\mathrm{Groves}}
\newcommand\SatGroves{\mathrm{SatGroves}}

\newcommand\Multicliques{\mathrm{Multicliques}}
\newcommand\SatMulticliques{\mathrm{SatMulticliques}}

\newcommand\Vines{\mathrm{Vines}}
\newcommand\SatVines{\mathrm{SatVines}}
\newcommand\Vineyards{\mathrm{Vineyards}}
\newcommand\SatVineyards{\mathrm{SatVineyards}}
\newcommand\VineyardShuffles{\mathrm{VineyardShuffles}}
\newcommand\SatVineyardShuffles{\mathrm{SatVineyardShuffles}}
\newcommand\GroveShuffles{\mathrm{GroveShuffles}}
\newcommand\SatGroveShuffles{\mathrm{SatGrovesShuffles}}

\newcommand\PermutationFlows{\mathrm{PermutationFlows}}
\newcommand\SatPermutationFlows{\mathrm{SatPermutationFlows}}
\newcommand\PermutationFlowShuffles{\mathrm{PermutationFlowShuffles}}
\newcommand\SatPermutationFlowShuffles{\mathrm{SatPermutationFlowShuffles}}

\newcommand\SplitAncestors{\mathrm{SplitAncestors}}
\newcommand\Covers{\mathrm{Covers}}
\newcommand\DKK{\mathrm{DKK}}
\newcommand\Extensions{\mathrm{Extensions}}
\newcommand\FinalSummaries{\mathrm{FinalSummaries}}
\newcommand\InPerm{\mathrm{InPerm}}  
\newcommand\RouteMatchings{\mathrm{RouteMatchings}}
\newcommand\OutPerm{\mathrm{OutPerm}}  

\newcommand\Prefixes{\mathrm{Prefixes}} 

\newcommand\Routes{\mathrm{Routes}}
\newcommand\Slacks{\mathrm{Slacks}}
\newcommand\Splits{\mathrm{Splits}}
\newcommand\Suffixes{\mathrm{Suffixes}} 

\newcommand\indeg{\mathrm{indeg}} 
\newcommand\outdeg{\mathrm{outdeg}} 
\newcommand\inedge{\mathrm{in}} 
\newcommand\outedge{\mathrm{out}} 
\newcommand\inpath{\mathrm{In}} 
\newcommand\outpath{\mathrm{Out}} 
\newcommand\initial{\mathrm{initial}}
\newcommand\length{\mathrm{length}}

\newcommand\minext{\mathrm{minext}}
\renewcommand\next{\mathrm{next}} 

\newcommand\oru{\mathrm{oru}}
\newcommand\car{\mathrm{car}}
\newcommand\Cat{\mathrm{Cat}}
\newcommand\cie{\mathrm{cie}}
\newcommand\mar{\mathrm{mar}}
\newcommand\zigzag{\mathrm{zigzag}}
\newcommand\PS{\mathrm{PS}}
\newcommand\PF{\mathrm{PF}}

\newcommand\supp{\mathrm{supp}}
\newcommand\terminal{\mathrm{terminal}}

\newcommand{\rank}{\mathrm{rank}}
\newcommand{\tail}{\mathrm{tail}}
\newcommand{\head}{\mathrm{head}}
\newcommand{\conv}{\mathrm{conv}}

\newcommand\des{\mathsf{des}}
\newcommand\asc{\mathsf{asc}}

\newcommand{\calC}{\mathcal{C}} 
\newcommand{\calD}{\mathcal{D}} 
\newcommand{\calF}{\mathcal{F}} 
\newcommand{\calG}{\mathcal{G}} 
 
\newcommand{\calL}{\mathcal{L}}
\newcommand{\calM}{\mathcal{M}}
\newcommand{\calN}{\mathcal{N}}

\newcommand{\calP}{\mathcal{P}} 
\newcommand{\calQ}{\mathcal{Q}} 
\newcommand{\calR}{\mathcal{R}} 

\newcommand{\calV}{\mathcal{V}} 
\newcommand{\calW}{\mathcal{W}}

\newcommand\ba{\mathbf{a}}

\newcommand\bd{\mathbf{d}}
\newcommand\be{\mathbf{e}}
\newcommand\bo{\mathbf{o}}
\newcommand\bi{\mathbf{i}}
\newcommand\bj{\mathbf{j}}

\newcommand\bs{\mathbf{s}}

\newcommand\bw{\mathbf{w}}

\newcommand\by{\mathbf{y}}
\newcommand\bz{\mathbf{z}}

\newcommand{\bbR}{\mathbb{R}}
\newcommand{\bbZ}{\mathbb{Z}}

\newcommand{\hatE}{\hat{E}}
\newcommand{\hatF}{\hat{F}}
\newcommand{\hatG}{\hat{G}}
\newcommand{\hatH}{\hat{H}}

\newcommand{\precdot}{\prec\mathrel{\mkern-5mu}\mathrel{\cdot}}

\newcommand{\subsetdot}{\subset\mathrel{\mkern-7mu}\mathrel{\cdot}}
\newcommand\meet{\wedge}
\newcommand\join{\vee}
\newcommand{\succSplit}{\mathrm{succSplit}}

\DeclareMathOperator\vol{\mathrm{vol}}

\renewcommand\hat{\widehat} 
\renewcommand{\phi}{\varphi}
\renewcommand{\emptyset}{\varnothing}

\usepackage[backend=bibtex,maxbibnames=99]{biblatex}
\title[Permutation Flows I (Research Announcement)]{Permutation Flows I: Triangulations of Flow Polytopes (Research Announcement)}

\date{\today}

\author[Gonz\'alez D'Le\'on]{Rafael S. Gonz\'alez D'Le\'on }
\address[R.\ S.\ Gonz\'alez D'Le\'on]{Department of Mathematics and Statistics\\Loyola University Chicago\\1032 W. Sheridan Rd., Chicago, IL 60660\\United States} 
\email{\href{mailto:rgonzalezdleon@luc.edu}{\texttt{rgonzalezdleon@luc.edu}}}
\urladdr{\url{https://dleon.combinatoria.co}}

\author[Hanusa]{Christopher R.\ H.\ Hanusa}
\address[C.\ R.\ H.\ Hanusa]{Department of Mathematics \\ Queens College (CUNY) \\ 65-30 Kissena Blvd. \\ Queens, NY 11367\\ United States}
\email{\href{mailto:chanusa@qc.cuny.edu}{\texttt{chanusa@qc.cuny.edu}}}
\urladdr{\url{https://hanusa.xyz/}}

\author[Yip]{Martha Yip}
\address[M.\ Yip]{Department of Mathematics, University of Kentucky, 719 Patterson Office Tower, Lexington, KY 40506-0027, United States}
\email{\href{mailto:mpyi222@uky.edu}{\texttt{mpyi222@uky.edu}}}
\urladdr{\url{https://www.ms.uky.edu/~myip/}}


\begin{document}

\begin{abstract}
We introduce a new broadly unifying family of combinatorial objects, which we call permutation flows, associated to an acyclic directed graph $G$ together with a framing~$F$. This new family is combinatorially rich and contains as special cases various families of combinatorial objects that are frequently studied in the literature, as is the case of classical permutations, circular permutations, multipermutations, Stirling permutations, alternating permutations, $k$-ary sequences, and Catalan objects and their generalizations. When permutation flows are decorated with compatible shuffles, they also include the combinatorics of parking functions and their generalizations. 

This model is geometrically rich. We show that permutation flow shuffles define a family of unimodular triangulations of the flow polytope $\mathcal{F}_G(\mathbf{a})$ on $G$ with an integer balanced netflow vector $\mathbf{a}$ where only the last entry is negative. As an application we provide a new proof of the Lidskii volume formula of  
Baldoni and Vergne for this family of polytopes and a reformulation of the same formula where every term is explained by the nature of the combinatorial objects involved. Permutation flow triangulations extend the Danilov, Karzanov, and Koshevoy triangulations that were defined for the case where $\mathbf{a}= \mathbf{e}_0- \mathbf{e}_n$. For this particular case we provide a formula for the $h^*$-polynomial of the flow polytope as the descent enumerating polynomial of permutation flows. Permutation flow triangulations include among others, a family of triangulations that refines the subdivisions of a family of polytopes studied by Pitman and Stanley. They also include as special cases triangulations of products of dilated simplices.

The model comes with an order structure induced by intuitive operators on permutation flows which we call the weak order. This order includes as special cases the weak order on permutations, the Tamari lattice, order ideals in Young's lattice, and their generalizations, among others. It was conjectured in 2020 by the three authors, together with Benedetti, Harris, and Morales, that this poset is in general a lattice. This conjecture has been recently established with independent proofs by Bell and Ceballos, and by Berggren and Serhiyenko. Further properties and applications of the weak order will appear in a subsequent article.
\end{abstract}


\maketitle

\renewcommand{\baselinestretch}{0.5}\normalsize
\tableofcontents
\renewcommand{\baselinestretch}{1.0}\normalsize

\section{Introduction}

This preliminary version of this article is presented as a research announcement to serve as supporting material for the research minicourse  ``Permutation Flows'' given by the third author during the CIRM Research School ``Beyond Permutahedra and Associahedra'' taking place December 1--5 at the Centre International de Rencontres Math\'ematiques CIRM, Marseille, France. As such, this version is currently incomplete and under construction; the final version will include additional examples, figures, and applications. 

This article is the first part of a series developing the theory of a new family of combinatorial objects that we call \defn{permutation flows}. This is a combinatorially rich family that contains as special cases various families of objects that are frequently studied in the literature, including the case of classical permutations, circular permutations, multipermutations, Stirling permutations, alternating permutations, $k$-ary sequences, and Catalan objects and their generalizations. When permutation flows are decorated with compatible \defn{shuffles}, they also include the combinatorics of parking functions and their generalizations.

Permutation flows are obtained as lifts of integer flows on an acyclic directed graph $G$ after enriching the graph with the additional information of total orderings on their sets of incoming and outgoing edges at each vertex, which is known as a \defn{framing} $F$ after the work of Danilov, Karzanov, and Koshevoy in \cite{DanilovKarzanovKoshevoy2012} on a family of regular unimodular triangulations of flow polytopes of with netflow $\ba=\be_0-\be_n$.  The family of permutation flows is defined then on a \defn{framed graph} $(G,F)$ adding structure to the set of integer flows that encodes the information on a triangulation. In particular, directing the edges of the dual graph of the triangulation in a coherent manner gives rise to the \defn{weak order} on permutation flows which generalizes known structures on the classical combinatorial objects as is the case of the weak order on permutations, the Tamari lattice on Catalan objects, lower principal ideals in Young's lattice, the Boolean algebra, and some of their generalizations.

\subsection{Combinatorics of flow polytopes}
\phantom{W}

Flow polytopes are a remarkable family of polytopes appearing as feasible sets in the minimization of cost functions on flows in networks.  In this context a natural interest is the study of their extreme flows, explored in particular in the works of Zangwill \cite{Zangwill1968}, Roy \cite{Roy1969}, and Florian, Rossin-Arthiat and deWerra \cite{FlorianRossin-ArthiatdeWerra1971}, and their adjacency relation, as studied by Gallo and Sodini in \cite{GalloSodini1978}. Further understanding of higher dimensional faces starts in the work of Hille \cite{Hille2003}, setting a connection with toric quiver varieties and quiver representations~\cite{AltmannHille1999}.

A renaissance of the study of these objects comes from the (unpublished) work of Postnikov and Stanley \cite{PostnikovStanley2000} and a striking conjectural formula for the volume of a polytope studied by Chan, Robbins and Yuen \cite{ChanRobbinsYuen2000} as the product $\prod_{i=0}^{n-2}\Cat(i)$ of consecutive Catalan numbers (see also \cite{Kirillov2001}). This conjecture was swiftly proven by Zeilberger \cite{Zeilberger1999} using an identity of Morris \cite{Morris1982}, but it is still an open problem to provide a combinatorial proof of this formula.

Given an (acyclic directed) graph $G$ on $n+1$ vertices and $m$ directed edges, the flow polytope $\calF_G(\ba)$ is defined as the set of real nonnegative flows satisfying flow conservation at every vertex and with a balanced (zero sum) netflow $\ba \in \bbR^{n+1}$ on its vertices. The Catalan formula has revealed the rich combinatorics underlying the calculations of volumes and lattice points of these polytopes. A technique to perform such computations is a beautiful formula given by Baldoni and Vergne \cite{BaldoniVergne2008} generalizing an older formula of Lidskii \cite{Lidskii1984} for the Kostant partition function of the root system of type $A_n$ (see also \cite{BaldoniDeLoeraVergne2004}). Their proof uses residue calculations, establishing a further connection with representation theory. It turns out that in the case when $\ba=\be_0-\be_n$ the Lidskii formula reduces to a simpler formula in which the volume is expressed exclusively by the enumeration of lattice points of a closely related flow polytope defined on the same graph. This special case appears earlier in~\cite{PostnikovStanley2000}.

Danilov, Karzanov, and Koshevoy \cite{DanilovKarzanovKoshevoy2012} construct a family of regular unimodular triangulations of flow polytopes where $\ba=\be_0-\be_n$ based on the concept of framing of the graph and of coherence of routes according to such framing. A unimodular triangulation translates the volume computation into the enumerative problem of counting maximal simplicies of the triangulation. M\'esz\'aros, Morales, and Striker \cite{MeszarosMoralesStriker2019} show that this triangulation then can be used to give a geometric proof of the case of the Lidskii formula where $\ba=\be_0-\be_n$. A geometric proof of the general Lidskii formula is given by M\'esz\'aros and Morales, using a subdivision strategy generalizing work of Postnikov and Stanley in \cite{MeszarosMorales2019}. The subdivision is based on their concept of \defn{bipartite noncrossing trees} introduced in \cite{MezsarosMorales2015} in the more general context of flow polytopes of signed graphs.

A combinatorial interpretation of the generalized Lidskii formula is given by the authors together with Benedetti, Harris, Khare, and Morales in \cite{BenedettiGonzalezDleonHanusaHarrisKhareMoralesYip2019}, in which \defn{gravity} and \defn{unified diagrams} are defined based on the combinatorics of decorated lattice paths and generalized parking functions to express the calculation of the Lidskii volume formula as an enumeration in a familiar combinatorial family. They compute the volume of the flow polytope of the \defn{caracol graph} $\car(n)$ in the cases when $\ba=\be_0-\be_n$ and $\ba=\sum_{i=0}^{n-1}(\be_i-\be_n)$. The former case is precisely the Catalan number $\Cat(n-2)$ and the latter is $\Cat(n-2)\cdot n^{n-2}$, a product of a Catalan number and a parking function number. A natural question that arises with such an interpretation is to further understand the connection between the combinatorial objects and the geometric objects involved.

A tool to answer this question is the dual graph of a triangulation of $\calF_G(\ba)$, which is the graph whose vertex set is the set of maximal simplices of the triangulation and includes an edge whenever two maximal simplices intersect on a facet. Additional structure in this dual graph has the potential to match similar structure already known for the combinatorial objects of interest. A particular structure is given by a natural orientation of the dual graph of the triangulation from the work in \cite{DanilovKarzanovKoshevoy2012} which induces a poset structure on maximal simplices. Computational examples led the authors together with Benedetti, Harris, and Morales in 2020 to conjecture that this poset structure, which we like to think of as the \defn{weak order} of $(G,F)$, always has the structure of a \defn{lattice}. Two vivid examples of the lattice conjecture correspond to two particular framings on the graph $\car(n)$ which generate  orders $W(\car(n),F)$ isomorphic to the Tamari lattice and to the lower ideal generated by the staircase partition in Young's lattice, both lattices on Catalan objects. These two lattices were also observed in relation to polytopes in the work of Pitman and Stanley in their study of subdivisions of a polytope related to parking functions \cite{PitmanStanley2002}, however in the Pitman-Stanley polytope these subdivisions are not triangulations. 

The first and third author together with Bell and Mayorga Cetina show in \cite{BellGonzalezMayorgaYip2023} that more generally, for the $\nu$-caracol graph $\car(\nu)$ for a weak composition $\nu$, the similar two framing strategies provide realizations of two well-studied lattices in the literature: the $\nu$-Tamari lattice introduced by
Pr\'eville-Ratelle and Viennot \cite{Preville-RatelleViennot2017}, and further developed by Ceballos, Padrol, and Sarmiento~\cite{CeballosPadrolSarmiento2019, CeballosPadrolSarmiento2020}, and the lower ideal $I(\nu)$ in Young's lattice determined by co-area when $\nu$ is considered as a lattice path. The work in \cite{BellGonzalezMayorgaYip2023} revealed the unifying character of triangulations of flow polytopes with respect to different order structures on combinatorial objects that are relevant  in the literature.  

The first and third authors together with Morales, Philippe, and Tamayo Jim\'enez provide in \cite{GonzalezDleonMoralesPhilippeTamayoYip2025} another relevant example of this phenomenon by building on the techniques developed in \cite{BellGonzalezMayorgaYip2023} to answer, in the case of a strict composition $\bs$, a conjecture of Ceballos and Pons \cite{CeballosPons2024-1,CeballosPons2024-2} on the geometric realization of their $\bs$-permutahedron, whose $1$-skeleton is given by their $\bs$-weak order on $\bs$-trees (or equivalently, on $\bs$-Stirling permutations). The lattice structure in \cite{GonzalezDleonMoralesPhilippeTamayoYip2025} is based on a framing of the $\bs$-\defn{oruga graph} $\oru(\bs)$ and in the particular case when $\bs=\sum_{i=0}^{n-1}(\be_i-\be_n)$, the polytope $\calF_{\oru(n)}$ is the $n$-hypercube; the triangulation corresponding to the framing given by the authors naturally gives the (right) weak order on the set $\sym_n$ of permutations of $[n]:=\{1,2,\dots, n\}$. In the same spirit of \cite{BellGonzalezMayorgaYip2023}, we show in Part II \cite{GonzalezHanusaMoralesYip2026} that a fully twisted framing on $\oru(n)$ realizes the lattice order on circular permutations introduced and studied by Abram, Chapelier-Laget, and Reutenauer \cite{AbramChapelier-LagetReutenauer2021}. (See Figures \ref{fig:222poset} and \ref{fig:222final_summaries}.)

The work of the first and third author together with Morales, Philippe, and Tamayo Jim\'enez shows that the permutree lattices and the permutree lattice quotients of Pilaud and Pons \cite{PilaudPons2018} are also examples of $W(G,F)$ after consecutive applications, starting on the planar framing of $\oru(n)$, of the \defn{M-move} operation on framed graphs found by the third author that replaces an edge $(i,j)$ by the pair of edges $(0,j)$ and $(i,n)$ while respecting the framing orders at $i$ and $j$. It will be shown in Part II \cite{GonzalezHanusaMoralesYip2026} that M-moves always induce lattice quotients of $W(G,F)$. 
A special case of this result appears in the dissertation of Tamayo Jim\'enez~\cite[Chapter 7]{TamayoJimenez2023}. Permutree lattices contain as particular examples the weak order on permutations, the Tamari lattice on binary trees, the boolean algebra on binary sequences, and Reading's Cambrian lattices \cite{Reading2006}.

Two proofs giving a positive answer to the lattice conjecture have been recently independently obtained by Bell and Ceballos~\cite{BellCeballos2024} and by Berggren and Serhiyenko~\cite{BerggrenSerhiyenko2024}. Both proofs use vastly different techniques and provide new surprising connections that support further the unifying character of the combinatorics of flow polytopes. As a consequence of the work announced by the authors in \cite{BellCeballos2024}, $W(G,F)$ is an HH-lattice, which in particular is congruence uniform and therefore semidistributive (see \cite[Theorem 3.2]{BellCeballos2024}). Their work connects triangulations of flow polytopes to further examples of lattices: the alt $\nu$-Tamari lattices introduced by Ceballos and Chenevi\`ere \cite{CeballosCheneviere2023}, the cross-Tamari lattices introduced by the same authors in \cite{BellCeballos2024}, the Grassmann-Tamari lattices of Santos, Stump, and Welker \cite{SantosStumpWelker2017}, McConville's grid Tamari lattices \cite{McConville2017}, and Pilaud's $(\varepsilon,I,J)$-Cambrian lattices, and others. On the other hand, the work in \cite{BerggrenSerhiyenko2024} builds on a connection established by work of the third author together with Bell, Braun, Bruegge, Hanely, Peterson, and  Serhiyenko in~\cite{BellBraunBrueggeHanelyPetersonSerhiyenkoYip2024} between DKK triangulations of graphs with ample framings and $\tau$-tilting posets of gentle algebras. The work in \cite{BerggrenSerhiyenko2024} further develops a theory that connects the DKK triangulations of flow polytopes to the representation theory of special families of gentle algebras whose $\tau$-tilting posets are lattices. Additional properties of $W(G,F)$ will be presented in Part II of this work \cite{GonzalezHanusaMoralesYip2026} including an explicit description of the join and meet of $W(G,F)$ in terms of permutation flows.

Further work of various other authors reveals deep connections between flow polytopes and other interesting areas of mathematics, such as diagonal harmonics \cite{MeszarosMoralesRhoades2017,LiuMeszarosMorales2019}, Schubert polynomials \cite{MeszarosStDizier2020}, Gelfand-Tsetlin polytopes \cite{LiuMeszarosStDizier2019}, among others \cite{Meszaros2015,MeszarosSimpsonWellner2019, BenedettiHanusaHarrisMoralesSimpson2020,KapoorMeszarosSetiabrata2021,MoralesShi2021,CorteelKimMeszaros2021,DuganHegartyMoralesRaymond2025}.

\subsection{Permutation flow triangulations}
\phantom{W}

In this work we develop a combinatorial treatise encoding the information stored on  maximal simplices, in a triangulation of a flow polytope that is based on a framing, in terms of various families of combinatorial objects, all of them in bijective correspondence, each of them showcasing a different aspect of the problem, and, to our judgment, all of them interesting for their own sake. These, in increasing order of encoding, are the families of \defn{cliques}, \defn{vines}, \defn{groves}, 
and \defn{permutation flows} in the simpler case when $\ba=\be_0-\be_n$. In the more general case, for a netflow vector $\ba$ that has its only negative entry last, the families are \defn{cliques}, \defn{multicliques}, \defn{vineyard shuffles}, \defn{grove shuffles}, 
and \defn{permutation flow shuffles}, whose objects involve multiple components (specifically $|a_n|$ of them) and hence carry the additional information of a shuffle that organizes all these components together in a globally coherent manner. 

The model of permutation flows forms the basis for the definition of a family of unimodular triangulations of the flow polytope $\calF_G(\ba)$ in the case where $\ba=(a_0,a_1,\dots, a_n)$ satisfies $a_v\ge 0$ for $v<n$, which we call the \defn{permutation flow triangulation}. 

An \defn{($\ba$-)augmented graph} $\hatG$ is an acyclic directed graph with half-edges such that the set~$X$ of input half-edges satisfies both that $|X|=|a_n|$ and that there are $a_v$ input half-edges at vertex $v$ for every $v<n$. A \defn{framed augmented graph} is a pair $(\hatG,\hatF)$ where $\hatG$ is an augmented graph and $\hatF$ is a framing (see Section~\ref{sec:augmented_graphs}) of $\hatG$ defined at every vertex\footnote[2]{This is a slight departure from the usual definition of Danilov, Karzanov and Koshevoy, where only the internal vertices $\{1,\dots, n -1\}$ are required to have a framing.} and including half-edges. We represent the lattice points of $\calF_G(\ba)$ as \defn{route matchings} on $\calG$---a coherent set of routes in $\calG$, one for each input half-edge (see Definition~\ref{def:route_matchings}). A \defn{clique} is a set of route matchings which is coherent (locally and globally according to Definition~\ref{def:cliques}). We say that a clique is \defn{saturated} if it is maximal with respect to inclusion of route matchings. Denote $\Delta_{\calC}$ the convex hull of the lattice points indexed by the route matchings on $\calC$. We prove the following theorem in Section \ref{sec:triangulation_flow_polytopes}.

\begin{theorem}\label{theorem:triangulation}
	Let $(\hatG,\hatF)$ be a framed $\ba$-augmented graph of $G$. The set
    $$\{\triangle_{\calC}\mid \calC \in \SatCliques(\hatG,\hatF)\}$$
    is a unimodular triangulation of $\calF_G(\ba)$.
\end{theorem} 

The triangulation of Theorem \ref{theorem:triangulation} generalizes and coincides with the one of Danilov, Karzanov, and Koshevoy in \cite{DanilovKarzanovKoshevoy2012} in the case where $\ba=\be_0-\be_n$, which we denote from now on as $\DKK(G,F)$. The proof of Theorem \ref{theorem:triangulation} provides an alternative proof to the fact that $\DKK(G,F)$ is a unimodular triangulation of $\calF_G$. 

\subsection{Relation between the combinatorial families in this article}
\phantom{W}

The family that most directly encodes a permutation flow triangulation is then the family of cliques. This particular family has the advantage that naturally has the structure of a simplicial complex and the face structure of this complex can be seen directly from cliques by set inclusion. In Section \ref{sec:combinatorial_families_on_augmented_graphs} we translate the facial structure across the bijective correspondences among combinatorial families, providing the bijections with the additional feature of being poset isomorphisms and giving a way to check within the same family whether two objects are related by face membership (see an example in Section \ref{sec:face_poset}).

Even though the family of cliques directly carry the structure of the triangulation, the family that we highlight the most in this work is the family of \defn{permutation flow shuffles} (and permutation flows in the case when $\ba=\be_0-\be_n$). As we show in Sections \ref{sec:permutationflows} and \ref{sec:combinatorial_families_on_augmented_graphs}, permutation flows are easily constructed by a simple combinatorial procedure on a framed augmented graph and a compatible shuffle can be assigned to this object in a similar manner in which a classical decorated Dyck path (sometimes called a shuffle or a parking function) is related to its underlying Dyck path. In Section \ref{sec:permutationflows} we provide a short discussion on the combinatorics of permutation flows (without the shuffles) to give space for the reader to get acquainted with the basic definitions and to appreciate the relationships with classical combinatorial families. Most of the proofs and details are delayed until Section \ref{sec:combinatorial_families_on_augmented_graphs} where the more general families are defined. In particular, of high relevance is the concept of a \defn{split} which becomes more evident and relevant for the more general framed augmented graphs and where the compatible shuffles are applied. We show in Section \ref{sec:hstarpoly} that the \defn{weak order} $W(G,F)$ can be defined by applying simple \defn{raising} and \defn{lowering operators} on permutation flows, and hence all permutation flows can be constructed recursively starting from the bottom permutation flow $\pi^0$ or the top permutation flow $\pi^1$. Further properties of the weak order on permutation flows, including results on inversions, will be presented in Part II \cite{GonzalezHanusaMoralesYip2026}.

Cliques and \defn{multicliques} are closely related by switching the point of view from indexing route matchings to indexing input half-edges. From this perspective every input half-edge is associated to an equal number of routes after taking multiplicity into account. These two families coincide in the case when $\ba=\be_0-\be_n$. In the work of M\'eszaros, Morales, and Striker~\cite{MeszarosMoralesStriker2019}, a bijection between $\SatCliques(G,F)$ and the lattice points of a closely related polytope~$\calF_G^{\bbZ}(\bd)$ is given by using the concept of a \defn{prefix} (see Lemma \ref{lemma:lemma_MMS}). Moving from the perspective of routes to a perspective of prefixes we encounter the concepts of \defn{vines} and \defn{vineyard shuffles}, where the concept of split arises. The multiplicity of routes in a multiclique can be naturally encoded by a shuffle and such shuffles can be conveniently associated to the set of splits of a vineyard. Vineyards act like genealogical trees by telling the story of the prefixes as flow moves through the vertices of the graph from the sources to the sink.

To construct an example of a valid simplex with the definitions of cliques, multicliques or vineyard shuffles, requires a significant amount of work checking that the coherence of all the pieces is satisfied. The family of \defn{groves} was introduced in the work of the second author and Brunner in \cite{BrunnerHanusa2024} and is based on the sequences of bipartite noncrossing trees developed by M\'eszaros and Morales in \cite{MeszarosMorales2019}. The family of \defn{grove shuffles} stores the information on vineyard shuffles at each vertex, encoding the coherence of routes in the noncrossing property of bipartite noncrossing forests. This makes these objects both, more natural to construct a valid example of a simplex and more suitable to be used in the proofs of theorems that require to check that coherence is satisfied. Most of the arguments in this article are given in terms of grove shuffles for this reason. 

An intermediate step between grove shuffles and permutation flows is provided by a natural labeling of the grove using a subset of the edges of the graph $G$ according to the framing $\hatF$. The fact that the dimension of $\calF_G(\ba)$ is equal to $d:=m-n$ is reflected on a  bipartition of the set of edges of $E=T\sqcup S$, which we call respectively \defn{splits and slacks}. Splits of $G$ are then used to label prefixes, routes, and splits on vineyard and grove shuffles. The encoding of a grove with its natural labeling is not only more efficient in terms of the information that we need to keep in the object, but also directly translates it into the concept of a permutation flow. Thus a permutation flow is the most optimal member of the family to construct and for recording the information of its associated simplex in the triangulation.

\subsection{A new (proof of the) Lidskii volume formula}
\phantom{W}

The model of permutation flows is versatile and elegant in its simplicity. We provide in Section \ref{sec:Lidskii} a new geometric proof of the Lidskii volume formula in \cite{BaldoniVergne2008} using the permutation flow triangulation of $\calF_G(\ba)$. This new proof gives also a direct interpretation of the formula which mirrors the purely combinatorial work of the authors in \cite{BenedettiGonzalezDleonHanusaHarrisKhareMoralesYip2019}. Hence, we obtain a bijection between gravity diagrams and permutation flows and a bijection between unified diagrams and permutation flow shuffles providing the original objects in \cite{BenedettiGonzalezDleonHanusaHarrisKhareMoralesYip2019} with a geometric connection to $\calF_G(\ba)$. This is a significant advance with respect to the interpretation presented in~\cite{BenedettiGonzalezDleonHanusaHarrisKhareMoralesYip2019}.

Moreover, our results give a reformulation of the Lidskii formula in which every term is explained by a unique object in the sum, in some sense, simplifying and illuminating the formula in terms of the combinatorial objects.

Recall that $X$ is the set of input half-edges of the $\ba$-augmented graph $\hatG$ with $|X|=|a_n|$ and $a_v$ input half-edges at every $v<n$. We denote $v_x$ the unique vertex of the input half-edge $x\in X$. A \defn{weak composition} $\bj$ of $d$ with parts indexed by $X$  is a function $\bj:X\rightarrow \bbZ_{\ge 0}$ with $d=\sum_{x\in X}\bj(x)$. Denote by $\hat\bj$ the weak composition of $d$ with $n$ parts that is the coarsening of $\bj$ given by $$\hat \bj_v = \sum_{\substack{x\in X\\v_x=v}} \bj(x). $$

Define $\bo=(o_0,\ldots, o_{n-1}, 0)$ to be the weak composition $o_v = \outdeg(v)-1$ for each $v$. Denote by $\hat \bj \rhd \bo$ the domination order $\sum_{v=1}^l \hat \bj_v \geq \sum_{v=1}^l o_v$ for all $l$. Denote by $K_G$ the Kostant partition function with respect to $G$, where $K_G(\cdot)$ is the number of lattice points of $\calF_G(\cdot)$.

We prove the following theorem.

\begin{theorem}[Reformulated generalized Lidskii formula]\label{theorem:reformulated_lidskii_volume}
    Let $\ba$ be such that $a_i\ge 0$ for $i=0,\dots, n-1$ and $a_n<0$.
The normalized volume of the flow polytope $\calF_G(\ba)$ is
\[
\vol \calF_G(\ba) = \sum_{\bj} \binom{d}{\bj} K_G(\hat \bj-\bo),
\]
where the sum is over weak compositions $\bj$ whose coarsening $\hat\bj$ satisfies $\hat \bj \rhd \bo$.
\end{theorem}

Saturated permutation flows are also in bijection with a subset of permutations of $X \cup [d]$ by a map that we call the \defn{final summary}. We prove the following theorem in Section~\ref{sec:permutation_flow_shuffles}.

\begin{theorem}
\label{thm:finalsummary}
    $\FinalSummaries(\hatG,\hatF)$ is in bijection with $\SatPermutationFlows(\hatG,\hatF)$.
\end{theorem}

When Theorem \ref{thm:finalsummary} is combined with the information of a shuffle of a permutation flow shuffle we obtain a new family of combinatorial objects indexing the simplices of the permutation flow triangulation of $\calF_G(\ba)$. 

Theorem \ref{thm:finalsummary} also connects with the work of the three authors and Morales in \cite{GonzalezHanusaMoralesYip2023} when~$G$ is a \defn{spinal graph} (a graph with a Hamiltonian path). Two families of $G$-cyclic orders (circular permutations) are given in bijection with the family of integer flows in $\calF_G(\bd)$: the families $\mathcal{A}^\uparrow_G$ and $\mathcal{A}^\downarrow_G$ of \defn{upper} and \defn{lower $G$-cyclic orders}.

\begin{question} \label{question:upper_and_lower_g_cyclic_orders}
    Are there framings $F^{\uparrow}$ and $F^{\downarrow}$ of $G$ such that $\mathcal{A}^\uparrow_G= \FinalSummaries(G,F^{\uparrow})$ and $\mathcal{A}^\downarrow_G=\FinalSummaries(G,F^{\downarrow})$? 
\end{question}

\subsection{\texorpdfstring{Shellability, the $h^*$-polynomial, and the $G$-Eulerian polynomials}{Shellability, the h*-polynomial, and the G-Eulerian polynomials}}
\phantom{W}

In Section \ref{sec:hstarpoly} we derive some properties of the weak order $W(G,F)$ to show the following. 

\begin{theorem}\label{theorem:linear_extensions_W(G,F)_are_shellings}
    Any linear extension of $W(G,F)$ is a shelling of $\DKK(G,F)$.
\end{theorem}

We provide in Section \ref{sec:hstarpoly} notions of \defn{descent} and \defn{ascent} on permutation flows which give rise to the definition of the \defn{$(G,F)$-Eulerian polynomial} $A_{(G,F)}(x)$ as the descent enumerator on $\SatPermutationFlows(G,F)$.

Using the fact that $\DKK(G,F)$ is a unimodular triangulation together with Theorem \ref{theorem:linear_extensions_W(G,F)_are_shellings} we prove in Theorem \ref{theorem:h_star_first} that the $h^*$-polynomial of $\calF_G$ is given by $A_{(G,F)}(t)$. Moreover, since the $h^*$-polynomial depends only on $\calF_G$ and not on the additional information of a triangulation we conclude that $A_{(G,F)}(t)$ is independent of the framing $F$, which we then denote as $A_{G}(t)$ and call it the \defn{$G$-Eulerian polynomial}. We have the following theorem.

\begin{theorem}\label{theorem:h_star_second}
    The $h^*$-polynomial of $\calF_G$ is $A_{G}(t)$.
\end{theorem}

Theorem \ref{theorem:h_star_second} and an answer to Question \ref{question:upper_and_lower_g_cyclic_orders} could answer the following related conjecture. 
\begin{conjecture}[Conjecture 7.6\cite{GonzalezHanusaMoralesYip2023}]
Given a spinal graph $G$, we have that 
\begin{equation*} 
P_{\mathcal{A}^\downarrow_G,\des}(t)\lhd h_{\calF_G}^*(t)\lhd 
P_{\mathcal{A}^{\uparrow}_G,\des}(t),
\end{equation*}
where $P_{\mathcal{A},\des}(t) =\sum_{\gamma\in \mathcal{A}}t^{\des(\pi(\gamma))}$ and $\lhd$ indicates the dominance order of the coefficients.
\end{conjecture}

\subsection{Organization of the article}
\phantom{W}

The remainder of this article is arranged as follows. In Section~\ref{sec:flow_polytopes_intro}, we present previous results about flow polytopes that involve the face descriptions, the Lidskii volume formulas, DKK triangulations, and the weak order. Section~\ref{sec:permutationflows} is an overview of the combinatorics of permutation flows and related objects in the case where $\ba=(1,0,\ldots,0,-1)$. We include a discussion about the face structure of the triangulation in terms of these objects and examples.  In Section~\ref{sec:hstarpoly} we derive key properties of the weak order to conclude that any linear extension gives a shelling of the DKK triangulation and establish that the $h^*$-polynomial is the $G$-Eulerian polynomial. In Section~\ref{sec:combinatorial_families_on_augmented_graphs}, we derive the main combinatorial results for the families of objects that support permutation flows when the netflow vector only has its only negative entry last. For each family we present a criterion for face membership within that family. Section~\ref{sec:triangulation_flow_polytopes} provides a proof of Theorem~\ref{theorem:triangulation} and an example of its application. Finally, in Section~\ref{sec:Lidskii}, we provide a new geometric proof of the Lidskii volume formula using the triangulation of Section~\ref{sec:triangulation_flow_polytopes} together with the combinatorial families developed in Section~\ref{sec:combinatorial_families_on_augmented_graphs}. As a corollary, we obtain formulas to count permutation flows and permutation flow shuffles.

\newpage
\section{Flow Polytopes: volumes, faces,  triangulations, and the weak order}
\label{sec:flow_polytopes_intro}

For a pair of integers $a<n$ we will denote \[[a,n]:=\{a,a+1,\dots, n\} \textup{ and }[n]:=\{1,\dots, n\}.\] Let $G=(V,E)$ be an \defn{(acyclic) directed graph} (which we will simply refer as a graph in this work) on the vertex set $V=[0,n]$ with unique \defn{sink} vertex $n$.
$G$ may have multiple edges between the same pair of vertices. We denote the number of edges by $m:=|E|$. 

For an edge $e=(v,w) \in E$ we define $\tail(e)=v$ and $\head(e)=w$. Often we will write $(v,\cdot)$ or $(\cdot,w)$ to indicate an arbitrary edge $e$ with $\tail(e)=v$ or $\head(e)=w$ respectively. Without loss of generality, we may assume that the vertices are labeled so that every edge $e=(v,w)$ satisfies $v<w$. For every vertex $v$, denote by $\inedge(v)=\{e \in E \mid \head(e)=v\}$ the \defn{set of incoming edges at $v$} and $\outedge(v)=\{e \in E \mid \tail(e)=v\}$ the \defn{set of outgoing edges at~$v$}. 
We also define $\indeg(v)=|\inedge(v)|$ and $\outdeg(v)=|\outedge(v)|$, which respectively denote the in-degree and out-degree at vertex $v$. 

Given a \defn{balanced} vector $\ba=(a_0,a_1,\ldots, a_n)\in \bbR^{n+1}$ (in which $\sum_v a_v=0$), an \defn{$\ba$-flow} on~$G$ is a tuple $(\phi(e))_{e\in E}\in \bbR_{\geq0}^E$ such that for all $v\in[0,n]$, it satisfies the \defn{conservation of flow} condition
\begin{equation}\label{eqn:defining_equations}
   a_v + \sum_{e\in \inedge(v)} \phi(e) = \sum_{e\in \outedge(v)} \phi(e). 
\end{equation}

The definition in terms of the non-negativity of its values and the equations \eqref{eqn:defining_equations} implies that the set $\calF_G(\ba)$ of all $\ba$-flows on $G$ is a convex polytope, which is called the \defn{flow polytope of $G$ with netflow $\ba$}. 
An $\ba$-flow is said to be an \defn{integer $\ba$-flow} if $\phi(e)\in \bbZ_{\geq0}$ for all $e\in E$, and we denote the set of integer $\ba$-flows by $\calF_G^{\bbZ}(\ba)$. 
Its cardinality $K_G(\ba):=|\calF_G^{\bbZ}(\ba)|$ is called the \defn{Kostant partition function} of $G$ with netflow $\ba$. It is well-known that the set of vertices of $\calF_G(\ba)$ is a subset of $\calF_G^{\bbZ}(\ba)$ and hence $\calF_G(\ba)$ is a lattice polytope. For simplicity, when $\ba= \be_0- \be_n=(1,0,\ldots, 0,-1)$, we will omit $\ba$ from the notation and write $\calF_G$.  The \defn{dimension} of $\calF_G(\ba)$ is $d:= \dim \calF_G = m-n$ for all $\ba$.

As an example, consider the \defn{Pitman-Stanley graph} $\PS_n$ on vertex set $V=[0,3]$ and edge set $E$ with edges of the form $(i-1,i)$ and $(i-1,n)$ for $i\in E$. This graph was defined in \cite{BaldoniVergne2008} where the authors show that the flow polytope $\calF_{\PS_n}(\ba)$ is integrally equivalent to the polytope $\Pi(\ba)$ defined by Pitman and Stanley in \cite{PitmanStanley2002}. In Figure \ref{fig:PS_polytope_with_integer_flows_versionM} we illustrate the flow polytope $\calF_{\PS_3}(1,1,1,-3)$ and highlight all the integer flows in $\calF^{\bbZ}_{\PS_3}(1,1,1,-3)$.

\begin{figure}[ht!]
    \centering
    \scalebox{0.7}{    \input{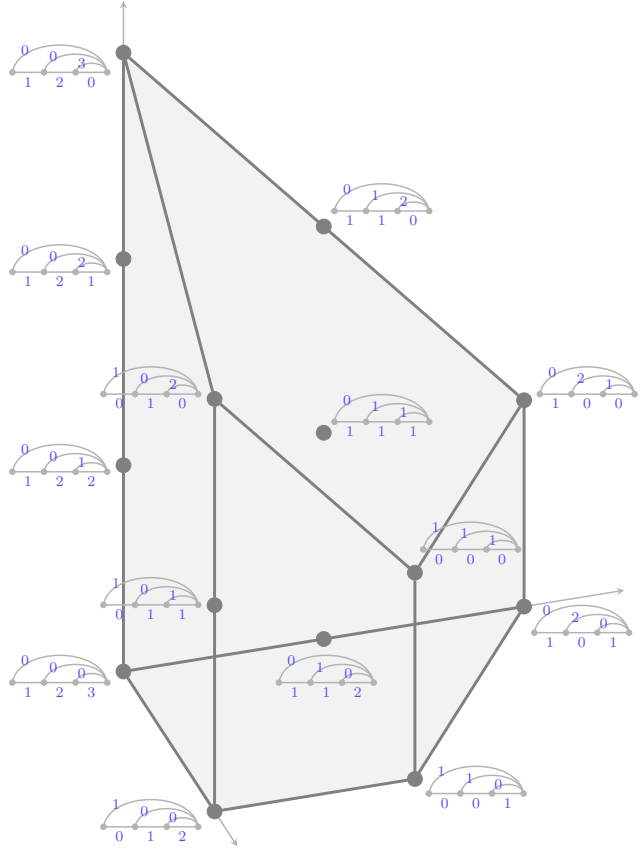}}
    \caption{The Pitman-Stanley polytope $\calF_{\PS_3}(1,1,1,-3)$ and its integer flows.}
    \label{fig:PS_polytope_with_integer_flows_versionM}
\end{figure}


\subsection{The Lidskii volume formulas} 
\label{sec:Lidskii_intro}
\phantom{W}

The \defn{normalized volume} of a $d$-dimensional lattice polytope is $d!$ times its Euclidean volume, which is an integer for any lattice polytope.

The following volume formula, which generalizes the formula of Lidskii \cite{Lidskii1984} for the case of the  \defn{complete graph}  $G=K_{n}$ on $n+1$ vertices, was proven by Baldoni and Vergne in \cite{BaldoniVergne2008}.  
A combinatorial proof using polytopal subdivisions was given by M\'esz\'aros and Morales~\cite{MeszarosMorales2019}.

\begin{theorem}[Generalized Lidskii formula \cite{BaldoniVergne2008,MeszarosMorales2019}]\label{thm.genlidskii}Let $\ba$ be a balanced netflow vector such that $a_i\ge 0$ for $i=0,\dots, n-1$ and $a_n\le 0$. The normalized volume of the flow polytope $\calF_G(\ba)$ is
\[
\vol \calF_G(\ba) = \sum_{\bj\rhd\bo} \binom{d}{\bj} \ba^{\bj} K_G(\bj-\bo),
\]
where $\bo=(o_0,\ldots, o_{n-1}, 0)$ with $o_v = \outdeg(v)-1$ for each $v$, for $\bj = (j_0,\ldots, j_n)\in \bbZ_{\geq0}^{n+1}$, $\bj \rhd \bo$ denotes the \defn{dominance order} defined by $\sum_{v=1}^k j_v \geq \sum_{v=1}^k o_v$ for each $k$, $\ba^{\bj} := a_0^{j_0} \cdots a_{n-1}^{j_{n-1}}$, and $K_G$ is the Kostant partition function with respect to $G$.
\end{theorem}

In the special case when $\ba=\be_0- \be_n$, it was already shown by Postnikov and Stanley \cite{PostnikovStanley2000} that the flow polytope $\calF_G$ has normalized volume
\begin{equation}\label{eqn:lidski_outdegree_formula}
\vol \calF_G=K_G(\tilde \bd),
\end{equation}
where $\tilde\bd = d\be_0 -\bo \in \bbZ^{n+1}$. A similar formula in terms of in-degrees exists if we define $\bi=(0,i_1,\ldots,i_n)$ where $i_v=\indeg(v)-1$ and reverse all the edges of the graph $G$ as
\begin{equation}\label{eqn:lidski_indegree_formula}
\vol \calF_G=K_G(\bd),
\end{equation}
where $\bd = \bi-d\be_n$. 

That is, the normalized volume of $\calF_G$ can be computed either by the number of integer $\tilde \bd$-flows or the number of integer $\bd$-flows on $G$. A computation of volume is then translated into a combinatorial enumeration of integer flows.  Proofs of Equations \eqref{eqn:lidski_outdegree_formula} and \eqref{eqn:lidski_indegree_formula} are obtained as a consequence of the work in \cite{MeszarosMorales2019} and \cite{MeszarosMoralesStriker2019} based on a technique envisioned first in \cite{PostnikovStanley2000}. 

\begin{remark}
    Lidskii formulas to determine the number of integer points $K_G(\ba)$ in $\calF_G^{\bbZ}(\ba)$ are also given and studied in \cite{BaldoniVergne2008, BaldoniDeLoeraVergne2004, MeszarosMorales2019,KapoorMeszarosSetiabrata2021}.
\end{remark}

\subsection{Faces of flow polytopes}
\label{sec:faces_of_flow_polytopes}
\phantom{W}

Faces of flow polytopes have been studied by various authors. Notably, Zangwill~\cite{Zangwill1968} studied extreme points in the context of flow optimization on uncapacitated networks, Roy~\cite{Roy1969} and Florian, Rossin-Arthiat and deWerra~\cite{FlorianRossin-ArthiatdeWerra1971} did it in the capacitated case. Gallo and Sodini~\cite{GalloSodini1978} described the adjacency relations among these extreme points in terms of minimal cycles, which form the $1$-skeleton of the flow polytope. Hille~\cite{Hille2003} described  all the faces of $\calF_G(\ba)$ for $\ba$ in generic position. The combinatorics of faces of particular families of flow polytopes have also been studied including generalized Pitman-Stanley polytopes~\cite{DuganHegartyMoralesRaymond2025} and the complete graph~\cite{Dugan2024}.

We will say that a subgraph $H$ of a connected graph $G$ is \defn{$\ba$-effective} if it is connected and it is $\ba$-flow supporting---that is, if $\calF_H(\ba)$ is not empty.

\begin{theorem}[\cite{Hille2003} Theorem 3.2]\label{Theorem:Hille}
    The faces of $\calF_G(\ba)$ are of the form $\calF_H(\ba)$ where $H$ is an $\ba$-effective subgraph of $G$.
\end{theorem}

By Theorem \ref{Theorem:Hille}, the face lattice of $\calF_H(\ba)$ is isomorphic to the lattice of $\ba$-effective subgraphs of $G$ where the relation is induced from the subgraph relation.

An important special case of Theorem \ref{Theorem:Hille} (see also \cite{Zangwill1968}) in our context is the characterization of vertices in the flow polytope. 

\begin{corollary}[\cite{Zangwill1968, Hille2003}]%
\label{corollary:vertices_flow_polytope}
The vertices of $\calF_G(\ba)$ are $\ba$-effective trees on $G$. In the case when $\ba=\be_0- \be_n$ the vertices are indexed by paths from the unique source to the unique sink.    
\end{corollary}

The reader can verify that in Figure \ref{fig:PS_polytope_with_integer_flows_versionM}, the subgraphs supporting the flow of the vertices of $\PS_3(1,1,1,-3)$ are precisely the $2^{3}=8$ spanning trees of $\PS_3$.

\subsection{Triangulations of flow polytopes}
\label{sec:triangulations_intro}
\phantom{W}

Danilov, Karzanov, and Koshevoy~\cite{DanilovKarzanovKoshevoy2012} introduce a family of triangulations of the flow polytope $\calF_G$ based on the concept of a framing of a graph and the coherence of routes on a framed graph.

\subsubsection{Paths, prefixes, suffixes, and routes}
\label{subsec.paths0}
\phantom{W}

A \defn{path} in $G$ is a sequence of edges $P=(e_1,\ldots,e_j)$ such that $\head(e_i)=\tail(e_{i+1})$ for $1\leq i\leq j-1$; we denote $|P|=j$. We say that $v$ is a vertex in a path $P$ if either $v=\tail(e_i)$ or $v=\head(e_i)$ for some $i\in [j]$. A \defn{route} is a path that cannot be extended by adding an additional, either initial or final, edge. If $v$ is a vertex in a path $P$ we denote by $Pv$ and $vP$ the subpaths of $P$ before $v$ and after $v$ respectively.  If $v$ and $w$ are vertices in a path $P$ with $v<w$ we denote by $vPw$ the subpath of $P$ which starts at $v$ and ends at $w$. If $P$ and $Q$ are paths that share a vertex $v$, we denote by $PvQ$ the path that concatenates the subpaths $Pv$ and $vQ$ into a new path. We denote by $x_0$ (and $y_0$, respectively) the path of length 0 starting and ending at vertex $0$ (or at vertex $n$, respectively). We call a path that starts at~$0$ a \defn{prefix} and a path that ends at $n$ a \defn{suffix}. We denote by $\Prefixes(G)$, $\Suffixes(G)$, and $\Routes(G)$ respectively the set of prefixes, suffixes, and routes of $G$ and by $\Prefixes(v)$ and $\Suffixes(v)$ those prefixes and suffixes that end (or start) at $v$. 

\subsubsection{Framings, coherence, and the DKK triangulation}
\label{subsec.framings0}
\phantom{W}

Danilov, Karzanov, and Koshevoy~\cite{DanilovKarzanovKoshevoy2012} introduce the concept of a framing of $G$. A \defn{framing of $G$ at vertex $v$} is a choice of linear orders $\preceq_{\inedge(v)}$ and $\preceq_{\outedge(v)}$ on the sets $\inedge(v)$ and $\outedge(v)$. In this article, we always assume that $\inedge(v)$ and $\outedge(v)$ are ordered sets, ordered via $\preceq_{\inedge(v)}$ and $\preceq_{\outedge(v)}$ respectively. A \defn{framing of~$G$} is the collection $F$ of framings at every vertex of~$G$.\footnote[2]{This is a slight departure from the usual definition of Danilov, Karzanov and Koshevoy, where only the internal vertices $\{1,\dots, n -1\}$ are required to have a framing.}  A \defn{framed graph} is a graph~$G$ with a framing $F$ and is denoted $(G,F)$. A subgraph $H\subseteq G$ inherits a framing $F_H$ from $F$ which yields the framed graph $(H,F_H)$. In this article, all illustrations of framed graphs are drawn so that the framing at each vertex, both for incoming edges and outgoing edges, is indicated from bottom to top. See Figure~\ref{fig:framed_graph}.

\begin{figure}[bth!]
\begin{tikzpicture}
\begin{scope}[xshift=0, yshift=0, scale=1.5]
    \vertex[fill=orange, minimum size=4pt, label=below:{\tiny\textcolor{orange}{$0$}}](v0) at (0,0) {};
	\vertex[fill=orange, minimum size=4pt, label=below:{\tiny\textcolor{orange}{$1$}}](v1) at (1,0) {};
	\vertex[fill=orange, minimum size=4pt, label=below:{\tiny\textcolor{orange}{$2$}}](v2) at (2,0) {};
	\vertex[fill=orange, minimum size=4pt, label=below:{\tiny\textcolor{orange}{$3$}}](v3) at (3,0) {};
	\vertex[fill=orange, minimum size=4pt, label=below:{\tiny\textcolor{orange}{$4$}}](v4) at (4,0) {};
	\vertex[fill=orange, minimum size=4pt, label=below:{\tiny\textcolor{orange}{$5$}}](v5) at (5,0) {};

    \draw[-stealth, thick] (v0) to [out=45,in=135] (v1);
    \draw[-stealth, thick, color=lavgray] (v0) to [out=-45,in=-135] (v1);
    \draw[-stealth, thick] (v0) to [out=60,in=120] (v2);
    \draw[-stealth, thick] (v1) to [out=0,in=180] (v2);
    \draw[-stealth, thick] (v1) .. controls (2.0, 1.0) and (2.5, -0.5) .. (v3);
    \draw[-stealth, thick, color=lavgray] (v1) .. controls (2.0, -1.0) and (2.5, 0.5) .. (v3);	
    \draw[-stealth, thick, color=lavgray] (v2) to [out=-60,in=-120] (v4);
    \draw[-stealth, thick] (v2) to [out=60,in=120] (v4);
    \draw[-stealth, thick, color=lavgray] (v3) to [out=0,in=180] (v4);
    \draw[-stealth, thick] (v3) to [out=60,in=120] (v5);
    \draw[-stealth, thick, color=lavgray] (v4) to [out=-45,in=-135] (v5);
    \draw[-stealth, thick] (v4) to [out=45,in=135] (v5);

    \node[] at (-0.3,0.7){\footnotesize\textcolor{black}{$(G,F)$}};

    \node[] at (1, .65) {\footnotesize\textcolor{black}{$t_1$}};
    \node[] at (0.55, .32) {\footnotesize\textcolor{black}{$t_2$}};
    \node[] at (0.5, -.15) {\footnotesize\textcolor{cadmiumgreen}{$s_0$}};

    \node[] at (1.9, .45) {\footnotesize\textcolor{black}{$t_3$}};
    \node[] at (1.5, .1) {\footnotesize\textcolor{black}{$t_4$}};
    \node[] at (1.7, -.27) {\footnotesize\textcolor{cadmiumgreen}{$s_1$}};

    \node[] at (3, .65) {\footnotesize\textcolor{black}{$t_5$}};
    \node[] at (3, -.4) {\footnotesize\textcolor{cadmiumgreen}{$s_2$}};
    
    \node[] at (4, .65) {\footnotesize\textcolor{black}{$t_6$}};
    \node[] at (3.5, .1) {\footnotesize\textcolor{cadmiumgreen}{$s_3$}};

    \node[] at (4.5, .3) {\footnotesize\textcolor{black}{$t_7$}};
    \node[] at (4.5, -.15) {\footnotesize\textcolor{cadmiumgreen}{$s_4$}};
\end{scope}
\end{tikzpicture}
\vspace{-.2in}
\caption{A framed graph $(G,F)$. The framing of the incoming and outgoing edges at each vertex is ordered pictorially from bottom to top. 
}
\label{fig:framed_graph}
\end{figure}
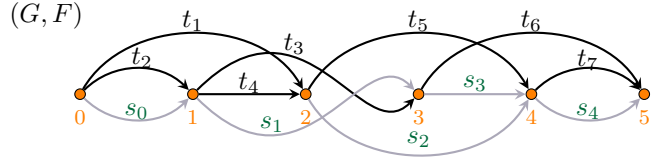

A framing~$F$ extends to total orderings of $\Prefixes(v)$ and $\Suffixes(v)$, as follows. Let $P$ and $P'$ be two distinct prefixes ending at $v$. Proceed backward from $v$ along $P$ and $P'$ to the first vertex $w$ where they diverge and inherit the ordering of $P$ and $P'$ from $\preceq_{\inedge(w)}$. Similarly, if $Q$ and $Q'$ are two distinct suffixes starting at $v$, proceed forward along $Q$ and $Q'$ to the vertex $w'$ where they first diverge and inherit the ordering of $Q$ and $Q'$ from $\preceq_{\outedge(w')}$. 

Following the same strategy we can use a framing $F$ to establish the order on any pair of paths either ending at $v$ or starting at $v$ whenever one is not a subpath of another. We denote by $\inpath(v)$ (and $\outpath(v)$) the sets of paths ending at (or starting at) $v$ together with their respective total orders.

A fundamental concept introduced in \cite{DanilovKarzanovKoshevoy2012} is the notion of coherence of routes, which we extend to paths, and we describe  also in the language of common subpaths used in \cite{GonzalezDleonMoralesPhilippeTamayoYip2025}.
		
\begin{definition}[Path Coherence] 
Two paths $P$ and $Q$ are said to have a \defn{conflict at a vertex~$v$} if $P$ and $Q$ both visit $v$ and the subpaths $Pv$ and $Qv$ have the opposite relative ordering in $\inpath(v)$ from the subpaths $vP$ and $vQ$ in $\outpath(v)$. Otherwise we say that $P$ and $Q$ are \defn{coherent at $v$}.  
We say that $P$ and $Q$ are \defn{coherent} if they are coherent at every vertex. We say that a set of paths is \defn{coherent} if the paths are pairwise coherent. Whenever $P$ and~$Q$ have a conflict at a vertex $v$ and are such that $vP\preceq_{\outpath(v)}vQ$ we say respectively that $P$ is \defn{descending} and $Q$ is \defn{ascending} at the conflict at $v$. 

It is relevant to note that if two paths $P$ and $Q$ are in conflict at a vertex $v$ they will also have a conflict at every vertex on a common subpath $O=uPw=uQw$ that includes $v$. We will say then that $P$ and $Q$ have a \defn{conflict} on the subpath $O$. Finally, every pair of paths $P$ and $Q$ uniquely define a sequence $O_1,O_2,\dots,O_k$ of maximal subpaths $O_1=v_1Pw_1=v_1Qw_1$, $O_2=v_2Pw_2=v_2Qw_2$, $...$ , $O_k=v_kPw_k=v_kQw_k$ where \[v_1\leq w_1<v_2\leq w_2<\cdots <v_k\leq w_k\] and they are in conflict. 
When $k=0$ we have that $P$ and $Q$ are coherent and when $k=1$ we say that they have a \defn{minimal conflict}. 
\end{definition}

Sets of coherent routes play the central role in the triangulation method in~\cite{DanilovKarzanovKoshevoy2012}.

\begin{definition}[Clique]\label{def:cliquesG}
A set $\calC$ of routes whose elements are pairwise coherent is called a \defn{clique}.  We denote by $\Cliques(G,F)$ the set of cliques of $(G,F)$.

A partial order on $\Cliques(G,F)$ is given by the containment order ($\calC \subseteq \calC'$). 
By definition, $\Cliques(G,F)$ is closed under inclusion; as a consequence, it is a simplicial complex and $\rank(\calC):=|\calC|-1$ is the \defn{rank} or \defn{dimension} of a simplex $\calC \in \Cliques(G,F)$. This order is the \defn{face poset} of the simplicial complex $\Cliques(G,F)$. We say that a clique $\calC$ is \defn{saturated} if it is maximal with respect to this order, that is, for every $P \in \Routes(G) \setminus \calC$ there is $Q \in \calC$ such that $P$ and $Q$ are in conflict. These are the top dimensional elements in $\Cliques(G,F)$. We denote by $\SatCliques(G,F)$ the set of saturated cliques. For any pair $\calC \in \Cliques(G,F)$ and $\calC' \in \SatCliques(G,F)$ such that $\calC \subseteq \calC'$ we say that  $\calC'$ is a \defn{saturation} of $\calC$.
\end{definition}

As a corollary of Theorem \ref{theorem:DKK} below, saturated cliques all have the same rank, $d=m-n$, the dimension of $\calF_{G}$. 
Let $\bz(P)\in \bbZ^E$ denote the \defn{indicator function} of a path $P$ and for $\calC \subseteq \Routes (G)$ define $\triangle_{\calC}$ to be the convex hull
\begin{equation}
\label{eq:triangle}
    \triangle_{\calC}:=\conv\big(\{\bz(P)\mid P \in \calC\}\big).
\end{equation}

Danilov, Karzanov, and Koshevoy prove the following pivotal theorem.

\begin{theorem}[\cite{DanilovKarzanovKoshevoy2012}]
\label{theorem:DKK}
Let $(G,F)$ be a framed graph of $G$. The collection
\[\DKK(G,F)=\{\triangle_{\calC}\mid \calC \in \SatCliques(G,F)\}
\]
is a regular unimodular triangulation of $\calF_G$.
\end{theorem}

Combining Theorem \ref{theorem:DKK} with the simple formulas given in equations \eqref{eqn:lidski_outdegree_formula} and \eqref{eqn:lidski_indegree_formula}  obtained from Theorem \ref{thm.genlidskii} when $\ba=\be_0- \be_n$, we conclude that the number of top dimensional simplices in $\DKK(G,F)$ is $K(\bd)=K(\tilde \bd)$. A pair of bijections $\DKK(G,F)\rightarrow \calF_G^{\bbZ}(\bd)$ and  $\DKK(G,F)\rightarrow \calF_G^{\bbZ}(\tilde \bd)$ were obtained in \cite{MeszarosMoralesStriker2019}. 

\begin{lemma}[Lemma 7.9 \cite{MeszarosMoralesStriker2019}]
\label{lemma:lemma_MMS}
The map $\varphi:\SatCliques(G,F)\rightarrow  \calF_G^{\bbZ}(\bd)$ defined by $\varphi(C)(e)=n(e)-1$ where $n(e)$ is defined as the number of distinct prefixes of routes $R \in \calC$ that contain~$e$.
\end{lemma}

The special case of our Lemma \ref{lemma:bijection} in Section \ref{sec:Lidskii} when $\ba=\be_0- \be_n$ gives a pair of bijections $\DKK(G,F)\rightarrow \calF_G^{\bbZ}(\bd)$ and  $\DKK(G,F)\rightarrow \calF_G^{\bbZ}(\tilde \bd)$ that are different from the ones in Lemma~\ref{lemma:lemma_MMS}.

A further consequence of the work in \cite{DanilovKarzanovKoshevoy2012} is a characterization of when two top dimensional simplices in $\DKK(G,F)$ intersect in a facet.

\begin{theorem}[Danilov et al.~\cite{DanilovKarzanovKoshevoy2012}]
\label{theorem:neighboring_simplices}
    Two top dimensional simplices $\triangle_{\calC},\triangle_{\calC'} \in \DKK(G,F)$  intersect in a common facet if their defining saturated cliques differ precisely in two routes $P\in \calC$ and $Q\in \calC'$ which have a minimal conflict at some vertex $v$. Moreover, the routes $PvQ$ and $QvP$ are members of both $\calC$ and~$\calC'$.
\end{theorem}

\begin{definition}
\label{def:minimal_exchange_quadrangle_cliques}
The set $\{P,Q,PvQ,QvP\}$ that satisfies the conditions of Theorem \ref{theorem:neighboring_simplices} is called a \defn{minimal exchange quadrangle}.  
\end{definition}

\subsection{The (right) weak order and the lattice property}
\label{sec:weakorder}
\phantom{W}

It turns out that Theorem \ref{theorem:neighboring_simplices} hides a suitable partial order on $\SatCliques(G,F))$. 

We first define the cover relations on pairs of elements of $\SatCliques(G,F)$ that have a minimal conflict. From Theorem \ref{theorem:neighboring_simplices}, two adjacent $\calC,\calC' \in \SatCliques(G,F)$ differ precisely at two routes $P\in \calC$ and $Q\in \calC'$ that are in conflict at a vertex $v$ (or, equivalently, in conflict at a maximal common subpath $O=v_1Pv_2=v_1Qv_2\subseteq P\cap Q$ containing $v$) at which $P$ is descending and $Q$ is ascending.  In this situation we define the weak cover relation $\calC \lessdot \calC'$.

\begin{proposition}
    The graph on vertex set $\SatCliques(G,F)$ with directed edge set $\{(\calC,\calC')\mid \calC \lessdot \calC'\}$ is free of cycles. Consequently, the transitive closure on $\lessdot$ defines a poset. 
\end{proposition}
\begin{proof}
We use the map $\varphi:\SatCliques(G,F)\rightarrow \calF_G^{\bbZ}(\bd)$ of Lemma \ref{lemma:lemma_MMS} together with the lexicographic order on the integer flow vectors in $\calF_G^{\bbZ}(\bd)$ such that edges are also ordered according to $e'=(v',w')<_{lex} e=(v,w)$ whenever $v'<v$ or $v'=v$ and $e\prec_{\outedge(v)} e'$.

Now let $\calC$ and $\calC'$ have a minimal conflict where $P\in \calC$ and $Q\in \calC'$ are in conflict at the maximal common subpath $O=v_1Pv_2=v_1Qv_2$ at which $P$ is descending and $Q$ is ascending, and $e \prec_{\outedge(v_2)}e'$ are respectively the edges of $P$ and $Q$ leaving $v_2$. In this case, we note that $e'$ and $e$  are consecutive in $<_{lex}$, all prefixes of edges $e''<_{lex}e'$ are equal in both $\calC$ and $\calC'$ and hence $\varphi(\calC)(e'')=\varphi(\calC')(e'')$ but $\varphi(\calC)(e')=\varphi(\calC')(e)+1$, which implies that
$$\varphi(\calC')>\varphi(\calC)$$
in the lexicographic order on $\calF_G^{\bbZ}(\bd)$. We have then shown that the value $\varphi$ always increases strictly lexicographically on $\calF_G^{\bbZ}(\bd)$ whenever $\calC \lessdot \calC'$ and hence cycles cannot exist.
\end{proof}

\begin{definition}[Weak order]\label{def:weak_order_cliques}
    We define the \defn{(right) weak order}  to be the partial order $W(G,F)$ on $\SatCliques(G,F)$ given by the transitive closure of $\lessdot$.
\end{definition}

Recall that a poset $\calL$ is said to be a \defn{lattice} if for every $x,y\in \calL$ there exist
\begin{itemize}
    \item a unique maximal common lower bound called the \defn{meet} of $x$ and $y$, denoted $x\meet y$ and 
    \item a unique minimal common upper bound called the \defn{join} of $x$ and $y$, denoted $x\join y$
\end{itemize}
 
In 2020, the three authors together with Benedetti, Harris and Morales, conjectured the following property which has been announced recently as a theorem with independent proofs by Bell and Ceballos in \cite{BellCeballos2024} and by  Berggren and Serhiyenko in \cite{BerggrenSerhiyenko2024}. 

\begin{theorem}[\cite{BellCeballos2024},\cite{BerggrenSerhiyenko2024}]
\label{theorem:lattice_property}
    $W(G,F)$ is a lattice.
\end{theorem} 

\begin{remark} 
The definition that was used in \cite[Corollary 4.9]{BerggrenSerhiyenko2024}, which they call the \emph{DKK lattice}, is isomorphic to the dual order $W(G,F)^*$ of $W(G,F)$. The authors of \cite{BellCeballos2024} call this order the \emph{framing lattice} and as a further result of their proof of \cite[Theorem 3.2]{BellCeballos2024} is that  $W(G,F)$ is an HH-lattice, which in particular is semidistributive and congruence uniform. We believe that the weak order is the right terminology for $W(G,F)$ because, as we will show in part II of this series \cite{GonzalezHanusaMoralesYip2026}, this order has analogous properties for permutation flows as the weak order on classical permutations~\cite{BB05} and their generalizations (see for example the weak order on Stirling permutations defined in \cite{CeballosPons2024-1}).
\end{remark}

\newpage
\section{Permutation Flow combinatorics for the working mathematician}
\label{sec:permutationflows}

The focus of this section is to provide readers with a working knowledge of what a permutation flow is and the combinatorics and geometry behind them. The case when $\ba=\be_0-\be_n$ is treated here to provide the foundations of these rich combinatorial objects, to be developed in more generality for $\ba$ in which only $a_n$ is negative in Section~\ref{sec:augmented_graphs}; most proofs are delayed until then. Permutation flows encode the combinatorics of a triangulation of the flow polytope $\calF_G(\ba)$ that generalizes Theorem \ref{theorem:DKK} and which is developed in Section~\ref{sec:triangulation_flow_polytopes}. The following visualization shows a progression of the combinatorial objects that are introduced in this section.

\vspace{-.1in}
\noindent
\begin{center}
\tcbox[on line,left=2pt,right=2pt,top=3pt,bottom=3pt,colback=black!05,colframe=black!50,arc=3mm,boxrule=0.75pt]{\(\begin{array}{ c c c c c c c }
\mathrm{\begin{array}{ c } \mathrm{Permutation} \\
\mathrm{Flows} \end{array}} & \!\!\longrightarrow\!\! &

\mathrm{~~Cliques~~} & \!\!\longrightarrow\!\! & 
\mathrm{~~Vines~~} & \!\!\longrightarrow\!\! 
& \mathrm{~~Groves~~}   \medskip\\
\pi & \mapsto & \mathcal{C} & \mapsto  & \calV & \mapsto  & \Gamma \\
\end{array}\)}
\end{center}

\subsection{Splits and slacks}
\label{sec:splits_and_slacks}
\phantom{W}

Let $(G,F)$ be a framed graph. We will call an edge $e=(v,w) \in E(G)$ a \defn{split} in $G$ (at~$v$) if there is another $e' \in \outedge(v)$ such that $e' \prec_{\outedge(v)} e$. An edge in $E(G)$ that is not a split in  $G$ is called a \defn{slack} in $G$.  Define $\Splits(G)$ to be the set of splits and $\Slacks(G)$  the set of slacks; these partition $E(G)$ into $\Splits(G)\sqcup\Slacks(G)$. These definitions also apply on any subgraph $H \subseteq G$ with the inherited framing. Splits are a key concept throughout the article.

For notational convenience, we will denote every slack in $G$ leaving vertex $v$ by $s_v$. All splits in $G$ will be denoted by $t_j$ for $j\in [d]$. We use a canonical ordering in which the indices of edges with the same tail vertex increase from the maximal to the next-to-minimal edges in $\outedge(v)$, and increases as the tail vertices of the edges increase, as in Figure~\ref{fig:framed_graph} on page~\pageref{fig:framed_graph}. We will keep this notation fixed for every subgraph $H\subseteq G$. Using this notational convention, for a proper subgraph $H\subsetneq G$ we will always have 
$\Splits(H)\subsetneq \Splits(G)=\{t_j\in E(G)\}$. 

\begin{example}
The acyclic directed graph $G$ in Figure~\ref{fig:framed_graph} has $n+1=6$ vertices and $m=12$ edges, so $d=m-n=7$. 
Slacks are labeled $s_v$ for $v\in[0,4]$ and splits are labeled $t_j$ in the canonical ordering for $j\in[7]$.
An example of a route in $G$ is the sequence of edges $t_2, s_1, s_3, t_7$.
We read the framing $F$ through the relative order of the incoming and outgoing edges at each vertex from bottom to top. As examples, $\outedge(1)=(s_1, t_5, t_4)$ and $\inedge(3)=(t_4,s_1)$. 
\end{example}

\subsection{Permutation Flows on Framed Graphs}
\label{sec:permuflows_10001}
\phantom{W}

An \defn{unshuffle}\footnote[2]{Remark~\ref{remark:shuffle} discusses this language.}  of two words $c$ and $c'$ (say, for simplicity, on disjoint alphabets) is another word~$c''$ whose underlying alphabet is the union of those of $c$ and $c'$ and where each of $c$ and~$c'$ is a subword of $c''$.

Recall that $x_0$ is the path of length 0 starting and ending at $0$.

\begin{definition}[Permutation flow]
\label{def:permutationflow}
A \defn{permutation flow} on a framed graph $(G,F)$ is a tuple $(\pi(e))_{e\in E}$ such that each $\pi(e)$ is a (possibly empty) word without repeated letters from a subset of $\{x_0\}\sqcup\Splits(G)$ that satisfies the conditions (i) and (ii) below, which rely on the following two definitions:

\begin{itemize}[leftmargin=0.3in]
    \item[$\bullet$~] The \defn{support} of $\pi$, denoted $\supp(\pi)$, is the subgraph of $G$ for which $\pi(e)\ne\emptyset$.
    \item[$\bullet$~] At a vertex $v\geq 1$ with $\inedge(v)=(e_0,\ldots,e_i)$, the \defn{$v$-th summary} $\zeta_v$ of $\pi$ is the concatenated word $\pi(e_0)\cdots\pi(e_i)$. When $\pi$ is non-empty, define $\zeta_0=x_0$.
\end{itemize}

\begin{itemize}
    \item[(i)] At a vertex $v<n$ with outgoing edges $\outedge(v)=(e_0',e_1', \ldots, e_j')$,  the concatenated word $\pi(e_0') \pi(e_1') \cdots \pi(e_j')$ is an unshuffle of the $v$-th summary $\zeta_v$ and a (possibly empty) subword of $e_1'\cdots e_j'$.    
    \item[(ii)] If $e_h'\in \pi(e_0')\cdots\pi(e_j')$ for some $1\leq h\leq j$, then $e_h'$ is the first letter of $\pi(e_h')$ and $e_h'$ is a split in $\supp(\pi)$.
\end{itemize}

We denote by $\PermutationFlows(G,F)$ the set of permutation flows on $(G,F)$, including the \defn{empty permutation flow} on $(G,F)$ for which $\pi(e)=\emptyset$ for all $e\in E$.
\end{definition}

\begin{remark}
    Condition (i) is related to the idea of conservation of flow in the sense that at each inner vertex $v$, the sequence of outgoing letters is obtained from the sequence of incoming letters together with the letters from the splits of $\outedge(v)$. 
\end{remark}

As with the combinatorial objects that follow, of particular interest are the maximal objects subject to a natural ordering. For permutation flows, this occurs when the unshuffle in condition (ii) above involves the entire word $t_{j_1}t_{j_2}\cdots t_{j_k}$, which implies the first letter of $\pi(t)$ is $t$ for all $t\in\Splits(G)$. Such a permutation flow will be called \defn{saturated}. We denote by \defn{$\SatPermutationFlows(G,F)$} the set of all saturated permutations flows on $(G,F)$. 

In permutation flows and in the objects that follow, the concept of a split is essential. 

\begin{definition}[Permutation flow split]
Let $\pi\in\PermutationFlows(G,F)$. An edge $t\in E$ is a \defn{split} in $\pi$ if $t$ is the first letter of $\pi(t)$. We denote by $\Splits(\pi)$ the \defn{set of splits} in~$\pi$.

If $t\in\Splits(\pi)$ with $v=\tail(t)$, let $e^*$ be the highest edge satisfying $e^*\prec_{\outedge(v)} t$, $\pi(e^*)\ne e^*$, and $\pi(e^*)\ne \emptyset$. 
Let $e$ be the last letter of $\pi(e^*)$.
We say that \defn{$t$ is a direct split of $e$ at $v$}.
See Figure~\ref{fig:direct_split}.
\end{definition}

This definition extends to general permutation flows directly in Definition~\ref{def:permuflow_splits}.

\begin{figure}[!h]
    \centering
\begin{tikzpicture}
\begin{scope}[scale=1.5, yscale=1.0]

\vertex[fill=orange, minimum size=4pt, label=below:{\tiny\textcolor{orange}{$v$}}](v0) at (0,0) {};
	
\draw[-stealth, thick, color=black!50] (v0) .. controls (.3, .4) and (.6,.6) .. (1,.7);
\draw[-stealth, thick, color=black!50] (-1,0) to (v0);
\draw[-stealth, thick, color=black!50] (v0) to (1,0);

\node[] at (1.2,.7) {\footnotesize \textcolor{black!50}{$t$}};
\node[] at (1.2,0) {\footnotesize \textcolor{black!50}{$e^*$}};
\node[] at (-1.2,0) {\footnotesize \textcolor{black!50}{$e'$}};

\node[] at (.5, 0.75){\footnotesize$t\cdots$};
\node[] at (.5, 0.1){\footnotesize$\cdots e$};
\node[] at (-.5, 0.1){\footnotesize$\cdots e\cdots$};

\end{scope}
\end{tikzpicture}    \caption{$t$ is a direct split of $e$ at $v$.}
    \label{fig:direct_split}
\end{figure}
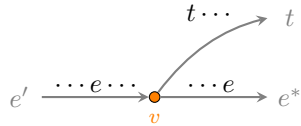

\begin{example}
\label{ex:permutation_flow_10001}
Figure~\ref{fig:permutation_flow_example} shows two examples $\pi$ and $\rho$ of permutation flows on the framed graph in Figure~\ref{fig:framed_graph}. Occurrences of $x_0$ and $t_j$ for $j\in[d]$ are replaced by their indices.

We verify that $\rho$ satisfies the conditions of Definition~\ref{def:permutationflow}.
Condition~(i) is satisfied because the first letter of every split $t_j$ is $t_j$. Condition~(ii) is satisfied at all inner vertices; for example, at vertex $4$, $\inedge(4) = (s_2,s_3,t_5)$, $\outedge(4) = (s_4,t_7)$, and the concatenation of the outgoing words $\rho(s_4) \rho(t_7) = t_4t_2t_3t_7x_0t_5t_1$ is an unshuffle of $\zeta_4=\rho(s_2)\rho(s_3)\rho(t_5) = t_4t_2t_3x_0t_5t_1$ and $\rho(t_7) = t_7$.
The vertex summaries for $\rho$ are 
\[\textup{$\zeta_0=x_0$, $\zeta_1=x_0t_2$, $\zeta_2=t_4t_2t_1$, $\zeta_3=t_3x_0$, $\zeta_4=t_4t_2t_3x_0t_5t_1$, and $\zeta_5=t_4t_2t_3t_7x_0t_5t_1t_6$.}\]
Note that $\rho$ is saturated and $\pi$ is not. For example, $\pi(t_2)$ does not start with $t_2$. 

Note that $\Splits(\pi)=\{t_1,t_3,t_5\}$ and $\Splits(\rho)=\Splits(G)$. In $\rho$, $t_5$ is a direct split of $t_2$ at vertex $2$ and $t_1$ is a direct split of $x_0$ at vertex $0$.
\end{example}

\begin{figure}[ht!]
    \centering
    \begin{tikzpicture}[scale=1.0]

\begin{scope}[scale=1.3, xshift=-180]
\node[] at (-0.3,0.6){\textcolor{blue}{$\pi$}};
	\vertex[fill=orange, minimum size=4pt, label=below:{\tiny\textcolor{orange}{$0$}}](v0) at (0,0) {};
	\vertex[fill=orange, minimum size=4pt, label=below:{\tiny\textcolor{orange}{$1$}}](v1) at (1,0) {};
	\vertex[fill=orange, minimum size=4pt, label=below:{\tiny\textcolor{orange}{$2$}}](v2) at (2,0) {};
	\vertex[fill=orange, minimum size=4pt, label=below:{\tiny\textcolor{orange}{$3$}}](v3) at (3,0) {};
	\vertex[fill=orange, minimum size=4pt, label=below:{\tiny\textcolor{orange}{$4$}}](v4) at (4,0) {};
	\vertex[fill=orange, minimum size=4pt, label=below:{\tiny\textcolor{orange}{$5$}}](v5) at (5,0) {};		

\draw[-stealth, thick] (v0) to [out=45,in=135] (v1);
\draw[-stealth, thick,color=lavgray] (v0) to [out=-45,in=-135] (v1);
\draw[-stealth, thick] (v0) to [out=60,in=120] (v2);
\draw[-stealth, thick] (v1) to [out=0,in=180] (v2);
\draw[-stealth, thick] (v1) .. controls (2.0, 1.0) and (2.5, -0.5) .. (v3);
\draw[-stealth, thick,color=lavgray] (v1) .. controls (2.0, -1.0) and (2.5, 0.5) .. (v3);	
\draw[-stealth, thick,color=lavgray] (v2) to [out=-45,in=-135] (v4);
\draw[-stealth, thick] (v2) to [out=60,in=120] (v4);
\draw[-stealth, thick,color=lavgray] (v3) to [out=0,in=180] (v4);
\draw[-stealth, thick] (v3) to [out=60,in=120] (v5);
\draw[-stealth, thick,color=lavgray] (v4) to [out=-45,in=-135] (v5);
\draw[-stealth, thick] (v4) to [out=45,in=135] (v5);

\node[] at (4, 0.66){\scriptsize\textcolor{blue}{$\emptyset$}};
\node[] at (4.47, 0.3){\scriptsize\textcolor{blue}{$351$}};
\node[] at (4.5, -0.12){\scriptsize\textcolor{blue}{$0$}};
\node[] at (3, 0.66){\scriptsize\textcolor{blue}{$51$}};
\node[] at (3.5, 0.1){\scriptsize\textcolor{blue}{$3$}};
\node[] at (3, -0.35){\scriptsize\textcolor{blue}{$0$}};
\node[] at (1.9, 0.45){\scriptsize\textcolor{blue}{$3$}};
\node[] at (1.5, 0.1){\scriptsize\textcolor{blue}{$0$}};
\node[] at (1.7, -.27){\scriptsize\textcolor{blue}{$\emptyset$}};
\node[] at (1, 0.66){\scriptsize\textcolor{blue}{$1$}};
\node[] at (0.5, 0.32){\scriptsize\textcolor{blue}{$0$}};
\node[] at (0.5, -0.12){\scriptsize\textcolor{blue}{$\emptyset$}};
\end{scope}

\begin{scope}[scale=1.3]
\node[] at (-0.3,0.6){\textcolor{blue}{$\rho$}};
	\vertex[fill=orange, minimum size=4pt, label=below:{\tiny\textcolor{orange}{$0$}}](v0) at (0,0) {};
	\vertex[fill=orange, minimum size=4pt, label=below:{\tiny\textcolor{orange}{$1$}}](v1) at (1,0) {};
	\vertex[fill=orange, minimum size=4pt, label=below:{\tiny\textcolor{orange}{$2$}}](v2) at (2,0) {};
	\vertex[fill=orange, minimum size=4pt, label=below:{\tiny\textcolor{orange}{$3$}}](v3) at (3,0) {};
	\vertex[fill=orange, minimum size=4pt, label=below:{\tiny\textcolor{orange}{$4$}}](v4) at (4,0) {};
	\vertex[fill=orange, minimum size=4pt, label=below:{\tiny\textcolor{orange}{$5$}}](v5) at (5,0) {};		

\draw[-stealth, thick] (v0) to [out=45,in=135] (v1);
\draw[-stealth, thick,color=lavgray] (v0) to [out=-45,in=-135] (v1);
\draw[-stealth, thick] (v0) to [out=60,in=120] (v2);
\draw[-stealth, thick] (v1) to [out=0,in=180] (v2);
\draw[-stealth, thick] (v1) .. controls (2.0, 1.0) and (2.5, -0.5) .. (v3);
\draw[-stealth, thick,color=lavgray] (v1) .. controls (2.0, -1.0) and (2.5, 0.5) .. (v3);	
\draw[-stealth, thick,color=lavgray] (v2) to [out=-45,in=-135] (v4);
\draw[-stealth, thick] (v2) to [out=60,in=120] (v4);
\draw[-stealth, thick,color=lavgray] (v3) to [out=0,in=180] (v4);
\draw[-stealth, thick] (v3) to [out=60,in=120] (v5);
\draw[-stealth, thick,color=lavgray] (v4) to [out=-45,in=-135] (v5);
\draw[-stealth, thick] (v4) to [out=45,in=135] (v5);

\node[] at (4, 0.66){\scriptsize\textcolor{blue}{$6$}};
\node[] at (4.47, 0.3){\scriptsize\textcolor{blue}{$7051$}};
\node[] at (4.5, -0.12){\scriptsize\textcolor{blue}{$423$}};
\node[] at (3, 0.66){\scriptsize\textcolor{blue}{$51$}};
\node[] at (3.5, 0.1){\scriptsize\textcolor{blue}{$30$}};
\node[] at (3, -0.35){\scriptsize\textcolor{blue}{$42$}};
\node[] at (1.9, 0.45){\scriptsize\textcolor{blue}{$3$}};
\node[] at (1.5, 0.1){\scriptsize\textcolor{blue}{$42$}};
\node[] at (1.7, -.27){\scriptsize\textcolor{blue}{$0$}};
\node[] at (1, 0.66){\scriptsize\textcolor{blue}{$1$}};
\node[] at (0.5, 0.32){\scriptsize\textcolor{blue}{$2$}};
\node[] at (0.5, -0.12){\scriptsize\textcolor{blue}{$0$}};
\end{scope}

\end{tikzpicture}\vspace{-.2in}
    \caption{Two permutation flows on the framed graph $(G,F)$ of Figure~\ref{fig:framed_graph}. Note that occurrences of $x_0$ and $t_j$ are replaced by their indices for legibility.}
    \label{fig:permutation_flow_example}
\end{figure}
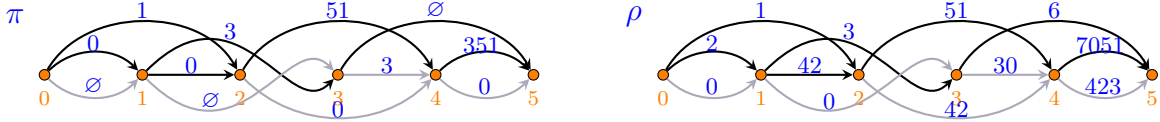

For $\pi\in\SatPermutationFlows(G,F)$, $\zeta(\pi)=\zeta_n$ is a permutation of $\{x_0\}\sqcup\{t_j\}_{j\in[d]}$. We call it the \defn{final summary} of $\pi$ and denote the set of these permutations by $\FinalSummaries(G,F)$. These permutations are all unique, as given by the Theorem \ref{thm:finalsummary_10001}, which is a corollary of Theorem~\ref{thm:finalsummary} when restricted to framed graphs. 

\begin{theorem}
\label{thm:finalsummary_10001}
    $\SatPermutationFlows(G,F)$ is in bijection with $\FinalSummaries(G,F)$.
\end{theorem}

The general case of permutation flows (and their companion permutation flow shuffles) is treated in Section~\ref{sec:permutation_flow_shuffles}.

\subsection{Combinatorial families supporting permutation flows}
\label{sec:objects_10001}
\phantom{W}

In this section we highlight combinatorial objects that are in bijection with permutation flows. 
The correspondences are proven in Section~\ref{sec:augmented_graphs}.

\subsubsection{Integer $\bd$-flows}
\label{sec:integerflows_10001}
\phantom{W}

Recall that 
\[\bd = \left(0, \indeg(1)-1, \ldots, \indeg(n-1)-1, 
    \hbox{$-\sum_{i=1}^{n-1} (\indeg(i)-1)$} \right) \in \bbZ^{n+1}.
\]

Saturated permutation flows are in bijection with integer $\bd$-flows on $G$, as proved in general in Theorem~\ref{thm:integer_flows_a}. Suppose that $\phi\in\SatPermutationFlows(G,F)$. Define $\psi\in\calF_G^\mathbb{Z}(\bd)$ by setting $\psi(e)$ to be one less than the length of $\pi(e)$ for every $e\in E$. The inverse rule involves using $\psi(e)$ for $e\in\outedge(v)$ to determine the number of elements of $\zeta_v$ that are included in each word $\pi(e)$. See Figure~\ref{fig:an_integer_flow}. 

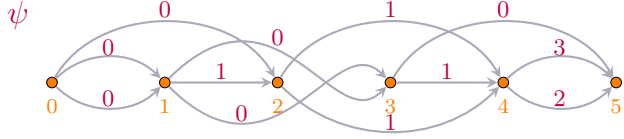
\begin{figure}[!h]
    \centering
    \vspace{-.2in}
    \begin{tikzpicture}
\begin{scope}[xshift=180, yshift=0, scale=1.5]

\node[] at (-0.3,0.6){\textcolor{purple}{$\psi$}};

    \vertex[fill=orange, minimum size=4pt, label=below:{\tiny\textcolor{orange}{$0$}}](v0) at (0,0) {};
    \vertex[fill=orange, minimum size=4pt, label=below:{\tiny\textcolor{orange}{$1$}}](v1) at (1,0) {};
    \vertex[fill=orange, minimum size=4pt, label=below:{\tiny\textcolor{orange}{$2$}}](v2) at (2,0) {};
	\vertex[fill=orange, minimum size=4pt, label=below:{\tiny\textcolor{orange}{$3$}}](v3) at (3,0) {};
	\vertex[fill=orange, minimum size=4pt, label=below:{\tiny\textcolor{orange}{$4$}}](v4) at (4,0) {};
	\vertex[fill=orange, minimum size=4pt, label=below:{\tiny\textcolor{orange}{$5$}}](v5) at (5,0) {};

\draw[-stealth, thick, color=lavgray] (v0) to [out=45,in=135] (v1);
\draw[-stealth, thick, color=lavgray] (v0) to [out=-45,in=-135] (v1);
\draw[-stealth, thick, color=lavgray] (v0) to [out=60,in=120] (v2);
\draw[-stealth, thick, color=lavgray] (v1) to [out=0,in=180] (v2);
\draw[-stealth, thick, color=lavgray] (v1) .. controls (2.0, 1.0) and (2.5, -0.5) .. (v3);
\draw[-stealth, thick, color=lavgray] (v1) .. controls (2.0, -1.0) and (2.5, 0.5) .. (v3);	
\draw[-stealth, thick, thick, color=lavgray] (v2) to [out=-45,in=-135] (v4);
\draw[-stealth, thick, color=lavgray] (v2) to [out=60,in=120] (v4);
\draw[-stealth, thick, color=lavgray] (v3) to [out=0,in=180] (v4);
\draw[-stealth, thick, color=lavgray] (v3) to [out=60,in=120] (v5);
\draw[-stealth, thick, color=lavgray] (v4) to [out=-45,in=-135] (v5);
\draw[-stealth, thick, color=lavgray] (v4) to [out=45,in=135] (v5);

\node[] at (0.5, -.14) {\scriptsize\textcolor{purple}{$0$}};
\node[] at (0.5, .31) {\scriptsize\textcolor{purple}{$0$}};
\node[] at (1, .65) {\scriptsize\textcolor{purple}{$0$}};
\node[] at (2, .4) {\scriptsize\textcolor{purple}{$0$}};
\node[] at (1.5, 0.1) {\scriptsize\textcolor{purple}{$1$}};
\node[] at (1.68, -.27) {\scriptsize\textcolor{purple}{$0$}};

\node[] at (3, .65) {\scriptsize\textcolor{purple}{$1$}};
\node[] at (3.5, 0.1) {\scriptsize\textcolor{purple}{$1$}};
\node[] at (3, -.35) {\scriptsize\textcolor{purple}{$1$}};
\node[] at (4, .65) {\scriptsize\textcolor{purple}{$0$}};
\node[] at (4.5, .31) {\scriptsize\textcolor{purple}{$3$}};
\node[] at (4.5, -.14) {\scriptsize\textcolor{purple}{$2$}};
\end{scope}

\end{tikzpicture}
    \vspace{-.2in}
    \caption{The integer $\bd$-flow $\psi$ on $(G,F)$ corresponding to the permutation flow $\rho$ in Figure \ref{fig:permutation_flow_example}.}
    \label{fig:an_integer_flow}
\end{figure}

\subsubsection{Cliques and the DKK triangulation}
\label{sec:cliques_10001}
\phantom{W}

Permutation flows are in bijection with cliques and saturated permutation flows are in bijection with saturated cliques. 

Starting with a permutation flow $\pi$, we can extract a clique $\calC=\calC(\pi)$ of routes. For all letters $l\in\zeta(\pi)$, form the path $P_l$ that consists of all edges $e$ such that $l\in \pi(e)$. If $P_l$ is a route, include $P_l$ in $\calC$. Otherwise, the first edge $e_l$ of $P_l$ satisfies $e_l\in\Splits(\pi;l',e')$ for some letter $l'$ and some edge $e'$. Using this letter $l'$, replace $P_l$ by $P_{l'}vP_{l}$, and continue backward until $P_l$ is a route; include the result in $\calC$.

A clique is \defn{saturated} if it contains $d+1$ routes.

\begin{example}\label{ex:clique_10001}
The cliques corresponding to the permutation flows $\pi$ and $\rho$ from Example~\ref{ex:permutation_flow_10001} are 
\begin{gather*}
    \calC(\pi)=\{\;
    {\color{brown} t_1}t_5t_7, \;\;
    {\color{brown}s_0} t_4s_2s_4, \;\;
    t_2t_4{\color{brown} t_5}t_7, \;\;
    t_2{\color{brown}t_3}s_3t_7\;\}
\end{gather*}
and
\begin{gather*}
    \calC(\rho)=\{\;
    {\color{brown} s_0}s_1s_3t_7, \;\;
    {\color{brown} t_1}t_5t_7, \;\;
    {\color{brown} t_2}t_4s_2s_4, \;\;
    t_2{\color{brown} t_3}s_3s_4, \;\;
    s_0{\color{brown} t_4}s_2s_4, \;\;
    t_2t_4{\color{brown} t_5}t_7, \;\;
    s_0s_1{\color{brown} t_6}, \;\;
    t_2t_3s_3{\color{brown} t_7}\;\}.
\end{gather*}
The colored letters indicate the letter $l$ that generated route $P_l$. Note that $\calC(\rho)$ has eight routes and is therefore saturated. Further notice that $\calC(\pi)\subseteq\calC(\rho)$, even if the routes do not correspond to the same letters.
\end{example}

The general case of cliques involves both cliques and multicliques, as discussed in Sections~\ref{sec:route_matchings} and \ref{sec:multicliques}.

\subsubsection{Vines} 
\label{sec:vines_10001}
\phantom{W}

Note that both $\Prefixes(G)$ and $\Suffixes(G)$ have a poset structure where the order relation is given by the inclusion of edges. The Hasse diagram of $\Prefixes(G)$ is a tree because every prefix $P\ne x_0$ covers exactly one element---the prefix that deletes the last edge of $P$. In addition, every prefix contains $x_0$, so $x_0$ is the minimal element of $\Prefixes(G)$ in the prefix containment order. We have the following definition.

\begin{definition}[Vine]\label{def:vine_G}
A \defn{vine} on a graph $G$ is a subposet $\calV\subseteq \Prefixes(G)$, when non-empty, its elements satisfy
\begin{enumerate}[label=(\alph*)]
    \item The paths in $\calV$ are coherent, 
    \item $\calV$ is closed under containment of prefixes, and
    \item every element in $\calV$ is extendable to a route in $\calV$; that is, the maximal elements in $\calV$ under containment are routes.
\end{enumerate}

We denote by $\Vines(G,F)$ the set of vines of $(G,F)$. There is a natural partial order on vines given by containment $\calV\subseteq\calW$. Denote by $\SatVines(G,F)$ the vines that are maximal with respect to this order, which we call \defn{saturated}. 
\end{definition}

\begin{definition}[Vine split]
\label{def:vine_splits}
Let $P \in \calV$ be a prefix covered by $\{Q_0,Q_1,\ldots, Q_l\}$ in the prefix containment order. For $i\in[l]$, we say that $Q_i$ is a \defn{direct split of $P$}. We denote by $\Splits(\calV)$ the set of all direct splits in $\calV$.
\end{definition}
We refer the reader to Definition~\ref{def:vineyard_splits} for the formal definition of a split in the language of vines. 

By Definition~\ref{def:vine_G}(c), every vine is generated from its set of routes $\Routes(\calV)$ by taking all their prefixes, giving rise to the following poset isomorphism.

\begin{theorem}
    The map $$\calV:\Cliques(G,F) \rightarrow \Vines(G,F)$$ where $\calV(\calC)$ is defined by $\Routes(\calV(\calC))=\calC$, is a poset isomorphism whose inverse
    $$\calC: \Vines(G,F) \rightarrow \Cliques(G,F)$$ is defined by $\calC(\calV)=\Routes(\calV)$.
\end{theorem}

\begin{example}
\label{ex:vines_10001}
    We continue Example~\ref{ex:clique_10001}. From the cliques $\calC(\pi)$ and $\calC(\rho)$, create the vines $\calV(\pi)$ and $\calV(\rho)$ by taking all prefixes of paths. The prefixes that occur are
\begin{gather*}
    x_0,\; 
    {\color{orange} s_0},\; 
    {\color{orange} s_0s_1},\; 
    {\color{orange} s_0s_1s_3},\; 
    {\color{orange} s_0s_1s_3t_7},\;
    {\color{orange} s_0s_1t_6},\;
    {\color{orange} s_0t_4},\;    
    {\color{orange} s_0t_4s_2},\;    
    {\color{orange} s_0t_4s_2s_4},\;  
    t_1,\;
    t_1t_5,\; 
    t_1t_5t_7 \\
    t_2,\; 
    t_2t_3,\;
    t_2t_3s_3,\;
    {\color{orange} t_2t_3s_3s_4},\;
    t_2t_3s_3t_7,\;
    t_2t_4,\;
    t_2t_4s_2,\;
    t_2t_4s_2s_4,\;
    t_2t_4t_5,\;
    t_2t_4t_5t_7,
\end{gather*}
where those that are {\color{orange} orange} occur in $\calV(\rho)$ but not in $\calV(\pi)$.  The vines can be represented visually as in Figure~\ref{fig:vines_1001}. Note that $t_2t_3s_3t_7$ is a direct split of $t_2t_3s_3$ in $\calV(\rho)$ but $t_2t_3s_3t_7\notin\Splits(\calV(\pi))$.
\end{example}

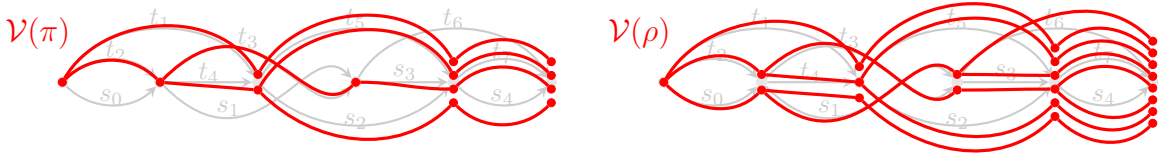
\begin{figure}[ht!]
\begin{tikzpicture}
\begin{scope}[scale=1.3, yshift=0, yscale=1.0]
\node at (-0.25,0.5){\textcolor{red}{$\calV(\pi)$}};

	\vertex[fill, color=black!10, minimum size=4pt](v0) at (0,0) {};
	\vertex[fill, color=black!10, minimum size=4pt](v1) at (1,0) {};
	\vertex[fill, color=black!10, minimum size=4pt](v2) at (2,0) {};
	\vertex[fill, color=black!10, minimum size=4pt](v3) at (3,0) {};
	\vertex[fill, color=black!10, minimum size=4pt](v4) at (4,0) {};
	\vertex[fill, color=black!10, minimum size=4pt](v5) at (5,0) {};		

    \draw[-stealth, thick,color=black!20] (v0) to [out=45,in=135] (v1);
    \draw[-stealth, thick,color=black!20] (v0) to [out=-45,in=-135] (v1);
    \draw[-stealth, thick,color=black!20] (v0) to [out=60,in=120] (v2);
    \draw[-stealth, thick,color=black!20] (v1) to [out=0,in=180] (v2);
    \draw[-stealth, thick,color=black!20] (v1) .. controls (2.0, 1.0) and (2.5, -0.5) .. (v3);
    \draw[-stealth, thick,color=black!20] (v1) .. controls (2.0, -1.0) and (2.5, 0.5) .. (v3);	
    \draw[-stealth, thick,color=black!20] (v2) to [out=-45,in=-135] (v4);
    \draw[-stealth, thick,color=black!20] (v2) to [out=60,in=120] (v4);
    \draw[-stealth, thick,color=black!20] (v3) to [out=0,in=180] (v4);
    \draw[-stealth, thick,color=black!20] (v3) to [out=60,in=120] (v5);
    \draw[-stealth, thick,color=black!20] (v4) to [out=-45,in=-135] (v5);
    \draw[-stealth, thick,color=black!20] (v4) to [out=45,in=135] (v5);

    \node[] at (1, .65) {\footnotesize\textcolor{black!20}{$t_1$}};
    \node[] at (0.55, .32) {\footnotesize\textcolor{black!20}{$t_2$}};
    \node[] at (0.5, -.15) {\footnotesize\textcolor{black!20}{$s_0$}};

    \node[] at (1.9, .45) {\footnotesize\textcolor{black!20}{$t_3$}};
    \node[] at (1.5, .1) {\footnotesize\textcolor{black!20}{$t_4$}};
    \node[] at (1.7, -.27) {\footnotesize\textcolor{black!20}{$s_1$}};

    \node[] at (3, .65) {\footnotesize\textcolor{black!20}{$t_5$}};
    \node[] at (3, -.4) {\footnotesize\textcolor{black!20}{$s_2$}};
    
    \node[] at (4, .65) {\footnotesize\textcolor{black!20}{$t_6$}};
    \node[] at (3.5, .1) {\footnotesize\textcolor{black!20}{$s_3$}};

    \node[] at (4.5, .3) {\footnotesize\textcolor{black!20}{$t_7$}};
    \node[] at (4.5, -.15) {\footnotesize\textcolor{black!20}{$s_4$}};

\vnode[color=red](vv0) at (0,0) {};
\vnode[color=red](vv12) at (1,0) {};
\vnode[color=red](vv23) at (2,0.08) {};
\vnode[color=red](vv22) at (2,-0.08) {};
\vnode[color=red](vv31) at (3,0) {};
\vnode[color=red](vv46) at (4,0.21) {};
\vnode[color=red](vv45) at (4,0.07) {};
\vnode[color=red](vv43) at (4,-0.07) {};
\vnode[color=red](vv42) at (4,-0.21) {};
\vnode[color=red](vv57) at (5,0.21) {};
\vnode[color=red](vv56) at (5,0.07) {};
\vnode[color=red](vv54) at (5,-0.07) {};
\vnode[color=red](vv52) at (5,-0.21) {};

\draw[very thick, color=red] (vv0) 
    to [out=60,in=120] (vv23) 
    to [out=60,in=120] (vv46)
    to [out=45,in=135] (vv57);

\draw[very thick, color=red] (vv12) 
    to [out=-5,in=175] (vv22)
    to [out=-45,in=-135] (vv42)
    to [out=-45,in=-135] (vv52);

\draw[very thick, color=red] (vv0) 
    to [out=45,in=135] (vv12)
    .. controls (2.0, 1.0) and (2.5, -0.5) .. (vv31)
    to [out=0,in=180] (vv43);

\draw[very thick, color=red] (vv43)
    to [out=45,in=135] (vv54);

\draw[very thick, color=red] (vv22) 
    to [out=60,in=120] (vv45)
    to [out=45,in=135] (vv56);

\end{scope}

\begin{scope}[scale=1.3, xshift=175, yscale=1.0]
\node at (-0.25,0.5){\textcolor{red}{$\calV(\rho)$}};

	\vertex[fill, color=black!10, minimum size=4pt](v0) at (0,0) {};
	\vertex[fill, color=black!10, minimum size=4pt](v1) at (1,0) {};
	\vertex[fill, color=black!10, minimum size=4pt](v2) at (2,0) {};
	\vertex[fill, color=black!10, minimum size=4pt](v3) at (3,0) {};
	\vertex[fill, color=black!10, minimum size=4pt](v4) at (4,0) {};
	\vertex[fill, color=black!10, minimum size=4pt](v5) at (5,0) {};		

    \draw[-stealth, thick,color=black!20] (v0) to [out=45,in=135] (v1);
    \draw[-stealth, thick,color=black!20] (v0) to [out=-45,in=-135] (v1);
    \draw[-stealth, thick,color=black!20] (v0) to [out=60,in=120] (v2);
    \draw[-stealth, thick,color=black!20] (v1) to [out=0,in=180] (v2);
    \draw[-stealth, thick,color=black!20] (v1) .. controls (2.0, 1.0) and (2.5, -0.5) .. (v3);
    \draw[-stealth, thick,color=black!20] (v1) .. controls (2.0, -1.0) and (2.5, 0.5) .. (v3);	
    \draw[-stealth, thick,color=black!20] (v2) to [out=-45,in=-135] (v4);
    \draw[-stealth, thick,color=black!20] (v2) to [out=60,in=120] (v4);
    \draw[-stealth, thick,color=black!20] (v3) to [out=0,in=180] (v4);
    \draw[-stealth, thick,color=black!20] (v3) to [out=60,in=120] (v5);
    \draw[-stealth, thick,color=black!20] (v4) to [out=-45,in=-135] (v5);
    \draw[-stealth, thick,color=black!20] (v4) to [out=45,in=135] (v5);

    \node[] at (1, .65) {\footnotesize\textcolor{black!20}{$t_1$}};
    \node[] at (0.55, .32) {\footnotesize\textcolor{black!20}{$t_2$}};
    \node[] at (0.5, -.15) {\footnotesize\textcolor{black!20}{$s_0$}};

    \node[] at (1.9, .45) {\footnotesize\textcolor{black!20}{$t_3$}};
    \node[] at (1.5, .1) {\footnotesize\textcolor{black!20}{$t_4$}};
    \node[] at (1.7, -.27) {\footnotesize\textcolor{black!20}{$s_1$}};

    \node[] at (3, .65) {\footnotesize\textcolor{black!20}{$t_5$}};
    \node[] at (3, -.4) {\footnotesize\textcolor{black!20}{$s_2$}};
    
    \node[] at (4, .65) {\footnotesize\textcolor{black!20}{$t_6$}};
    \node[] at (3.5, .1) {\footnotesize\textcolor{black!20}{$s_3$}};

    \node[] at (4.5, .3) {\footnotesize\textcolor{black!20}{$t_7$}};
    \node[] at (4.5, -.15) {\footnotesize\textcolor{black!20}{$s_4$}};

\vnode[color=red](vv0) at (0,0) {};
\vnode[color=red](vv12) at (1,0.08) {};
\vnode[color=red](vv11) at (1,-0.08) {};
\vnode[color=red](vv23) at (2,0.16) {};
\vnode[color=red](vv22) at (2,0) {};
\vnode[color=red](vv21) at (2,-0.16) {};
\vnode[color=red](vv32) at (3,0.08) {};
\vnode[color=red](vv31) at (3,-0.08) {};
\vnode[color=red](vv46) at (4,0.35) {};
\vnode[color=red](vv45) at (4,0.21) {};
\vnode[color=red](vv44) at (4,0.07) {};
\vnode[color=red](vv43) at (4,-0.07) {};
\vnode[color=red](vv42) at (4,-0.21) {};
\vnode[color=red](vv41) at (4,-0.35) {};
\vnode[color=red](vv58) at (5,0.42) {};
\vnode[color=red](vv57) at (5,0.30) {};
\vnode[color=red](vv56) at (5,0.18) {};
\vnode[color=red](vv55) at (5,0.06) {};
\vnode[color=red](vv54) at (5,-0.06) {};
\vnode[color=red](vv53) at (5,-0.18) {};
\vnode[color=red](vv52) at (5,-0.30) {};
\vnode[color=red](vv51) at (5,-0.42) {};

\draw[very thick, color=red] (vv0) 
    to [out=-45,in=-135] (vv11)
    .. controls (2.0, -1.0) and (2.5, 0.5) .. (vv32)
    to [out=0,in=180] (vv44)
    to [out=45,in=135] (vv55);

\draw[very thick, color=red] (vv11) 
    to [out=-5,in=175] (vv21)
    to [out=-45,in=-135] (vv41)
    to [out=-45,in=-135] (vv51);

\draw[very thick, color=red] (vv32)
    to [out=45,in=135] (vv58);

\draw[very thick, color=red] (vv0) 
    to [out=60,in=120] (vv23) 
    to [out=60,in=120] (vv46)
    to [out=45,in=135] (vv57);

\draw[very thick, color=red] (vv12) 
    to [out=-5,in=175] (vv22)
    to [out=-45,in=-135] (vv42)
    to [out=-45,in=-135] (vv52);

\draw[very thick, color=red] (vv0) 
    to [out=45,in=135] (vv12)
    .. controls (2.0, 1.0) and (2.5, -0.5) .. (vv31)
    to [out=0,in=180] (vv43)
    to [out=-45,in=-135] (vv53);

\draw[very thick, color=red] (vv43)
    to [out=45,in=135] (vv54);

\draw[very thick, color=red] (vv22) 
    to [out=60,in=120] (vv45)
    to [out=45,in=135] (vv56);

\end{scope}

\end{tikzpicture}
    \caption{The vines $\calV(\pi)$ and $\calV(\rho)$ corresponding to the permutation flows $\pi$ and $\rho$ in Figure~\ref{fig:permutation_flow_example}.}
    \label{fig:vines_1001}
\end{figure}

The general case of vines involves vines, vineyards, and vineyard shuffles and is treated in Section~\ref{sec:vineyardshuffles}.

\subsubsection{Groves}
\label{sec:groves_10001}
\phantom{W}

Groves are a family of objects introduced in~\cite{BrunnerHanusa2024} based on the related concept of \textit{noncrossing bipartite trees} that was defined in~\cite{MeszarosMorales2019}. There, it was used to describe a decomposition process of $\calF_G(\ba)$ which extended (unpublished) work of Postnikov and Stanley for the case when $\ba=\be_0- \be_n$, and which allowed the authors to provide a geometric and combinatorial proof of the Lidskii volume formula of Theorem \ref{thm.genlidskii}. The same underlying idea behind noncrossing bipartite trees is also implicit in the concept of \textit{monotone correspondence} in the proof of the triangulation in \cite{DanilovKarzanovKoshevoy2012}. 

The main idea behind a grove is the codification of the information present in a vine but from the perspective of observers standing on the vertices of the graph $G$. The noncrossing property of the bipartite trees in a grove directly captures the coherence of the routes involved, so they can be used to construct a clique inductively across vertices from $v=0,\dots,n$.

\begin{definition}[Noncrossing bipartite tree]\label{def:noncrossing_bipartite_tree_G}
A \defn{noncrossing bipartite forest} $\gamma_v$ of $(G,F)$ on a vertex $v$ is a bipartite graph on a vertex set with bipartition denoted $V(\gamma_v)=L(\gamma_v)\sqcup R(\gamma_v)$ and edge set denoted $E(\gamma_v)$, where 
\begin{enumerate}[label=(\alph*)]
    \item The \defn{left vertices} $L(\gamma_v)\subseteq\Prefixes(v)$ are prefixes of $G$ that end at $v$ and the \defn{right vertices} $R(\gamma_v)\subseteq\hat \outedge(v)$ are edges of $G$; they inherit their corresponding total orders. 
    When $v=n$, define $R(\gamma_v)=\{n\}$.
    \item The edges of $\gamma_v$ are \defn{noncrossing}. That is, whenever $P \prec Q$ and $e \prec e'$ then $(P,e')$ and $(Q,e)$ cannot both be in $E(\gamma_v)$.
    \item Every vertex of $\gamma_v$ is incident to at least one edge.
\end{enumerate}
\end{definition}

\begin{definition}[Grove]\label{def:grove_10001}
 A \defn{grove} of $(G,F)$, when non-empty, is a sequence $\Gamma=(\gamma_v)_{v \in [0,n]}$ of $n+1$ noncrossing bipartite forests satisfying
 \begin{enumerate}[label=(\alph*)]
     \item $L(\gamma_{0})=\{x_0\}$,    
     \item $R(\gamma_n)=\{y_0\}$, and
     \item for every $Pe  \in \Prefixes(G)\setminus \{x_0\}$ with $e=(v,w)$, we have that $(P,e)\in E(\gamma_{v})$ if and only if $Pe \in L(\gamma_{w})$.
 \end{enumerate}
We denote by $\Groves(G,F)$ the set of groves of $(G,F)$. There is a natural partial order on $\Groves(G,F)$ by the relation $\Gamma \subseteq \Theta$ whenever $\Gamma$ is a subgraph of $\Theta$. Denote by $\SatGroves(G,F)$ the set of groves that are maximal with respect to this order, which we call \defn{saturated}.
\end{definition}

We denote by $\Prefixes(\Gamma):=\bigsqcup_{v=0}^nL(\gamma_v)$ the set of all left vertices and $E(\Gamma):=\bigsqcup_{v=0}^nE(\gamma_v)$ the set of all grove edges. 

\begin{definition}[Grove split]
Let $v\in V(G)$ and let $\outedge(v)=\{e_0,\ldots,e_l\}$. When $P\in L(\gamma_v)$, we say that that the (grove) edge $(P,e_i)$ is a \defn{direct split of $P$ at $v$} for all $i\in[l]$. 

We denote by $\Splits(\Gamma)$ the set of all direct splits in $\Gamma$ and by $\Splits(\Gamma,(P,e))$ the set of all direct splits of any prefix $Q$ that extends $Pe$.
\end{definition}

We refer the reader to Definition~\ref{def:grove_splits} for the formal definition of a split in the language of groves. 

\begin{example}
\label{ex:groves_10001}
    We continue Example~\ref{ex:vines_10001}. The groves $\Gamma(\pi)$ and $\Gamma(\rho)$ in Figure~\ref{fig:grove_1001} correspond to the permutation flows $\pi$ and $\rho$ in Figure~\ref{fig:permutation_flow_example}. Exemplifying Definition~\ref{def:grove_10001}(c), we see that in $\Gamma(\rho)$, the three edges in $\gamma_3$ are $(s_0s_1,t_6)$, $(s_0s_1,s_3)$, and $(t_2t_3,s_3)$, and the corresponding prefixes $s_0s_1t_6$, $s_0s_1s_3$, and $t_2t_3s_3$ appear as left vertices in $\gamma_5$, $\gamma_4$, and $\gamma_4$, respectively. 
\end{example}

\begin{figure}[ht!]
    \begin{tikzpicture}[scale=1.8]
\begin{scope}[xshift=-5, yshift=-25, scale=0.5]

    \node[] at (-0.5,1.6) {$\Gamma(\pi)$};
    \node[] at (0.5,-1.1) {$\gamma_0$};
    \node[] at (0.5+2.7,-1.1) {$\gamma_1$};
    \node[] at (0.5+5.6,-1.1) {$\gamma_2$};
     \node[] at (0.5+8.5,-1.1) {$\gamma_3$};
    \node[] at (0.5+11.9,-1.1) {$\gamma_4$};
     \node[] at (0.5+15.5,-1.1) {$\gamma_5$};

    \draw[-,thick] (0,0.5)--(1,1.);
    \draw[-,thick] (0,0.5)--(1,0);
    \vertex[fill=black, minimum size=3pt] at (0,0.5) {};
    \vertex[fill=black, minimum size=3pt] at (1,1.0) {};
    \vertex[fill=black, minimum size=3pt] at (1,0.) {};
\node[anchor=east] at (-0.05,0.5) {\tiny{$x_0$}};
\node[anchor=west] at (1.05,1.0) {\tiny\textcolor{black}{$t_1$}};
\node[anchor=west] at (1.05,0.0) {\tiny\textcolor{black}{$t_2$}};

\begin{scope}[xshift=2.7cm]
    \draw[-,thick] (0,0.5)--(1,1.);
    \draw[-,thick] (0,0.5)--(1,0);
    \vertex[fill=black, minimum size=3pt] at (0,0.5) {};
    \vertex[fill=black, minimum size=3pt] at (1,1.0) {};
    \vertex[fill=black, minimum size=3pt] at (1,0.) {};
\node[anchor=east] at (-0.05,0.5) {\tiny{$t_2$}};
\node[anchor=west] at (1.05,1.0) {\tiny\textcolor{black}{$t_3$}};
\node[anchor=west] at (1.05,0.0) {\tiny\textcolor{black}{$t_4$}};

\end{scope}

\begin{scope}[xshift=5.6cm]
    \draw[-,thick] (0,1)--(1,1);	    
    \draw[-,thick] (0,0)--(1,1);		
    \draw[-,thick] (0,0)--(1,0);		
    \vertex[fill=black, minimum size=3pt] at (0,1) {};
    \vertex[fill=black, minimum size=3pt] at (0,0) {};
    \vertex[fill=black, minimum size=3pt] at (1,1) {};
    \vertex[fill=black, minimum size=3pt] at (1,0) {};
\node[anchor=east] at (-0.05,1) {\tiny{$t_1$}};
\node[anchor=east] at (-0.05,0) {\tiny{$t_2t_4$}};
\node[anchor=west] at (1.05,1) {\tiny\textcolor{black}{$t_5$}};
\node[anchor=west] at (1.05,0) {\tiny\textcolor{black}{$s_2$}};
\end{scope}

\begin{scope}[xshift=8.5cm] 
    \draw[-,thick] (0,0.5)--(1,0.5);	    
    \vertex[fill=black, minimum size=3pt] at (0,0.5) {};
    \vertex[fill=black, minimum size=3pt] at (1,0.5) {};
\node[anchor=east] at (-0.05,0.5) {\tiny{$t_2t_3$}};
\node[anchor=west] at (1.05,0.5) {\tiny\textcolor{black}{$s_3$}};
\end{scope}

\begin{scope}[xshift=11.9cm]
\draw[-,thick] (0.0,1.1)--(1.0,0.9);  
\draw[-,thick] (0.0,0.7)--(1.0,0.9);  
\draw[-,thick] (0.0,0.3)--(1.0,0.9);  
\draw[-,thick] (0.0,-0.1)--(1.0,0.1);  
\vertex[fill=black, minimum size=3pt] at (0.0,1.1) {};
\vertex[fill=black, minimum size=3pt] at (0.0,0.7) {};
\vertex[fill=black, minimum size=3pt] at (0.0,0.3) {};
\vertex[fill=black, minimum size=3pt] at (0.0,-0.1) {};
\vertex[fill=black, minimum size=3pt] at (1.0,0.9) {};
\vertex[fill=black, minimum size=3pt] at (1.0,0.1) {};
\node[anchor=east] at (-0.05,1.1) {\tiny{$t_1t_5$}};
\node[anchor=east] at (-0.05,0.7) {\tiny{$t_2t_4t_5$}};
\node[anchor=east] at (-0.05,0.3) {\tiny{$t_2t_3s_3$}};
\node[anchor=east] at (-0.05,-0.1) {\tiny{$t_2t_4s_2$}};
\node[anchor=west] at (1.05,0.9) {\tiny\textcolor{black}{$t_7$}};
\node[anchor=west] at (1.05,0.1) {\tiny\textcolor{black}{$s_4$}};
\end{scope}

\begin{scope}[xshift=15.5cm]
\draw[-,thick] (0.0,1.1)--(1.0,0.5);  
\draw[-,thick] (0.0,0.7)--(1.0,0.5);  
\draw[-,thick] (0.0,0.3)--(1.0,0.5);  
\draw[-,thick] (0.0,-0.1)--(1.0,0.5);  

\vertex[fill=black, minimum size=3pt] at (0.0,1.1) {};
\vertex[fill=black, minimum size=3pt] at (0.0,0.7) {};
\vertex[fill=black, minimum size=3pt] at (0.0,0.3) {};
\vertex[fill=black, minimum size=3pt] at (0.0,-0.1) {};
\vertex[fill=black, minimum size=3pt] at (1.0,0.5) {};
\node[anchor=east] at (-0.05,1.1) {\tiny{$t_1t_5t_7$}};
\node[anchor=east] at (-0.05,0.7) {\tiny{$t_2t_3t_5t_7$}};
\node[anchor=east] at (-0.05,0.3) {\tiny{$t_2t_3s_3t_7$}};
\node[anchor=east] at (-0.05,-0.1) {\tiny{$t_2t_4s_2s_4$}};

\node[anchor=west] at (1.05,0.5) {\tiny\textcolor{black}{$y_0$}};
\end{scope}

\end{scope}

\begin{scope}[xshift=-5, yshift=-85, scale=0.5]

    \node[] at (-0.5,1.6) {$\Gamma(\rho)$};
    \node[] at (0.5,-1.3) {$\gamma_0$};
    \node[] at (0.5+2.7,-1.3) {$\gamma_1$};
    \node[] at (0.5+5.6,-1.3) {$\gamma_2$};
     \node[] at (0.5+8.5,-1.3) {$\gamma_3$};
    \node[] at (0.5+11.9,-1.3) {$\gamma_4$};
     \node[] at (0.5+15.5,-1.3) {$\gamma_5$};

    \draw[-,thick] (0,0.5)--(1,1.1);
    \draw[-,thick] (0,0.5)--(1,0.5);
    \draw[-,thick] (0,0.5)--(1,-0.1);
    \vertex[fill=black, minimum size=3pt] at (0,0.5) {};
    \vertex[fill=black, minimum size=3pt] at (1,1.1) {};
    \vertex[fill=black, minimum size=3pt] at (1,0.5) {};
    \vertex[fill=black, minimum size=3pt] at (1,-0.1) {};
\node[anchor=east] at (-0.05,0.5) {\tiny{$x_0$}};
\node[anchor=west] at (1.05,1.1) {\tiny\textcolor{black}{$t_1$}};
\node[anchor=west] at (1.05,0.5) {\tiny\textcolor{black}{$t_2$}};
\node[anchor=west] at (1.05,-0.1) {\tiny\textcolor{black}{$s_0$}};

\begin{scope}[xshift=2.7cm]
    \draw[-,thick] (0,0.8)--(1,1.1);
    \draw[-,thick] (0,0.8)--(1,0.5);	    
    \draw[-,thick] (0,0.2)--(1,0.5);		
    \draw[-,thick] (0,0.2)--(1,-0.1);	
    \vertex[fill=black, minimum size=3pt] at (0,0.8) {};
    \vertex[fill=black, minimum size=3pt] at (0,0.2) {};
    \vertex[fill=black, minimum size=3pt] at (1,1.1) {};
    \vertex[fill=black, minimum size=3pt] at (1,0.5) {};
    \vertex[fill=black, minimum size=3pt] at (1,-0.1) {};
\node[anchor=east] at (-0.05,0.8) {\tiny{$t_2$}};
\node[anchor=east] at (-0.05,0.2) {\tiny{$s_0$}};
\node[anchor=west] at (1.05,1.1) {\tiny\textcolor{black}{$t_3$}};
\node[anchor=west] at (1.05,0.5) {\tiny\textcolor{black}{$t_4$}};
\node[anchor=west] at (1.05,-0.1) {\tiny\textcolor{black}{$s_1$}};
\end{scope}

\begin{scope}[xshift=5.6cm]
    \draw[-,thick] (0,1.1)--(1,0.8);
    \draw[-,thick] (0,0.5)--(1,0.8);     
    \draw[-,thick] (0,0.5)--(1,0.2);      
    \draw[-,thick] (0,-0.1)--(1,0.2);    
\vertex[fill=black, minimum size=3pt] at (0,1.1) {};
\vertex[fill=black, minimum size=3pt] at (0,0.5) {};
\vertex[fill=black, minimum size=3pt] at (0,-0.1) {};
\vertex[fill=black, minimum size=3pt] at (1.0,0.8) {};
\vertex[fill=black, minimum size=3pt] at (1.0,0.2) {};
\node[anchor=east] at (-0.05,1.1) {\tiny{$t_1$}};
\node[anchor=east] at (-0.05,0.5) {\tiny{$t_2t_4$}};
\node[anchor=east] at (-0.05,-0.1) {\tiny{$s_0t_4$}};
\node[anchor=west] at (1.05,0.8) {\tiny\textcolor{black}{$t_5$}};
\node[anchor=west] at (1.05,0.2) {\tiny\textcolor{black}{$s_2$}};
\end{scope}

\begin{scope}[xshift=8.5cm]
    \draw[-,thick] (0,1)--(1,1);	    
    \draw[-,thick] (0,1)--(1,0);		
    \draw[-,thick] (0,0)--(1,0);		
    \vertex[fill=black, minimum size=3pt] at (0,1) {};
    \vertex[fill=black, minimum size=3pt] at (0,0) {};
    \vertex[fill=black, minimum size=3pt] at (1,1) {};
    \vertex[fill=black, minimum size=3pt] at (1,0) {};
\node[anchor=east] at (-0.05,1) {\tiny{$s_0s_1$}};
\node[anchor=east] at (-0.05,0) {\tiny{$t_2t_3$}};
\node[anchor=west] at (1.05,1) {\tiny\textcolor{black}{$t_6$}};
\node[anchor=west] at (1.05,0) {\tiny\textcolor{black}{$s_3$}};
\end{scope}

\begin{scope}[xshift=11.9cm]
\draw[-,thick] (0.0,1.5)--(1.0,0.9); 
\draw[-,thick] (0.0,1.1)--(1.0,0.9);  
\draw[-,thick] (0.0,0.7)--(1.0,0.9);  
\draw[-,thick] (0.0,0.3)--(1.0,0.9);  
\draw[-,thick] (0.0,0.3)--(1.0,0.1);  
\draw[-,thick] (0.0,-0.1)--(1.0,0.1);  
\draw[-,thick] (0.0,-0.5)--(1.0,0.1);  
\vertex[fill=black, minimum size=3pt] at (0.0,1.5) {};
\vertex[fill=black, minimum size=3pt] at (0.0,1.1) {};
\vertex[fill=black, minimum size=3pt] at (0.0,0.7) {};
\vertex[fill=black, minimum size=3pt] at (0.0,0.3) {};
\vertex[fill=black, minimum size=3pt] at (0.0,-0.1) {};
\vertex[fill=black, minimum size=3pt] at (0.0,-0.5) {};
\vertex[fill=black, minimum size=3pt] at (1.0,0.9) {};
\vertex[fill=black, minimum size=3pt] at (1.0,0.1) {};
\node[anchor=east] at (-0.05,1.5) {\tiny{$t_1t_5$}};
\node[anchor=east] at (-0.05,1.1) {\tiny{$t_2t_4t_5$}};
\node[anchor=east] at (-0.05,0.7) {\tiny{$s_0s_1s_3$}};
\node[anchor=east] at (-0.05,0.3) {\tiny{$t_2t_3s_3$}};
\node[anchor=east] at (-0.05,-0.1) {\tiny{$t_2t_4s_2$}};
\node[anchor=east] at (-0.05,-0.5) {\tiny{$s_0t_4s_2$}};
\node[anchor=west] at (1.05,0.9) {\tiny\textcolor{black}{$t_7$}};
\node[anchor=west] at (1.05,0.1) {\tiny\textcolor{black}{$s_4$}};
\end{scope}

\begin{scope}[xshift=15.5cm]
\draw[-,thick] (0.0,1.55)--(1.0,0.5); 
\draw[-,thick] (0.0,1.25)--(1.0,0.5); 
\draw[-,thick] (0.0,0.95)--(1.0,0.5);  
\draw[-,thick] (0.0,0.65)--(1.0,0.5);  
\draw[-,thick] (0.0,0.35)--(1.0,0.5);  
\draw[-,thick] (0.0,0.05)--(1.0,0.5);  
\draw[-,thick] (0.0,-0.25)--(1.0,0.5);  
\draw[-,thick] (0.0,-0.55)--(1.0,0.5);  
\vertex[fill=black, minimum size=3pt] at (0.0,1.55) {};
\vertex[fill=black, minimum size=3pt] at (0.0,1.25) {};
\vertex[fill=black, minimum size=3pt] at (0.0,0.95) {};
\vertex[fill=black, minimum size=3pt] at (0.0,0.65) {};
\vertex[fill=black, minimum size=3pt] at (0.0,0.35) {};
\vertex[fill=black, minimum size=3pt] at (0.0,0.05) {};
\vertex[fill=black, minimum size=3pt] at (0.0,-0.25) {};
\vertex[fill=black, minimum size=3pt] at (0.0,-0.55) {};
\vertex[fill=black, minimum size=3pt] at (1.0,0.5) {};
\node[anchor=east] at (-0.05,1.55) {\tiny{$s_0s_1t_6$}};
\node[anchor=east] at (-0.05,1.25) {\tiny{$t_1t_5t_7$}};
\node[anchor=east] at (-0.05,0.95) {\tiny{$t_2t_3t_5t_7$}};
\node[anchor=east] at (-0.05,0.65) {\tiny{$s_0s_1s_3t_7$}};
\node[anchor=east] at (-0.05,0.35) {\tiny{$t_2t_3s_3t_7$}};
\node[anchor=east] at (-0.05,0.05) {\tiny{$t_2t_3s_3s_4$}};
\node[anchor=east] at (-0.05,-0.25) {\tiny{$t_2t_4s_2s_4$}};
\node[anchor=east] at (-0.05,-0.55) {\tiny{$s_0t_4s_2s_4$}};

\node[anchor=west] at (1.05,0.5) {\tiny\textcolor{black}{$y_0$}};
\end{scope}

\end{scope}
\end{tikzpicture}
    \caption{The groves $\Gamma(\pi)$ and $\Gamma(\rho)$, corresponding to the permutation flows $\pi$ and $\rho$ in Figure~\ref{fig:permutation_flow_example}.}
    \label{fig:grove_1001}
\end{figure}
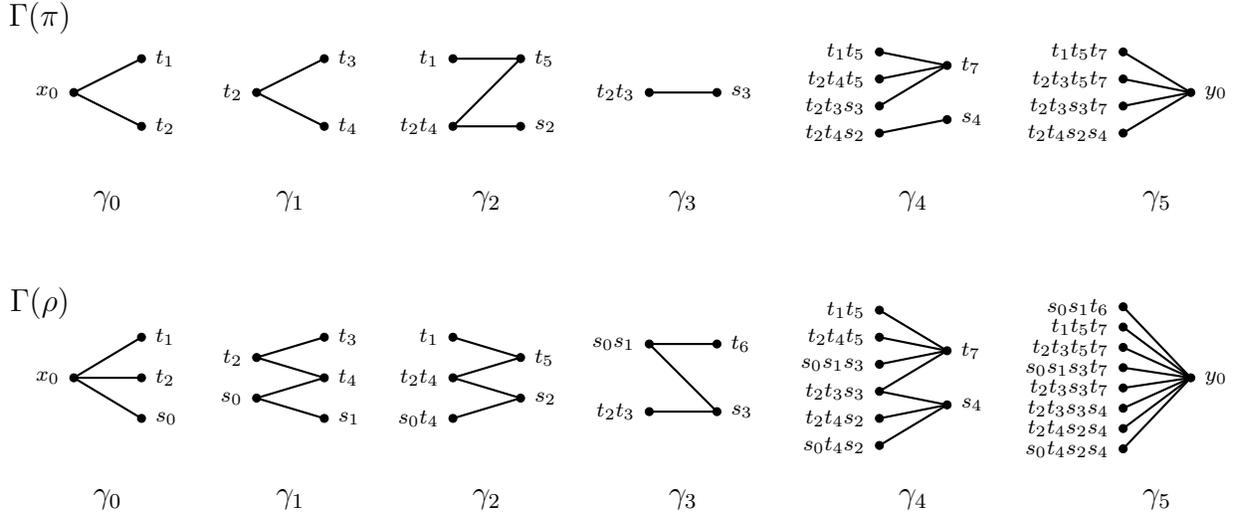

There is a natural labeling $\lambda(\cdot)$ of groves defined on prefixes $P$ and edges $(P,e)$. We build this labeling from $\gamma_0$ to $\gamma_n$ as follows. Initiate the process by defining $\lambda(x_0)=x_0$. When every prefix $P\in L(\gamma_v)$ is labeled, label the edges of $\gamma_v$ by 
\[\lambda((P,e)) := \left.\begin{cases}
        e & \textup{if there is an edge $(P,e')\prec(P,e)$ in $\gamma_v$} \\
        \lambda(P) & \textup{if there is no edge $(P,e')\prec(P,e)$ in $\gamma_v$}
    \end{cases}\right\}.\]
Finally, pass these labels forward to prefixes $Pe\in L(\gamma_{\head(e)})$ by defining $\lambda(Pe)=\lambda((P,e))$. See Figure~\ref{fig:grove_path_labels_1001}.

\begin{figure}[ht!]
    \input{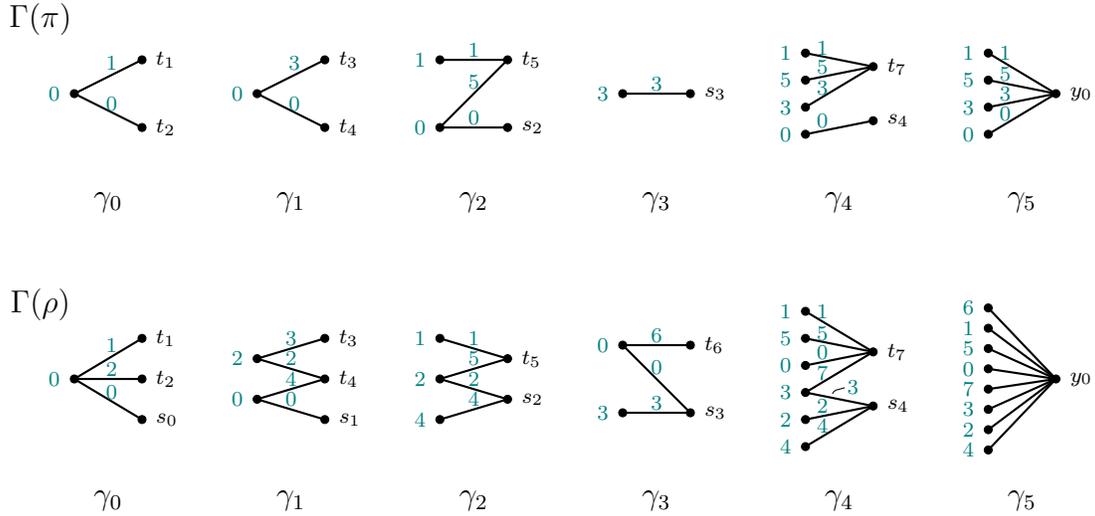}
    \caption{The natural labelings of left vertices and edges of $\Gamma(\pi)$ and $\Gamma(\rho)$. The edge labels on $\gamma_5$ in $\Gamma(\pi)$ have been suppressed for legibility; they are the same as the labels on their incident left vertices.}
    \label{fig:grove_path_labels_1001}
\end{figure}

The correspondences between groves and the other objects in this section are straightforward. Let $\Gamma\in\Groves(G,F)$. The corresponding clique $\calC(\Gamma)$ is exactly $L(\gamma_n)$. The corresponding vine $\calV(\Gamma)$ is $\Prefixes(\Gamma)$. The corresponding permutation flow $\pi=\pi(\Gamma)$ is constructed by assigning to $\pi(e)$ the word formed by reading the labels of the edges incident to $e\in \gamma_{\tail(e)}$ from lowest to highest. As a consequence, when $\Gamma$ is saturated, the corresponding integer flow $\psi=\psi(\Gamma)$ is determined by defining $\psi(e)$ to be one less than the number of incident edges to~$e$ in $\gamma_{\tail(e)}$.

The general case of groves, which also involves grove shuffles, is treated in Section~\ref{sec:grove_shuffles}.

\subsection{The face poset of \texorpdfstring{$DKK(G,F)$}{DKK(G,F)}}
\label{sec:face_poset}
\phantom{W}

Every framing on $G$ induces a unimodular triangulation of $\calF_G$. We provide the combinatorial description of face inclusion in this simplicial complex for each object in Section~\ref{sec:permutationflows}. We use $\subseteq$ to represent inclusion and $\subsetdot$ to represent a covering relation in this face poset. 

Two cliques satisfy $\calC\subseteq \calD$ when $\calC$ is contained in $\calD$; if further $|\calD|=|\calC|+1$, then $\calC\subsetdot \calD$. Apply Equation~\eqref{eq:triangle} to determine the corresponding simplex $\triangle_{\calC}$ to $\calC$.

Two vines satisfy $\calV\subseteq\calW$ when $\calV$ is a subposet of $\calW$. If further $|\Routes(\calW)|=|\Routes(\calV)|+1$, then $\calV\subsetdot \calW$. This is equivalent to pruning the vine by deleting a path connecting a leaf of the vine to its first edge that is a direct split.

Two groves satisfy $\Gamma\subseteq \Theta$ when $\Gamma$ is a subgraph of $\Theta$. If further $|L(\theta_n)|=|L(\gamma_n)|+1$, then $\Gamma\subsetdot \Theta$.

Two permutation flows satisfy $\pi\subseteq \rho$ when $\pi$ is a split reduction of $\rho$, as defined in Definition~\ref{definition:split_reduction}. If further the corresponding final summaries satisfy $|\zeta(\rho)|=|\zeta(\pi)|+1$, then $\pi\subsetdot \rho$.

\begin{example}
Figures~\ref{fig:face_oru} and \ref{fig:face_poset_oru} give two different views of the face poset of the oruga graph $\oru(2)$ with the planar framing. Figure~\ref{fig:face_oru} labels each face of the triangulation of $\calF_G$ by its corresponding permutation flow. The two simplices correspond to the two saturated permutation flows. Figure~\ref{fig:face_poset_oru} gives the face poset and includes the empty permutation flow, corresponding to the empty face of the polytope. When one permutation flow covers another, they differ by a split reduction.
\end{example}

\begin{figure}[h!]
    \centering
    \scalebox{.7}{\input{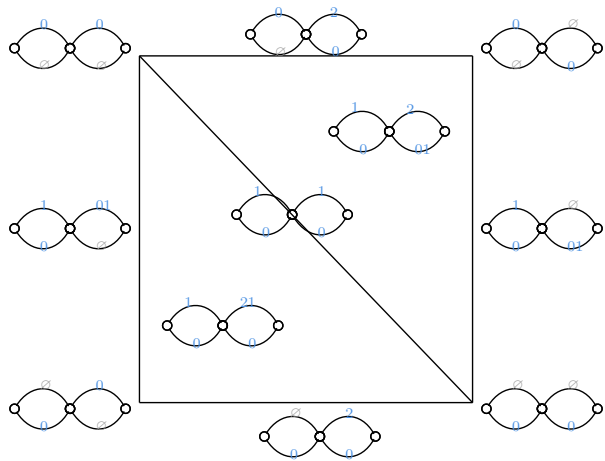}}
    \caption{The faces of the triangulation $\DKK(\oru(2),F)$ in terms of permutation flows, where $F$ is the planar framing.}
    \label{fig:face_oru}
\end{figure}

\begin{figure}[h!]
    \centering
    \scalebox{.7}{\input{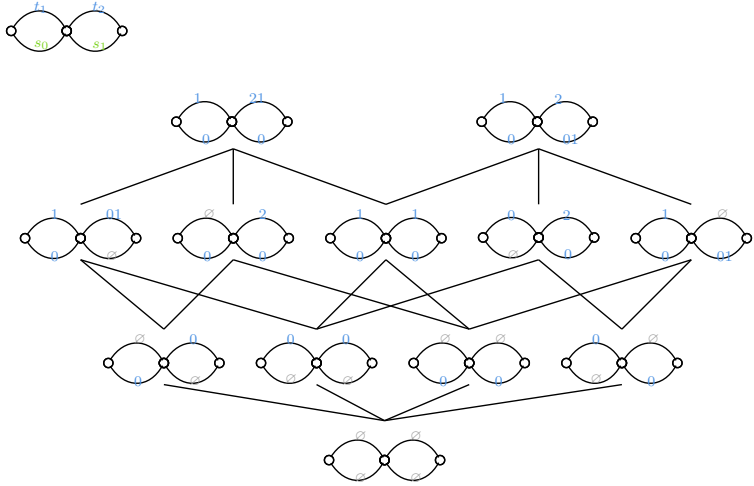}}
    \caption{Face poset of $\DKK(\oru(2),F)$ in terms of permutation flows, where $F$ is the planar framing.}
    \label{fig:face_poset_oru}
\end{figure}

\subsection{An abundance of examples}
\label{sec:permuflow_examples}
\phantom{W}

The Lidskii formula of Theorem \ref{thm.genlidskii} already reveals the rich combinatorics that is hidden behind the calculation of volumes of flow polytopes. The three authors together with Benedetti, Harris, Khare, and Morales developed in \cite{BenedettiGonzalezDleonHanusaHarrisKhareMoralesYip2019} a combinatorial model that gives an interpretation to the Lidskii formula in terms of certain decorated lattice paths, which are also known as generalized parking functions. To specific choices of a graph $G$ and a balanced netflow vector $\ba$, such that $a_i\ge 0$ for $i=0,\dots, n-1$ and $a_n\le 0$, we associate a specific combinatorial subfamily. In the case where $\ba = \be_0-\be_n$ we already recover subfamilies which are central in the combinatorics literature.  Combinatorial models giving an interpretation to the Kostant partition function also give, by Theorem \ref{thm.genlidskii}, an interpretation to the volume formula in this case, see for example \cite{BenedettiHanusaHarrisMoralesSimpson2020} for an interpretation in terms of juggling sequences. 

Among others, the following classical families of combinatorial numbers have appeared in the study of volumes of flow polytopes of the form $\calF_G$ (see for example \cite{MeszarosMorales2019,GonzalezDleonMoralesPhilippeTamayoYip2025,TamayoJimenez2023,BrunnerHanusa2024,GonzalezHanusaMoralesYip2023}):

\begin{itemize}
    \item The number $2^n$ of binary sequences of length $n$,
    \item the factorial numbers $n!=1\cdot 2\cdots n$ which enumerate permutations of $[n]$, 
    \item the Catalan numbers $C_n=\frac{1}{n+1}\binom{2n}{n}$ that enumerate Dyck paths of length $2n$,
    \item the Euler numbers that enumerate up-down permutations of $n$ (see \cite{Stanley2010}).
\end{itemize}
When a more general $\ba$ is allowed, the combinatorics of parking functions $(n+1)^{n-1}$ also comes into play.

All of these families admit a composition generalization in the context of the signature combinatorics studied by Ceballos and the first author in \cite{CeballosGonzalezDLeon2019}. For a composition $\bs=(s_1,\dots,s_n)$ there are volumes of flow polytopes given by the following numbers:
\begin{itemize}
    \item the number $$s_1\cdot s_2 \cdots s_n$$ which enumerates $\bs$-ary sequences of length $n$, where the entry $i$ is in $[s_i]$ for $i\in [n]$,
    \item the $\bs$-factorial number $$n\stackrel {\bs}{!}:=\prod_{i=1}^{n-1}(1+\sum_{j=1}^is_j)=(1+s_1)\cdot(1+s_1+s_2)\cdots (1+s_1+s_2+\cdots+s_{n-1})$$ which enumerates $\bs$-Stirling permutations (see \cite{Gessel2020, CeballosPons2024-1,CeballosGonzalezDLeon2019}),
    \item the $\bs$-multinomial number $$\binom{|\bs|}{\bs}:=\binom{s_1+s_2+\cdots+s_n}{s_1,s_2,\dots,s_n}$$ which enumerates $\bs$-multipermutations,
    \item the $\bs$-Catalan number given by Kreweras' \cite{Kreweras1965}  determinantal formula   $$\Cat(\bs):=\det\left [\binom{s_1+s_2+\dots+s_{n-j}+1}{j-i-1}\right]$$ which enumerates $\bs$-Dyck paths,
    \item the $\bs$-Euler number $E_{\bs}$ of up-down permutations with up-down pattern given by $\bs$, which also enumerates Standard Young Tableaux of the hook determined by $\bs$.
    \item the number $\PF(\bs)$ of generalized $\bs$-parking functions (see \cite{PitmanStanley2002}).
\end{itemize}

The graphs $G$ and netflow vectors $\ba$ which realize all of these are listed in Table \ref{table:signature_combinatorics} and illustrated in Figure \ref{fig:s-graphs}.

\begin{figure}
    \centering
    \scalebox{.75}{\input{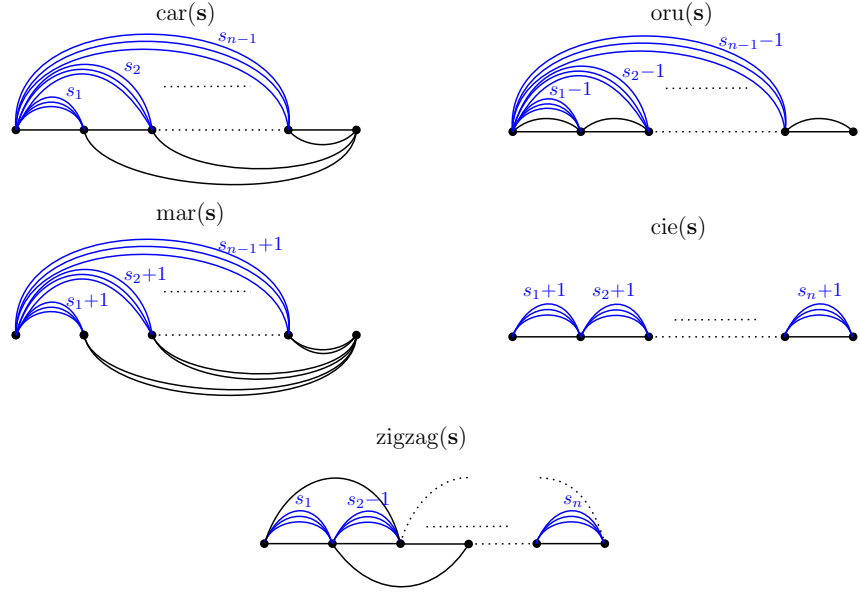}}
    \caption{Graphs from Table \ref{table:signature_combinatorics}}
    \label{fig:s-graphs}
\end{figure}

\begin{table}[]
    \centering
\begin{tabular}{|c|c|c|c|c|}
\hline 
 Graph & Name & $\ba$ & Volume & Reference \\
\hline 
 $\oru(\bs)$ & $\bs$-oruga & $\displaystyle e_{0} -e_{1}$ & $n\stackrel{\bs}{!}$ & \cite{GonzalezDleonMoralesPhilippeTamayoYip2025} \\
\hline 
 $\car(\bs)$ & $\bs$-caracol & $\displaystyle e_{0} -e_{1}$ & $\Cat(\bs)$ & \cite{BellGonzalezMayorgaYip2023} \\
\hline 
 $\mar(\bs)$ & $\bs$-mariposa & $\displaystyle e_{0} -e_{1}$ & $\prod_{i=1}^{s_{n-1}} s_{i}$ &  This work\\
\hline 
 $\cie(\bs)$ & $\bs$-cienpies & $\displaystyle e_{0} -e_{1}$ & $\displaystyle \binom{|\bs|}{\bs}$ &  This work\\
\hline 
 $\zigzag(\bs)$ & $\bs$-zigzag & $\displaystyle e_{0} -e_{1}$ & $\displaystyle E_{\bs}$ &  \cite{GonzalezHerreraLauveNair2025} \\
\hline 
 $\PS_n$ & Pitman-Stanley & $\displaystyle \sum\limits _{i=0}^{n-1} s_{i}( e_{i} -e_{n})$ & $\PF(\bs)$ & \cite{MeszarosMorales2019}  \\
 \hline
\end{tabular}
    \caption{Table of $\bs$-combinatorics appearing in the literature as volumes of flow polytopes.}
    \label{table:signature_combinatorics}
\end{table}

There is another family of graphs which is also of central relevance.
When $G=K_{n}$ is the complete graph and $\ba=\be_1-\be_n$, the volume of the flow polytope $\calF_{K_n}$ is the product of consecutive Catalan numbers 
\begin{equation}\label{equation:product_of_catalan}
  \prod_{i=1}^{n-2}\Cat(i).  
\end{equation}
The polytope $\calF_{K_n}$ was studied by Chan, Robbins, and Yuen \cite{ChanRobbinsYuen2000} who conjecture that their volume was given by \eqref{equation:product_of_catalan}. This conjecture was proved by Zeilberger \cite{Zeilberger1999} using a constant term identity and it is still an open problem to find a combinatorial proof of of this formula. When $\ba=(1,\ldots,1,-n)$ the polytope is known as the Tesler polytope and it was proved in \cite{MeszarosMoralesRhoades2017} that its volume is given by \eqref{equation:product_of_catalan} times the number of standard Young tableaux of staircase shape $(n-2,n-3,\dots,2,1)$. Other related polytopes on the complete graph have also been studied \cite{CorteelKimMeszaros2021}. Other interesting families of flow polytopes connected to many areas are the Gelfand-Tsetlin polytopes studied in \cite{LiuMeszarosStDizier2019} and also the generalized Pitman-Stanley polytopes defined and studied in \cite{DuganHegartyMoralesRaymond2025}. 

When we move to the context of permutation flows, the number $\vol \calF_G$ translates bijectively into a particular family of permutations of $\sym_{[0,d]}$. Indeed, the quantity $\vol \calF_G$ is given by the number of maximal simplices in $\DKK(G,F)$ which are in bijection with the set $\PermutationFlows(G,F))$ by Theorem \ref{theorem:triangulation}. Theorem \ref{thm:finalsummary_10001} gives a bijection of this set  to $\FinalSummaries(G,F)$. Often the final summaries of permutation flows give already an enconding of the numbers $\vol \calF_G$ into a family that is of combinatorial interest, and in some other case a simple bijection will enconde these summaries into such a family. We illustrate this below with the examples of two framings on $\oru(3)$ and two framings on $\car(5)$ (see Figures \ref{fig:222poset} and \ref{fig:caracolposets}).

More generally, in the context of the permutation flow shuffle triangulation of Section \ref{sec:triangulation_flow_polytopes}, we use Theorems \ref{theorem:triangulation} and \ref{thm:finalsummary} to encode the volumes as pairs $(\zeta(\pi),\sigma_{\pi})$ where $\zeta(\pi) \in \FinalSummaries(\hatG,\hatF)$ and $(\pi,\sigma_{\pi}) \in \SatPermutationFlowShuffles(\hatG,\hatF)$. 

In the following, we discuss two fundamental examples where we organize the elements of $\PermutationFlows(G,F)$ according to the structure of the weak order $W(G,F)$.

\begin{example}[The oruga graph]
  One of the most fundamental and visited examples in the literature of flow polytopes is given by the polytope $\calF_{\oru(n)}$ on the oruga graph $\oru(n)$. This graph is defined by an edge set with two copies of an edge
$(i-1,i)$ for $i\in[n]$. The polytope $\calF_{\oru(n)}$ is integrally equivalent to the $n$-hypercube.  

In Figure \ref{fig:222poset} we illustrate the weak order $W(\oru(3),F)$ on $\SatPermutationFlows(\oru(3),F)$ using two framings. In the left a framing that we call \defn{planar} and in the right a framing we call \defn{twisted}.

It was shown in \cite{GonzalezDleonMoralesPhilippeTamayoYip2025} that $W(\oru(3),F)$ with the planar framing is isomorphic to the (right) weak order on the set $\sym_3$ of permutations of $[3]$. Indeed, the bijection is obtained by reading the final summary of the permutation flow without the leading $0$ as is illustrated in Figure \ref{fig:222final_summaries} (left). On the other hand the authors of \cite{GonzalezHanusaMoralesYip2026} found, when $F$ is the twisted framing, that  the final summary map $\zeta$ directly gives an isomorphism between $W(\oru(3),F)$ an the lattice of circular permutations introduced by Abram, Chapelier-Laget, and Reutenauer in \cite{AbramChapelier-LagetReutenauer2021} as is illustrated in Figure \ref{fig:222final_summaries} (right).
A generalization $\oru(s)$ of $\oru(n)$ to a signature $\bs$ was defined and studied in \cite{GonzalezDleonMoralesPhilippeTamayoYip2025}, such that $\oru(n)=\oru(1,1,\dots,1)$.
\end{example}

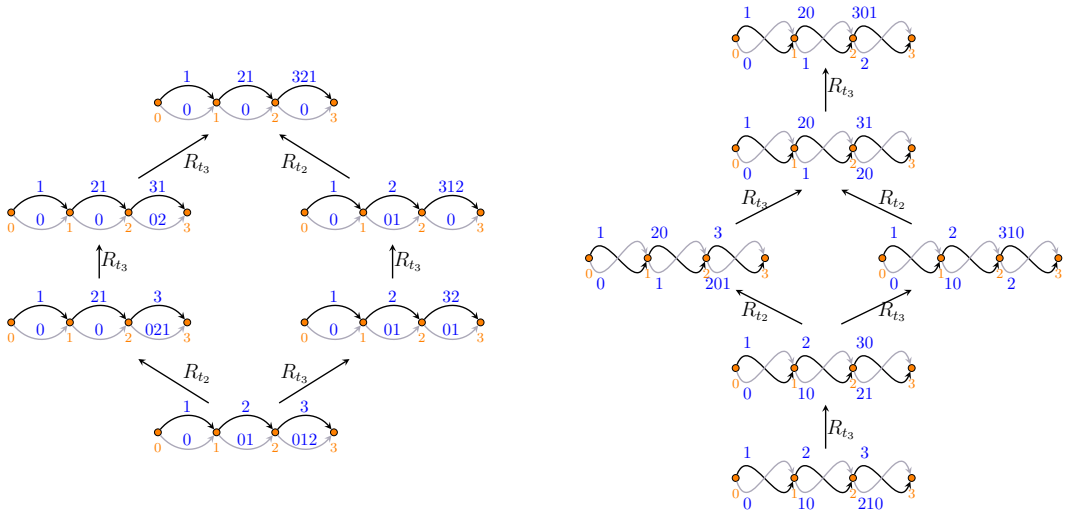
\begin{figure}[t]
    \centering
    \vspace{-.3in}
    \raisebox{0.85cm}{
    \scalebox{0.65}{\begin{tikzpicture}[xscale=0.8]

\begin{scope}
    \draw[-stealth, thick] (1.3,0.6) to (-0.5,1.5);
    \draw[-stealth, thick] (3.1,0.6) to (4.9,1.5);
    \draw[-stealth, thick] (-1.5,3.15) to (-1.5,3.85);
    \draw[-stealth, thick] (6,3.15) to (6,3.85);
    \draw[-stealth, thick] (-0.5,5.2) to (1.3,6.1);
    \draw[-stealth, thick] (4.9,5.2) to (3.1,6.1);
\node[] at (1,1.2) {$R_{t_2}$};
\node[] at (3.5,1.2) {$R_{t_3}$};
\node[] at (6.4,3.5) {$R_{t_3}$};
\node[] at (-1.1,3.5) {$R_{t_3}$};
\node[] at (1,5.5) {$R_{t_3}$};
\node[] at (3.5,5.5) {$R_{t_2}$};
\end{scope}

\begin{scope}[scale=1.5]
    \vertex[fill=orange, minimum size=4pt, label=below:{\tiny\textcolor{orange}{$0$}}](v0) at (0,0) {};
	\vertex[fill=orange, minimum size=4pt, label=below:{\tiny\textcolor{orange}{$1$}}](v1) at (1,0) {};
	\vertex[fill=orange, minimum size=4pt, label=below:{\tiny\textcolor{orange}{$2$}}](v2) at (2,0) {};
	\vertex[fill=orange, minimum size=4pt, label=below:{\tiny\textcolor{orange}{$3$}}](v3) at (3,0) {};

    \draw[-stealth, thick, color=lavgray] (v0) to [out=-45,in=225] (v1);
    \draw[-stealth, thick] (v0) to [out=45,in=135] (v1);
    \draw[-stealth, thick, color=lavgray] (v1) to [out=-45,in=225] (v2);
    \draw[-stealth, thick] (v1) to [out=45,in=135] (v2);
    \draw[-stealth, thick, color=lavgray] (v2) to [out=-45,in=225] (v3);
    \draw[-stealth, thick] (v2) to [out=45,in=135] (v3);

    \node[] at (0.5, .35) {\footnotesize\textcolor{blue}{$1$}};
    \node[] at (0.5, -.1) {\footnotesize\textcolor{blue}{$0$}};

    \node[] at (1.5, .35) {\footnotesize\textcolor{blue}{$2$}};
    \node[] at (1.5, -.1) {\footnotesize\textcolor{blue}{$01$}};
    
     \node[] at (2.5, .35) {\footnotesize\textcolor{blue}{$3$}};
     \node[] at (2.5, -.1) {\footnotesize\textcolor{blue}{$012$}};
\end{scope}

\begin{scope}[scale=1.5,shift={(-2.5,1.5)}]
    \vertex[fill=orange, minimum size=4pt, label=below:{\tiny\textcolor{orange}{$0$}}](v0) at (0,0) {};
	\vertex[fill=orange, minimum size=4pt, label=below:{\tiny\textcolor{orange}{$1$}}](v1) at (1,0) {};
	\vertex[fill=orange, minimum size=4pt, label=below:{\tiny\textcolor{orange}{$2$}}](v2) at (2,0) {};
	\vertex[fill=orange, minimum size=4pt, label=below:{\tiny\textcolor{orange}{$3$}}](v3) at (3,0) {};

    \draw[-stealth, thick, color=lavgray] (v0) to [out=-45,in=225] (v1);
    \draw[-stealth, thick] (v0) to [out=45,in=135] (v1);
    \draw[-stealth, thick, color=lavgray] (v1) to [out=-45,in=225] (v2);
    \draw[-stealth, thick] (v1) to [out=45,in=135] (v2);
    \draw[-stealth, thick, color=lavgray] (v2) to [out=-45,in=225] (v3);
    \draw[-stealth, thick] (v2) to [out=45,in=135] (v3);

    \node[] at (0.5, .35) {\footnotesize\textcolor{blue}{$1$}};
    \node[] at (0.5, -.1) {\footnotesize\textcolor{blue}{$0$}};

    \node[] at (1.5, .35) {\footnotesize\textcolor{blue}{$21$}};
    \node[] at (1.5, -.1) {\footnotesize\textcolor{blue}{$0$}};
    
     \node[] at (2.5, .35) {\footnotesize\textcolor{blue}{$3$}};
     \node[] at (2.5, -.1) {\footnotesize\textcolor{blue}{$021$}};
\end{scope}

\begin{scope}[scale=1.5,shift={(2.5,1.5)}]
    \vertex[fill=orange, minimum size=4pt, label=below:{\tiny\textcolor{orange}{$0$}}](v0) at (0,0) {};
	\vertex[fill=orange, minimum size=4pt, label=below:{\tiny\textcolor{orange}{$1$}}](v1) at (1,0) {};
	\vertex[fill=orange, minimum size=4pt, label=below:{\tiny\textcolor{orange}{$2$}}](v2) at (2,0) {};
	\vertex[fill=orange, minimum size=4pt, label=below:{\tiny\textcolor{orange}{$3$}}](v3) at (3,0) {};

    \draw[-stealth, thick, color=lavgray] (v0) to [out=-45,in=225] (v1);
    \draw[-stealth, thick] (v0) to [out=45,in=135] (v1);
    \draw[-stealth, thick, color=lavgray] (v1) to [out=-45,in=225] (v2);
    \draw[-stealth, thick] (v1) to [out=45,in=135] (v2);
    \draw[-stealth, thick, color=lavgray] (v2) to [out=-45,in=225] (v3);
    \draw[-stealth, thick] (v2) to [out=45,in=135] (v3);

    \node[] at (0.5, .35) {\footnotesize\textcolor{blue}{$1$}};
    \node[] at (0.5, -.1) {\footnotesize\textcolor{blue}{$0$}};

    \node[] at (1.5, .35) {\footnotesize\textcolor{blue}{$2$}};
    \node[] at (1.5, -.1) {\footnotesize\textcolor{blue}{$01$}};
    
     \node[] at (2.5, .35) {\footnotesize\textcolor{blue}{$32$}};
     \node[] at (2.5, -.1) {\footnotesize\textcolor{blue}{$01$}};
\end{scope}

\begin{scope}[scale=1.5,shift={(-2.5,3)}]
    \vertex[fill=orange, minimum size=4pt, label=below:{\tiny\textcolor{orange}{$0$}}](v0) at (0,0) {};
	\vertex[fill=orange, minimum size=4pt, label=below:{\tiny\textcolor{orange}{$1$}}](v1) at (1,0) {};
	\vertex[fill=orange, minimum size=4pt, label=below:{\tiny\textcolor{orange}{$2$}}](v2) at (2,0) {};
	\vertex[fill=orange, minimum size=4pt, label=below:{\tiny\textcolor{orange}{$3$}}](v3) at (3,0) {};

    \draw[-stealth, thick, color=lavgray] (v0) to [out=-45,in=225] (v1);
    \draw[-stealth, thick] (v0) to [out=45,in=135] (v1);
    \draw[-stealth, thick, color=lavgray] (v1) to [out=-45,in=225] (v2);
    \draw[-stealth, thick] (v1) to [out=45,in=135] (v2);
    \draw[-stealth, thick, color=lavgray] (v2) to [out=-45,in=225] (v3);
    \draw[-stealth, thick] (v2) to [out=45,in=135] (v3);

    \node[] at (0.5, .35) {\footnotesize\textcolor{blue}{$1$}};
    \node[] at (0.5, -.1) {\footnotesize\textcolor{blue}{$0$}};

    \node[] at (1.5, .35) {\footnotesize\textcolor{blue}{$21$}};
    \node[] at (1.5, -.1) {\footnotesize\textcolor{blue}{$0$}};
    
     \node[] at (2.5, .35) {\footnotesize\textcolor{blue}{$31$}};
     \node[] at (2.5, -.1) {\footnotesize\textcolor{blue}{$02$}};
\end{scope}

\begin{scope}[scale=1.5,shift={(2.5,3)}]
    \vertex[fill=orange, minimum size=4pt, label=below:{\tiny\textcolor{orange}{$0$}}](v0) at (0,0) {};
	\vertex[fill=orange, minimum size=4pt, label=below:{\tiny\textcolor{orange}{$1$}}](v1) at (1,0) {};
	\vertex[fill=orange, minimum size=4pt, label=below:{\tiny\textcolor{orange}{$2$}}](v2) at (2,0) {};
	\vertex[fill=orange, minimum size=4pt, label=below:{\tiny\textcolor{orange}{$3$}}](v3) at (3,0) {};

    \draw[-stealth, thick, color=lavgray] (v0) to [out=-45,in=225] (v1);
    \draw[-stealth, thick] (v0) to [out=45,in=135] (v1);
    \draw[-stealth, thick, color=lavgray] (v1) to [out=-45,in=225] (v2);
    \draw[-stealth, thick] (v1) to [out=45,in=135] (v2);
    \draw[-stealth, thick, color=lavgray] (v2) to [out=-45,in=225] (v3);
    \draw[-stealth, thick] (v2) to [out=45,in=135] (v3);

    \node[] at (0.5, .35) {\footnotesize\textcolor{blue}{$1$}};
    \node[] at (0.5, -.1) {\footnotesize\textcolor{blue}{$0$}};

    \node[] at (1.5, .35) {\footnotesize\textcolor{blue}{$2$}};
    \node[] at (1.5, -.1) {\footnotesize\textcolor{blue}{$01$}};
    
     \node[] at (2.5, .35) {\footnotesize\textcolor{blue}{$312$}};
     \node[] at (2.5, -.1) {\footnotesize\textcolor{blue}{$0$}};
\end{scope}

\begin{scope}[scale=1.5,shift={(0,4.5)}]
    \vertex[fill=orange, minimum size=4pt, label=below:{\tiny\textcolor{orange}{$0$}}](v0) at (0,0) {};
	\vertex[fill=orange, minimum size=4pt, label=below:{\tiny\textcolor{orange}{$1$}}](v1) at (1,0) {};
	\vertex[fill=orange, minimum size=4pt, label=below:{\tiny\textcolor{orange}{$2$}}](v2) at (2,0) {};
	\vertex[fill=orange, minimum size=4pt, label=below:{\tiny\textcolor{orange}{$3$}}](v3) at (3,0) {};

    \draw[-stealth, thick, color=lavgray] (v0) to [out=-45,in=225] (v1);
    \draw[-stealth, thick] (v0) to [out=45,in=135] (v1);
    \draw[-stealth, thick, color=lavgray] (v1) to [out=-45,in=225] (v2);
    \draw[-stealth, thick] (v1) to [out=45,in=135] (v2);
    \draw[-stealth, thick, color=lavgray] (v2) to [out=-45,in=225] (v3);
    \draw[-stealth, thick] (v2) to [out=45,in=135] (v3);

    \node[] at (0.5, .35) {\footnotesize\textcolor{blue}{$1$}};
    \node[] at (0.5, -.1) {\footnotesize\textcolor{blue}{$0$}};

    \node[] at (1.5, .35) {\footnotesize\textcolor{blue}{$21$}};
    \node[] at (1.5, -.1) {\footnotesize\textcolor{blue}{$0$}};
    
     \node[] at (2.5, .35) {\footnotesize\textcolor{blue}{$321$}};
     \node[] at (2.5, -.1) {\footnotesize\textcolor{blue}{$0$}};
\end{scope}

\end{tikzpicture}}}
    \hspace{0.8cm}
    \scalebox{0.65}{\begin{tikzpicture}[xscale=0.8]

\begin{scope}
    \draw[-stealth, thick] (2.3,0.6) to (2.3,1.5);
    \draw[-stealth, thick] (2.7,3.15) to (4.5,3.85);
    \draw[-stealth, thick] (1.7,3.15) to (0,3.85);
    \draw[-stealth, thick] (0,5.2) to (1.7,5.9);
    \draw[-stealth, thick] (4.5,5.2) to (2.7,5.9);
    \draw[-stealth, thick] (2.3,7.5) to (2.3,8.4);
\node[] at (2.7,1.0) {$R_{t_3}$};
\node[] at (0.5,3.4) {$R_{t_2}$};
\node[] at (4,3.4) {$R_{t_3}$};
\node[] at (0.5,5.7) {$R_{t_3}$};
\node[] at (4,5.7) {$R_{t_2}$};
\node[] at (2.7,8) {$R_{t_3}$};
\end{scope}

\begin{scope}[scale=1.5]
    \vertex[fill=orange, minimum size=4pt, label=below:{\tiny\textcolor{orange}{$0$}}](v0) at (0,0) {};
	\vertex[fill=orange, minimum size=4pt, label=below:{\tiny\textcolor{orange}{$1$}}](v1) at (1,0) {};
	\vertex[fill=orange, minimum size=4pt, label=below:{\tiny\textcolor{orange}{$2$}}](v2) at (2,0) {};
	\vertex[fill=orange, minimum size=4pt, label=below:{\tiny\textcolor{orange}{$3$}}](v3) at (3,0) {};
	
    \draw[-stealth, thick] (v0) .. controls (0.25, .5) and (0.75,-.5) .. (v1);
    \draw[-stealth, thick] (v1) .. controls (1.25, .5) and (1.75,-.5) .. (v2);
    \draw[-stealth, thick] (v2) .. controls (2.25, .5) and (2.75,-.5) .. (v3);
 
    \draw[-stealth, thick, color=lavgray] (v0) .. controls (0.25, -.6) and (0.75,.6) .. (v1);
    \draw[-stealth, thick, color=lavgray] (v1) .. controls (1.25, -.6) and (1.75,.6) .. (v2);
    \draw[-stealth, thick, color=lavgray] (v2) .. controls (2.25, -.6) and (2.75,.6) .. (v3);

    \node[] at (0.2, .35) {\footnotesize\textcolor{blue}{$1$}};
    \node[] at (0.2, -.35) {\footnotesize\textcolor{blue}{$0$}};
    \node[] at (1.2, .35) {\footnotesize\textcolor{blue}{$2$}};
    \node[] at (1.2, -.35) {\footnotesize\textcolor{blue}{$10$}};
    \node[] at (2.2, .35) {\footnotesize\textcolor{blue}{$3$}};
    \node[] at (2.3, -.35) {\footnotesize\textcolor{blue}{$210$}};
\end{scope}

\begin{scope}[scale=1.5, shift={(0,1.5)}]
    \vertex[fill=orange, minimum size=4pt, label=below:{\tiny\textcolor{orange}{$0$}}](v0) at (0,0) {};
	\vertex[fill=orange, minimum size=4pt, label=below:{\tiny\textcolor{orange}{$1$}}](v1) at (1,0) {};
	\vertex[fill=orange, minimum size=4pt, label=below:{\tiny\textcolor{orange}{$2$}}](v2) at (2,0) {};
	\vertex[fill=orange, minimum size=4pt, label=below:{\tiny\textcolor{orange}{$3$}}](v3) at (3,0) {};
	
    \draw[-stealth, thick] (v0) .. controls (0.25, .5) and (0.75,-.5) .. (v1);
    \draw[-stealth, thick] (v1) .. controls (1.25, .5) and (1.75,-.5) .. (v2);
    \draw[-stealth, thick] (v2) .. controls (2.25, .5) and (2.75,-.5) .. (v3);
 
    \draw[-stealth, thick, color=lavgray] (v0) .. controls (0.25, -.6) and (0.75,.6) .. (v1);
    \draw[-stealth, thick, color=lavgray] (v1) .. controls (1.25, -.6) and (1.75,.6) .. (v2);
    \draw[-stealth, thick, color=lavgray] (v2) .. controls (2.25, -.6) and (2.75,.6) .. (v3);

    \node[] at (0.2, .35) {\footnotesize\textcolor{blue}{$1$}};
    \node[] at (0.2, -.35) {\footnotesize\textcolor{blue}{$0$}};
    \node[] at (1.2, .35) {\footnotesize\textcolor{blue}{$2$}};
    \node[] at (1.2, -.35) {\footnotesize\textcolor{blue}{$10$}};
    \node[] at (2.2, .35) {\footnotesize\textcolor{blue}{$30$}};
    \node[] at (2.2, -.35) {\footnotesize\textcolor{blue}{$21$}};
\end{scope}

\begin{scope}[scale=1.5, shift={(-2.5,3)}]
    \vertex[fill=orange, minimum size=4pt, label=below:{\tiny\textcolor{orange}{$0$}}](v0) at (0,0) {};
	\vertex[fill=orange, minimum size=4pt, label=below:{\tiny\textcolor{orange}{$1$}}](v1) at (1,0) {};
	\vertex[fill=orange, minimum size=4pt, label=below:{\tiny\textcolor{orange}{$2$}}](v2) at (2,0) {};
	\vertex[fill=orange, minimum size=4pt, label=below:{\tiny\textcolor{orange}{$3$}}](v3) at (3,0) {};
	
    \draw[-stealth, thick] (v0) .. controls (0.25, .5) and (0.75,-.5) .. (v1);
    \draw[-stealth, thick] (v1) .. controls (1.25, .5) and (1.75,-.5) .. (v2);
    \draw[-stealth, thick] (v2) .. controls (2.25, .5) and (2.75,-.5) .. (v3);
 
    \draw[-stealth, thick, color=lavgray] (v0) .. controls (0.25, -.6) and (0.75,.6) .. (v1);
    \draw[-stealth, thick, color=lavgray] (v1) .. controls (1.25, -.6) and (1.75,.6) .. (v2);
    \draw[-stealth, thick, color=lavgray] (v2) .. controls (2.25, -.6) and (2.75,.6) .. (v3);

    \node[] at (0.2, .35) {\footnotesize\textcolor{blue}{$1$}};
    \node[] at (0.2, -.35) {\footnotesize\textcolor{blue}{$0$}};
    \node[] at (1.2, .35) {\footnotesize\textcolor{blue}{$20$}};
    \node[] at (1.2, -.35) {\footnotesize\textcolor{blue}{$1$}};
    \node[] at (2.2, .35) {\footnotesize\textcolor{blue}{$3$}};
    \node[] at (2.2, -.35) {\footnotesize\textcolor{blue}{$201$}};
\end{scope}

\begin{scope}[scale=1.5, shift={(2.5,3)}]
    \vertex[fill=orange, minimum size=4pt, label=below:{\tiny\textcolor{orange}{$0$}}](v0) at (0,0) {};
	\vertex[fill=orange, minimum size=4pt, label=below:{\tiny\textcolor{orange}{$1$}}](v1) at (1,0) {};
	\vertex[fill=orange, minimum size=4pt, label=below:{\tiny\textcolor{orange}{$2$}}](v2) at (2,0) {};
	\vertex[fill=orange, minimum size=4pt, label=below:{\tiny\textcolor{orange}{$3$}}](v3) at (3,0) {};
	
    \draw[-stealth, thick] (v0) .. controls (0.25, .5) and (0.75,-.5) .. (v1);
    \draw[-stealth, thick] (v1) .. controls (1.25, .5) and (1.75,-.5) .. (v2);
    \draw[-stealth, thick] (v2) .. controls (2.25, .5) and (2.75,-.5) .. (v3);
 
    \draw[-stealth, thick, color=lavgray] (v0) .. controls (0.25, -.6) and (0.75,.6) .. (v1);
    \draw[-stealth, thick, color=lavgray] (v1) .. controls (1.25, -.6) and (1.75,.6) .. (v2);
    \draw[-stealth, thick, color=lavgray] (v2) .. controls (2.25, -.6) and (2.75,.6) .. (v3);

    \node[] at (0.2, .35) {\footnotesize\textcolor{blue}{$1$}};
    \node[] at (0.2, -.35) {\footnotesize\textcolor{blue}{$0$}};
    \node[] at (1.2, .35) {\footnotesize\textcolor{blue}{$2$}};
    \node[] at (1.2, -.35) {\footnotesize\textcolor{blue}{$10$}};
    \node[] at (2.2, .35) {\footnotesize\textcolor{blue}{$310$}};
    \node[] at (2.2, -.35) {\footnotesize\textcolor{blue}{$2$}};
\end{scope}

\begin{scope}[scale=1.5, shift={(0,4.5)}]
    \vertex[fill=orange, minimum size=4pt, label=below:{\tiny\textcolor{orange}{$0$}}](v0) at (0,0) {};
	\vertex[fill=orange, minimum size=4pt, label=below:{\tiny\textcolor{orange}{$1$}}](v1) at (1,0) {};
	\vertex[fill=orange, minimum size=4pt, label=below:{\tiny\textcolor{orange}{$2$}}](v2) at (2,0) {};
	\vertex[fill=orange, minimum size=4pt, label=below:{\tiny\textcolor{orange}{$3$}}](v3) at (3,0) {};
	
    \draw[-stealth, thick] (v0) .. controls (0.25, .5) and (0.75,-.5) .. (v1);
    \draw[-stealth, thick] (v1) .. controls (1.25, .5) and (1.75,-.5) .. (v2);
    \draw[-stealth, thick] (v2) .. controls (2.25, .5) and (2.75,-.5) .. (v3);
 
    \draw[-stealth, thick, color=lavgray] (v0) .. controls (0.25, -.6) and (0.75,.6) .. (v1);
    \draw[-stealth, thick, color=lavgray] (v1) .. controls (1.25, -.6) and (1.75,.6) .. (v2);
    \draw[-stealth, thick, color=lavgray] (v2) .. controls (2.25, -.6) and (2.75,.6) .. (v3);

    \node[] at (0.2, .35) {\footnotesize\textcolor{blue}{$1$}};
    \node[] at (0.2, -.35) {\footnotesize\textcolor{blue}{$0$}};
    \node[] at (1.2, .35) {\footnotesize\textcolor{blue}{$20$}};
    \node[] at (1.2, -.35) {\footnotesize\textcolor{blue}{$1$}};
    \node[] at (2.2, .35) {\footnotesize\textcolor{blue}{$31$}};
    \node[] at (2.2, -.35) {\footnotesize\textcolor{blue}{$20$}};
\end{scope}

\begin{scope}[scale=1.5, shift={(0,6)}]
    \vertex[fill=orange, minimum size=4pt, label=below:{\tiny\textcolor{orange}{$0$}}](v0) at (0,0) {};
	\vertex[fill=orange, minimum size=4pt, label=below:{\tiny\textcolor{orange}{$1$}}](v1) at (1,0) {};
	\vertex[fill=orange, minimum size=4pt, label=below:{\tiny\textcolor{orange}{$2$}}](v2) at (2,0) {};
	\vertex[fill=orange, minimum size=4pt, label=below:{\tiny\textcolor{orange}{$3$}}](v3) at (3,0) {};
	
    \draw[-stealth, thick] (v0) .. controls (0.25, .5) and (0.75,-.5) .. (v1);
    \draw[-stealth, thick] (v1) .. controls (1.25, .5) and (1.75,-.5) .. (v2);
    \draw[-stealth, thick] (v2) .. controls (2.25, .5) and (2.75,-.5) .. (v3);
 
    \draw[-stealth, thick, color=lavgray] (v0) .. controls (0.25, -.6) and (0.75,.6) .. (v1);
    \draw[-stealth, thick, color=lavgray] (v1) .. controls (1.25, -.6) and (1.75,.6) .. (v2);
    \draw[-stealth, thick, color=lavgray] (v2) .. controls (2.25, -.6) and (2.75,.6) .. (v3);

    \node[] at (0.2, .35) {\footnotesize\textcolor{blue}{$1$}};
    \node[] at (0.2, -.35) {\footnotesize\textcolor{blue}{$0$}};
    \node[] at (1.2, .35) {\footnotesize\textcolor{blue}{$20$}};
    \node[] at (1.2, -.35) {\footnotesize\textcolor{blue}{$1$}};
    \node[] at (2.2, .35) {\footnotesize\textcolor{blue}{$301$}};
    \node[] at (2.2, -.35) {\footnotesize\textcolor{blue}{$2$}};
\end{scope}

\end{tikzpicture}}
    
    \caption{The weak orders for the graph $\oru(3)$ corresponding to the framings as depicted.
    The edges are labeled by the corresponding raising operator. }
    \label{fig:222poset}
\end{figure}

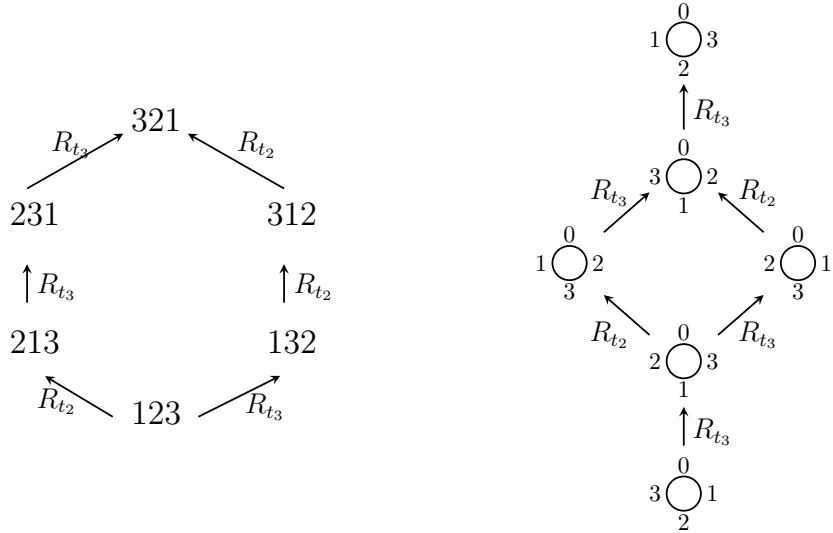
\begin{figure}[t]
    \centering
   
    \scalebox{0.9}{\tikzset{every picture/.style={line width=0.75pt}} 

\begin{tikzpicture}[x=0.75pt,y=0.75pt,yscale=-0.8,xscale=0.8]

\draw[-stealth]    (60,200) -- (127.4,161.49) ;
\draw[-stealth]    (240,200) -- (172.6,161.49) ;
\draw[-stealth]    (240,280) -- (240,253) ;
\draw[-stealth]    (180,360) -- (237.32,331.34) ;
\draw[-stealth]    (120,360) -- (72.57,331.54) ;
\draw[-stealth]    (60,280) -- (60,253) ;

\draw (65,338) node [anchor=north west][inner sep=0.75pt]    {$R_{t_{2}}$};
\draw (211,342) node [anchor=north west][inner sep=0.75pt]    {$R_{t_{3}}$};
\draw (75,158) node [anchor=north west][inner sep=0.75pt]    {$R_{t_{3}}$};
\draw (205,158) node [anchor=north west][inner sep=0.75pt]    {$R_{t_{2}}$};
\draw (245,258) node [anchor=north west][inner sep=0.75pt]    {$R_{t_{2}}$};
\draw (65,258) node [anchor=north west][inner sep=0.75pt]    {$R_{t_{3}}$};

\draw (131,348) node [anchor=north west][inner sep=0.75pt]  [font=\large  ]  {$123$};
\draw (46,298) node [anchor=north west][inner sep=0.75pt]  [font=\large  ]  {$213$};
\draw (226,298) node [anchor=north west][inner sep=0.75pt]  [font=\large  ]  {$132$};
\draw (226,208) node [anchor=north west][inner sep=0.75pt]  [font=\large  ]  {$312$};
\draw (131,142) node [anchor=north west][inner sep=0.75pt]  [font=\large  ]  {$321$};
\draw (46,208) node [anchor=north west][inner sep=0.75pt]  [font=\large  ]  {$231$};

\begin{scope}[scale=0.8,xshift=150,yshift=70]
\draw (450,116-90) circle (15);
\draw (450,91-90) node   [font=\footnotesize]  {$0$};
\draw (425,116-90) node   [font=\footnotesize]  {$1$};
\draw (450,141-90) node   [font=\footnotesize]  {$2$};
\draw (475,116-90) node   [font=\footnotesize]  {$3$};

\draw (450,206-60) circle (15);
\draw (450,181-60) node   [font=\footnotesize]  {$0$};
\draw (450,231-60) node   [font=\footnotesize]  {$1$};
\draw (475,206-60) node   [font=\footnotesize]  {$2$};
\draw (425,206-60) node   [font=\footnotesize]  {$3$};

\draw (350,243-11-10) circle (15);
\draw (350,243-25-11-10) node   [font=\footnotesize]  {$0$};
\draw (325,243-11-10) node   [font=\footnotesize]  {$1$};
\draw (375,243-11-10) node   [font=\footnotesize]  {$2$};
\draw (350,243+25-11-10) node   [font=\footnotesize]  {$3$};

\draw (550,243-11-10) circle (15);
\draw (550,243-25-11-10) node   [font=\footnotesize]  {$0$};
\draw (575,243-11-10) node   [font=\footnotesize]  {$1$};
\draw (525,243-11-10) node   [font=\footnotesize]  {$2$};
\draw (550,243+25-11-10) node   [font=\footnotesize]  {$3$};

\draw (450,298+10) circle (15);
\draw (450,273+10) node   [font=\footnotesize]  {$0$};
\draw (450,323+10) node   [font=\footnotesize]  {$1$};
\draw (425,298+10) node   [font=\footnotesize]  {$2$};
\draw (475,298+10) node   [font=\footnotesize]  {$3$};

\draw (450,384+40) circle (15);
\draw (450,359+40) node [font=\footnotesize] {$0$};
\draw (475,384+40) node [font=\footnotesize] {$1$};
\draw (450,409+40) node [font=\footnotesize] {$2$};
\draw (425,384+40) node [font=\footnotesize] {$3$};


\draw[-stealth]    (450,105) -- (450,65);
\draw (475,90) node {$R_{t_{3}}$};

\draw[-stealth]    (380,195) -- (420,160) ;
\draw (385,160) node {$R_{t_{3}}$};

\draw[-stealth]    (520,195) -- (480,160) ;
\draw (515,160) node {$R_{t_{2}}$};

\draw[-stealth]    (420,285) -- (380,250) ;
\draw (385,285) node {$R_{t_{2}}$};

\draw[-stealth]    (480,285) -- (520,250) ;
\draw (515,285) node {$R_{t_{3}}$};

\draw[-stealth]    (450,382) -- (450,348);
\draw (475,370) node {$R_{t_{3}}$};

\end{scope}

\end{tikzpicture}}
    
    \caption{The weak orders for the graph $\oru(3)$ corresponding to the framings of Figure \ref{fig:222poset} encoded as final summaries.}
    \label{fig:222final_summaries}
\end{figure}

\begin{figure}[t]
    \centering
    \scalebox{0.9}{\input{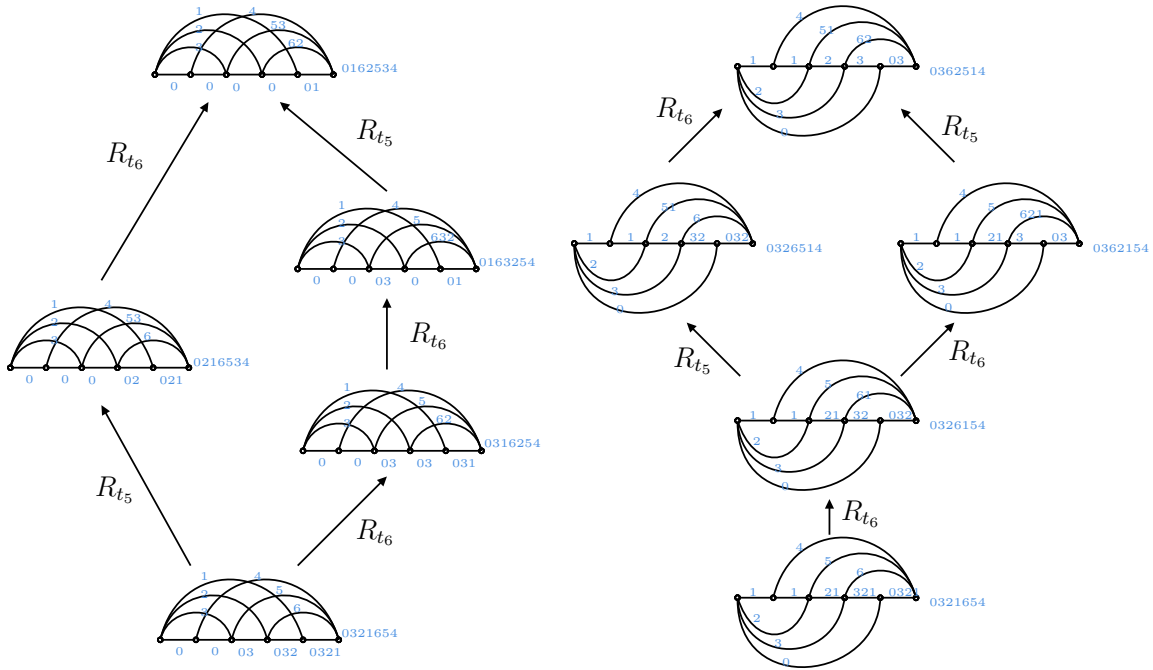}}
    \caption{The weak orders for the graph $\car(3)$ corresponding to two different framings: the length framing and the planar framing.}   
    \label{fig:caracolposets}
\end{figure}

\begin{figure}[t]
    \centering
 
    \scalebox{0.9}{\tikzset{every picture/.style={line width=0.75pt}} 

\begin{tikzpicture}[x=0.75pt,y=0.75pt,yscale=-1,xscale=1]

\begin{scope}[yshift=20]
\draw[-stealth] (60,240) -- (120,162.26);
\draw[-stealth] (240,200) -- (180,161.49);
\draw[-stealth] (240,280) -- (240,240);
\draw[-stealth] (180,360) -- (240,320);
\draw[-stealth] (120,360) -- (60,290);

\draw (65,328.4) node [anchor=north west][inner sep=0.75pt]    {$R_{t_{5}}$};
\draw (211,342.4) node [anchor=north west][inner sep=0.75pt]    {$R_{t_{6}}$};
\draw (71,178.4) node [anchor=north west][inner sep=0.75pt]    {$R_{t_{6}}$};
\draw (205,158.4) node [anchor=north west][inner sep=0.75pt]    {$R_{t_{5}}$};
\draw (245,258.4) node [anchor=north west][inner sep=0.75pt]    {$R_{t_{6}}$};

\draw (136,152.4) node [anchor=north west][inner sep=0.75pt]  [font=\large]  {$321$};
\draw (44,264.4) node [anchor=north west][inner sep=0.75pt]  [font=\large]  {$231$};
\draw (226,212.4) node [anchor=north west][inner sep=0.75pt]  [font=\large]  {$312$};
\draw (226,292.4) node [anchor=north west][inner sep=0.75pt]  [font=\large]  {$132$};
\draw (136,352.4) node [anchor=north west][inner sep=0.75pt]  [font=\large]  {$123$};
\end{scope}

\begin{scope}[scale=17, xshift=20, yshift=20, yscale=-1]
	\draw[fill, color=gray!33] (0,0) rectangle (1,1);
	\draw[fill, color=gray!33] (1,1) rectangle (2,2);
	\draw[fill, color=gray!33] (2,2) rectangle (3,3);
	\draw[step=1.0,very thin, color=gray!100] (0,0) grid (3,3);
					
	\draw[very thick, color=red] (0,0)--(0,3)--(3,3);
\end{scope}

\begin{scope}[scale=17, xshift=20, yshift=16, yscale=-1]
	\draw[fill, color=gray!33] (0,0) rectangle (1,1);
	\draw[fill, color=gray!33] (1,1) rectangle (2,2);
	\draw[fill, color=gray!33] (2,2) rectangle (3,3);
	\draw[step=1.0,very thin, color=gray!100] (0,0) grid (3,3);
					
	\draw[very thick, color=red] (0,0)--(0,2)--(1,2)--(1,3)--(3,3);
\end{scope}

\begin{scope}[scale=17, xshift=16, yshift=12, yscale=-1]
	\draw[fill, color=gray!33] (0,0) rectangle (1,1);
	\draw[fill, color=gray!33] (1,1) rectangle (2,2);
	\draw[fill, color=gray!33] (2,2) rectangle (3,3);
	\draw[step=1.0,very thin, color=gray!100] (0,0) grid (3,3);
					
	\draw[very thick, color=red] (0,0)--(0,2)--(2,2)--(2,3)--(3,3);
\end{scope}

\begin{scope}[scale=17, xshift=24, yshift=12, yscale=-1]
	\draw[fill, color=gray!33] (0,0) rectangle (1,1);
	\draw[fill, color=gray!33] (1,1) rectangle (2,2);
	\draw[fill, color=gray!33] (2,2) rectangle (3,3);
	\draw[step=1.0,very thin, color=gray!100] (0,0) grid (3,3);
					
	\draw[very thick, color=red] (0,0)--(0,1)--(1,1)--(1,3)--(3,3);
\end{scope}

\begin{scope}[scale=17, xshift=20, yshift=8, yscale=-1]
	\draw[fill, color=gray!33] (0,0) rectangle (1,1);
	\draw[fill, color=gray!33] (1,1) rectangle (2,2);
	\draw[fill, color=gray!33] (2,2) rectangle (3,3);
	\draw[step=1.0,very thin, color=gray!100] (0,0) grid (3,3);
					
	\draw[very thick, color=red] (0,0)--(0,1)--(1,1)--(1,2)--(2,2)--(2,3)--(3,3);
\end{scope}

\draw[-stealth] (479+60,277-62) -- (479+30+5,303-116);
\draw[-stealth] (479-60,277-62) -- (479-30-5,303-116);

\draw[-stealth] (479+30+5,303) -- (479+60,277);
\draw[-stealth] (479-30-5,303) -- (479-60,277);

\draw[-stealth] (479,395) -- (479,370);

\draw (479-60,195) node {$R_{t_{6}}$};
\draw (479+60,195) node {$R_{t_{5}}$};

\draw (479-60,300) node {$R_{t_{5}}$};
\draw (479+60,300) node {$R_{t_{6}}$};

\draw (479+20,382) node {$R_{t_{6}}$};

\end{tikzpicture}}
    \caption{The weak orders for the graph $\car(3)$ corresponding to the framings in Figure \ref{fig:caracolposets} encoded with objects directly translated from the final summaries.}   
    \label{fig:caracolposets_objects}
\end{figure}
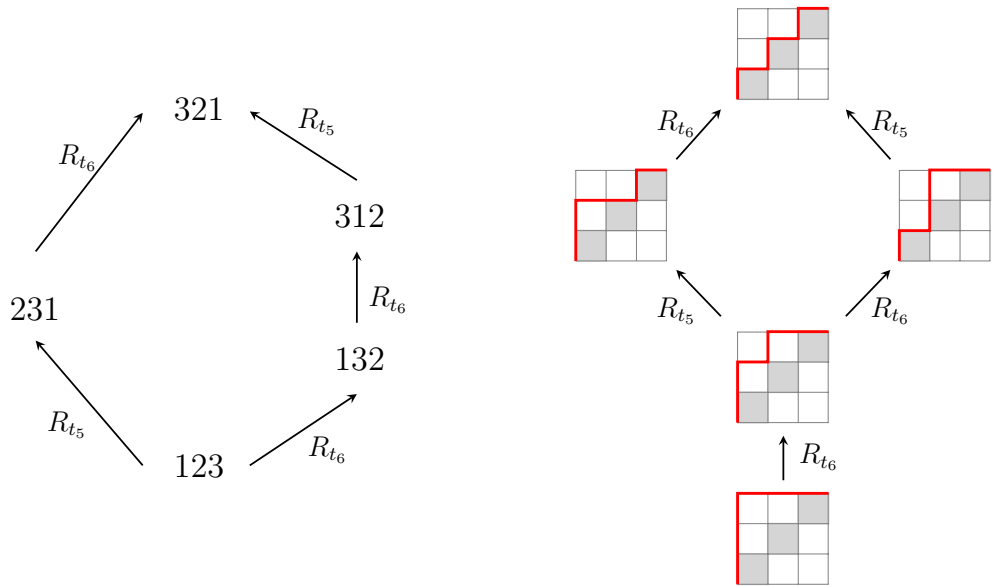

\begin{example}[The Caracol graph] Another example that is fundamental in the recent combinatorial study of flow polytopes is known as the \defn{Caracol} graph $\car(n)$ as was named by the authors of \cite{BenedettiGonzalezDleonHanusaHarrisKhareMoralesYip2019}. It is the graph with edges of the form $(0,i)$, $(i-1,i)$ and $(i,n)$ for $i\in [n-1]$ together with the edge $(n-1,n)$. The authors there show that the polytope $\calF_{\car(n)}$ has volume given by the Catalan number $\Cat(n-2)$. In Figure \ref{fig:caracolposets} two framings of $\car(5)$ are illustrated, the \defn{length} framing on the left and the \defn{planar} framing on the right. In work of the first and third author together with von Bell and Mayorga Cetina
\cite{BellGonzalezMayorgaYip2023} it was shown that $W(\car(n),F)$ with the length framing is isomorphic to the Tamari lattice and with the planar framing is isomorphic to the order ideal determined by the staircase partition $(n-3,n-4,\dots, 2,1)$. Indeed, in terms of permutation flows the isomorphism is given both by the final summary map $\zeta$. In the case of the length framing it is enough to restrict the final summary to the letters in  $\{1,2,3\}$ and then mirror the values of the labels,  $1\mapsto 3$, $2\mapsto 2$, and $3\mapsto 1$, to obtain a description of the Tamari lattice in terms of $231$-avoiding permutations (See Figure \ref{fig:caracolposets_objects} left).  In the planar framing it is enough to replace the elements of the set $\{4,5,6\}$ by north steps $N$ and the elements of the set $\{1,2,3\}$ by east steps $E$ to get a Dyck path that has the corresponding partition in Young's lattice as co-area (See Figure \ref{fig:caracolposets_objects} right).

A generalization $\car(\nu)$ of $\car(n)$ to a weak composition $\nu$ (equivalent to $\car(\bs)$) was defined and studied in \cite{BellGonzalezMayorgaYip2023}, such that $\car(n)=\car(1,1,\dots,1)$.
\end{example}

\begin{example}[Permutree lattices]
In work of the first and third author together with Morales, Philippe, and Tamayo Jim\'enez the authors show that the families of permutree lattices and permutree lattice quotients as defined by Pilaud and Pons in \cite{PilaudPons2018} can be realized by indexing the simplexes of $\DKK(G,F)$ where 
the framed graph $(G,F)$ is obtained from the planar framing of $\oru(n)$ after a series of applications of an operation on framed graphs called \defn{M-move} which  replaces an edge $(i,j)$ by the pair of edges $(0,j)$ and $(i,n)$ while respecting the framing orders at $i$ and $j$. (see the dissertation of Tamayo Jim\'enez~\cite[Chapter 7]{TamayoJimenez2023}).  Permutree lattices contain as particular examples the weak order on permutations, the Tamari lattice on binary trees, the boolean algebra on binary sequences, and the Cambrian lattices.
\end{example}

\newpage
\section{The weak order on permutation flows and the \texorpdfstring{$h^*$}{h*}-polynomial of \texorpdfstring{$\calF_G$}{F sub G}}\label{sec:hstarpoly}

The combinatorics of permutation flows given in Section \ref{sec:permutationflows} translates into useful combinatorial descriptions of the weak order $W(G,F)$ depending on the supporting family of objects we choose to use. In this section we describe how this order works on $\SatGroves(G,F)$ and $\SatPermutationFlows(G,F)$, and combine these together with the original definition on $\SatCliques(G,F)$ to give a combinatorial description of the $h^*$-polynomial of $\calF_G$.

In Definitions~\ref{def:minimal_exchange_quadrangle_cliques} and \ref{def:weak_order_cliques} we define cover relations on minimal exchange quadrangles that lead to the definition of the weak order. We translate these definitions to the context of saturated groves and saturated permutation flows.

\subsection{Raising and lowering operations on Groves}
\label{sec:raising_operators_groves}
\phantom{W}

\begin{definition}\label{def:minimal_exchange_quadrangle_groves}
An \defn{exchange quadrangle} on $\Gamma \in \SatGroves(G,F)$ at an edge $t=(v,w)$ is a tuple $(P,Q,e,t)$ formed by a pair of consecutive prefixes $P\precdot_{L(\gamma_v)} Q$ and a pair of consecutive edges $e\precdot_{L(\gamma_v)}t$ such that $(P,e)$, $(Q,t)$, and exactly one of either $(P,t)$ or $(Q,e)$ are in $\gamma_v$. Furthermore, we say that the exchange quadrangle is \defn{minimal} if $\Splits(\Gamma, (P,t))=\emptyset$ or $\Splits(\Gamma, (Q,e))=\emptyset$ depending on which edge is in $\gamma_n$. We will say that the edge $(P,t)$ is \defn{ascending} and the edge $(Q,e)$ is \defn{descending} at the quadrangle. 
\end{definition}

Let $\Gamma \in \Groves(G,F)$, $P$ be a prefix in $L(\gamma_v)$ that has direct splits at some vertex $v$, and $(P,t)$ be one of the edges in $\gamma_v$. We denote by $\Gamma \setminus (P,t)\in \Groves(G,F)$ to be the grove obtained by removing $(P,t)$ and all prefixes and grove edges that are extensions of $Pt$. 

Observe that the maximality of $\Gamma \in \SatGroves(G,F)$, implies that a minimal exchange quadrangle on $\Gamma$ has the property that $\Gamma \setminus (P,t)$  (or $\Gamma \setminus (Q,e)$) has a unique saturation that contains $(Q,e)$ (or $(P,t)$).
Indeed, if there were a choice for this saturation then $\Splits(\Gamma;(P,t))\neq \emptyset$ ($\Splits(\Gamma;(Q,e))\neq \emptyset$). We will denote by $R_t(\Gamma) \in \SatGroves(G,F)$ the unique grove obtained by exchanging the unique saturation containing $(P,t)$ to the unique saturation containing $(Q,e)$ and $L_t(\Gamma)\in  \SatGroves(G,F)$ the unique grove obtained from exchanging the unique saturation containing $(Q,e)$ to the unique saturation containing $(P,t)$. These define a pair of operations
\begin{align*}
R_{t},L_{t}&:\SatGroves(G,F)\rightarrow \SatGroves(G,F).
\end{align*}

\begin{proposition}\label{proposition:weak_order_groves}
    The poset $W(G,F)$ is isomorphic to the poset on $\SatGroves(G,F)$ where cover relations are of the form $\Gamma \lessdot R_t(\Gamma)$ (equivalently $L_t(\Gamma) \lessdot \Gamma $) for $t \in \Splits(\Gamma)$.
\end{proposition}

\begin{remark} Moves of the form $\Gamma \mapsto R_t(\Gamma)$ are referred as to \defn{\textup{S}-to-\textup{Z}} and the ones of the form $\Gamma \mapsto L_t(\Gamma)$ as \defn{\textup{Z}-to-\textup{S}} in the context of the zigzag graph in \cite{BrunnerHanusa2024}.
\end{remark}

\subsection{Raising and lowering operations on Permutation Flows}
\label{sec:raising_operators_permuflows}
\phantom{W}

\begin{definition}
\label{def:ascent_descent}
For $\pi \in \SatPermutationFlows(G,F)$ an \defn{ascent} is a pair $(e,t)$ where $t$ is a direct split of $e$ at a vertex $v$, and if $e^*$ is the outgoing carrier of $e$ at $v$ then $\pi(e^*)\neq e$ and $\Splits(\pi;e,e^*)=\emptyset$.

A \defn{descent} is a pair $(t,e)$ where $t$ is a split, $e$ is the second letter of $\pi(t)$, and $\Splits(\pi;t,t)=\emptyset$. Equivalently, $t$ and $e$ appear consecutively in all carriers of $t$.
\end{definition}

If $(e,t)$ is an ascent of $\pi$ and $e^*$ is as in Definition~\ref{def:ascent_descent}, let $\pi-(e,e^*)$ be the element of $\PermutationFlows(G,F)$ obtained from $\pi$ by removing $e$ from $\pi(e^*)$ and from every edge thereafter. If $(t,e)$ is a descent of $\pi$ we define $\pi-(t,e)$ by removing $e$ from $\pi(t)$ and from every edge thereafter.

\begin{definition}
\label{def:raising}
Let $(G,F)$ be a framed graph and 
define the \defn{raising operator} 
\[
R_{t}:\SatPermutationFlows(G,F)\rightarrow \SatPermutationFlows(G,F)
\] as follows. 
Let $\pi\in\SatPermutationFlows(G,F)$ and $(e,t)$ be an ascent of $\pi$ where $e$ is the last entry of $\pi(e^*)$. Define $R_t(\pi)$ to be the permutation flow $\pi'$ obtained as the unique saturation of $\pi-(e,e^*)$ in which the second letter of $\pi'(t)$ is $e$. Otherwise, if there is no ascent at $t$ then $R_t$ acts trivially and we define $R_t(\pi)=\pi$. 
 
In a similar manner we define the \defn{lowering operator} 
\[L_{t}:\SatPermutationFlows(G,F)\rightarrow \SatPermutationFlows(G,F)\] as the inverse operation that takes $\pi'$ to $\pi$. That is, if $(t,e)$ is a descent of $\pi$ with $e^* \prec_{\hat\outedge(v)} t$ then $L_t(\pi')$ is the unique saturation of $\pi'-(e;t)$ in which the last letter of $\pi(e^*)$ is $e$. Otherwise, if there is no descent at $t$ then $L_t$ acts trivially and we define $L_t(\pi')=\pi$. 
\end{definition}

\begin{remark}
The raising and lowering operators have equivalent descriptions.

Equivalently, delete the letter $e$ in $e^*$ and all edges thereafter; then add the letter $e$ after $t$ in all carriers of $t$. 

This operation can equivalently be described as deleting the letter $e$ in $t$ and all edges thereafter and reinserting the letter $e$ as follows. Let $l$ be the last letter of $\pi(e^*)$. Insert $e$ after $l$ on carriers of $l$ until $l$ splits, and if $t'$ is the highest of these splits, $e$ should be inserted after $t'$ in carriers of $t'$, up until the point where $t'$ splits, and so on.
\end{remark}

\begin{example}
\label{ex:raising}
Consider the action of the raising operators $R_{t_1}$ through $R_{t_7}$ on the permutation flow $\pi$ from Figure~\ref{fig:permutation_flow_example}. 

\begin{itemize}
    \item We see that $R_{t_1}$, $R_{t_2}$, and $R_{t_4}$ act trivially on $\pi$ because the words on the edges preceding $t_1$, $t_2$, and $t_4$ in the order on outgoing edges ($t_2$, $s_0$, and $s_1$, respectively) are all singletons ($t_2$, $x_0$, and $x_0$, respectively).
    \item We see that $R_{t_3}$ acts trivially on $\pi$ because the last letter of $\pi(t_4)$ is $t_2$, and so is the last letter of $\pi(s_2)$, and $s_2$ comes after $t_4$.
    \item The actions of the raising operators $R_{t_5}$, $R_{t_6}$, and $R_{t_7}$ on $\pi$ are given in Figure~\ref{fig:3liftings}. The raising operator $R_{t_5}$ raises the letter $t_2$, deletes every occurrence of $t_2$ after the edge $s_2$, and replaces every occurrence of $t_5$ with $t_5t_2$.
    The operator $R_{t_6}$ raises the letter $x_0$, and $R_{t_7}$ raises the letter $t_3$.  
\end{itemize}
We invite the reader to check that $\pi$ has three non-trivial lowering operators: $L_{t_4}$, $L_{t_5}$, and~$L_{t_7}$.
\end{example}

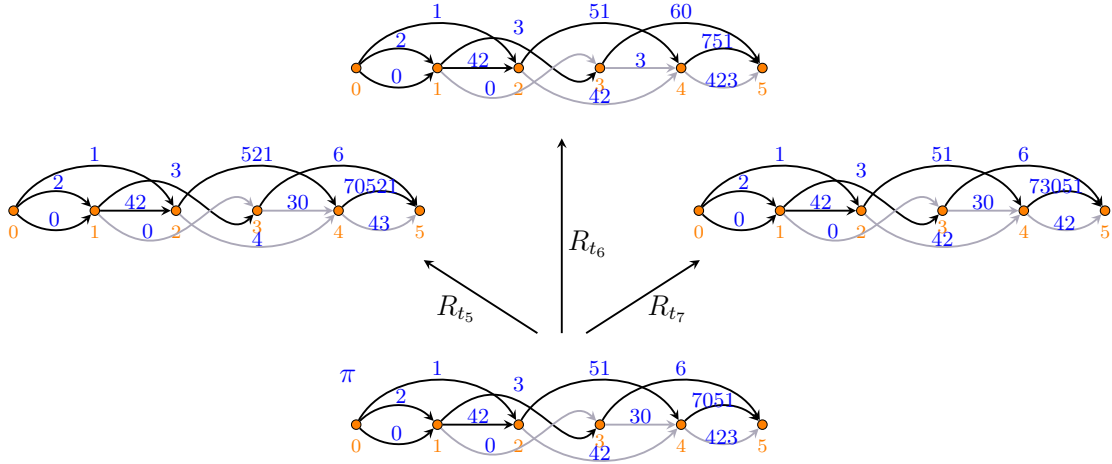
\begin{figure}[t]
    \centering
    \scalebox{0.9}{\begin{tikzpicture}[scale=1.2]
\begin{scope}[xshift=72, yshift=30]
    \draw[-stealth, thick] (0,0.6) to (0,3);
    \draw[-stealth, thick] (0.3,0.6) to (1.7,1.5);
    \draw[-stealth, thick] (-0.3,0.6) to (-1.7,1.5);
\node[] at (-1.3,0.9) {$R_{t_5}$};
\node[] at (0.32,1.7) {$R_{t_6}$};
\node[] at (1.3,0.9) {$R_{t_7}$};
\end{scope}

\begin{scope}[xshift=-120, yshift=90, scale=1.0]
	\vertex[fill=orange, minimum size=4pt, label=below:{\tiny\textcolor{orange}{$0$}}](v0) at (0,0) {};
	\vertex[fill=orange, minimum size=4pt, label=below:{\tiny\textcolor{orange}{$1$}}](v1) at (1,0) {};
	\vertex[fill=orange, minimum size=4pt, label=below:{\tiny\textcolor{orange}{$2$}}](v2) at (2,0) {};
	\vertex[fill=orange, minimum size=4pt, label={[label distance=-2pt]below:{\tiny\textcolor{orange}{$3$}}}](v3) at (3,0) {};
	\vertex[fill=orange, minimum size=4pt, label=below:{\tiny\textcolor{orange}{$4$}}](v4) at (4,0) {};
	\vertex[fill=orange, minimum size=4pt, label=below:{\tiny\textcolor{orange}{$5$}}](v5) at (5,0) {};		

\draw[-stealth, thick] (v0) to [out=45,in=135] (v1);
\draw[-stealth, thick] (v0) to [out=-45,in=-135] (v1);
\draw[-stealth, thick] (v0) to [out=60,in=120] (v2);
\draw[-stealth, thick] (v1) to [out=0,in=180] (v2);
\draw[-stealth, thick] (v1) .. controls (2.0, 1.0) and (2.5, -0.5) .. (v3);
\draw[-stealth, thick, color=lavgray] (v1) .. controls (2.0, -1.0) and (2.5, 0.5) .. (v3);	
\draw[-stealth, thick, color=lavgray] (v2) to [out=-45,in=-135] (v4);
\draw[-stealth, thick] (v2) to [out=60,in=120] (v4);
\draw[-stealth, thick, color=lavgray] (v3) to [out=0,in=180] (v4);
\draw[-stealth, thick] (v3) to [out=60,in=120] (v5);
\draw[-stealth, thick, color=lavgray] (v4) to [out=-45,in=-135] (v5);
\draw[-stealth, thick] (v4) to [out=45,in=135] (v5);

\node[] at (4,0.7){\scriptsize\textcolor{blue}{$6$}};
\node[] at (4.4, 0.3){\scriptsize\textcolor{blue}{$70521$}};
\node[] at (3,0.7){\scriptsize\textcolor{blue}{$521$}};
\node[] at (2, 0.5){\scriptsize\textcolor{blue}{\scriptsize$3$}};
\node[] at (1, 0.7){\scriptsize\textcolor{blue}{\scriptsize$1$}};
\node[] at (1.5, 0.1){\scriptsize\textcolor{blue}{$42$}};
\node[] at (0.55, 0.35){\scriptsize\textcolor{blue}{$2$}};
\node[] at (0.5, -0.1){\scriptsize\textcolor{blue}{$0$}};
\node[] at (1.65, -0.25){\scriptsize\textcolor{blue}{$0$}};
\node[] at (3, -0.35){\scriptsize\textcolor{blue}{$4$}};
\node[] at (3.5, 0.1){\scriptsize\textcolor{blue}{$30$}};
\node[] at (4.5,-0.15){\scriptsize\textcolor{blue}{$43$}};
\end{scope}

\begin{scope}[xshift=0, yshift=140, scale=1.0]
	\vertex[fill=orange, minimum size=4pt, label=below:{\tiny\textcolor{orange}{$0$}}](v0) at (0,0) {};
	\vertex[fill=orange, minimum size=4pt, label=below:{\tiny\textcolor{orange}{$1$}}](v1) at (1,0) {};
	\vertex[fill=orange, minimum size=4pt, label=below:{\tiny\textcolor{orange}{$2$}}](v2) at (2,0) {};
	\vertex[fill=orange, minimum size=4pt, label={[label distance=-2pt]below:{\tiny\textcolor{orange}{$3$}}}](v3) at (3,0) {};
	\vertex[fill=orange, minimum size=4pt, label=below:{\tiny\textcolor{orange}{$4$}}](v4) at (4,0) {};
	\vertex[fill=orange, minimum size=4pt, label=below:{\tiny\textcolor{orange}{$5$}}](v5) at (5,0) {};		

\draw[-stealth, thick] (v0) to [out=45,in=135] (v1);
\draw[-stealth, thick] (v0) to [out=-45,in=-135] (v1);
\draw[-stealth, thick] (v0) to [out=60,in=120] (v2);
\draw[-stealth, thick] (v1) to [out=0,in=180] (v2);
\draw[-stealth, thick] (v1) .. controls (2.0, 1.0) and (2.5, -0.5) .. (v3);
\draw[-stealth, thick, color=lavgray] (v1) .. controls (2.0, -1.0) and (2.5, 0.5) .. (v3);	
\draw[-stealth, thick, color=lavgray] (v2) to [out=-45,in=-135] (v4);
\draw[-stealth, thick] (v2) to [out=60,in=120] (v4);
\draw[-stealth, thick, color=lavgray] (v3) to [out=0,in=180] (v4);
\draw[-stealth, thick] (v3) to [out=60,in=120] (v5);
\draw[-stealth, thick, color=lavgray] (v4) to [out=-45,in=-135] (v5);
\draw[-stealth, thick] (v4) to [out=45,in=135] (v5);

\node[] at (4,0.7){\scriptsize\textcolor{blue}{$60$}};
\node[] at (4.45, 0.33){\scriptsize\textcolor{blue}{$751$}};
\node[] at (3,0.7){\scriptsize\textcolor{blue}{$51$}};
\node[] at (2, 0.5){\scriptsize\textcolor{blue}{$3$}};
\node[] at (1, 0.7){\scriptsize\textcolor{blue}{$1$}};
\node[] at (1.5, 0.1){\scriptsize\textcolor{blue}{$42$}};
\node[] at (0.55, 0.35){\scriptsize\textcolor{blue}{$2$}};
\node[] at (0.5, -0.1){\scriptsize\textcolor{blue}{$0$}};
\node[] at (1.65, -0.25){\scriptsize\textcolor{blue}{$0$}};
\node[] at (3, -0.35){\scriptsize\textcolor{blue}{$42$}};
\node[] at (3.5, 0.1){\scriptsize\textcolor{blue}{$3$}};
\node[] at (4.5,-0.15){\scriptsize\textcolor{blue}{$423$}};
\end{scope}

\begin{scope}[xshift=120, yshift=90, scale=1.0]
	\vertex[fill=orange, minimum size=4pt, label=below:{\tiny\textcolor{orange}{$0$}}](v0) at (0,0) {};
	\vertex[fill=orange, minimum size=4pt, label=below:{\tiny\textcolor{orange}{$1$}}](v1) at (1,0) {};
	\vertex[fill=orange, minimum size=4pt, label=below:{\tiny\textcolor{orange}{$2$}}](v2) at (2,0) {};
	\vertex[fill=orange, minimum size=4pt, label={[label distance=-2pt]below:{\tiny\textcolor{orange}{$3$}}}](v3) at (3,0) {};
	\vertex[fill=orange, minimum size=4pt, label=below:{\tiny\textcolor{orange}{$4$}}](v4) at (4,0) {};
	\vertex[fill=orange, minimum size=4pt, label=below:{\tiny\textcolor{orange}{$5$}}](v5) at (5,0) {};		

\draw[-stealth, thick] (v0) to [out=45,in=135] (v1);
\draw[-stealth, thick] (v0) to [out=-45,in=-135] (v1);
\draw[-stealth, thick] (v0) to [out=60,in=120] (v2);
\draw[-stealth, thick] (v1) to [out=0,in=180] (v2);
\draw[-stealth, thick] (v1) .. controls (2.0, 1.0) and (2.5, -0.5) .. (v3);
\draw[-stealth, thick, color=lavgray] (v1) .. controls (2.0, -1.0) and (2.5, 0.5) .. (v3);	
\draw[-stealth, thick, color=lavgray] (v2) to [out=-45,in=-135] (v4);
\draw[-stealth, thick] (v2) to [out=60,in=120] (v4);
\draw[-stealth, thick, color=lavgray] (v3) to [out=0,in=180] (v4);
\draw[-stealth, thick] (v3) to [out=60,in=120] (v5);
\draw[-stealth, thick, color=lavgray] (v4) to [out=-45,in=-135] (v5);
\draw[-stealth, thick] (v4) to [out=45,in=135] (v5);

\node[] at (4,0.7){\scriptsize\textcolor{blue}{$6$}};
\node[] at (4.4, 0.3){\scriptsize\textcolor{blue}{$73051$}};
\node[] at (3,0.7){\scriptsize\textcolor{blue}{$51$}};
\node[] at (2, 0.5){\scriptsize\textcolor{blue}{$3$}};
\node[] at (1, 0.7){\scriptsize\textcolor{blue}{$1$}};
\node[] at (1.5, 0.1){\scriptsize\textcolor{blue}{$42$}};
\node[] at (0.55, 0.35){\scriptsize\textcolor{blue}{$2$}};
\node[] at (0.5, -0.1){\scriptsize\textcolor{blue}{$0$}};
\node[] at (1.65, -0.25){\scriptsize\textcolor{blue}{$0$}};
\node[] at (3, -0.35){\scriptsize\textcolor{blue}{$42$}};
\node[] at (3.5, 0.1){\scriptsize\textcolor{blue}{$30$}};
\node[] at (4.5,-0.15){\scriptsize\textcolor{blue}{$42$}};
\end{scope}

\begin{scope}[xshift=0, yshift=15, scale=1.0]
\node[] at (-0.1,0.6){\textcolor{blue}{$\pi$}};
	\vertex[fill=orange, minimum size=4pt, label=below:{\tiny\textcolor{orange}{$0$}}](v0) at (0,0) {};
	\vertex[fill=orange, minimum size=4pt, label=below:{\tiny\textcolor{orange}{$1$}}](v1) at (1,0) {};
	\vertex[fill=orange, minimum size=4pt, label=below:{\tiny\textcolor{orange}{$2$}}](v2) at (2,0) {};
	\vertex[fill=orange, minimum size=4pt, label={[label distance=-2pt]below:{\tiny\textcolor{orange}{$3$}}}](v3) at (3,0) {};
	\vertex[fill=orange, minimum size=4pt, label=below:{\tiny\textcolor{orange}{$4$}}](v4) at (4,0) {};
	\vertex[fill=orange, minimum size=4pt, label=below:{\tiny\textcolor{orange}{$5$}}](v5) at (5,0) {};		

\draw[-stealth, thick] (v0) to [out=45,in=135] (v1);
\draw[-stealth, thick] (v0) to [out=-45,in=-135] (v1);
\draw[-stealth, thick] (v0) to [out=60,in=120] (v2);
\draw[-stealth, thick] (v1) to [out=0,in=180] (v2);
\draw[-stealth, thick] (v1) .. controls (2.0, 1.0) and (2.5, -0.5) .. (v3);
\draw[-stealth, thick, color=lavgray] (v1) .. controls (2.0, -1.0) and (2.5, 0.5) .. (v3);	
\draw[-stealth, thick, color=lavgray] (v2) to [out=-45,in=-135] (v4);
\draw[-stealth, thick] (v2) to [out=60,in=120] (v4);
\draw[-stealth, thick, color=lavgray] (v3) to [out=0,in=180] (v4);
\draw[-stealth, thick] (v3) to [out=60,in=120] (v5);
\draw[-stealth, thick, color=lavgray] (v4) to [out=-45,in=-135] (v5);
\draw[-stealth, thick] (v4) to [out=45,in=135] (v5);

\node[] at (4,0.7){\textcolor{blue}{\scriptsize$6$}};
\node[] at (4.4, 0.3){\scriptsize\textcolor{blue}{$7051$}};
\node[] at (3,0.7){\scriptsize\textcolor{blue}{$5$}\textcolor{blue}{$1$}};
\node[] at (2, 0.5){\scriptsize\textcolor{blue}{\scriptsize$3$}};
\node[] at (1, 0.7){\scriptsize\textcolor{blue}{\scriptsize$1$}};
\node[] at (1.5, 0.1){\scriptsize\textcolor{blue}{$42$}};
\node[] at (0.55, 0.35){\scriptsize\textcolor{blue}{$2$}};
\node[] at (0.5, -0.1){\scriptsize\textcolor{blue}{$0$}};
\node[] at (1.65, -0.25){\scriptsize\textcolor{blue}{$0$}};
\node[] at (3, -0.35){\scriptsize\textcolor{blue}{$42$}};
\node[] at (3.5, 0.1){\scriptsize\textcolor{blue}{{\scriptsize}$30$}};
\node[] at (4.5,-0.15){\scriptsize\textcolor{blue}{$423$}};
\end{scope}
\end{tikzpicture}}
    \caption{The effect of the nontrivial raising operators on the permutation flow $\pi$ from Figure~\ref{fig:permutation_flow_example}. See Example~\ref{ex:raising}.}
    \label{fig:3liftings}
\end{figure}

\begin{proposition}\label{proposition:weak_order_permutationflows}
    $W(G,F)$ is isomorphic to the poset on $\SatPermutationFlows(G,F)$ where cover relations are of the form $\pi \lessdot R_t(\pi)$ (equivalently $L_t(\pi) \lessdot \pi $) for $t \in \Splits(\pi)$.
\end{proposition}

A consequence of the description of $W(G,F)$ in terms of raising and lowering operators is that we can easily identify unique bottom and top elements in $\SatPermutationFlows(G,F)$ by lowering or raising all possible letters.
The \defn{bottom element} of $\SatPermutationFlows(G,F)$, denoted~$\pi^0$ satisfies $\pi^0(s_0)=x_0$, $\pi^0(t)=t$ for each $t \in \Splits(G,F)$, and for every inner vertex $v$, $\pi^0(s_v)=\zeta_v$ is equal to the summary at $v$. On the other end, the \defn{top element} of $\SatPermutationFlows(G,F)$, denoted~$\pi^1$ satisfies $\pi^1(s_v)$ equals the first letter of $\zeta_v$, $\pi^1(t)=t$ for $t \in \Splits(G,F)$ if $t\neq \max(\outedge(v))$, and if $t=\max(\outedge(v))$ then $\pi^1(t)$ is obtained from~$\zeta_v$ by replacing its first letter by $t$.

\begin{figure}[ht!]
    \centering
    \begin{tikzpicture}[scale=1.0]

\begin{scope}[scale=1.3, xshift=-180]
\node[] at (-0.3,0.6){\textcolor{blue}{$\pi^0$}};
	\vertex[fill=orange, minimum size=4pt, label=below:{\tiny\textcolor{orange}{$0$}}](v0) at (0,0) {};
	\vertex[fill=orange, minimum size=4pt, label=below:{\tiny\textcolor{orange}{$1$}}](v1) at (1,0) {};
	\vertex[fill=orange, minimum size=4pt, label=below:{\tiny\textcolor{orange}{$2$}}](v2) at (2,0) {};
	\vertex[fill=orange, minimum size=4pt, label=below:{\tiny\textcolor{orange}{$3$}}](v3) at (3,0) {};
	\vertex[fill=orange, minimum size=4pt, label=below:{\tiny\textcolor{orange}{$4$}}](v4) at (4,0) {};
	\vertex[fill=orange, minimum size=4pt, label=below:{\tiny\textcolor{orange}{$5$}}](v5) at (5,0) {};		

\draw[-stealth, thick] (v0) to [out=45,in=135] (v1);
\draw[-stealth, thick,color=lavgray] (v0) to [out=-45,in=-135] (v1);
\draw[-stealth, thick] (v0) to [out=60,in=120] (v2);
\draw[-stealth, thick] (v1) to [out=0,in=180] (v2);
\draw[-stealth, thick] (v1) .. controls (2.0, 1.0) and (2.5, -0.5) .. (v3);
\draw[-stealth, thick,color=lavgray] (v1) .. controls (2.0, -1.0) and (2.5, 0.5) .. (v3);	
\draw[-stealth, thick,color=lavgray] (v2) to [out=-45,in=-135] (v4);
\draw[-stealth, thick] (v2) to [out=60,in=120] (v4);
\draw[-stealth, thick,color=lavgray] (v3) to [out=0,in=180] (v4);
\draw[-stealth, thick] (v3) to [out=60,in=120] (v5);
\draw[-stealth, thick,color=lavgray] (v4) to [out=-45,in=-135] (v5);
\draw[-stealth, thick] (v4) to [out=45,in=135] (v5);

\node[] at (4, 0.66){\scriptsize\textcolor{blue}{$6$}};
\node[] at (4.47, 0.3){\scriptsize\textcolor{blue}{$7$}};
\node[] at (4.5, -0.12){\scriptsize\textcolor{blue}{$413025$}};
\node[] at (3, 0.66){\scriptsize\textcolor{blue}{$5$}};
\node[] at (3.5, 0.1){\scriptsize\textcolor{blue}{$302$}};
\node[] at (3, -0.35){\scriptsize\textcolor{blue}{$41$}};
\node[] at (1.9, 0.45){\scriptsize\textcolor{blue}{$3$}};
\node[] at (1.5, 0.1){\scriptsize\textcolor{blue}{$4$}};
\node[] at (1.7, -.27){\scriptsize\textcolor{blue}{$02$}};
\node[] at (1, 0.66){\scriptsize\textcolor{blue}{$1$}};
\node[] at (0.5, 0.32){\scriptsize\textcolor{blue}{$2$}};
\node[] at (0.5, -0.12){\scriptsize\textcolor{blue}{$0$}};
\end{scope}

\begin{scope}[scale=1.3]
\node[] at (-0.3,0.6){\textcolor{blue}{$\pi^1$}};
	\vertex[fill=orange, minimum size=4pt, label=below:{\tiny\textcolor{orange}{$0$}}](v0) at (0,0) {};
	\vertex[fill=orange, minimum size=4pt, label=below:{\tiny\textcolor{orange}{$1$}}](v1) at (1,0) {};
	\vertex[fill=orange, minimum size=4pt, label=below:{\tiny\textcolor{orange}{$2$}}](v2) at (2,0) {};
	\vertex[fill=orange, minimum size=4pt, label=below:{\tiny\textcolor{orange}{$3$}}](v3) at (3,0) {};
	\vertex[fill=orange, minimum size=4pt, label=below:{\tiny\textcolor{orange}{$4$}}](v4) at (4,0) {};
	\vertex[fill=orange, minimum size=4pt, label=below:{\tiny\textcolor{orange}{$5$}}](v5) at (5,0) {};		

\draw[-stealth, thick] (v0) to [out=45,in=135] (v1);
\draw[-stealth, thick,color=lavgray] (v0) to [out=-45,in=-135] (v1);
\draw[-stealth, thick] (v0) to [out=60,in=120] (v2);
\draw[-stealth, thick] (v1) to [out=0,in=180] (v2);
\draw[-stealth, thick] (v1) .. controls (2.0, 1.0) and (2.5, -0.5) .. (v3);
\draw[-stealth, thick,color=lavgray] (v1) .. controls (2.0, -1.0) and (2.5, 0.5) .. (v3);	
\draw[-stealth, thick,color=lavgray] (v2) to [out=-45,in=-135] (v4);
\draw[-stealth, thick] (v2) to [out=60,in=120] (v4);
\draw[-stealth, thick,color=lavgray] (v3) to [out=0,in=180] (v4);
\draw[-stealth, thick] (v3) to [out=60,in=120] (v5);
\draw[-stealth, thick,color=lavgray] (v4) to [out=-45,in=-135] (v5);
\draw[-stealth, thick] (v4) to [out=45,in=135] (v5);

\node[] at (4, 0.66){\scriptsize\textcolor{blue}{$620$}};
\node[] at (4.47, 0.3){\scriptsize\textcolor{blue}{$7351$}};
\node[] at (4.5, -0.12){\scriptsize\textcolor{blue}{$4$}};
\node[] at (3, 0.66){\scriptsize\textcolor{blue}{$51$}};
\node[] at (3.5, 0.1){\scriptsize\textcolor{blue}{$3$}};
\node[] at (3, -0.35){\scriptsize\textcolor{blue}{$4$}};
\node[] at (1.9, 0.45){\scriptsize\textcolor{blue}{$32$}};
\node[] at (1.5, 0.1){\scriptsize\textcolor{blue}{$4$}};
\node[] at (1.7, -.27){\scriptsize\textcolor{blue}{$0$}};
\node[] at (1, 0.66){\scriptsize\textcolor{blue}{$1$}};
\node[] at (0.5, 0.32){\scriptsize\textcolor{blue}{$2$}};
\node[] at (0.5, -0.12){\scriptsize\textcolor{blue}{$0$}};
\end{scope}
\end{tikzpicture}\vspace{-.2in}
    \caption{Top and bottom permutation flows on the framed graph $(G,F)$ of Figure~\ref{fig:framed_graph}.
    }
    \label{fig:permutation_flow_bottom_top}
\end{figure}
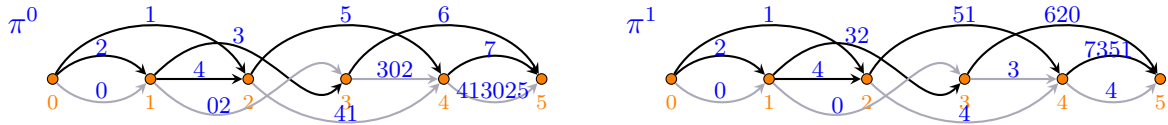

\begin{example}
The weak orders for different framings of $\oru(3)$ are shown in Figure~\ref{fig:222poset}.
\end{example}

\subsection{Shellability of \texorpdfstring{$DKK(G,F)$}{DKK(G,F)}}
\label{sec:shellability}
\phantom{W}

We show that every linear extension of $W(G,F)$ indexes a shelling of the triangulation of $\calF_G$ given by $DKK(G,F)$. We use the different combinatorial interpretations of $W(G,F)$ on $\SatCliques(G,F)$, $\SatGroves(G,F)$, and $\SatPermutationFlows(G,F)$, whenever we see one of the descriptions to be more convenient.

\begin{definition}[Shelling c.f. \cite{Ziegler1995}] A \defn{shelling} of a pure simplicial complex  is an ordering $F_1,\dots,F_s$ of its facets such that for every $i<j$ there exists a $k < j$ such that
$F_i\cap F_j \subseteq F_k\cap F_j$ and $|F_k\cap F_j|=|F_j|-1$.
\end{definition}

To prove shellability we will use a few lemmas. In particular, the following lemma can be interpreted as a convexity of intervals of $W(G,F)$ with respect of the inclusion of coherent routes.

\begin{lemma}
\label{lemma:nonrevisiting_property}    Suppose that $\calC \preceq \calD\preceq \calC'$ and $R\in \calC\cap \calC'$. Then $R\in \calD$.
\end{lemma}
\begin{proof}

We will prove instead the following statement: if $\calC \precdot \calD\preceq \calC' $ and if $P\in \calC$ is such that $P\not \in \calD$, then there exist $Q \in\calC'$ such that $P$ and $Q$ are in conflict at some vertex $v$ in which $vP\prec_{\Suffixes(v)} vQ$.

Since $P$ is fixed we note that  $\calC \precdot \calD$ implies that there is a vertex $v$ and $Q \in \calD$ with a minimal conflict with $P$ at $v$ such that $\{P, Q, P vQ, QvP \}$ is the minimal exchange quadrangle between $\calC$ and $\calD$ where $P$ is descending and $Q$ is ascending. The base case is then satisfied when $\calC'=\calD$.  

By the induction assumption we know for  $\calC \precdot \calD\prec \calD' \precdot \calC'$ that there is such a vertex $v$ and $Q \in \calD'$ in which $P$ and $Q$ are in conflict at $v$.  Furthermore, without loss of generality we assume that $v$ is the minimal vertex of such a conflict and $v'$ the maximal satisfying that the conflict happens at the maximal subpath $vPv'=vQv'$. In particular, for every vertex~$u$ in $vPv'=vQv'$ we have $Qu\prec_{\Prefixes(u)}Pu$ and $uP\prec_{\Suffixes(u)}uQ$. We will show that the same holds for $\calC'$.

If $Q\in \calC'$ then $Q$ satisfies the desired conditions. So we will consider when  $Q\not \in \calC'$ and hence there is a minimal exchange quadrangle $\{Q, R, QwR, RwQ \}$ between $\calD'$ and $\calC'$ in which $Q$ is descending and $R\in \calC'$ is ascending. We show that at least one of the routes~$R$, $QwR$ or $RwQ$ (all which belong to $\calC'$) is in conflict with $P$ as we wanted. Consider $wQw'=wRw'$ to be the maximal subpath of this conflict. We break down the argument into the following cases, represented visually in Figure~\ref{fig:routesinconflict_v2}.
\begin{itemize}
    \item[Case 1.] When $vQv'\cap wQw' =\emptyset$, there are two subcases: 
    \begin{itemize}
        \item[Case 1a.] $v<v'<w<w'$. In this case we have that 
        $QwRv=Qv\prec_{\Prefixes(v)}Pv$ and $v'P\prec_{\Suffixes(v')}v'Q\prec_{\Suffixes(v')} v'QwR$, and so
        $P$ and $QwR$ are in conflict at $v$ with the desired conditions.
        \item[Case 1b.] $w<w'<v<v'$. In this case we have that $ RwQvP\prec_{\Prefixes(v)}Qv\prec_{\Prefixes(v)}v$ and $v'P\prec_{\Suffixes(v')}v'Q=v'RwQ$, and so $P$ and $RwQ$ are in conflict at $v$ with the desired conditions.
    \end{itemize}
    \item[Case 2.] When $vQv'\cap wQw' \neq \emptyset$, let $uQu'=vQv'\cap wQw'$. Hence, we have that $Ru\prec_{\Prefixes(u)}Qu\prec_{\Prefixes(u)}Pu$ and $u'P\prec_{\Suffixes(u')}u'Q\prec_{\Suffixes(u')}u'R$, and so $P$ and $R$ are in conflict at $u$ with the desired conditions. \qedhere
\end{itemize}
\end{proof}

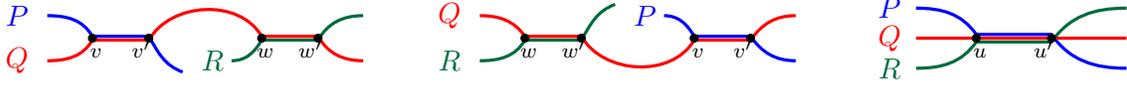
\begin{figure}[!t]
    \centering
    \begin{tikzpicture}



\begin{scope}[scale=1,shift={(0.25,0)}]

    \node at (1,-0.3){\small\textcolor{red}{$Q$}};
    \node at (1,0.3){\small\textcolor{blue}{$P$}};
    \node at (3.6,-0.3){\small\textcolor{cadmiumgreen}{$R$}};

    \draw[-, very thick, color=blue] (1.4,0.3) to [out=0,in=120] (2,0.025) to [out=0,in=180] (2.75,0.025) to [out=-60,in=160] (3.2,-0.45); 
    \draw[-, very thick, color=red] (1.4,-0.3) to [out=0,in=240] (2,-0.025) to [out=0,in=180] (2.75,-0.025) to [out=60,in=120] (4.25,0.025) to [out=0,in=180] (5,0.025) to [out=-60,in=180] (5.6,-0.3);
    \draw[-, very thick, color=cadmiumgreen] (3.85,-0.3) to [out=0,in=240] (4.25,-0.025) to [out=0,in=180] (5,-0.025) to [out=60,in=180] (5.6,0.3);

	\vertex[fill=black, minimum size=3pt](v2) at (2,0) {};
	\vertex[fill=black, minimum size=3pt](v3) at (2.75,0) {};
    \node at (2.05,-0.2){\tiny\textcolor{black}{$v$}};
    \node at (2.65,-0.15){\tiny\textcolor{black}{$v'$}};
	\vertex[fill=black, minimum size=3pt](v4) at (4.25,0) {};
	\vertex[fill=black, minimum size=3pt](v5) at (5,0) {};
    \node at (4.3,-0.2){\tiny\textcolor{black}{$w$}};
    \node at (4.9,-0.15){\tiny\textcolor{black}{$w'$}};
\end{scope}

\begin{scope}[scale=1,shift={(6,0)}]

    \node at (1,0.3){\small\textcolor{red}{$Q$}};
    \node at (1,-0.3){\small\textcolor{cadmiumgreen}{$R$}};
    \node at (3.6,0.3){\small\textcolor{blue}{$P$}};

    \draw[-, very thick, color=cadmiumgreen] (1.4,-0.3) to [out=0,in=240] (2,-0.025) to [out=0,in=180] (2.75,-0.025) to [out=60,in=200] (3.2,0.45); 
    \draw[-, very thick, color=red] (1.4,0.3) to [out=0,in=125] (2,0.025) to [out=0,in=180] (2.75,0.025) to [out=-60,in=240] (4.25,-0.025) to [out=0,in=180] (5,-0.025) to [out=55,in=180] (5.6,0.3);
    \draw[-, very thick, color=blue] (3.85,0.3) to [out=0,in=120] (4.25,0.025) to [out=0,in=180] (5,0.025) to [out=-55,in=180] (5.6,-0.3);

	\vertex[fill=black, minimum size=3pt](v2) at (2,0) {};
	\vertex[fill=black, minimum size=3pt](v3) at (2.75,0) {};
    \node at (2.05,-0.2){\tiny\textcolor{black}{$w$}};
    \node at (2.65,-0.15){\tiny\textcolor{black}{$w'$}};
	\vertex[fill=black, minimum size=3pt](v4) at (4.25,0) {};
	\vertex[fill=black, minimum size=3pt](v5) at (5,0) {};
    \node at (4.3,-0.2){\tiny\textcolor{black}{$v$}};
    \node at (4.9,-0.15){\tiny\textcolor{black}{$v'$}};
\end{scope}

\begin{scope}[scale=1,shift={(12,0)}]

    \node at (0.85,0.4){\small\textcolor{blue}{$P$}};
    \node at (0.85,0.0){\small\textcolor{red}{$Q$}};
    \node at (0.85,-0.4){\small\textcolor{cadmiumgreen}{$R$}};

    \draw[-, very thick, color=blue] (1.2,0.4) to [out=0,in=125] (2,0.05) to [out=0,in=180] (3,0.05) to [out=-55,in=180] (4,-0.4);
    \draw[-, very thick, color=red] (1.2,0) to [out=0,in=180] (2,0.0) to [out=0,in=180] (3,0.0) to [out=0,in=180] (4,0);
    \draw[-, very thick, color=cadmiumgreen] (1.2,-0.4) to [out=0,in=235] (2,-0.05) to [out=0,in=180] (3,-0.05) to [out=55,in=180] (4,0.4); 

	\vertex[fill=black, minimum size=3pt](v2) at (2,0) {};
	\vertex[fill=black, minimum size=3pt](v3) at (3,0) {};
    \node at (2.05,-0.2){\tiny\textcolor{black}{$u$}};
    \node at (2.9,-0.15){\tiny\textcolor{black}{$u'$}};

\end{scope}
\end{tikzpicture}
    \caption{(Left) In Case 1a, $P$ and $QwR$ are in conflict. (Center) In Case 1b, $P$ and $RwQ$ are in conflict. (Right) In Case 2, $P$ and $R$ are in conflict.}
    \label{fig:routesinconflict_v2}
\end{figure}

\begin{definition}[Minimal and maximal saturations]
Let $\Gamma \in \Groves(G,F)$. We define the \defn{minimal saturation} of $\Gamma$ to be $\underline \Gamma \in \SatGroves(G,F)$ by uniquely constructing $\underline \gamma_v$ as follows. Inductively at every $v=0,1,\dots,n$, add:
\begin{enumerate}
    \item $x_0$ to $L(\underline \gamma_0)$ if $x_0\not \in L(\underline \gamma_0)$ and at each $v\neq 0$ all prefixes of the form $Pe$ whenever $e=(w,v)$ and $(P,e)\in E(\underline \gamma_w)$;
    \item $e$ to $R(\underline \gamma_v)$ for every $e\in \outedge(v)\setminus \in R(\gamma_v)$;
    \item to $E(\underline \gamma_v)$ every edge of the form $(P,e)$ for every $P\in L(\underline \gamma_v)$ and where $e$ is the maximal edge in $\outedge(v)$ that is connected with the maximal $Q \in L(\gamma_v)$ such that $Q\prec_{\Prefixes(v)}P$;
    \item $(P,e)$ for every $e\in  R(\underline \gamma_v)$ where $P\in L(\gamma_v)$ is the minimal edge in $L(\gamma_v)$ connected to the minimal $e' \in R(\gamma_v)$ with $e' \succeq_{\outedge(v)}e$.
\end{enumerate}

Similarly, define the \defn{maximal saturation} of $\Gamma$ to be $\overline \Gamma \in \SatGroves(G,F)$ by uniquely constructing $\overline{\gamma}_v$ as follows: inductively at every $v=0,1,\dots,n$, add
\begin{enumerate}
    \item $x_0$ to $L(\overline \gamma_0)$ if $x_0\not \in L(\overline \gamma_0)$ and at each $v\neq 0$ all prefixes of the form $Pe$ whenever $e=(w,v)$ and $(P,e)\in E(\overline \gamma_w)$;
    \item $e$ to $R(\overline \gamma_v)$ for every $e\in \outedge(v)\setminus \in R(\gamma_v)$;
    \item to $E(\overline \gamma_v)$ every edge of the form $(P,e)$ for every  $P\in L(\overline \gamma_v)$ where $e$ is the minimal edge in $\outedge(v)$ that is connected with the minimal $Q \in L(\gamma_v)$ such that $P\prec_{\Prefixes(v)}Q$;
    \item $(P,e)$ for every $e\in R(\overline \gamma_v)$ where $P\in L(\gamma_v)$ is the maximal edge in $L(\gamma_v)$ connected to the maximal $e' \in R(\gamma_v)$ with $e \succeq_{\outedge(v)}e'$.
\end{enumerate}

We define \defn{minimal and maximal saturations} on the families  $\Cliques(G,F)$, $\Vines(G,F)$, and  $\PermutationFlows(G,F)$ to be the corresponding images of the minimal and maximal saturations under the corresponding bijections to $\Groves(G,F)$.
\end{definition}

\begin{proposition}\label{proposition:saturations_sec4} Let $\Theta \in \SatGroves(G,F)$ and $\Gamma \in \Groves(G,F)$ such that  $\Gamma \subseteq \Theta$. We have that $\underline \Gamma \leq \Theta \leq \overline \Gamma.$ 

Similarly, for  $\calC \in Cliques(G,F)$ and $\calD \in \SatCliques(G,F)$ such that  $\calC \subseteq \calD$, we have that $\underline \calC \leq \calD \leq \overline \calC.$
\end{proposition}
\begin{proof} We only show, inductively on $W(G,F)$, that $\underline \Gamma \leq \Theta$ since the dual argument is similar. Consider a possible counterexample in $W(G,F)$ such that $\Gamma \subseteq \Theta$ but $\underline \Gamma \not \leq \Theta$. 
Since $\Gamma \neq \Theta$ we can choose a minimal exchange quadrangle $(P,Q,e,t)$ at some $t=(v,w)$ in $\Theta$ such that $(P,t) \in E(\Theta)$ but $(P,t) \not \in E(\Gamma)$. This is possible since otherwise inductively from $v=0,\dots,n$ we would conclude that $\underline \Gamma = \Theta$. So using such a minimal exchange quadrangle it is still true that $\Gamma \subseteq L_t(\Theta)$ since $(P,t)\not \in E(\Gamma)$. Since $\Theta$ was a minimal counterexample we have that $\underline \Gamma \leq L_t(\Theta) \lessdot \Theta$, so such a counterexample does not exist. 
\end{proof}

\begin{corollary}\label{corollary:intersection_is_in_meet}
    Let $\calC,\calD \in W(G,F)$. Then
    $\calC \cap \calD \subseteq \calC \wedge \calD$.
\end{corollary}
\begin{proof}
    By Proposition \ref{proposition:saturations_sec4} we have that $\underline{\calC \cap \calD} \leq \calC$ and $\underline{\calC \cap \calD} \leq \calD$. Since a meet is a maximal common lower bound, we have $\underline{\calC \cap \calD}  \leq \calC \wedge \calD \leq \calC$ and $\underline{\calC \cap \calD}  \leq \calC \wedge \calD \leq  \calD$. Then Lemma \ref{lemma:nonrevisiting_property} implies the corollary.
\end{proof}

\begin{proof}[Proof of Theorem \ref{theorem:linear_extensions_W(G,F)_are_shellings} ]
    Let $\calC_1,\calC_2,\dots,C_k$ be a linear extension of $W(G,F)$. Then for $i<j$ we know by Corollary \ref{corollary:intersection_is_in_meet} that $\calC_i \cap \calC_j \subseteq \calC_i \wedge \calC_j$. Lemma \ref{lemma:nonrevisiting_property}  then implies that for any $\calC_r$ such that $\calC_i \wedge \calC_j\leq \calC_r \lessdot \calC_j$ we have that $\calC_i \cap \calC_j \subseteq \calC_r$ and so  $\calC_i \cap \calC_j \subseteq \calC_r \cap \calC_j $. Such $\calC_r$ also satisfies that $r<j$ because the listing is a linear extension and $|\calC_r \cap \calC_j|=|\calC_j|-1$ because of the definition of $W(G,F)$.
\end{proof}

\subsection{The \texorpdfstring{$G$}{G}-Eulerian Polynomial}
\label{sec:G-Eulerian}
\phantom{W}

We review some results from Ehrhart theory, which are found for example in~\cite{BeckRobins2015}.
Ehrhart showed~\cite{Ehrhart1962} that for a lattice $d$-dimensional polytope $P$, the function $|tP\cap \mathbb{Z}^d|$ that enumerates lattice points of the $t$-th dilation of $P$ is a polynomial in $t$.  This is the \defn{Ehrhart polynomial} of $P$. This polynomial is related via a transformation to the \defn{$h^*$-polynomial} of $P$ which carries the same information as the Ehrhart polynomial. The $h^*$-polynomial is known to have nonnegative coefficients~\cite{Stanley1980} and coincides with the $h$-polynomial of a shellable unimodular triangulation (see for example \cite[Theorem 10.3]{BeckRobins2015}). This polynomial can then be computed from a shelling $\calC_0,\calC_1,\dots, \calC_s$ of the triangulation as
\begin{equation}\label{equation:h_star_from_shelling}
    h^*(t)=\sum_{i=0}^s t^{|\calR(\calC_i)|},
\end{equation}
where $\calR(\calC_k)$ is the minimal face in the simplicial complex $\calC_k \setminus \cup_{i=0}^{k-1} \calC_i$, which coincides with the number of faces $\calC_j$ with $j<k$ and $|\calC_j\cap \calC_k|=|\calC_k|-1$. In our case, where the shelling comes from a linear extension of $W(G,F)$, $|\calR(\calC_k)|$ is the number of elements covered by $\calC_k$ in $W(G,F)$.

\begin{definition}[Eulerian polynomial of a framed graph]
Let $(G,F)$ be a framed graph. 
The \defn{$(G,F)$-Eulerian polynomial} is
\[A_{(G,F)}(t)=\sum_{\pi \in \SatPermutationFlows(G,F)}t^{\des(\pi)}.
\]
\end{definition}

\begin{theorem}\label{theorem:h_star_first}
    Let $(G,F)$ be a framed graph. The $h^*$-polynomial of $\calF_G$ is $A_{(G,F)}(t)$.
\end{theorem}
\begin{proof}
By Theorem \ref{theorem:linear_extensions_W(G,F)_are_shellings}, every linear extension of $W(G,F)$ induces a shelling of $\DKK(G,F)$.  We can use Equation \eqref{equation:h_star_from_shelling} to read the $h^*$-polynomial from any such linear extension by determining for every $\calC \in W(G,F)$ the number $|\calR(\calC)|$ of elements covered by $\calC$ in $W(G,F)$. By Proposition \ref{proposition:weak_order_permutationflows} this is precisely $\des(\pi(\calC))$.
\end{proof}

Since the $h^*$-polynomial is a geometric invariant of $\calF_G$, the $(G,F)$-Eulerian polynomial does not depend on the framing $F$.

\begin{corollary}\label{corollary:same_Eulerian_for_different_framings}
    For any two framings $F$ and $F'$ of $G$ we have that 
    $$A_{(G,F)}(t)=A_{(G,F')}(t).$$
\end{corollary}

Corollary \ref{corollary:same_Eulerian_for_different_framings} motivates the following definition.

\begin{definition}[Eulerian polynomial]\label{def:G-Eulerian_polynomial}
    The \defn{$G$-Eulerian polynomial} $A_{G}(t)$ is defined to be $A_{G}(t)=A_{(G,F)}(t)$ for any framing $F$ of $G$.
\end{definition}

\begin{corollary}
    For any framed graph $(G,F)$ we have that
    $$A_{G}(t)=\sum_{\pi \in \SatPermutationFlows(G,F)}t^{\asc(\pi)}.$$
\end{corollary}
\begin{proof}
 The reverse framing of $F$ interchanges $\des$ and $\asc$. Then use Corollary \ref{corollary:same_Eulerian_for_different_framings}. 
\end{proof}

Theorem \ref{theorem:h_star_first} can then be reformulated as Theorem \ref{theorem:h_star_second} after Definition \ref{def:G-Eulerian_polynomial}.

\newpage
\section{Combinatorial objects on framed augmented graphs}
\label{sec:combinatorial_families_on_augmented_graphs}

In this section we discuss several families of combinatorial objects which are in bijection and form the combinatorial backbone for describing a triangulation of the flow polytope $\calF_{G}(\ba)$ in Section \ref{sec:triangulation_flow_polytopes}. Each family is interesting in its own right and allows one to directly see different aspects of the structure of the triangulation. 

The following schematic diagram shows the order in which we will introduce the families of combinatorial objects associated to an augmented graph $(\hatG,\hatF)$. The geometric information that is encoded in each object is represented in an increasingly compact manner as we follow the arrows to the right.

\vspace{-.1in}
\noindent
\begin{center}
\tcbox[on line,left=2pt,right=2pt,top=3pt,bottom=3pt,colback=black!05,colframe=black!50,arc=3mm,boxrule=0.75pt]{\(~\begin{array}{ c c c c c c c c c }
\mathrm{Cliques} & \!\!\longrightarrow\!\! & 
\mathrm{Multicliques} & \!\!\longrightarrow\!\! & 
\mathrm{\begin{array}{ c }\mathrm{Vineyard}\\
\mathrm{Shuffles} \end{array}} & \!\!\longrightarrow\!\! 
& \mathrm{\begin{array}{ c } \mathrm{Grove}\\
\mathrm{Shuffles} \end{array}} & \!\!\longrightarrow\!\! 
& \mathrm{\begin{array}{ c } \mathrm{Permutation\, Flow}\\
\mathrm{Shuffles} \end{array}}\\
&&&&&&&&
\\
\mathcal{C} & \mapsto  & \mathcal{M} & \mapsto  & (\calV ,\sigma_{\calV} ) & \mapsto  & ( \Gamma,\sigma_{\Gamma} ) & \mapsto  & ( \pi,\sigma_{\pi} )
\end{array}\)}
\end{center}

The combinatorial structure that is most directly related to the triangulation of $\calF_G(\ba)$ is a clique. 
A clique is defined as a subset of lattice points of $\calF_G(\ba)$ that satisfy some pairwise coherency conditions, hence the set of cliques is closed under inclusion and naturally will carry the structure of a simplicial complex. The simple inclusion check for membership of faces translates to a partial order on the set of cliques. Moving through the bijections from the left to the right allows us to define a partial order on each set of combinatorial objects, so these bijections are in fact poset isomorphisms.

We choose to build the combinatorial objects from left-to-right from vertices $v=0$ to $v=n$. An alternative approach would be to build vineyard shuffles, grove shuffles, and permutation flow shuffles from right-to-left, using a perspective of suffixes instead of prefixes. 

A novel feature for the story for more general netflow vectors is the concept of shuffle. There is only one shuffle, the identity shuffle, when $\ba=\be_0-\be_n$. The diagram of diagram of bijections reduces to the one in Section~\ref{sec:permutationflows}. In addition, the concepts of cliques and multicliques coincide.

\subsection{Framed augmented graphs}
\label{sec:augmented_graphs}
\phantom{W}

The new combinatorial objects are defined on a framed augmented graph $(\hatG,\hatF)$, which extends the notion of a framed graph to include the information of the more general netflow vector $\ba=(a_0,a_1,\ldots, a_n)\in \bbZ^{n+1}$, where $\sum_v a_v=0$ and $a_v\geq 0$ for $v=0,\ldots, n-1$. We refer the reader back to Section~\ref{subsec.paths0} for the basics on framed graphs, which we extend in this section.

\begin{definition}[Framed augmented graph]\label{def:framed_augmented_graph}
Given a graph $G=(V,E)$ and a netflow vector~$\ba$, define the \defn{($\ba$-)augmented graph} $\hatG(\ba)=(V,\hat{E})$, where $\hat{E}=E\cup X\cup Y$ is an edge set including the edges $E$ of $G$, a set $X$ of $|a_n|=\sum_{v=0}^{n-1} a_v$ \defn{inflow half-edges}, and the set $Y=\{y\}$ consisting of one \defn{outflow half-edge} $y$. The inflow half-edges should satisfy $\big\lvert\{x\in X \mid \head(x)=v\}\big\rvert=a_v$ for $v=0,\ldots, n-1$ and the unique outflow half-edge satisfies $\tail(y)=n$. In most of what follows, whenever $\ba$ is understood from the context, we will use $\hatG$ instead of $\hatG(\ba)$ to simplify the notation.

The sets $\hat{\inedge}(v)$ and $\hat{\outedge}(v)$ of incoming and outgoing incidences in $\hatG$ are defined in a similar manner as $\inedge(v)$ and $\outedge(v)$ for $G$ but now include the half-edges from $X\cup Y$. 
In a parallel manner, a \defn{framing} of $\hatG$ is a collection of total orders of the sets $\hat{\inedge}(v)$ and $\hat{\outedge}(v)$ for $v \in [0,n]$. A framing of $\hatG$ induces a total order on $X$ respecting $x\prec x'$ when $\head(x)<\head(x')$, and so we will label the inflow half-edges $x_1,x_2,\dots,x_{|a_n|}$ to reflect this choice of total order. For a framed graph $(G,F)$, a framing $\hatF$ of $\hatG$ is \defn{consistent with~$F$}  if the elements of $E$ in $\hatF$ have the same relative orders as in $F$. We call $(\hatG,\hatF)$ a \defn{framed augmented graph} of $G$.
\end{definition}

\begin{example}
Figure~\ref{fig:auggraph} shows an example of a framed graph $(G,F)$ and a framed $\ba$-augmented graph $(\hatG,\hatF)$ consistent with $F$ with inflow half-edge set $X=\{x_1, x_2, x_3, x_4\}$ and outflow half-edge $y$. The inflow half-edges are incident with vertices $0$, $0$, $1$, and $3$, respectively consistent with netflow vector $\ba=(2,1,0,1,0,-4)$.
\end{example}

\begin{figure}[t!]
    \centering
    \vspace{-.6in}
    \begin{tikzpicture}
\begin{scope}[scale=1.5]
\node at (-1.5,0){$(G,F)$};

\vertex[fill=orange, minimum size=4pt, label=below:{\tiny\textcolor{orange}{$0$}}](v0) at (0,0) {};
\vertex[fill=orange, minimum size=4pt, label=below:{\tiny\textcolor{orange}{$1$}}](v1) at (1,0) {};
\vertex[fill=orange, minimum size=4pt, label=below:{\tiny\textcolor{orange}{$2$}}](v2) at (2,0) {};
\vertex[fill=orange, minimum size=4pt, label=below:{\tiny\textcolor{orange}{$3$}}](v3) at (3,0) {};
\vertex[fill=orange, minimum size=4pt, label=below:{\tiny\textcolor{orange}{$4$}}](v4) at (4,0) {};
\vertex[fill=orange, minimum size=4pt, label=below:{\tiny\textcolor{orange}{$5$}}](v5) at (5,0) {};		

\draw[-stealth, thick, color=black!30] (v0) .. controls (1.2, 1.6) and (2.5, -0.3) .. (v3);
\draw[-stealth, thick, color=black!30] (v0) .. controls (0.9, 1.0) and (1.5, -0.7) .. (v2);
\draw[-stealth, thick, color=black!30] (v0) to [out=30,in=150] (v1);
\draw[-stealth, thick, color=black!30] (v0) to [out=-30,in=-150] (v1);

\draw[-stealth, thick, color=black!30] (v1) .. controls (1.9, 1.0) and (2.5, -0.7) .. (v3);
\draw[-stealth, thick, color=black!30] (v1) to [out=30,in=150] (v2);
\draw[-stealth, thick, color=black!30] (v1) .. controls (2.0, -1.2) and (2.5, 0.7) .. (v3);	

\draw[-stealth, thick, color=black!30] (v2) to [out=45,in=135] (v4);
\draw[-stealth, thick, color=black!30] (v2) .. controls (3.0, -1.0) and (3.1, 0.3) .. (v4);	

\draw[-stealth, thick, color=black!30] (v3) to [out=60,in=120] (v5);
\draw[-stealth, thick, color=black!30] (v3) to [out=-30,in=-150] (v4);
\draw[-stealth, thick, color=black!30] (v3) .. controls (4.0, -1.0) and (4.5, 0.0) .. (v5);

\draw[-stealth, thick, color=black!30] (v4) to [out=30,in=150] (v5);
\draw[-stealth, thick, color=black!30] (v4) to [out=-30,in=-150] (v5);

\end{scope}

\begin{scope}[scale=1.5, yshift=-50]
\node at (-1.5,0){$(\hatG,\hatF)$};

\vertex[fill=black!30, minimum size=4pt, label=below:{\tiny\textcolor{black!30}{$0$}}](v0) at (0,0) {};
\vertex[fill=black!30, minimum size=4pt, label=below:{\tiny\textcolor{black!30}{$1$}}](v1) at (1,0) {};
\vertex[fill=black!30, minimum size=4pt, label=below:{\tiny\textcolor{black!30}{$2$}}](v2) at (2,0) {};
\vertex[fill=black!30, minimum size=4pt, label=below:{\tiny\textcolor{black!30}{$3$}}](v3) at (3,0) {};
\vertex[fill=black!30, minimum size=4pt, label=below:{\tiny\textcolor{black!30}{$4$}}](v4) at (4,0) {};
\vertex[fill=black!30, minimum size=4pt, label=below:{\tiny\textcolor{black!30}{$5$}}](v5) at (5,0) {};		

\draw[-stealth, thick, color=black!30] (v0) .. controls (1.2, 1.6) and (2.5, -0.3) .. (v3);
\draw[-stealth, thick, color=black!30] (v0) .. controls (0.9, 1.0) and (1.5, -0.7) .. (v2);
\draw[-stealth, thick, color=black!30] (v0) to [out=30,in=150] (v1);
\draw[-stealth, thick, color=black!30] (v0) to [out=-30,in=-150] (v1);

\draw[-stealth, thick, color=black!30] (v1) .. controls (1.9, 1.0) and (2.5, -0.7) .. (v3);
\draw[-stealth, thick, color=black!30] (v1) to [out=30,in=150] (v2);
\draw[-stealth, thick, color=black!30] (v1) .. controls (2.0, -1.2) and (2.5, 0.7) .. (v3);	

\draw[-stealth, thick, color=black!30] (v2) to [out=45,in=135] (v4);
\draw[-stealth, thick, color=black!30] (v2) .. controls (3.0, -1.0) and (3.1, 0.3) .. (v4);	

\draw[-stealth, thick, color=black!30] (v3) to [out=60,in=120] (v5);
\draw[-stealth, thick, color=black!30] (v3) to [out=-30,in=-150] (v4);
\draw[-stealth, thick, color=black!30] (v3) .. controls (4.0, -1.0) and (4.5, 0.0) .. (v5);

\draw[-stealth, thick, color=black!30] (v4) to [out=30,in=150] (v5);
\draw[-stealth, thick, color=black!30] (v4) to [out=-30,in=-150] (v5);

\draw[-stealth, thick, color=black] (-0.5,0.2) .. controls (-0.4, 0.2) and (-0.15, 0.1) .. (v0);
\draw[-stealth, thick, color=black] (-0.5, -.4) .. controls (-0.4, -.4) and (-0.25, -.3) .. (v0);
\draw[-stealth, thick, color=black] (0.5, -.4) .. controls (0.6, -.4) and (0.75, -.3) .. (v1);
\draw[-stealth, thick, color=black] (2.5, -.5) .. controls (2.6, 0) and (2.7, 0.1) .. (v3);

\draw[-stealth, thick, color=black] (v5) to [out=20, in=160] (5.5, 0);

\node[] at (-0.6, -0.4){\scriptsize\textcolor{red}{$x_1$}};
\node[] at (-0.6, 0.2){\scriptsize\textcolor{darkyellow}{$x_2$}};
\node[] at (0.4, -0.4){\scriptsize\textcolor{cadmiumgreen}{$x_3$}};
\node[] at (2.5, -0.6){\scriptsize\textcolor{blue}{$x_4$}};
\node[] at (5.6, 0){\scriptsize\textcolor{black}{$y$}};

\node[] at (.5, .6){\scriptsize\textcolor{black}{$t_1$}};
\node[] at (.4, .3){\scriptsize\textcolor{black}{$t_2$}};
\node[] at (.3, .1){\scriptsize\textcolor{black}{$t_3$}};
\node[] at (1.5, .38){\scriptsize\textcolor{black}{$t_4$}};
\node[] at (1.4, .12){\scriptsize\textcolor{black}{$t_5$}};
\node[] at (2.4, .4){\scriptsize\textcolor{black}{$t_6$}};
\node[] at (3.6, .55){\scriptsize\textcolor{black}{$t_7$}};
\node[] at (3.5, -0.2){\scriptsize\textcolor{black}{$t_8$}};
\node[] at (4.3, .15){\scriptsize\textcolor{black}{$t_9$}};
\node[] at (.3, -.2){\scriptsize\textcolor{black}{$s_0$}};
\node[] at (1.5, -.45){\scriptsize\textcolor{black}{$s_1$}};
\node[] at (2.8, -.45){\scriptsize\textcolor{black}{$s_2$}};
\node[] at (3.7, -.5){\scriptsize\textcolor{black}{$s_3$}};
\node[] at (4.3, -.1){\scriptsize\textcolor{black}{$s_4$}};

\end{scope}

\end{tikzpicture}
    \vspace{-.15in}
     \caption{A framed graph $(G,F)$, and a framed $\ba$-augmented graph $(\hatG,\hatF)$ associated to the netflow vector $\ba=(2,1,0,1,0,-4)$.}
    \label{fig:auggraph}
\end{figure}
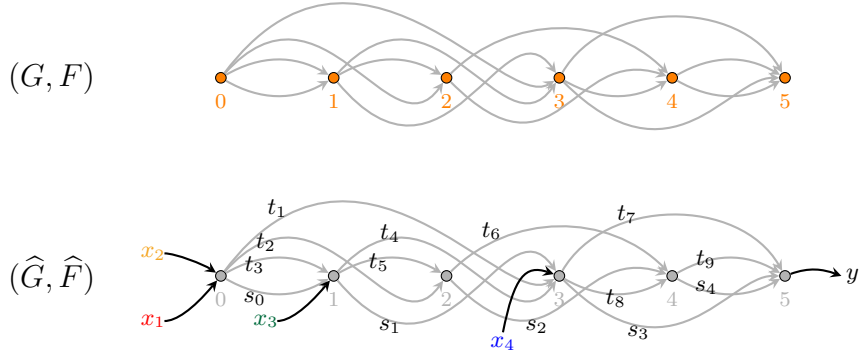

\subsection{Routes, route matchings, and cliques}
\label{sec:route_matchings}
\phantom{W}

A \defn{prefix} is any path that starts with an inflow half-edge, and a \defn{suffix} is any path that ends with the outflow half-edge $y$. A \defn{route} $P$ in $\hatG$ is a path that is simultaneously a prefix and a suffix. We denote by $\Prefixes(\hatG)$, $\Suffixes(\hatG)$, and $\Routes(\hatG)$ the sets of all prefixes, suffixes, and routes in $\hatG$, respectively. 

We say that a prefix (suffix) $P$ is a \defn{prefix (suffix) of $Q$} if every edge of $P$ is an edge of~$Q$. As such, $\Prefixes(\hatG)$ and $\Suffixes(\hatG)$ are posets under containment. Starting with a route~$P$ and an inner vertex $v$ in $P$, we can divide $P$ at $v$ into the prefix $Pv$ and the suffix $vP$. We denote by $\Prefixes(v)$ and $\Suffixes(v)$ the set of prefixes (suffixes) that end (begin) at vertex~$v$.

Define the \defn{initial edge} $\initial(P)$ and the \defn{terminal edge} $\terminal(P)$ of any path $P$ to be its first and last (full or half) edge, respectively. Whenever~$P$ is a prefix that ends at a vertex~$v$ and $e=(v,\cdot)$ is a (full or half) edge of $\hatG$, we denote by $Pe$ the \defn{extension of $P$ by $e$}.

The \defn{length} $\length(P)$ of a path $P$ is defined to be its number of full edges. For any inflow half-edge $x$, denote by $v_x$ to be its unique vertex. 

At any vertex $v$, the framing $\hatF$ extends to total orders on the sets $\Prefixes(v)$ and $\Suffixes(v)$, as in the case of a framed graph $(G,F)$. 

For any half-edge $x$, we denote by $\Routes(\hatG,x)$ the set of routes that start with $x$. By ignoring the common half-edge $x$ in each route, $\Routes(\hatG,x)$ inherits a total order $\prec_x$ from $\Suffixes(v_x)$.

\begin{definition}[Route matching] \label{def:route_matchings}
A \defn{route matching} is a set $\calP=\{P_x\mid x\in X\}$ of coherent routes with exactly one route starting at each inflow half-edge $x\in X$. 
We denote the set of route matchings by $\RouteMatchings(\hatG,\hatF)$.    
\end{definition}

Figure~\ref{fig:FSM_routematching} shows five examples of route matchings on the augmented graph $\hatG$ of Figure~\ref{fig:auggraph}.

\begin{figure}[ht!]
    \centering
    \vspace{-.3in}
    \scalebox{0.9}{\input{pictures/FSM_routematching}}
     \caption{A rank $4$ clique $\calC = \{\calP^0, \ldots, \calP^4\}$ of route matchings on the framed augmented graph $(\hatG, \hatF)$ from Figure~\ref{fig:auggraph}.
     }
    \label{fig:FSM_routematching}
\end{figure}

Route matchings are in correspondence with the lattice points in $\calF_G^\bbZ(\ba)$. 
Indeed, for a route $P$, define its \defn{indicator flow} $\bz(P)\in[0,1]^E$ to be the flow for which $\bz(P)(e)=1$ when $e\in P$ and $\bz(P)(e)=0$ if $e\notin P$. 
For a route matching $\calP=\{P_x\mid x\in X\}$, define its \defn{indicator flow} $\bz(\calP)\in\bbZ^E_{\geq 0}$ to be the flow $\bz(\calP)=\sum_{x\in X} \bz(P_x)$.
We have the following proposition.

\begin{proposition}\label{proposition:matchings_are_integer_flows}
The map $\bz$  is a bijection $$\bz:\RouteMatchings(\hatG,\hatF)\rightarrow \calF_G^\bbZ(\ba).$$
\end{proposition}
\begin{proof}
    Let $\calP\in \RouteMatchings(\hatG,\hatF)$.
    It has $a_0$ inflow half-edges incoming to vertex $0$ and the routes in $\calP$ dictate a unique way to distribute these prefixes across the set $\outedge(0)$ satisfying the  conservation of flow of Equation~\eqref{eqn:defining_equations}. Inductively at every vertex $v\in [0,n]$, there are $\bz(e)$ many prefixes for every $e\in \inedge(v)$ plus the $a_v$ prefixes corresponding to inflow half-edges, which have a unique extension to $\outedge(v)$ satisfying Equation~\eqref{eqn:defining_equations}. 
\end{proof}

An example illustrating the bijection of Proposition \ref{proposition:matchings_are_integer_flows} is given in Figure~\ref{fig:PS_polytope_with_route_matchings_versionM}.  Each lattice point of the flow polytope $\calF_{\PS_3}(1,1,1,-3)$ is represented as a route matching for the specific extended framing of the Pitman--Stanley graph $\PS_3$. The map $\bz$ can be visualized comparing  Figure~\ref{fig:PS_polytope_with_integer_flows_versionM} and Figure~\ref{fig:PS_polytope_with_route_matchings_versionM} side-by-side.

\begin{figure}[ht!]
    \centering
    \scalebox{0.8}{\input{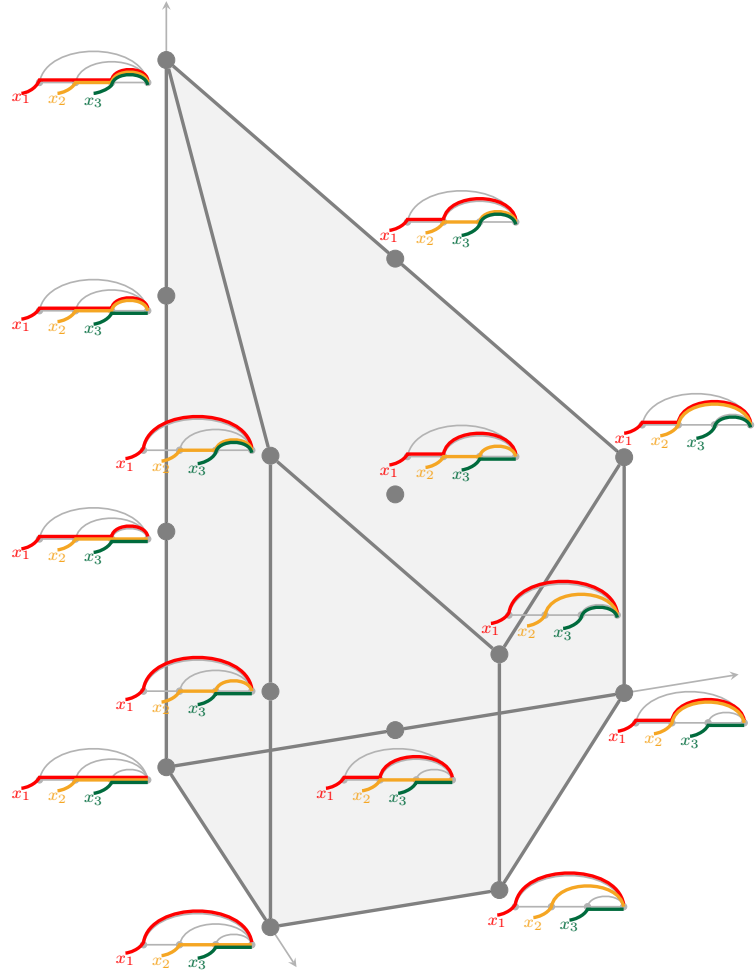}}
    \caption{The Pitman--Stanley polytope $\calF_{\PS_3}(1,1,1,-3)$ with lattice points represented as route matchings.
    }
    \label{fig:PS_polytope_with_route_matchings_versionM}
\end{figure}

We now extend the notion of coherence of routes to the notion of coherence of route matchings.
\begin{definition}[Coherence of route matchings] 
Two route matchings \[\calP=\{P_x \mid x\in X\} \textup{ and } \calQ=\{Q_x \mid x\in X\}\]  are said to be \defn{locally coherent} if the set of routes $\calP \cup \calQ$ is coherent.

$\calP$ and $\calQ$ are said to have a \defn{global conflict (between $x$ and $x'$)} if there exist two distinct inflow half-edges $x$ and $x'$ satisfying $P_x\preceq_{x}Q_x$ and $Q_{x'}\preceq_{x'}P_{x'}$. They are said to be \defn{globally coherent} if they do not have any global conflicts. 

Two route matchings $\calP$ and $\calQ$ are \defn{coherent} if they are both locally and globally coherent.
\end{definition}

\begin{definition}[Clique]
\label{def:cliques}
A set $\calC$ of route matchings whose elements are pairwise coherent is a \defn{clique}.  
We denote by $\Cliques(\hatG,\hatF)$ the set of cliques of $(\hatG,\hatF)$.
\end{definition}

A partial order on $\Cliques(\hatG,\hatF)$ is defined by containment $\calC' \subseteq \calC$.  By definition, $\Cliques(\hatG,\hatF)$ is closed under inclusion; as a consequence, it is a simplicial complex and the $\rank(\calC):=|\calC|-1$ is the \defn{rank} or \defn{dimension} of a simplex $\calC \in \Cliques(\hatG,\hatF)$. Note that $\rank(\calC')\le \rank(\calC)$ whenever $\calC' \subseteq \calC$. We denote by $\SatCliques(\hatG,\hatF)$ the set of cliques that are maximal with respect to the inclusion order, which we call \defn{saturated}. 
We will prove in Corollary~\ref{cor:dimension_d} that saturated cliques all have the same rank, $d=m-n$, the dimension of $\calF_{G}(\ba)$.

The definition of global coherence induces a total order $\preceq_{\calC}$ on the route matchings in a clique $\calC \in \Cliques(\hatG,\hatF)$ as follows. For two route matchings $\calP, \calQ \in \calC$, we say that $\calP \preceq_{\calC} \calQ$ if the routes $P_x$ and $Q_x$ satisfy  $P_x\preceq_{x}Q_x$ for each $x\in X$. As a consequence of coherence, any clique of rank $k$ can be written in \defn{standard form} $\calC=\{\calP^0,\ldots,\calP^k\}$ where $\calP^i\prec_{\calC} \calP^j$ whenever $i<j$.

\begin{example}\label{eg.cliques}
The five route matchings shown in Figure~\ref{fig:FSM_routematching} are pairwise coherent, so they form a rank $4$ clique $\calC=\{\calP^0, \ldots, \calP^4\}$, expressed in standard form.  
\end{example}

\begin{remark}
In the case $\ba = \be_0- \be_n$, each route matching consists of exactly one route and the notion of coherence of route matchings is the same as the notion of coherence of routes in the sense of~\cite{DanilovKarzanovKoshevoy2012}. The standard form $\calC = \{\calP^0,\ldots, \calP^k \}$ for a clique of route matchings then translates to a total order on the routes in a clique.
\end{remark}

\subsection{Multicliques} 
\label{sec:multicliques}
\phantom{W}

We can say more about the combinatorial structure of cliques when we consider all routes beginning at the same inflow half-edge together across the different route matchings in the clique. We introduce a second type of combinatorial object that can be used to describe cliques more efficiently in this way.

Let $\calM$ be a multiset of routes in $(\hatG, \hatF)$. The \defn{canonical partition} of $\calM$ is the multiset partition $\{\calM_x\}_{x \in X}$ where $\calM_x:=\{M\in \calM\mid \initial(M)=x\}$ is the multiset of routes starting at the inflow half-edge $x$.
Since each $\calM_x$ is totally ordered it can be written in \defn{standard form} as $\calM_x=\{M_x^0,\dots,M_x^k\}$ where $M_x^i\preceq_{x} M_x^j$ for $i<j$. 

Given a clique $\calC=\{\calP^0,\ldots,\calP^k\}$, consider the multiset of coherent routes $\calM(\calC)=\bigcup_{i=0}^k \calP^i$ that appear in the route matchings of $\calC$. Note that the canonical partition  $\{\calM(\calC)_x\}_{x \in X}$ satisfies $|\calM(\calC)_x|=k+1$ since $\calP^i$ contains a unique route starting at $x$ for every $x \in X$ and for every $i\in[0,k]$.
Furthermore, for each $i=1,\ldots,k$, there exists an $x\in X$ such that $M_x^{i-1} \neq M_x^i$ since all the route matchings $\calP^0,\ldots, \calP^k$ are distinct.
These properties are captured by the next definition.

\begin{definition}[Multiclique]\label{def:multicliques}
A multiset $\calM$ of routes is called a \defn{multiclique} if
\begin{enumerate}[label=(\alph*)]
    \item $\calM$ is coherent, 
    \item for some fixed $k\ge 0$ and for each $x \in X$ we have $|\calM_x|=k+1$, and
    \item for each $i=1,\dots,k$, there exists $x \in X$ such that $M_x^{i-1} \neq M_x^{i}$.
\end{enumerate}
The quantity $k$ is the \defn{rank} of $\calM$.
We denote by $\Multicliques(\hatG,\hatF)$ the set of multicliques of $(\hatG,\hatF)$. 
\end{definition}

For $\calM \in \Multicliques(\hatG,\hatF)$ we will use $\Routes(\calM)$ to denote the underlying set of distinct routes of $\calM$.

\begin{proposition}
The map \[\calM:\Cliques(\hatG,\hatF)\rightarrow \Multicliques(\hatG,\hatF),
\]
where $\calM(\calC) = \bigcup_{i=0}^k \calP^i$ as a multiset, is a bijection.
\end{proposition}
\begin{proof}
$\calM(\calC)$ is a multiclique by definition.
We construct the inverse map 
\[\calC:\Multicliques(\hatG,\hatF)\rightarrow \Cliques(\hatG,\hatF)
\]
as follows. Let $\calM \in \Multicliques(\hatG,\hatF)$, with canonical partition $\{\calM_x\}_{x \in X}$ where $\calM_x=\{M_x^0,\ldots,M_x^{k}\}$ is written in standard form. Then
\[\calC(\calM):=\{\calP^0(\calM),\ldots,\calP^{k}(\calM)\}\] where  $\calP^i(\calM):=\{M_x^i\mid x \in X\}$ for $i=0,\dots,k$.
This map is well defined because for every $0\le i<j \le k$ and every $x,x'\in X$ we have $M_x^i\preceq_{x} M_x^j$ and $M_{x'}^i\preceq_{x'} M_{x'}^j$. This ensures $\calC(\calM)$ is globally coherent, and hence is a valid clique. The corresponding total orders in $\calC$ and $\calM$ directly imply that these two maps are inverses of each other.
\end{proof}

Given a clique $\calC = \{\calP^0,\ldots, \calP^k\}$ on $(\hatG,\hatF)$, we can represent the multiclique $\calM=\calM(\calC)$ as an $|X|\times (k+1)$ array whose columns are the route matchings in $\calC$ and whose rows are the parts in the canonical partition of $\calM$ in standard form.
That is, the entry $M_x^i$ in row $\calM_x$ and column $\calM^i$ is the unique route in $\calP^i$ starting at the inflow half-edge $x$. See Figure~\ref{fig:cliques_and_multicliques} for the multiclique array representation of the clique $\calC$ of Figure~\ref{fig:FSM_routematching}.

\begin{figure}[h!]
\begin{center}
\scalebox{0.8}
{$
\begin{array}{ c|c|c|c|c|c }
&
\cellcolor{black!10} \calM^{0}& 
\cellcolor{black!14} \calM^{1} & 
\cellcolor{black!18} \calM^{2} & 
\cellcolor{black!22} \calM^{3} & 
\cellcolor{black!26} \calM^{4}\\
\rowcolor{red!30}
\calM_{x_1}& 
x_1s_0s_1t_7y & - & x_1t_3t_5t_6t_9y & x_1t_3t_4s_3y & x_1t_2s_2t_9y\\
\rowcolor{darkyellow!30}
\calM_{x_2}&
x_2t_2s_2t_9y & x_2t_2t_6t_9 y & - & x_2t_1s_3y & -\\
\rowcolor{darkgreen!30}
\calM_{x_3}&
x_3s_1t_8t_9y & - & - & - & -\\
\rowcolor{blue!30}
\calM_{x_4}&
x_4t_8t_9y & - &- & - & -
\end{array}
$}
\end{center}
    
    \caption{The multiclique $\calM(\calC)$ representing the clique $\calC$ from Figure~\ref{fig:FSM_routematching}. 
    A dash ($-$) indicates an entry that is equal to the one to its left.}
    \label{fig:cliques_and_multicliques}
\end{figure}

\begin{figure}[h!]
\begin{center}
\scalebox{0.78}
{$
\begin{array}{ c|c|c|c|c|c|c|c|c|c|c }
 & 
\calN^{0} & 
\calN^{1} & 
\cellcolor{black!10} \calN^{2} & 
\cellcolor{black!14} \calN^{3} & 
\calN^{4} & 
\cellcolor{black!18} \calN^{5} & 
\calN^{6} & 
\cellcolor{black!22} \calN^{7} & 
\cellcolor{black!26} \calN^{8} & 
\calN^{9}\\
\hline
\rowcolor{red!30}  \calN_{x_{1}} &  x_{1} s_{0} s_{1} t_{8} t_{9} y & x_{1} s_{0} s_{1} t_{7} y & - & - & x_{1} s_{0} t_{5} t_{6} t_{9} y & x_{1} t_{3} t_{5} t_{6} t_{9} y & - & x_{1} t_{3} t_{4} s_{3} y & x_{1} t_{2} s_{2} t_{9} y & -\\
\rowcolor{darkyellow!30}
\calN_{x_{2}} & x_{2} t_{2} s_{2} t_{9} y & - & - & x_{2} t_{2} t_{6} t_{9} y & - & - & x_{2} t_{1} s_{3} y & - & - & x_{2} t_{1} t_{8} s_4 y\\
\rowcolor{darkgreen!30}
\calN_{x_{3}} & x_{3} s_{1} t_{8} t_{9} y & - & - & - & - & - & - & - & - & -\\
\rowcolor{blue!30}
\calN_{x_4} & x_4t_8s_4y & - & x_4t_8t_9y & - & - & - & - & - & - & -
\end{array}
$}
\end{center}

    \caption{A saturated multiclique $\calN$ containing $\calM$ in Figure \ref{fig:cliques_and_multicliques} as a submulticlique. The shading of the column headings indicates the containment.   
    A dash ($-$) indicates an entry that is equal to the one to its left.}
    \label{fig:SatCliques_and_multicliques}
\end{figure}

Recall that the set $\Cliques(\hatG, \hatF)$ is naturally ordered by containment.
We can define the corresponding order relation on  $\Multicliques(\hatG,\hatF)$ consistent with the bijection $\calM$ as follows. Let $\calM,\calN\in \Multicliques(\hatG,\hatF)$. We say that $\calM\subseteq \calN$ whenever there exists a strictly increasing function 
\begin{equation}\label{eqn:strictly_increasing}
f:[0,\rank(\calM)]\rightarrow [0,\rank(\calN)]    
\end{equation}
such that for all $x \in X$ and for every $j\in [0,\rank(\calM)]$ and every $x\in X$, 
\begin{equation}\label{equation:face_membership}
    M_x^j=N_x^{f(j)}.
\end{equation}
We note that Equation \eqref{equation:face_membership} implies, in particular, that $\Routes(\calM) \subseteq \Routes(\calN)$. The strictly increasing function $f$ acts by selecting the route matchings that remain in the submulticlique $\calM$. 
Hence, we conclude with the following proposition.

\begin{theorem} \label{proposition:CtoM}
The map 
\[\calM:\Cliques(\hatG,\hatF)\rightarrow \Multicliques(\hatG,\hatF)
\]
is a poset isomorphism.
\qed
\end{theorem}

The poset $\Multicliques(\hatG,\hatF)$ is then graded with rank function given by $\rank(\calM)=k$, one less than the common multiplicity of all the sets $\calM_x$.
We denote by $\SatMulticliques(\hatG,\hatF)$ the set of multicliques that are maximal with respect to this partial order, which we call \defn{saturated}. 

\begin{example}\label{eg.multiclique}
Let $\calC = \{\calP^0,\ldots, \calP^4\}$ be the rank $4$ clique from Example~\ref{eg.cliques}.
Figure~\ref{fig:cliques_and_multicliques} shows the array that represents the multiclique $\calM=\calM(\calC)$.
Let $\calD = \{\calQ^0, \ldots, \calQ^9 \}$ be the clique corresponding to the multiclique $\calN = \calM(\calD)$ shown in Figure~\ref{fig:SatCliques_and_multicliques}.

The strictly increasing function $f:[0,4]\rightarrow [0,9]$ given by 
\[f(0) = 2, \ f(1) = 3, \ f(2) = 5, \ f(3) = 7, \ f(4) = 8
\]
satisfies $M_x^j = N_x^{f(j)}$ for $j\in [0,4]$ and $x\in X$, thus $\calM \leq \calN$. 
The clique $\calD$ is saturated since it contains $d+1=10$ route matchings. 
\end{example}

\begin{remark}\label{rem.multicliques} 
In the case $\ba = \be_0- \be_n$, $\hatG$ has a unique inflow half-edge $x$, thus the notions multiclique (of routes) and the notion of clique (of routes) coincide.
\end{remark}

\subsection{Vineyard shuffles}
\label{sec:vineyardshuffles}
\phantom{W}

The combinatorial information that is contained in a multiclique can be condensed while keeping track of the multiplicity of repeated routes, giving rise to the concepts of vineyards and vineyard shuffles. Recall from Section \ref{sec:route_matchings} that $\Prefixes(\hatG)$ has the structure of a poset given by containment of prefixes. Notice also that the Hasse diagram of $\Prefixes(\hatG)$ is a forest because every prefix $P\notin X$ covers exactly one element---the prefix that deletes the last edge of $P$. The connected components of $\Prefixes(\hatG)$ are trees in correspondence with the elements of $X$.

\begin{definition}[Vineyard]\label{def:vineyard}
Let $\hatG$ be an augmented graph with $X$ its set of inflow half-edges.
A \defn{vineyard} in $\hatG$ is a subposet $\calV\subseteq \Prefixes(\hatG)$, when non-empty, its elements satisfy
\begin{enumerate}[label=(\alph*)]
    \item $X\subseteq \calV$,
    \item $\calV$ is coherent, 
    \item $\calV$ is closed under containment of prefixes, and
    \item every element in $\calV$ is extendable to a route in $\calV$; that is, the maximal elements in $\calV$ under containment are routes.
\end{enumerate}
The Hasse diagram of $\calV$ is also a forest (a subforest of $\Prefixes(\hatG)$) whose trees are indexed by $X$. We call these connected components \defn{vines}, which we denote by $\calV_x$ for each $x\in X$. We denote by $\Vineyards(\hatG,\hatF)$ the set of vineyards of $(\hatG,\hatF)$.
\end{definition}

Note that every vineyard $\calV$ is defined by its set of routes, which we denote $\Routes(\calV)$. 
By part (d) of Definition \ref{def:vineyard}, each prefix $P\in \calV$ that is not a route has at least one route that extends it.  
For each $x\in X$, denote by
\begin{equation}\label{equation:order_of_routes_in_vine}
    \Routes(\calV;x):=\left\{V_x^0,V_x^1,\dots,V_x^{l_x} \right\}
\end{equation}
the set of routes that start with $x$, which are totally ordered by $\prec_x$ at $v_x$. Keeping this order for every $x\in X$, we can then further define for a $P \in \calV_x$ the set of routes that contain $P$ as
$$\Routes(\calV;P):=\left \{V_x^i,V_x^{i+1}\dots,V_x^{j}\right \}.$$
They form a consecutive interval in $\Routes(\calV;x)$ according to $\prec_x$. 

We call the route $V_x^i$ the \defn{minimal route extension} of $P$ with respect to $\calV$ and denote it by \defn{$\minext_{\calV}(P)$}. 
Note that $\minext_{\calV}(x)=V_x^0$ for every $x\in X$.

There is a natural partial order on $\Vineyards(\hatG,\hatF)$ given by containment $\calV \subseteq \calW$. (It suffices to check that $\Routes(\calV)\subseteq \Routes(\calW)$). Denote by $\SatVineyards(\hatG,\hatF)$ the set of vineyards that are maximal with respect to this order, which we call \defn{saturated}.

Given a prefix $P\in \calV$, its set of \defn{covers}
\begin{equation}\label{equation:covers}
    \Covers(\calV;P):=\left\{Q_0,Q_1,\dots,Q_l\right\},
\end{equation}
is defined as the set of prefixes in $\calV$ that extend $P$ by exactly one edge. It inherits a total ordering from $\hat \outedge(v)$ at the last vertex $v$ of $P$. We call $Q_0$ the \defn{minimal continuation} or \defn{next step} of $P$ in 
$\calV$ and denote it by \defn{$\next_{\calV}(P)$}.

\begin{example}\label{eg.subvines}
Figure~\ref{fig:FSM_subvines} shows the Hasse diagram (read from left to right) of a vineyard $\calV = \{\calV_{x_1}, \ldots, \calV_{x_4}\}$, superimposed on the augmented graph $\hatG$.
This representation of $\calV$ is useful in visualizing the prefix of $\hatG$ that corresponds to a particular element of $\calV$.

The reader can see that three elements of $\calV$ are superimposed onto vertex $2$ corresponding to the prefixes $x_1t_2$, $x_2t_2$, and $x_1t_3t_5$.

The vineyard $\calV$ is completely determined by its routes:
\begin{align*}
\Routes(\calV; x_1) 
    &= \{V_{x_1}^0, \ldots, V_{x_1}^3\}=\{x_1s_0s_1t_7y,\ x_1t_3t_5t_6t_9y,\  x_1t_3t_4s_3y,\ x_1t_2s_2t_8y\},\\
\Routes(\calV; x_2) 
    &= \{V_{x_2}^0, \ldots, V_{x_2}^2\}=\{x_2t_2s_2t_9y,\  x_2t_2t_6t_9y,\  x_2t_1s_3y\},\\
\Routes(\calV; x_3) 
    &= \{V_{x_3}^{0}\} =\{ x_3s_1t_8t_9y\},\\
\Routes(\calV; x_4) 
    &= \{V_{x_4}^{0}\} =\{ x_4t_8t_9y\}.
\end{align*}

The prefix $P=x_2t_2$ has $\Routes(\calV;P) = \{V_{x_2}^0, V_{x_2}^1 \}$ so $\minext_{\calV}(P)=V_{x_2}^0=x_2t_2s_2t_9y$. Furthermore, $\Covers(\calV;P)=\{x_2t_2s_2, x_2t_2t_6\}$ so $\next_{\calV}(P)=x_2t_2s_2$.
\end{example}

\begin{figure}[ht!]
    \centering\vspace{-.5in}
    \begin{tikzpicture}
\begin{scope}[scale=1.5, yshift=60, yscale=1.2]
\node at (-1.5,0){$\calV$};

\vertex[fill, color=black!10, label=below:{\tiny\textcolor{black!30}{$0$}}](v0) at (0,0) {};
\vertex[fill, color=black!10, label=below:{\tiny\textcolor{black!30}{$1$}}](v1) at (1,0) {};
\vertex[fill, color=black!10, label=below:{\tiny\textcolor{black!30}{$2$}}](v2) at (2,0) {};
\vertex[fill, color=black!10, label=below:{\tiny\textcolor{black!30}{$3$}}](v3) at (3,0) {};
\vertex[fill, color=black!10, label=below:{\tiny\textcolor{black!30}{$4$}}](v4) at (4,0) {};
\vertex[fill, color=black!10, label=below:{\tiny\textcolor{black!30}{$5$}}](v5) at (5,0) {};

\node[] at (-0.6, -0.5){\scriptsize\textcolor{red}{$\calV_{x_1}$}};
\node[] at (-0.7, 0.2){\scriptsize\textcolor{darkyellow}{$\calV_{x_2}$}};
\node[] at (0.4, -0.5){\scriptsize\textcolor{cadmiumgreen}{$\calV_{x_3}$}};
\node[] at (2.5, -0.7){\scriptsize\textcolor{blue}{$\calV_{x_4}$}};
\node[] at (5.5, -0.45){\scriptsize\textcolor{black!30}{$y$}};

\node[] at (.5, .5){\scriptsize\textcolor{black!30}{$t_1$}};
\node[] at (.4, .3){\scriptsize\textcolor{black!30}{$t_2$}};
\node[] at (.3, .1){\scriptsize\textcolor{black!30}{$t_3$}};
\node[] at (1.4, .35){\scriptsize\textcolor{black!30}{$t_4$}};
\node[] at (1.3, .15){\scriptsize\textcolor{black!30}{$t_5$}};
\node[] at (2.4, .4){\scriptsize\textcolor{black!30}{$t_6$}};
\node[] at (3.6, .55){\scriptsize\textcolor{black!30}{$t_7$}};
\node[] at (3.6, -0.27){\scriptsize\textcolor{black!30}{$t_8$}};
\node[] at (4.3, .25){\scriptsize\textcolor{black!30}{$t_9$}};
\node[] at (.3, -.15){\scriptsize\textcolor{black!30}{$s_0$}};
\node[] at (1.5, -.5){\scriptsize\textcolor{black!30}{$s_1$}};
\node[] at (2.8, -.45){\scriptsize\textcolor{black!30}{$s_2$}};
\node[] at (3.7, -.5){\scriptsize\textcolor{black!30}{$s_3$}};
\node[] at (4.2, -.17){\scriptsize\textcolor{black!30}{$s_4$}};

\draw[thin, color=black!30] (v0) .. controls (1.2, 1.6) and (2.5, -0.3) .. (v3);
\draw[thin, color=black!30] (v0) .. controls (0.9, 1.0) and (1.5, -0.7) .. (v2);
\draw[thin, color=black!30] (v0) to [out=30,in=150] (v1);
\draw[thin, color=black!30] (v0) to [out=-30,in=-150] (v1);

\draw[thin, color=black!30] (v1) .. controls (1.9, 1.0) and (2.5, -0.7) .. (v3);
\draw[thin, color=black!30] (v1) to [out=30,in=150] (v2);
\draw[thin, color=black!30] (v1) .. controls (2.0, -1.2) and (2.5, 0.7) .. (v3);	

\draw[thin, color=black!30] (v2) to [out=45,in=135] (v4);
\draw[thin, color=black!30] (v2) .. controls (3.0, -1.0) and (3.1, 0.3) .. (v4);	

\draw[thin, color=black!30] (v3) to [out=60,in=120] (v5);
\draw[thin, color=black!30] (v3) to [out=-30,in=-150] (v4);
\draw[thin, color=black!30] (v3) .. controls (4.0, -1.0) and (4.5, 0.0) .. (v5);

\draw[thin, color=black!30] (v4) to [out=30,in=150] (5,0);
\draw[thin, color=black!30] (v4) to [out=-30,in=-150] (v5);

\draw[thin, color=black!30] (-0.5,0.2) .. controls (-0.4, 0.2) and (-0.15, 0.1) .. (v0);
\draw[thin, color=black!30] (0.5, -.4) .. controls (0.6, -.4) and (0.75, -.3) .. (v1);
\draw[thin, color=black!30] (2.5, -.5) .. controls (2.6, 0) and (2.7, 0.1) .. (v3);

\draw[thin, color=black!30] (v5) to [out=20, in=160] (5.5, 0);

\vnode[color=red](vv0) at (0,-0.05) {};
\vnode[color=red](vv12) at (1,0.08) {};
\vnode[color=red](vv11) at (1,0) {};
\vnode[color=red](vv22) at (2,0.08) {};
\vnode[color=red](vv21) at (2,-0.08) {};
\vnode[color=red](vv32) at (3,0.16) {};
\vnode[color=red](vv31) at (3,-0.16) {};
\vnode[color=red](vv42) at (4,0.24) {};
\vnode[color=red](vv41) at (4,-0.16) {};
\vnode[color=red](vv54) at (5,0.40) {};
\vnode[color=red](vv53) at (5,0.32) {};
\vnode[color=red](vv52) at (5,-0.08) {};
\vnode[color=red](vv51) at (5,-0.24) {};
\vnode[color=red](vr4) at (5.5,0.40) {};
\vnode[color=red](vr3) at (5.5,0.32) {};
\vnode[color=red](vr2) at (5.5,-0.08) {};
\vnode[color=red](vr1) at (5.5,-0.24) {};

\draw[very thick, dashed, color=red] (-0.5, -.4) .. controls (-0.4, -.4) and (-0.25, -.3) .. (vv0);
\draw[very thick, color=red] (vv0) .. controls (0.9, 1.0) and (1.5, -0.7) .. (vv21);
\draw[very thick, color=red] (vv0) to [out=30,in=150] (vv12);
\draw[very thick, color=red] (vv0) to [out=-30,in=-150] (vv11);
\draw[very thick, color=red] (vv12) .. controls (1.9, 1.0) and (2.5, -0.7) .. (vv31);
\draw[very thick, color=red] (vv12) to [out=30,in=150] (vv22);
\draw[very thick, color=red] (vv11) .. controls (2.0, -1.2) and (2.5, 0.7) .. (vv32);
\draw[very thick, color=red] (vv22) to [out=45,in=135] (vv42);
\draw[very thick, color=red] (vv21) .. controls (3.1, -1.1) and (3.1, 0.3) .. (vv41);	
\draw[very thick, color=red] (vv32) to [out=40,in=150] (vv54);
\draw[very thick, color=red] (vv31) .. controls (4.1, -1.1) and (4.5, 0.0) .. (vv51);
\draw[very thick, color=red] (vv42) to [out=30,in=150] (vv53);
\draw[very thick, color=red] (vv41) to [out=30,in=150] (vv52);
\draw[very thick, color=red] (vv54) to [out=20, in=160] (5.5, 0.4);
\draw[very thick, color=red] (vv53) to [out=20, in=160] (5.5, 0.32);
\draw[very thick, color=red] (vv52) to [out=20, in=160] (5.5, -0.08);
\draw[very thick, color=red] (vv51) to [out=20, in=160] (5.5, -0.24);

\vnode[color=darkyellow](vv0) at (0,0.05) {};
\vnode[color=darkyellow](vv2) at (2,0) {};
\vnode[color=darkyellow](vv3) at (3,-0.08) {};
\vnode[color=darkyellow](vv43) at (4,0.16) {};
\vnode[color=darkyellow](vv42) at (4.0,-0.08) {};
\vnode[color=darkyellow](vv54) at (5,0.24) {};
\vnode[color=darkyellow](vv53) at (5,0) {};
\vnode[color=darkyellow](vv52) at (5,-0.16) {};
\vnode[color=darkyellow](vy4) at (5.5,0.24) {};
\vnode[color=darkyellow](vy3) at (5.5,0) {};
\vnode[color=darkyellow](vy2) at (5.5,-0.16) {};

\draw[very thick, dashed, color=darkyellow] (-0.5,0.2) .. controls (-0.4, 0.2) and (-0.15, 0.1) .. (vv0);
\draw[very thick, color=darkyellow] (vv0) .. controls (1.2, 1.6) and (2.5, -0.3) .. (vv3);
\draw[very thick, color=darkyellow] (vv0) .. controls (0.9, 1.0) and (1.5, -0.6) .. (vv2);
\draw[very thick, color=darkyellow] (vv2) to [out=45,in=135] (vv43);
\draw[very thick, color=darkyellow] (vv2) .. controls (3.0, -1.0) and (3.1, 0.3) .. (vv42);	
\draw[very thick, color=darkyellow] (vv3) .. controls (4.0, -1.0) and (4.5, 0.0) .. (vv52);
\draw[very thick, color=darkyellow] (vv43) to [out=30,in=150] (vv54);
\draw[very thick, color=darkyellow] (vv42) to [out=30,in=150] (vv53);
\draw[very thick, color=darkyellow] (vv54) to [out=20, in=160] (vy4);
\draw[very thick, color=darkyellow] (vv53) to [out=20, in=160] (vy3);
\draw[very thick, color=darkyellow] (vv52) to [out=20, in=160] (vy2);

\vnode[color=cadmiumgreen](vv1) at (1,-0.08) {};
\vnode[color=cadmiumgreen](vv3) at (3,0.08) {};
\vnode[color=cadmiumgreen](vv4) at (4,0.08) {};
\vnode[color=cadmiumgreen](vv5) at (5,0.16) {};
\vnode[color=cadmiumgreen](vg) at (5.5,0.16) {};

\draw[very thick, dashed, color=cadmiumgreen] (0.5, -.4) .. controls (0.6, -.4) and (0.75, -.3) .. (vv1);
\draw[very thick, color=cadmiumgreen] (vv1) .. controls (2.1, -1.25) and (2.5, 0.7) .. (vv3);
\draw[very thick, color=cadmiumgreen] (vv3) to [out=-30,in=-150] (vv4);
\draw[very thick, color=cadmiumgreen] (vv4) to [out=30,in=150] (vv5);
\draw[very thick, color=cadmiumgreen] (vv5) to [out=20, in=160] (5.5, 0.16);

\vnode[color=blue](vv3) at (3,0) {};
\vnode[color=blue](vv4) at (4,0) {};
\vnode[color=blue](vv5) at (5,0.08) {};
\vnode[color=blue](vb) at (5.5,0.08) {};

\draw[very thick, dashed, color=blue] (2.5, -.5) .. controls (2.6, 0) and (2.7, 0.1) .. (vv3);
\draw[very thick, color=blue] (vv3) to [out=-30,in=-150] (vv4);
\draw[very thick, color=blue] (vv4) to [out=30,in=150] (vv5);
\draw[very thick, color=blue] (vv5) to [out=20, in=160] (5.5, 0.08);

\end{scope}
\end{tikzpicture}
    \vspace{-.2in}
    \caption{The Hasse diagram of a vineyard $\calV=\{\calV_{x_1},\ldots, \calV_{x_4}\}$, intertwined and superimposed on $\hatG$. See Example~\ref{eg.subvines}.}
    \label{fig:FSM_subvines}
\end{figure}

\begin{definition}[Vineyard split]   \label{def:vineyard_splits}
Let $P \in \calV$ be a prefix with \[\Covers(\calV;P)=\{Q_0,Q_1,\ldots, Q_l\}\] as in Equation~\eqref{equation:covers}. 
If $|\Covers(\calV;P)|\geq 2$, we say that $P$ \defn{has direct splits} and define the \defn{direct splits} of $P$  to be its non-minimal covers $Q_1,\ldots, Q_l$.
Recursively, we define the \defn{splits} of $P$ in $\calV$ to be $\Splits(\calV;P)=\emptyset$ whenever $\Covers(\calV;P)=\emptyset$, otherwise we define
\[
\Splits(\calV;P)= \{Q_1,\ldots, Q_l\} \sqcup \bigsqcup_{i=0}^l \Splits(\calV;Q_i).
\]
Denote by $\Splits(\calV) = \bigsqcup_{x\in X}\Splits(\calV;x)$ the set of splits of the vineyard $\calV$. 
\end{definition}

\begin{proposition}\label{proposition:minext_bijects_routes_to_X_splits}
The map
\[ \minext_{\calV}: X \sqcup \Splits(\calV)  \rightarrow \Routes(\calV) \]
is a bijection. As a consequence, each set $\Splits(\calV;x)$ is totally ordered by $\prec_xF$.
\end{proposition}
\begin{proof}
It suffices to show for each $x\in X$ that $\minext_{\calV}:\{x\} \sqcup \Splits(\calV;x) \rightarrow \Routes(\calV;x)$ is a bijection.

Let $\Routes(\calV;x) = \{V_x^0, \ldots, V_x^{l_x} \}$.
We note that (i) $\minext_{\calV}(x) = V_x^0$ and (ii) a prefix $P\in \Splits(\calV;x)$ cannot have $\minext_{\calV}(P) = V_x^0$ because it would mean that $P$ is a minimal cover of some other prefix in $\calV_x$ and could not have been a split in $\calV_x$.

If there exists $P,Q\in \Splits(\calV;x)$ such that $\minext_{\calV}(P) = \minext_{\calV}(Q)$, then one is a prefix of the other, so without loss of generality we assume $P\subseteq Q$.  
$Q$ is a prefix of $\minext_{\calV}(Q) = \minext_{\calV}(P)$, which implies that $Q$ is a minimal cover of some prefix $R$ such that $P \subseteq R \subset Q$.
This contradicts the fact that $Q\in \Splits(\calV;x)$ unless $P=Q$.
Therefore, $\minext_{\calV}$ is injective on $\{x\} \sqcup \Splits(\calV;x)$. 

For all $i\in [l_x]$, $V_x^{i-1}$ and $V_x^{i}$ coincide at a maximal prefix $P$ and there are prefixes $Q, Q'\in \Covers(\calV;P)$ that satisfy $P\subset Q\subseteq V_x^{i-1}$ and $P\subset Q'\subseteq V_x^{i}$. 
Further, $Q\prec Q'$ so $Q'$ is a split of $P$ and since $V_x^{i}$ is an extension of $Q'$ but $V_x^{i-1}$ is not, then $\minext_{\calV}(Q')=V_x^{i}$, from which the surjectivity of $\minext_{\calV}$ follows. 
\end{proof}

From Proposition~\ref{proposition:minext_bijects_routes_to_X_splits} we have that for each $x\in X$, the map 
\[\minext_{\calV}: \{x\}\sqcup \Splits(\calV;x) \rightarrow \{V_x^0,\ldots, V_x^{l_x} \}\] 
is a bijection. For $i\in [0,l_x]$, define the prefix
\begin{equation}\label{eqn:split_associated_to_a_route}
S_x^i:= \minext_{\calV}^{-1}(V_x^i)    
\end{equation}
to be the \defn{$i$-th split of $x$}. In particular, as a convention, $S_x^0 =x$ which is not a split. For $S_x^i\in \Splits(\calV;x)$, let $\next_\calV(S_x^i)= S_x^{i+1}$ be the next split in the vine $\calV_x$. These concepts will be used later, including in Proposition~\ref{prop:is_a_shuffle}.

The following lemma is straightforward from Definition \ref{def:vineyard_splits}. 
\begin{lemma}\label{lemma:inclusion_of_vineyards}
    For vineyards satisfying $\calV \subseteq \calW$, we have  $\Splits(\calV)\subseteq  \Splits(\calW)$.
\end{lemma}

The \defn{support} of a vineyard $\calV$ is the subgraph of $\hatG$ with augmented edge set 
\[\supp(\calV) := \{e \in \hatE\mid e=\terminal(P)\text{ for some }P\in \calV\}.
\]

\begin{lemma}\label{lemma:proj_is_injective}
Let $\hatH = \supp(\calV)$.
The map $\terminal:\Splits(\calV)\rightarrow \Splits(\hatH)$ is injective.
\end{lemma}
\begin{proof}
Suppose two splits $P\neq Q\in \Splits(\calV)$ satisfy $\terminal(P)=\terminal(Q)=(v,w)$. If $Pv=Qv$ then $P=Q$, so assume without loss of generality that $Pv \prec_{\Prefixes(v)} Qv$. Then the prefixes $P$ and $\next_{\calV}(Qv)$ are in conflict, which contradicts $\calV$ being a vineyard. 
\end{proof}

The injectivity of the map $\terminal$ of Lemma \ref{lemma:proj_is_injective} leads to a concise way to label various elements of a vineyard. 
The lemma asserted that every split $P\in \Splits(\calV)$ maps injectively into the split $\terminal(P)\in \Splits(\hatH)$. Hence, we can use splits in $\hatH$ to label the splits of $\calV$ and by Proposition \ref{proposition:minext_bijects_routes_to_X_splits}, we can use this together with the set $X$ to label the routes of $\calV$.

\begin{definition}[Natural labeling of a vineyard]\label{def:natural_labeling_vineyard}
Let $\hatH=\supp(\calV)$. The \defn{natural labeling} of a vineyard $\calV$ is the map 
\[\lambda: \calV\rightarrow X\cup \Splits(\hatH)
\]
uniquely defined by assigning the label $\lambda(P):=\terminal(P)$ to every $P\in X \cup \Splits(\calV)$ and then extending the labeling to the entire set $\calV$ by letting $\lambda(P)=\lambda(Q)$ whenever $Q=\next_{\calV}(P)$. A particular byproduct of this construction is an injective labeling of routes
\[\lambda: \Routes(\calV)\rightarrow X\cup \Splits(\hatH).\]
\end{definition}

\begin{remark}
If $\calV$ is saturated in $\hatH$, then $\lambda$ is bijective. However, $\lambda$ is not bijective in general because there may be splits in $\Splits(\hatH)$ that are not in $\Splits(\calV)$. For example, consider the vine induced from two routes that converge and diverge multiple times. Each time they diverge at a split in $\hatH$, but only the first divergence is at a split in $\Splits(\calV)$. 
\end{remark}

\begin{example}\label{eg.splits}
The vineyard $\calV$ from Example~\ref{eg.subvines} has 
\begin{align*}
\Splits(\calV;{x_1}) &= \{S_{x_1}^1, S_{x_1}^2, S_{x_1}^3\} = \{x_1t_3,\  x_1t_3t_4,\  x_1t_2 \}, \\
\Splits(\calV;{x_2}) &= \{S_{x_2}^1, S_{x_2}^2\} = \{x_2t_2t_6,\  x_2t_1\},  \\
\Splits(\calV;{x_3}) &= \Splits(\calV;{x_4}) = \emptyset.
\end{align*}
Its support $\hatH = \supp(\calV)$ is $\hatG$ minus the single edge $s_4$.
The set of terminal edges of $X\sqcup \Splits(\calV)$ is $\{x_1, x_2, x_3, x_4, t_3, t_4, t_2, t_6, t_1 \}$, which gives the natural labeling of $\Routes(\calV)$.
For example, $\lambda(V_{x_2}^1) = \lambda(x_2t_2t_6t_9y)= t_6$.
The extension of $\lambda$ to all prefixes $P\in\calV$ is given in Figure~\ref{fig:FSM_unpacked_subvines}, where the label $\lambda(P)$ is shown beside each $P\in \calV$.
Prefixes labeled $t_j$ are depicted as $j$ in the figure.
\end{example}

\begin{remark}
As initially defined, the map $\minext_{\calV}: \calV \rightarrow \Routes(\calV)$ is a surjection.
However, Proposition~\ref{proposition:minext_bijects_routes_to_X_splits} states that it is a bijection when the domain is restricted to $D=X\sqcup \Splits(\calV)$, hence we may equivalently define the natural labeling of $\calV$ by letting $\lambda(R) = \terminal(\minext_{\calV}|_D^{-1}(R))$ for $R\in\Routes(\calV)$, then extending to all prefixes $P\in\calV$ by letting $\lambda(P) = \lambda(\minext_{\calV}(P))$.
\end{remark}

\begin{remark}
When $\ba=\be_0-\be_n$, there is only one vine in a vineyard.  
\end{remark}

\begin{figure}[ht!]
    \centering\vspace{-.5in}
    \input{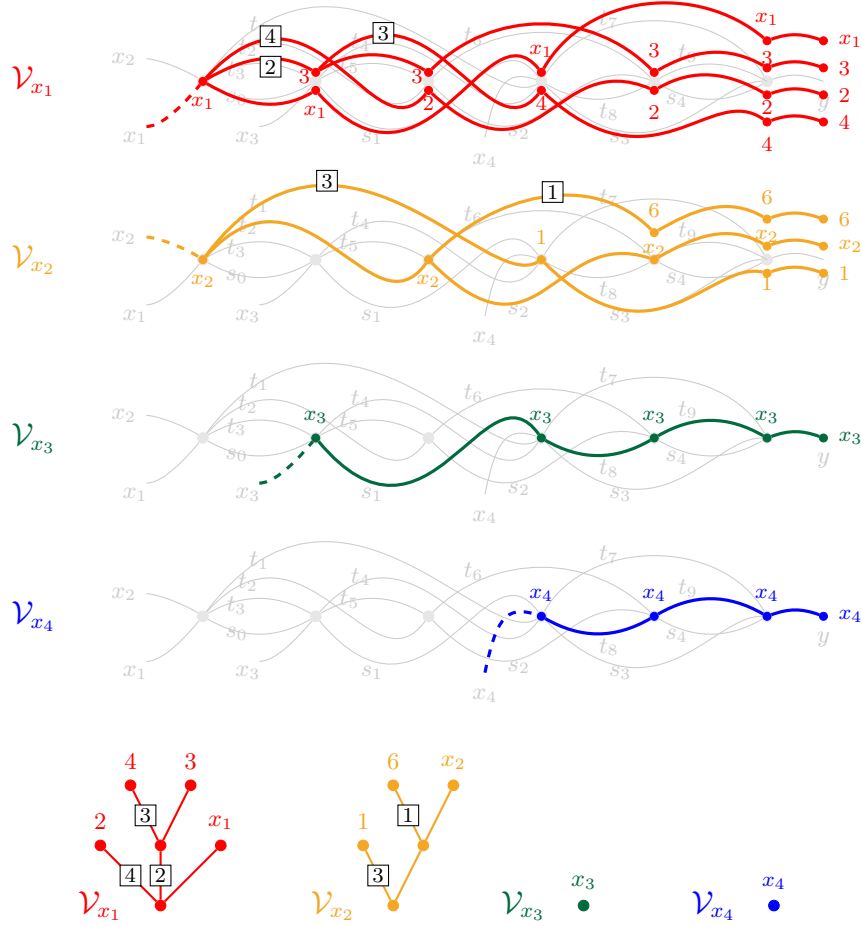}
    \caption{Two representations of a vineyard shuffle $(\calV, \sigma_\calV)$ corresponding to the multiclique $\calM$ from Figure~\ref{fig:cliques_and_multicliques}. Each prefix is the vineyard in labeled by its natural labeling $\lambda(P)$. See Example~\ref{eg.splits}.}
    \label{fig:FSM_unpacked_subvines}
\end{figure}

\begin{figure}[ht!]
    \centering\vspace{-.5in}
    \input{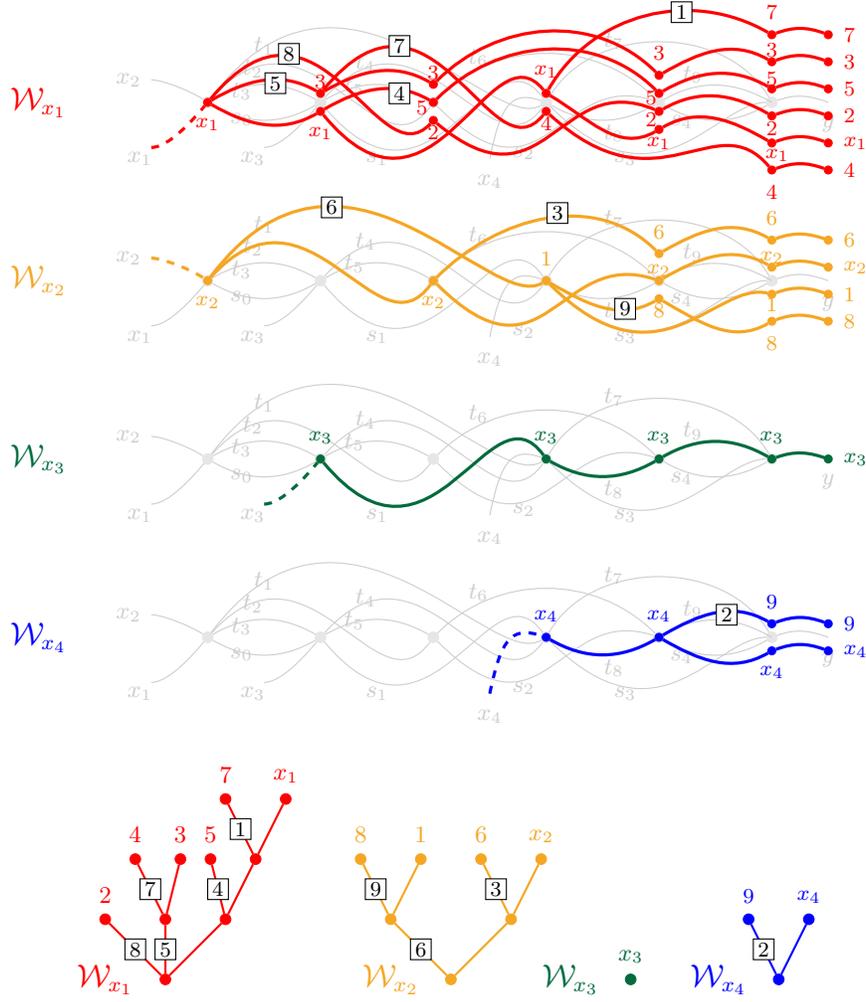}
    \caption{Two representations of the saturated vineyard shuffle $(\calW, \sigma_\calW)$ corresponding to the saturated multiclique $\calN$ from Figure~\ref{fig:SatCliques_and_multicliques}.
    }
    \label{fig:FSM_unpacked_vines}
\end{figure}

\subsubsection{Vineyard shuffles from multicliques}
\label{sec:vineyard_shuffles_from_multicliques}
\phantom{W}

We can condense the information contained in a multiclique using the concept of a vineyard plus an additional piece of information which keeps track of the multiplicity of routes, namely a shuffle surjection.

\begin{definition}\label{def:shuffle}
    For positive integers $n_1,\ldots, n_r, k$, an  \defn{$(n_1,\dots,n_r;k)$-shuffle} is a surjection $\sigma:[n_1+\dots + n_r]\rightarrow [k]$ satisfying
\begin{align*}
   \sigma(1)<&\dots <\sigma(n_1)\\
    \sigma(n_1+1) <&\dots <\sigma(n_1+n_2)\\
    &\,\,\,\,\vdots\\
    \sigma(n_1+\dots+n_{r-1}+1) <&\dots <\sigma(n_1+\dots+n_r).
\end{align*}
\end{definition}

\begin{remark} \label{remark:shuffle}
    The word ``shuffle'' is used in multiple ways in the English language and in the mathematics literature. We clarify the distinctions and its applicability here. 
    
    The colloquial use of the word shuffle corresponds to the shuffling of a deck of cards that has been split into two or more blocks (see \cite{DiaconisFulman2023}). A shuffle of $(1,2,3)$ and $(4,5)$ of this type is $(4,1,5,2,3)$ --- the elements of each block appear in their original order. This is the concept normally referred to as a shuffle in the combinatorics literature.
    
    In the algebra literature, however, this concept is called an {\em unshuffle} and a shuffle  frequently refers to its inverse function. (See, for example, \cite{LodayVallette2012}.)  Under this definition, an example of a shuffle of $(1,2,3)$ and $(4,5)$ is then $(2,4,5,1,3)$ --- where the first three entries are in increasing order as are the last two. 
    
    This concept of shuffle has been further extended to include surjections, as in \cite{PalaciosRonco2006}; this concept is sometimes referred to as a \textit{quasi-shuffle} (see \cite{Hoffman2000}). A shuffle surjection of $(1,2,3)$ and $(4,5)$  is for example $(1,3,4,2,3)$---again, the first three entries are in increasing order, as are the last two, but entries may repeat in different blocks. 
\end{remark}

Let $\calM$ be a multiclique of rank $k$ and recall that $\Routes(\calM)$ denotes the set of distinct routes in $\calM$. Since this set of routes is coherent, we can use it to induce the vineyard $\calV(\calM)$, and $\Routes(\calV(\calM))=\Routes(\calM)$. For each $x\in X$ denote the ordered set of distinct routes beginning on the inflow half-edge $x$ by $\Routes(\calV(\calM);x)=\{V_x^0,\ldots, V_x^{l_x}\}$. Recall that $\{\calM_x\mid x\in X\}$ is the canonical partition of $\calM$ where $\calM_x=\{M_x^0,M_x^1,\dots,M_x^{k}\}$ is the ordered multiset of routes beginning on the inflow half-edge $x$.

The multiplicity of the routes in $\calM$ can be recovered via the function 
\[
    \tau_{\calM}:\Routes(\calM)\rightarrow [0,k]
\]
defined by
\begin{equation}\label{equation:definition_shuffle_from_multiclique_on_routes}
    \tau_{\calM}(R) = \min \{j \mid M_{\initial(R)}^j=R \}.
\end{equation}
Given a route $V_x^i \in \Routes(\calM;x)$, define the \defn{interval of $V_x^i$ in $\calM_x$} by 
\[ [V_x^i]_\calM:=[\tau_\calM(V_x^i), \tau_\calM(V_x^{i+1})-1],
\]
where $\tau_\calM(V_x^{l_x+1})=k+1$ by convention.
Then the multiplicity of $V_x^i$ in $\calM_x$ is the length of the interval of $V_x^i$ in $\calM$, given by $\tau_\calM(V_x^{i+1}) - \tau_\calM(V_x^i)$.
As an example, consider the multiclique $\calM$ in Figure~\ref{fig:cliques_and_multicliques}. 
The interval of the route $V_{x_2}^1=x_2t_2t_6t_9y$ is $[V_{x_2}^1]_\calM=[1,2]$.

Recall from Section~\ref{sec:multicliques} that given multicliques $\calM$ and $\calN$ we have $\calM \leq \calN$ if and only if there exists a strictly increasing function $f:[0,\rank(\calM)] \rightarrow [0, \rank(\calN)]$ satisfying $M_x^j = N_x^{f(j)}$ for all $x\in X$ and $j\in [0, \rank(\calM)]$.
We give an alternative characterization of this order relation on multicliques.

\begin{proposition}\label{proposition:face_membership_rho}
Let $\calM, \calN\in \Multicliques(\hatG,\hatF)$. Then $\calM \le \calN$ if and  only if 
\begin{enumerate}
    \item $\Routes(\calM) \subseteq \Routes(\calN)$, and
    \item if $[P]_\calM \cap [Q]_\calM \neq \emptyset$, then $[P]_\calN\cap [Q]_\calN \neq\emptyset$, for $P,Q\in \Routes(\calM)$.
\end{enumerate}  
\end{proposition}
\begin{proof}
First suppose $\calM \leq \calN$, so that by definition there exists a strictly increasing function~$f$ that satisfies $M_x^j = N_x^{f(j)}$ for all $j\in [0, \rank(\calM)]$.  
It follows that $\Routes(\calM) \subseteq \Routes(\calN)$. 

Next, suppose $[P]_\calM \cap [Q]_\calM \neq \emptyset$. For any $j\in [P]_\calM \cap [Q]_\calM$, we have $P, Q \in \calM^j = \calN^{f(j)}$, so $[P]_\calN \cap [Q]_\calN \neq \emptyset$. Therefore, $\calM$ and $\calN$ satisfy conditions (1) and (2).

Conversely, suppose we have two multicliques $\calM, \calN \in \Multicliques(\hatG,\hatF)$ that satisfy conditions (1) and (2). We will show that there exists a unique strictly increasing function $f:[0,\rank(\calM)]\rightarrow [0, \rank(\calN)]$ such that $M_x^j = N_x^{f(j)}$ for all $x\in X$ and $j\in [0, \rank(\calM)]$.

Let $j\in [0,\rank(\calM)]$. For all $x\in X$, condition (1) implies that $[M_x^j]_\calN \neq \emptyset$. For any $x,x'\in X$, the intervals $[M_x^j]_\calM \cap [M_{x'}^j]_\calM \neq \emptyset$, so by condition (2), we have $[M_x^j]_\calN \cap [M_{x'}^j]_\calN \neq \emptyset$. By Helly's property, this implies
\begin{equation}\label{eqn.fj}
\bigcap_{x\in X} [M_x^j]_\calN \neq \emptyset.    
\end{equation}
But the columns of $\calN$ are distinct (because each column is a route matching in a coherent clique), therefore the intersection in Equation~\eqref{eqn.fj} is a unique integer. We define $f(j)$ to be that unique integer. From this definition, we have that $N_x^{f(j)}= M_x^j$ for all $j\in [0,\rank(\calM)]$.

To see that $f$ is strictly increasing, let $a,b\in [0,\rank(\calM)]$ with $a<b$. Since the columns of $\calM$ are distinct, then there exists at least one route in $\calM$ whose interval ends in the $a$-th column of $\calM$, say, $M_x^a$. Then $a<b$ implies $M_x^a \prec_x M_x^b$ and therefore $\tau_\calN(M_x^a) < \tau_\calN(M_x^b)$. This means that the interval $[M_x^a]_\calN$ ends before the interval $[M_x^b]_\calN$ begins in $\calN$, and since $f(a)\in [M_x^a]_\calN$ and $f(b)\in [M_x^b]_\calN$, then $f(a)< f(b)$ as desired. Therefore, $\calM \leq \calN$.
\end{proof}

Recall from Proposition~\ref{proposition:minext_bijects_routes_to_X_splits} that $\minext_{\calV(\calM)}: X \sqcup \Splits(\calV(\calM)) \rightarrow \Routes(\calV(\calM))$ is a bijection, so we may equivalently choose to represent the information of $\tau_{\calV(\calM)}$ from the perspective of the splits of the vineyard by the function
\[
\sigma_{\calV(\calM)}:X\sqcup\Splits(\calV(\calM))\rightarrow [0,k]
\] defined by 
\begin{equation}
\sigma_{_{\calV(\calM)}}=\tau_{_{\calV(\calM)}}\circ \minext_{\calV(\calM)}.   
\end{equation}
Observe that $\tau_{\calV(\calM)}(R) =0$ if and only if $R=V_x^0$ for some $x\in X$, and hence $\sigma_{\calV(\calM)}(P)=0$ if and only if $\minext_{\calV(\calM)}(P)=V_x^0$ for some $x\in X$.
The proof of Proposition~\ref{proposition:minext_bijects_routes_to_X_splits} shows that $\minext_{\calV(\calM)}$ restricts to the two bijective correspondences
\begin{align*}
X&\longleftrightarrow \{V_x^0\mid x\in X\}\\
\Splits(\calV(\calM))&\longleftrightarrow \Routes(\calV(\calM))\setminus\{V_x^0\mid x\in X\}.
\end{align*}
This means that the restriction $\sigma_{_{\calV(\calM)}}:\Splits(\calV(\calM))\rightarrow [k]$ still enables us to recover the multiplicity of routes in $\calM$.

\begin{proposition}\label{prop:is_a_shuffle}
 The function $\sigma_{\calV(\calM)}:\Splits(\calV(\calM))\rightarrow [k]$ defined by 
 \[\sigma_{\calV(\calM)}(P)=\min\{j\mid \minext_{\calV(\calM)}(P)=M_{\initial(P)}^j\}
 \]
 is an $(l_{x_1},l_{x_2},\dots;k)$-shuffle.
\end{proposition}

\begin{proof}
Since the order in $\calM_x$ is induced by $\preceq_x$, this definition implies that  whenever $i<j$ we have $\sigma_{\calV(\calM)}(S_x^i) < \sigma_{\calV(\calM)}(S_x^j)$, where $S_x^i=\minext_{\calV(\calM)}^{-1}(V_x^i)$ as in Equation \eqref{eqn:split_associated_to_a_route}.
Condition (c) of Definition~\ref{def:multicliques} says that for every $h\in [k]$ there is an $x\in X$ such that $M_x^{h-1}\prec_x M_x^{h}$ and so for the unique split $P$ such that $\minext_{\calV}(P)=M_x^{h}$ we have $\sigma_{\calV(\calM)}(P)=h$. 
So $\sigma_{\calV(\calM)}$ is a surjection and then ordering the sets $\Splits(\calV(\calM);x)$ in a compatible manner with the order on $X$, we can see that $\sigma_{\calV(\calM)}$ is a shuffle. 
\end{proof}

Proposition~\ref{prop:is_a_shuffle} motivates the following definition.

\begin{definition}[Vineyard shuffle]
\label{def:vineyard_shuffles}
A \defn{vineyard shuffle} of rank $k$ is a pair $(\calV, \sigma) $ where~$\calV$ is a vineyard and $\sigma_{\calV}:\Splits(\calV)\rightarrow [k]$ is a surjective function satisfying
\[
\sigma_\calV(S) < \sigma_\calV(S') \hbox{ for every pair } S \prec_x S' \hbox{ in } \Splits(\calV;x) \hbox{ for } x\in X.
\]
We denote by  $\VineyardShuffles(\hatG,\hatF)$ the set of vineyard shuffles of $(\hatG,\hatF)$.
\end{definition}

Note, as with the examples from multicliques, a vineyard shuffle can always be extended to a function $\sigma_{\calV}:X\sqcup \Splits(\calV)\rightarrow [0,k]$ by defining $\sigma_{\calV}(x)=0$ for every $x\in X$.

\begin{remark}
When $\ba=\be_0-\be_n$, there is only one vineyard shuffle associated to an underlying vineyard because $\Splits(\calV)$ is totally ordered. 
\end{remark}

\begin{example}\label{eg.vineshuffles}
Continuing the example of the vineyard $\calV$ from Examples~\ref{eg.subvines} and~\ref{eg.splits}, letting
\begin{gather*}
\sigma_\calV(x_1t_3) = 2,\  \sigma_\calV(x_1t_3t_4) = 3,\  \sigma_\calV(x_1t_2) = 4,\\
\sigma_\calV(x_2t_2t_6) = 1,\  \sigma_\calV(x_2t_1) = 3, 
\end{gather*}
(and $\sigma_\calV(x_1) = \sigma_\calV(x_2) = \sigma_\calV(x_3) = \sigma_\calV(x_4) = 0$) defines a $(3,2,0,0;4)$-shuffle $\sigma_\calV$.
The vineyard shuffle $(\calV, \sigma_\calV)$ is induced by the multiclique $\calM$ from Figure~\ref{fig:cliques_and_multicliques}. 

Figure~\ref{fig:FSM_unpacked_subvines} shows two representations of $(\calV, \sigma_\calV)$.
In the first representation, the four vines of $\calV$ from Figure~\ref{fig:FSM_subvines} are depicted separately.
For each $P\in \Splits(\calV)$, the value $\sigma_\calV(P)$ is shown on the cover relation of $P$ in its vine.

A second representation of a vineyard shuffle is shown in the bottom row of the figure. Since prefixes that are not splits of $\calV$ do not carry the shuffle information, a vineyard shuffle can be more compactly represented by keeping only elements in $\Routes(\calV)$ and their common prefixes. In this representation, we choose to draw the Hasse diagram of each vine so that the leaves (routes of $\calV$) from right to left are in the order of $\Routes(\calV;x)$. Taking a pre-order traversal of a vine (starting with the right subtree), we see that the values $\sigma_\calV(P)$ are encountered in increasing order.
\end{example}

We have discussed how to obtain a vine from a multiclique.
We can construct the inverse map
\[
\calM:\VineyardShuffles(\hatG,\hatF)\rightarrow\Multicliques(\hatG,\hatF)
\] 
by letting $\calM(\calV,\sigma_{\calV})$ be the multiset with underlying set $\Routes(\calV)$ and
\[
\sigma_{\calV}(S_x^{i+1})-\sigma_{\calV}(S_x^i)
\]  
copies of $R=\minext_{\calV}(S_x^i)$ for every $x\in X$ and $i\in [0,l_x]$, where as a notational convenience we define here $\sigma_{\calV}(S_x^{l_x+1}):=k+1$ if $(\calV, \sigma_\calV)$ is a vineyard shuffle of rank $k$.

We can determine the partial order $(\calV,\sigma_{\calV})\le (\calW,\sigma_{\calW})$ on $\VineyardShuffles(\hatG,\hatF)$ induced by the bijection $(\calV(\calM),\sigma_{\calV(\calM)})$ by translating the language of multicliques and $\tau$ in Proposition~\ref{proposition:face_membership_rho} to the language of vineyards and shuffles.

Recall that for $S\in \Splits(\calV;x)$, $\next_\calV(S)$ is the next split in the vine $\calV_x$.
\begin{definition}\label{def:uniformity_of_sigma}
We define $(\calV,\sigma_{\calV})\le (\calW,\sigma_{\calW})$ on $\VineyardShuffles(\hatG,\hatF)$ whenever
\begin{enumerate}
\item $\calV\subseteq \calW$ and
\item  for splits $S \in \Splits(\calV;x)$ and $S' \in \Splits(\calV;x')$, 
\[\hbox{if } \sigma_\calV(S) \leq 
\sigma_\calV(S') \leq \sigma_\calV(\next_\calV(S))-1, \hbox{ then }
\sigma_\calW(S) \leq \sigma_\calW(S') 
\leq \sigma_\calW(\next_\calW(S))-1.
\]
\end{enumerate} 
\end{definition}

We denote by $\SatVineyardShuffles(\hatG,\hatF)$ the set of maximal vineyard shuffles with respect to the partial order given in Definition~\ref{def:uniformity_of_sigma} which we call \defn{saturated}. 

\begin{theorem}\label{theorem:isomorphism_multicliques_vineyard_shuffles}
The map 
\[
(\calV,\sigma_{\calV}):\Multicliques(\hatG,\hatF) \rightarrow \VineyardShuffles(\hatG,\hatF)
\] 
is a poset isomorphism.
\qed
\end{theorem}

\begin{example}\label{eg.vineyardshuffle}
Figure~\ref{fig:FSM_unpacked_vines} shows a saturated vineyard shuffle $(\calW,\sigma_\calW)$ of rank $9$.
This vineyard shuffle corresponds to the saturated multiclique $\calN$ from Figure~\ref{fig:SatCliques_and_multicliques} under the isomorphism of Theorem \ref{theorem:isomorphism_multicliques_vineyard_shuffles}.

Let $(\calV, \sigma_\calV)$ be the rank $4$ vineyard shuffle from Figure~\ref{fig:FSM_unpacked_subvines}.
We leave it to the reader to verify that $\Routes(\calV) \subset \Routes(\calW)$, so $\calV \subset \calW$.
To see that $(\calV,\sigma_\calV) \subset (\calW,\sigma_\calW)$, we check that condition (2) in Definition~\ref{def:uniformity_of_sigma} is satisfied by the splits in $\calV$.
Recall $\Splits(\calV)=\{S_{x_1}^1, S_{x_1}^2, S_{x_1}^3, S_{x_2}^1, S_{x_2}^2\}= \{x_1t_3, x_1t_3t_4, x_1t_2,x_2t_2t_6, x_2t_1 \}$, and from Figure~\ref{fig:FSM_unpacked_vines} we have
\begin{gather*}
\sigma_\calW(S_{x_1}^1) = 5,\  \sigma_\calW(S_{x_1}^2) = 7,\  \sigma_\calW(S_{x_1}^3) = 8,\\
\sigma_\calW(S_{x_2}^1) = 3,\  \sigma_\calW(S_{x_2}^2) = 6.
\end{gather*}
Obtaining the values of $\sigma_\calV$ from Example~\ref{eg.splits}, we see that condition (2) only needs to be verified on the pair of splits $S_{x_2}^1$ and $S_{x_1}^1$:
\[\sigma_\calV(S_{x_2}^1) =1 \leq \sigma_\calV(S_{x_1}^1) =2 \leq \sigma_\calV(S_{x_2}^{2})-1 =2,\]
and 
\[\sigma_\calW(S_{x_2}^1) =3 \leq \sigma_\calW(S_{x_1}^1) =5 \leq \sigma_\calW(S_{x_2}^{2})-1 =5,\]
Therefore, $(\calV, \sigma_\calV) \subset (\calW, \sigma_\calW)$.
\end{example}

\subsection{Grove shuffles} 
\label{sec:grove_shuffles}
\phantom{W}

Groves are a family of objects introduced in~\cite{BrunnerHanusa2024} based on the related concept of \textit{noncrossing bipartite trees} that was defined by M\'esz\'aros and Morales in~\cite{MeszarosMorales2019}. 
There, it was used to describe a decomposition process of $\calF_G(\ba)$ which extended (unpublished) work of Postnikov and Stanley for the case when $\ba=\be_0- \be_n$, and which allowed the authors to provide a geometric and combinatorial proof of the Lidskii volume formula of Theorem \ref{thm.genlidskii}. 
The same underlying idea behind noncrossing bipartite trees is also implicit in the concept of \textit{monotone correspondence} in the proof of the DKK triangulation~\cite{DanilovKarzanovKoshevoy2012}. 

The main idea behind a grove is the codification of the information present in a vineyard but from the perspective of observers standing on the vertices of the graph $\hatG$. The noncrossing property of the forests in the grove directly captures the coherence of the routes involved, so it can be used to construct a clique inductively across vertices from $v=0,\dots,n$.

\begin{definition}[Noncrossing bipartite forest]\label{def:noncrossing_bipartite_tree}
A \defn{noncrossing bipartite forest} $\gamma_v$ of $(\hatG,\hatF)$ on a vertex $v$ is a bipartite graph on a vertex set with bipartition denoted $V(\gamma_v)=L(\gamma_v)\sqcup R(\gamma_v)$ and edge set denoted $E(\gamma_v)$, where 
\begin{enumerate}[label=(\alph*)]
    \item The \defn{left vertices} $L(\gamma_v)\subseteq\Prefixes(v)$ are prefixes of $\hatG$ that end at $v$ and the \defn{right vertices} $R(\gamma_v)\subseteq\hat \outedge(v)$ are edges of $\hatG$; they inherit their corresponding total orders. 
    \item The edges of $\gamma_v$ are \defn{noncrossing}. 
    That is, whenever $P \prec_{L(\gamma_v)} Q$ and $e \prec_{R(\gamma_v)}e'$, then $(P,e')$ and $(Q,e)$ cannot both be in $E(\gamma_v)$.
    \item Every vertex of $\gamma_v$ is incident to at least one edge.
\end{enumerate}
\end{definition}

\begin{definition}[Grove]\label{def:grove}
 A \defn{grove} of $(\hatG,\hatF)$, when non-empty, is a sequence $\Gamma=(\gamma_v)_{v \in [0,n]}$ of $n+1$ noncrossing bipartite forests such that
 \begin{enumerate}[label=(\alph*)]
     \item $x \in L(\gamma_{v_x})$  for every $x\in X$.
     \item $R(\gamma_n)= \{y\}$.
     \item For every $Pe \in \Prefixes(\hatG)\setminus X$ with $e=(v,w)$, we have that $(P,e)\in E(\gamma_{v})$ if and only if $Pe \in L(\gamma_{w})$.
 \end{enumerate}
We denote by $\Groves(\hatG,\hatF)$ the set of groves of $(\hatG,\hatF)$.
\end{definition}

Given a grove $\Gamma$, we denote by $\Prefixes(\Gamma):=\bigsqcup_{v=0}^nL(\gamma_v)$ the set of all left vertices, $E(\Gamma):=\bigsqcup_{v=0}^nE(\gamma_v)$ the set of all grove edges, and note that $\bigsqcup_{v=0}^nR(\gamma_v) \subseteq E\cup\{y\}$. The \defn{support} of a grove $\Gamma$ is the subgraph $\supp(\Gamma)\subseteq \hatG$ induced by the augmented edge set $X\sqcup\bigsqcup_{v\in[0,n]}R(\gamma_v)$. 

There is a natural partial order on $\Groves(\hatG,\hatF)$ by the relation $\Gamma \subseteq \Theta$ whenever $\Gamma$ is a subgraph of $\Theta$. Denote by $\SatGroves(\hatG,\hatF)$ the set of groves that are maximal with respect to this order which we call \defn{saturated}.

\begin{example}\label{ex:partial_grove}
Figure~\ref{fig:partial_grove_definition} shows the noncrossing bipartite forests $\gamma_v$ that comprise a grove~$\Gamma$.  

Note that in $\gamma_2$ of $\Gamma$, the left vertices $L(\gamma_2)$ are $x_1t_2$, $x_2t_2$, and $x_1t_3t_5$. The prefixes $x_1t_2$ and $x_2t_2$ exist because $x_1$ and $x_2$ are adjacent to $t_2$ in $\gamma_0$ and the prefix $x_1t_3t_5$ exists since $x_1t_3$ is adjacent to $t_5$ in $\gamma_1$.  

$\Gamma$ is not maximal. For example, $R(\gamma_4)\subsetneq \hat\outedge(4)$ and $\gamma_1$ and $\gamma_3$ are not trees. 
The grove~$\Theta$ in Figure~\ref{fig:grove_definition} is maximal and satisfies $\Gamma\prec \Theta$ in the partial order on $\Groves(\hatG,\hatF)$. Maximality of $\Theta$ can be seen because for each $v\in [0,n]$, $\theta_v$ is a tree and $R(\theta_v)=\hat\outedge(v)$.
\end{example}

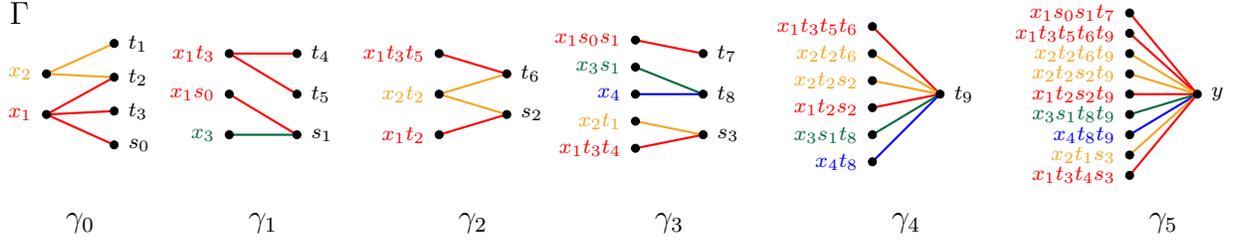
\begin{figure}[htb!]
    \vspace{-.1in}
    \begin{tikzpicture}[scale=1.8]
\begin{scope}[xshift=-5, yshift=-85, scale=0.5]
    \node[] at (-0.4,1.7) {$\Gamma$};
    \node[] at (0.5,-1.4) {$\gamma_0$};
    \node[] at (0.5+2.7,-1.4) {$\gamma_1$};
    \node[] at (0.5+5.8,-1.4) {$\gamma_2$};
     \node[] at (0.5+8.7,-1.4) {$\gamma_3$};
    \node[] at (0.5+12.2,-1.4) {$\gamma_4$};
     \node[] at (0.5+16,-1.4) {$\gamma_5$};

    \draw[-,darkyellow,thick] (0,0.8)--(1,1.25);
    \draw[-,darkyellow,thick] (0,0.8)--(1,0.75);
    \draw[-,red,thick] (0,0.2)--(1,0.75);
    \draw[-,red,thick] (0,0.2)--(1,0.25);
    \draw[-,red,thick] (0,0.2)--(1,-0.25);
    \vertex[fill=black, minimum size=3pt] at (0,0.8) {};
    \vertex[fill=black, minimum size=3pt] at (0,0.2) {};
    \vertex[fill=black, minimum size=3pt] at (1,1.25) {};
    \vertex[fill=black, minimum size=3pt] at (1,0.75) {};
    \vertex[fill=black, minimum size=3pt] at (1,0.25) {};
    \vertex[fill=black, minimum size=3pt] at (1,-0.25) {};
\node[anchor=east] at (-0.05,0.8) {\tiny\textcolor{darkyellow}{$x_2$}};
\node[anchor=east] at (-0.05,0.2) {\tiny\textcolor{red}{$x_1$}};
\node[anchor=west] at (1.05,1.25) {\tiny\textcolor{black}{$t_1$}};
\node[anchor=west] at (1.05,0.75) {\tiny\textcolor{black}{$t_2$}};
\node[anchor=west] at (1.05,0.25) {\tiny\textcolor{black}{$t_3$}};
\node[anchor=west] at (1.05,-0.25) {\tiny\textcolor{black}{$s_0$}};

\begin{scope}[xshift=2.7cm]
    \draw[-,red,thick] (0,1.1)--(1,1.1);
    \draw[-,red,thick] (0,1.1)--(1,0.5);	    
    \draw[-,red,thick] (0,0.5)--(1,-0.1);		
    \draw[-,cadmiumgreen,thick] (0,-0.1)--(1,-0.1);	
    \vertex[fill=black, minimum size=3pt] at (0,1.1) {};
    \vertex[fill=black, minimum size=3pt] at (0,0.5) {};
    \vertex[fill=black, minimum size=3pt] at (0,-0.1) {};
    \vertex[fill=black, minimum size=3pt] at (1,1.1) {};
    \vertex[fill=black, minimum size=3pt] at (1,0.5) {};
    \vertex[fill=black, minimum size=3pt] at (1,-0.1) {};
\node[anchor=east] at (-0.05,1.1) {\tiny\textcolor{red}{$x_1t_3$}};
\node[anchor=east] at (-0.05,0.5) {\tiny\textcolor{red}{$x_1s_0$}};
\node[anchor=east] at (-0.05,-0.1) {\tiny\textcolor{cadmiumgreen}{$x_3$}};
\node[anchor=west] at (1.05,1.1) {\tiny\textcolor{black}{$t_4$}};
\node[anchor=west] at (1.05,0.5) {\tiny\textcolor{black}{$t_5$}};
\node[anchor=west] at (1.05,-0.1) {\tiny\textcolor{black}{$s_1$}};
\end{scope}

\begin{scope}[xshift=5.8cm]
    \draw[-,red,thick] (0,1.1)--(1,0.8);
\draw[-,darkyellow,thick] (0,0.5)--(1,0.8);     
\draw[-,darkyellow,thick] (0,0.5)--(1,0.2);      
\draw[-,red,thick] (0,-0.1)--(1,0.2);    
\vertex[fill=black, minimum size=3pt] at (0,1.1) {};
\vertex[fill=black, minimum size=3pt] at (0,0.5) {};
\vertex[fill=black, minimum size=3pt] at (0,-0.1) {};
\vertex[fill=black, minimum size=3pt] at (1.0,0.8) {};
\vertex[fill=black, minimum size=3pt] at (1.0,0.2) {};
\node[anchor=east] at (-0.05,1.1) {\tiny\textcolor{red}{$x_1t_3t_5$}};
\node[anchor=east] at (-0.05,0.5) {\tiny\textcolor{darkyellow}{$x_2t_2$}};
\node[anchor=east] at (-0.05,-0.1) {\tiny\textcolor{red}{$x_1t_2$}};
\node[anchor=west] at (1.05,0.8) {\tiny\textcolor{black}{$t_6$}};
\node[anchor=west] at (1.05,0.2) {\tiny\textcolor{black}{$s_2$}};
\end{scope}

\begin{scope}[xshift=8.7cm]
    \draw[-,red,thick] (0,1.3)--(1,1.1);
    \draw[-,cadmiumgreen,thick] (0,0.9)--(1,0.5);	    
    \draw[-,blue,thick] (0,0.5)--(1,0.5);	    
    \draw[-,darkyellow,thick] (0,0.1)--(1,-0.1);		
    \draw[-,red,thick] (0,-0.3)--(1,-0.1);		
    \vertex[fill=black, minimum size=3pt] at (0,1.3) {};
    \vertex[fill=black, minimum size=3pt] at (0,0.9) {};
    \vertex[fill=black, minimum size=3pt] at (0,0.5) {};
    \vertex[fill=black, minimum size=3pt] at (0,0.1) {};
    \vertex[fill=black, minimum size=3pt] at (0,-0.3) {};
    \vertex[fill=black, minimum size=3pt] at (1,1.1) {};
    \vertex[fill=black, minimum size=3pt] at (1,0.5) {};
    \vertex[fill=black, minimum size=3pt] at (1,-0.1) {};
\node[anchor=east] at (-0.05,1.3) {\tiny\textcolor{red}{$x_1s_0s_1$}};
\node[anchor=east] at (-0.05,0.9) {\tiny\textcolor{cadmiumgreen}{$x_3s_1$}};
\node[anchor=east] at (-0.05,0.5) {\tiny\textcolor{blue}{$x_4$}};
\node[anchor=east] at (-0.05,0.1) {\tiny\textcolor{darkyellow}{$x_2t_1$}};
\node[anchor=east] at (-0.05,-0.3) {\tiny\textcolor{red}{$x_1t_3t_4$}};
\node[anchor=west] at (1.05,1.1) {\tiny\textcolor{black}{$t_7$}};
\node[anchor=west] at (1.05,0.5) {\tiny\textcolor{black}{$t_8$}};
\node[anchor=west] at (1.05,-0.1) {\tiny\textcolor{black}{$s_3$}};
\end{scope}

\begin{scope}[xshift=12.2cm]
\draw[-,red,thick] (0.0,1.5)--(1.0,0.5); 
\draw[-,darkyellow,thick] (0.0,1.1)--(1.0,0.5);  
\draw[-,darkyellow,thick] (0.0,0.7)--(1.0,0.5);  
\draw[-,red,thick] (0.0,0.3)--(1.0,0.5);  
\draw[-,cadmiumgreen,thick] (0.0,-0.1)--(1.0,0.5);  
\draw[-,blue,thick] (0.0,-0.5)--(1.0,0.5);  
\vertex[fill=black, minimum size=3pt] at (0.0,1.5) {};
\vertex[fill=black, minimum size=3pt] at (0.0,1.1) {};
\vertex[fill=black, minimum size=3pt] at (0.0,0.7) {};
\vertex[fill=black, minimum size=3pt] at (0.0,0.3) {};
\vertex[fill=black, minimum size=3pt] at (0.0,-0.1) {};
\vertex[fill=black, minimum size=3pt] at (0.0,-0.5) {};
\vertex[fill=black, minimum size=3pt] at (1.0,0.5) {};
\node[anchor=east] at (-0.05,1.5) {\tiny\textcolor{red}{$x_1t_3t_5t_6$}};
\node[anchor=east] at (-0.05,1.1) {\tiny\textcolor{darkyellow}{$x_2t_2t_6$}};
\node[anchor=east] at (-0.05,0.7) {\tiny\textcolor{darkyellow}{$x_2t_2s_2$}};
\node[anchor=east] at (-0.05,0.3) {\tiny\textcolor{red}{$x_1t_2s_2$}};
\node[anchor=east] at (-0.05,-0.1) {\tiny\textcolor{cadmiumgreen}{$x_3s_1t_8$}};
\node[anchor=east] at (-0.05,-0.5) {\tiny\textcolor{blue}{$x_4t_8$}};
\node[anchor=west] at (1.05,0.5) {\tiny\textcolor{black}{$t_9$}};
\end{scope}

\begin{scope}[xshift=16cm]
\draw[-,red,thick] (0.0,1.7)--(1.0,0.5); 
\draw[-,red,thick] (0.0,1.4)--(1.0,0.5); 
\draw[-,darkyellow,thick] (0.0,1.1)--(1.0,0.5);  
\draw[-,darkyellow,thick] (0.0,0.8)--(1.0,0.5);  
\draw[-,red,thick] (0.0,0.5)--(1.0,0.5);  
\draw[-,cadmiumgreen,thick] (0.0,0.2)--(1.0,0.5);  
\draw[-,blue,thick] (0.0,-0.1)--(1.0,0.5);  
\draw[-,darkyellow,thick] (0.0,-0.4)--(1.0,0.5);  
\draw[-,red,thick] (0.0,-0.7)--(1.0,0.5);  
\vertex[fill=black, minimum size=3pt] at (0.0,1.7) {};
\vertex[fill=black, minimum size=3pt] at (0.0,1.4) {};
\vertex[fill=black, minimum size=3pt] at (0.0,1.1) {};
\vertex[fill=black, minimum size=3pt] at (0.0,0.8) {};
\vertex[fill=black, minimum size=3pt] at (0.0,0.5) {};
\vertex[fill=black, minimum size=3pt] at (0.0,0.2) {};
\vertex[fill=black, minimum size=3pt] at (0.0,-0.1) {};
\vertex[fill=black, minimum size=3pt] at (0.0,-0.4) {};
\vertex[fill=black, minimum size=3pt] at (0.0,-0.7) {};
\vertex[fill=black, minimum size=3pt] at (1.0,0.5) {};
\node[anchor=east] at (-0.05,1.7) {\tiny\textcolor{red}{$x_1s_0s_1t_7$}};
\node[anchor=east] at (-0.05,1.4) {\tiny\textcolor{red}{$x_1t_3t_5t_6t_9$}};
\node[anchor=east] at (-0.05,1.1) {\tiny\textcolor{darkyellow}{$x_2t_2t_6t_9$}};
\node[anchor=east] at (-0.05,0.8) {\tiny\textcolor{darkyellow}{$x_2t_2s_2t_9$}};
\node[anchor=east] at (-0.05,0.5) {\tiny\textcolor{red}{$x_1t_2s_2t_9$}};
\node[anchor=east] at (-0.05,0.2) {\tiny\textcolor{cadmiumgreen}{$x_3s_1t_8t_9$}};
\node[anchor=east] at (-0.05,-0.1) {\tiny\textcolor{blue}{$x_4t_8t_9$}};
\node[anchor=east] at (-0.05,-0.4) {\tiny\textcolor{darkyellow}{$x_2t_1s_3$}};
\node[anchor=east] at (-0.05,-0.7) {\tiny\textcolor{red}{$x_1t_3t_4s_3$}};
\node[anchor=west] at (1.05,0.5) {\tiny\textcolor{black}{$y$}};
\end{scope}

\end{scope}
\end{tikzpicture}
    \centering\vspace{-.2in}

\caption{A grove $\Gamma=(\gamma_0,\gamma_1, \ldots,\gamma_5)$. See Example~\ref{ex:partial_grove}. }
\label{fig:partial_grove_definition}
\end{figure}

\begin{figure}[htb!]
    \vspace{-.1in}
    \begin{tikzpicture}[scale=1.8]
\begin{scope}[xshift=-5, yshift=-85, scale=0.5]
    \node[] at (-0.4,2.0) {$\Theta$};
    \node[] at (0.5,-1.7) {$\theta_0$};
    \node[] at (0.5+2.7,-1.7) {$\theta_1$};
    \node[] at (0.5+5.8,-1.7) {$\theta_2$};
     \node[] at (0.5+8.7,-1.7) {$\theta_3$};
    \node[] at (0.5+12.2,-1.7) {$\theta_4$};
     \node[] at (0.5+16,-1.7) {$\theta_5$};

    \draw[-,darkyellow,thick] (0,0.8)--(1,1.25);
    \draw[-,darkyellow,thick] (0,0.8)--(1,0.75);
    \draw[-,red,thick] (0,0.2)--(1,0.75);
    \draw[-,red,thick] (0,0.2)--(1,0.25);
    \draw[-,red,thick] (0,0.2)--(1,-0.25);
    \vertex[fill=black, minimum size=3pt] at (0,0.8) {};
    \vertex[fill=black, minimum size=3pt] at (0,0.2) {};
    \vertex[fill=black, minimum size=3pt] at (1,1.25) {};
    \vertex[fill=black, minimum size=3pt] at (1,0.75) {};
    \vertex[fill=black, minimum size=3pt] at (1,0.25) {};
    \vertex[fill=black, minimum size=3pt] at (1,-0.25) {};
\node[anchor=east] at (-0.05,0.8) {\tiny\textcolor{darkyellow}{$x_2$}};
\node[anchor=east] at (-0.05,0.2) {\tiny\textcolor{red}{$x_1$}};
\node[anchor=west] at (1.05,1.25) {\tiny\textcolor{black}{$t_1$}};
\node[anchor=west] at (1.05,0.75) {\tiny\textcolor{black}{$t_2$}};
\node[anchor=west] at (1.05,0.25) {\tiny\textcolor{black}{$t_3$}};
\node[anchor=west] at (1.05,-0.25) {\tiny\textcolor{black}{$s_0$}};

\begin{scope}[xshift=2.7cm]
    \draw[-,red,thick] (0,1.1)--(1,1.1);
    \draw[-,red,thick] (0,1.1)--(1,0.5);	    
    \draw[-,red,thick] (0,0.5)--(1,0.5);		
    \draw[-,red,thick] (0,0.5)--(1,-0.1);		
    \draw[-,cadmiumgreen,thick] (0,-0.1)--(1,-0.1);	
    \vertex[fill=black, minimum size=3pt] at (0,1.1) {};
    \vertex[fill=black, minimum size=3pt] at (0,0.5) {};
    \vertex[fill=black, minimum size=3pt] at (0,-0.1) {};
    \vertex[fill=black, minimum size=3pt] at (1,1.1) {};
    \vertex[fill=black, minimum size=3pt] at (1,0.5) {};
    \vertex[fill=black, minimum size=3pt] at (1,-0.1) {};
\node[anchor=east] at (-0.05,1.1) {\tiny\textcolor{red}{$x_1t_3$}};
\node[anchor=east] at (-0.05,0.5) {\tiny\textcolor{red}{$x_1s_0$}};
\node[anchor=east] at (-0.05,-0.1) {\tiny\textcolor{cadmiumgreen}{$x_3$}};
\node[anchor=west] at (1.05,1.1) {\tiny\textcolor{black}{$t_4$}};
\node[anchor=west] at (1.05,0.5) {\tiny\textcolor{black}{$t_5$}};
\node[anchor=west] at (1.05,-0.1) {\tiny\textcolor{black}{$s_1$}};
\end{scope}

\begin{scope}[xshift=5.8cm]
    \draw[-,red,thick] (0,1.25)--(1,0.8);
\draw[-,red,thick] (0,0.75)--(1,0.8);     
\draw[-,darkyellow,thick] (0,0.25)--(1,0.8);     
\draw[-,darkyellow,thick] (0,0.25)--(1,0.2);      
\draw[-,red,thick] (0,-0.25)--(1,0.2);    
\vertex[fill=black, minimum size=3pt] at (0,1.25) {};
\vertex[fill=black, minimum size=3pt] at (0,0.75) {};
\vertex[fill=black, minimum size=3pt] at (0,0.25) {};
\vertex[fill=black, minimum size=3pt] at (0,-0.25) {};
\vertex[fill=black, minimum size=3pt] at (1.0,0.8) {};
\vertex[fill=black, minimum size=3pt] at (1.0,0.2) {};
\node[anchor=east] at (-0.05,1.25) {\tiny\textcolor{red}{$x_1t_3t_5$}};
\node[anchor=east] at (-0.05,0.75) {\tiny\textcolor{red}{$x_1s_0t_5$}};
\node[anchor=east] at (-0.05,0.25) {\tiny\textcolor{darkyellow}{$x_2t_2$}};
\node[anchor=east] at (-0.05,-0.25) {\tiny\textcolor{red}{$x_1t_2$}};
\node[anchor=west] at (1.05,0.8) {\tiny\textcolor{black}{$t_6$}};
\node[anchor=west] at (1.05,0.2) {\tiny\textcolor{black}{$s_2$}};
\end{scope}

\begin{scope}[xshift=8.7cm]
    \draw[-,red,thick] (0,1.3)--(1,1.1);
    \draw[-,red,thick] (0,1.3)--(1,0.5);	    
    \draw[-,cadmiumgreen,thick] (0,0.9)--(1,0.5);	    
    \draw[-,blue,thick] (0,0.5)--(1,0.5);	    
    \draw[-,darkyellow,thick] (0,0.1)--(1,0.5);		
    \draw[-,darkyellow,thick] (0,0.1)--(1,-0.1);		
    \draw[-,red,thick] (0,-0.3)--(1,-0.1);		
    \vertex[fill=black, minimum size=3pt] at (0,1.3) {};
    \vertex[fill=black, minimum size=3pt] at (0,0.9) {};
    \vertex[fill=black, minimum size=3pt] at (0,0.5) {};
    \vertex[fill=black, minimum size=3pt] at (0,0.1) {};
    \vertex[fill=black, minimum size=3pt] at (0,-0.3) {};
    \vertex[fill=black, minimum size=3pt] at (1,1.1) {};
    \vertex[fill=black, minimum size=3pt] at (1,0.5) {};
    \vertex[fill=black, minimum size=3pt] at (1,-0.1) {};
\node[anchor=east] at (-0.05,1.3) {\tiny\textcolor{red}{$x_1s_0s_1$}};
\node[anchor=east] at (-0.05,0.9) {\tiny\textcolor{cadmiumgreen}{$x_3s_1$}};
\node[anchor=east] at (-0.05,0.5) {\tiny\textcolor{blue}{$x_4$}};
\node[anchor=east] at (-0.05,0.1) {\tiny\textcolor{darkyellow}{$x_2t_1$}};
\node[anchor=east] at (-0.05,-0.3) {\tiny\textcolor{red}{$x_1t_3t_4$}};
\node[anchor=west] at (1.05,1.1) {\tiny\textcolor{black}{$t_7$}};
\node[anchor=west] at (1.05,0.5) {\tiny\textcolor{black}{$t_8$}};
\node[anchor=west] at (1.05,-0.1) {\tiny\textcolor{black}{$s_3$}};
\end{scope}

\begin{scope}[xshift=12.2cm]
\draw[-,red,thick] (0.0,1.7)--(1.0,1); 
\draw[-,red,thick] (0.0,1.4)--(1.0,1);  
\draw[-,darkyellow,thick] (0.0,1.1)--(1.0,1);  
\draw[-,darkyellow,thick] (0.0,0.8)--(1.0,1);  
\draw[-,red,thick] (0.0,0.5)--(1.0,1);  
\draw[-,red,thick] (0.0,0.2)--(1.0,1);  
\draw[-,cadmiumgreen,thick] (0.0,-0.1)--(1.0,1.0);  
\draw[-,blue,thick] (0.0,-0.4)--(1.0,1.0);  
\draw[-,blue,thick] (0.0,-0.4)--(1.0,0);  
\draw[-,darkyellow,thick] (0.0,-0.7)--(1.0,0);  
\vertex[fill=black, minimum size=3pt] at (0.0,1.7) {};
\vertex[fill=black, minimum size=3pt] at (0.0,1.4) {};
\vertex[fill=black, minimum size=3pt] at (0.0,1.1) {};
\vertex[fill=black, minimum size=3pt] at (0.0,0.8) {};
\vertex[fill=black, minimum size=3pt] at (0.0,0.5) {};
\vertex[fill=black, minimum size=3pt] at (0.0,0.2) {};
\vertex[fill=black, minimum size=3pt] at (0.0,-0.1) {};
\vertex[fill=black, minimum size=3pt] at (0.0,-0.4) {};
\vertex[fill=black, minimum size=3pt] at (0.0,-0.7) {};
\vertex[fill=black, minimum size=3pt] at (1.0,1) {};
\vertex[fill=black, minimum size=3pt] at (1.0,0) {};
\node[anchor=east] at (-0.05,1.7) {\tiny\textcolor{red}{$x_1t_3t_5t_6$}};
\node[anchor=east] at (-0.05,1.4) {\tiny\textcolor{red}{$x_1s_0t_5t_6$}};
\node[anchor=east] at (-0.05,1.1) {\tiny\textcolor{darkyellow}{$x_2t_2t_6$}};
\node[anchor=east] at (-0.05,0.8) {\tiny\textcolor{darkyellow}{$x_2t_2s_2$}};
\node[anchor=east] at (-0.05,0.5) {\tiny\textcolor{red}{$x_1t_2s_2$}};
\node[anchor=east] at (-0.05,0.2) {\tiny\textcolor{red}{$x_1s_0s_1t_8$}};
\node[anchor=east] at (-0.05,-0.1) {\tiny\textcolor{cadmiumgreen}{$x_3s_1t_8$}};
\node[anchor=east] at (-0.05,-0.4) {\tiny\textcolor{blue}{$x_4t_8$}};
\node[anchor=east] at (-0.05,-0.7) {\tiny\textcolor{darkyellow}{$x_2t_1t_8$}};
\node[anchor=west] at (1.05,1) {\tiny\textcolor{black}{$t_9$}};
\node[anchor=west] at (1.05,0) {\tiny\textcolor{black}{$s_4$}};
\end{scope}

\begin{scope}[xshift=16cm]
\draw[-,red,thick] (0.0,2.3)--(1.0,0.5); 
\draw[-,red,thick] (0.0,2.0)--(1.0,0.5); 
\draw[-,red,thick] (0.0,1.7)--(1.0,0.5);  
\draw[-,darkyellow,thick] (0.0,1.4)--(1.0,0.5);  
\draw[-,darkyellow,thick] (0.0,1.1)--(1.0,0.5);  
\draw[-,red,thick] (0.0,0.8)--(1.0,0.5);  
\draw[-,red,thick] (0.0,0.5)--(1.0,0.5);  
\draw[-,cadmiumgreen,thick] (0.0,0.2)--(1.0,0.5);  
\draw[-,blue,thick] (0.0,-0.1)--(1.0,0.5);  
\draw[-,darkyellow,thick] (0.0,-0.4)--(1.0,0.5);  
\draw[-,red,thick] (0.0,-0.7)--(1.0,0.5);  
\draw[-,blue,thick] (0.0,-1.0)--(1.0,0.5);  
\draw[-,darkyellow,thick] (0.0,-1.3)--(1.0,0.5);  
\vertex[fill=black, minimum size=3pt] at (0.0,2.3) {};
\vertex[fill=black, minimum size=3pt] at (0.0,2.0) {};
\vertex[fill=black, minimum size=3pt] at (0.0,1.7) {};
\vertex[fill=black, minimum size=3pt] at (0.0,1.4) {};
\vertex[fill=black, minimum size=3pt] at (0.0,1.1) {};
\vertex[fill=black, minimum size=3pt] at (0.0,0.8) {};
\vertex[fill=black, minimum size=3pt] at (0.0,0.5) {};
\vertex[fill=black, minimum size=3pt] at (0.0,0.2) {};
\vertex[fill=black, minimum size=3pt] at (0.0,-0.1) {};
\vertex[fill=black, minimum size=3pt] at (0.0,-0.4) {};
\vertex[fill=black, minimum size=3pt] at (0.0,-0.7) {};
\vertex[fill=black, minimum size=3pt] at (0.0,-1.0) {};
\vertex[fill=black, minimum size=3pt] at (0.0,-1.3) {};
\vertex[fill=black, minimum size=3pt] at (1.0,0.5) {};
\node[anchor=east] at (-0.05,2.3) {\tiny\textcolor{red}{$x_1s_0s_1t_7$}};
\node[anchor=east] at (-0.05,2.0) {\tiny\textcolor{red}{$x_1t_3t_5t_6t_9$}};
\node[anchor=east] at (-0.05,1.7) {\tiny\textcolor{red}{$x_1s_0t_5t_6t_9$}};
\node[anchor=east] at (-0.05,1.4) {\tiny\textcolor{darkyellow}{$x_2t_2t_6t_9$}};
\node[anchor=east] at (-0.05,1.1) {\tiny\textcolor{darkyellow}{$x_2t_2s_2t_9$}};
\node[anchor=east] at (-0.05,0.8) {\tiny\textcolor{red}{$x_1t_2s_2t_9$}};
\node[anchor=east] at (-0.05,0.5) {\tiny\textcolor{red}{$x_1s_0s_1t_8t_9$}};
\node[anchor=east] at (-0.05,0.2) {\tiny\textcolor{cadmiumgreen}{$x_3s_1t_8t_9$}};
\node[anchor=east] at (-0.05,-0.1) {\tiny\textcolor{blue}{$x_4t_8t_9$}};
\node[anchor=east] at (-0.05,-0.4) {\tiny\textcolor{darkyellow}{$x_2t_1s_3$}};
\node[anchor=east] at (-0.05,-0.7) {\tiny\textcolor{red}{$x_1t_3t_4s_3$}};
\node[anchor=east] at (-0.05,-1.0) {\tiny\textcolor{blue}{$x_4t_8s_4$}};
\node[anchor=east] at (-0.05,-1.3) {\tiny\textcolor{darkyellow}{$x_2t_1t_8s_4$}};
\node[anchor=west] at (1.05,0.5) {\tiny\textcolor{black}{$y$}};
\end{scope}

\end{scope}
\end{tikzpicture}
    \centering\vspace{-.2in}
\caption{A saturated grove $\Theta=(\theta_0,\theta_1, \ldots,\theta_5)$ satisfying $\Gamma \subset \Theta$. See Example~\ref{ex:partial_grove}.  }
\label{fig:grove_definition}
\end{figure}
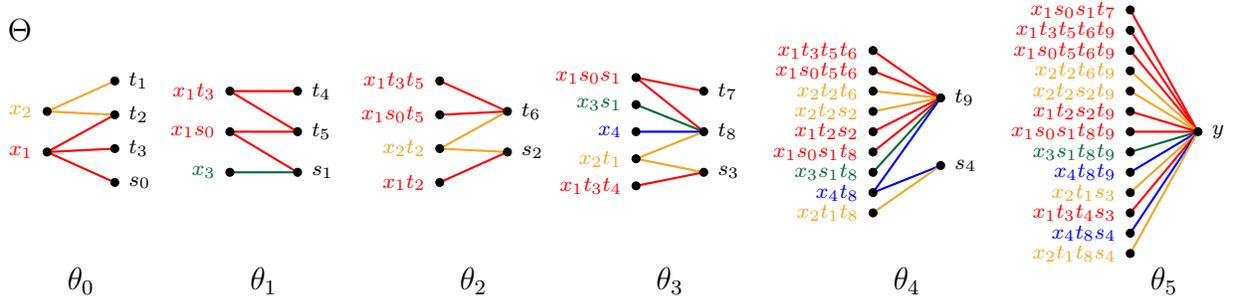

Vineyards and groves are intimately related. Define the grove $\Gamma(\calV)=(\gamma_0(\calV),\ldots,\gamma_n(\calV))$ of a vineyard $\calV$ as follows.  
For every $v\in [0,n]$, $\gamma_v(\calV)$ has the left vertex set
\[L(\gamma_v(\calV))=\{P \in \calV \mid P=Pv\},
\]
with the order induced from $\Prefixes(v)$, the right vertex set 
\[R(\gamma_v(\calV))=\{(v,w)\mid P\in \calV\text{ and }\terminal(P)=(v,w)\},
\]
with the order induced from $\hat\outedge(v)$, and edges $(P,e)$ whenever $Pe\in \calV$. 

\begin{proposition}\label{proposition:VineyardstoGroves}
    The map \[\Gamma: \Vineyards(\hatG,\hatF) \rightarrow \Groves(\hatG,\hatF)\] is a poset isomorphism.
\end{proposition}
\begin{proof}
 The inverse map is given by the function $$\calV:\Groves(\hatG,\hatF)\rightarrow\Vineyards(\hatG,\hatF)$$ where $\calV(\Gamma)$ is the vineyard formed by the set of prefixes \[\bigcup_{v\in[0,n]}L(\gamma_v)\cup\{Py\mid P \in L(\gamma_n)\},\] or, equivalently, induced by the set $\Routes(\calV(\Gamma))=\{Py\mid P \in L(\gamma_n)\}$.
\end{proof}

The notion of splits carries over from vineyards to groves. We define $\Splits(\cdot)$ on both prefixes $P$ and (grove) edges $(P,e)$.
Because of the correspondence in Definition~\ref{def:grove}(c),
splits are now more easily represented as edges in a forest $\gamma_v$. 

\begin{definition}[Grove split]
\label{def:grove_splits}
Let $\Gamma = (\gamma_v)_{v\in [0,n]}$ be a grove of $(\hatG,\hatF)$.
For $P \in L(\gamma_{v})$ let $I(P)$ be \defn{incidences} of $P$. That is, 
\begin{align}
  I(P):&=\{(P,e) \in E(\gamma_v) \mid e \in R(\gamma_v)\}\nonumber\\
  &=\{(P,e_0),(P,e_1),\dots,(P,e_l)\}.\label{equation:incident_edges}
\end{align}
If $|I(P)|\geq 2$, we say that $P$ \defn{has direct splits} and define the \defn{direct splits} of $P$ to be the edges $(P,e_i)$ for $i\in[l]$. We call $(P,e_0)$ the \defn{minimal continuation} or \defn{next step} of $P$ in $\Gamma$ and denote it by \defn{$\next_{\Gamma}(P)$}

For edges $(P,e)$ and left vertices $P$, we define the associated \defn{sets of splits} $\Splits(\Gamma;(P,e))$ and $\Splits(\Gamma;P)$, recursively as follows. Initialize the process by defining $\Splits(\Gamma;(P,y))=\emptyset$ for every edge in $\gamma_n$.

In $\gamma_v$ where $\Splits(\Gamma;(P,e))$ is defined on all its edges, we define $\Splits(\Gamma;P)$ on its left vertices by conditioning on $I(P)$ as in Equation~\eqref{equation:incident_edges}.
\begin{itemize}
    \item If $|I(P)|=1$, define $\Splits(\Gamma;P)=\Splits(\Gamma;(P, e_0))$.
    \item If $|I(P)|\geq 2$, define 
$\displaystyle\Splits(\Gamma;P):=\{(P,e_1),\dots,(P,e_l)\}\sqcup \bigsqcup_{i=0}^l \Splits(\Gamma;(P,e_i)).$
\end{itemize}
Finally, for an edge $(P,e)\in \gamma_u$, we define
\[
\Splits(\Gamma;(P,e)) = \Splits(\Gamma; Pe)
\]
if $e=(u,v)$ and $\Splits(\Gamma; Pe)$ is already defined. 
This process is well defined by starting at $v=n$ and proceeding right to left.

We denote by $\displaystyle\Splits(\Gamma)=\bigsqcup_{x\in X} \Splits(\Gamma;x)$ the set of splits in $\Gamma$. 
\end{definition}

The set $\Splits(\Gamma)$ has an order. 
Let $P\in L(\gamma_v)$ with $\initial(P)=x$ and $I(P)$ as in \eqref{equation:incident_edges}.
If $(P, e_i)$ is a direct split, then any  $(Q,e')\in \Splits(\Gamma;(P,e_i))$ and any $(R,e'')\in \Splits(\Gamma;(P,e_{i-1}))$, satisfy $(R,e'')\prec_x (P,e_i)\preceq_x (Q, e')$. 
This  order separates $\Splits(\Gamma)$ into the totally ordered chains $\Splits(\Gamma;x) = \{S_x^1,\ldots, S_x^{l_x}\}$ for every $x\in X$. 

\begin{definition}[Grove shuffle]
    A \defn{grove shuffle} of rank $k$ is a pair $(\Gamma,\sigma_{\Gamma})$ where $\Gamma$ is a grove and $\sigma_{\Gamma}:\Splits(\Gamma)\rightarrow [k]$ is a surjective function satisfying 
       \textup{$\sigma_{\Gamma}((P,e))<\sigma_{\Gamma}((P',e'))$
       for every pair $(P,e)\prec_x (P',e')$ in $\Splits(\Gamma; x)$ for some $x\in X$.}
    
    We denote by $\GroveShuffles(\hatG,\hatF)$ the set of grove shuffles of $(\hatG,\hatF)$.
\end{definition}

\begin{remark}
When $\ba=\be_0-\be_n$, there is a unique compatible shuffle associated to an underlying grove because the set $\Splits(\Gamma)$ is totally ordered.    
\end{remark}

For $S\in \Splits(\Gamma;x)$, let $\next_\Gamma(S)$ be the next split in $\Splits(\Gamma;x)$, with the understanding that if $S$ is the last split in $\Splits(\Gamma;x)$, then $\sigma_\Gamma(\next(S)) := \rank(\Gamma)+1$ for notational convenience.
\begin{definition}\label{def:grove_shuffles}
We define $(\Gamma, \sigma_\Gamma) \leq (\Theta; \sigma_\Theta)$ on $\GroveShuffles(\hatG,\hatF)$ whenever
\begin{enumerate}
    \item $\Gamma \subseteq \Theta$, and
\item  for splits $S, S' \in \Splits(\Gamma;x)$, 
\[\hbox{if } \sigma_\Gamma(S) \leq 
\sigma_\Gamma(S') \leq \sigma_\Gamma(\next_\Gamma(S))-1, \hbox{ then }
\sigma_\Theta(S) \leq \sigma_\Theta(S') 
\leq \sigma_\Theta(\next_\Theta(S))-1.
\]
\end{enumerate}
\end{definition}

We denote by $\SatGroveShuffles(\hatG,\hatF)$ the set of grove shuffles maximal with respect to the partial order given in Definition~\ref{def:grove_shuffles}, which we call \defn{saturated}.

\begin{theorem}\label{thm:VtoG}
    The map \[(\Gamma,\textup{id}): \VineyardShuffles(\hatG,\hatF) \rightarrow \GroveShuffles(\hatG,\hatF)\] is a poset isomorphism.
\end{theorem}
\begin{proof}
The bijection from Proposition~\ref{proposition:VineyardstoGroves} induces an order-preserving bijection between the ordered sets $\Splits(\calV)$ and $\Splits(\Gamma)$.
The shuffle \(\sigma_{\Gamma(\calV)}((P,e)):=\sigma_{\calV}(Pe)\) is valid because \[\Splits(\Gamma(\calV))=\{(P,e) \mid Pe \in \Splits(\calV)\}.\qedhere\]
\end{proof}

We are now able to establish that maximal objects are all of the same rank.

\begin{lemma}\label{lemma:max_grove_bijection}
$\Gamma$ is a saturated grove if and only if the map $\Splits(\Gamma)\rightarrow\Splits(\hatG)$ that takes $(P,e)\mapsto e$ is a bijection. Moreover, $(\Gamma,\sigma_\Gamma)$ is a saturated grove shuffle if and only if $\Gamma$ is a saturated grove and $\sigma_\Gamma$ is a bijection.
\end{lemma}
\begin{proof}
First, we verify that the map $\Splits(\Gamma) \rightarrow \Splits(\hatG): (P,e) \mapsto e$ is well-defined, because $(P,e)\in \Splits(\Gamma)$ means it is an edge in some forest $\gamma_v$ whose left vertex $P$ has incidence $I(P)\geq2$, and $e$ is not the smallest edge in $\hat\outedge(v)$ that $P$ is incident to.
That is, $e\in\Splits(\hatG)$.

We next show that this map is injective. 
Suppose for contradiction that the grove splits $(P,e)$ and $(P',e)$ map to $e\in \Splits(\hatG)$.
Then $P$ and $P'$ are left vertices in a forest $\gamma_v$.
Without loss of generality, we assume that $P \prec P'$.
Since $\gamma_v$ is noncrossing, then $(P,e)$ and $(P',e)$ cannot both be splits in the grove.
Therefore, the map is injective.

Now, $\Gamma$ is a saturated grove if and only if each noncrossing bipartite forest $\gamma_v \in \Gamma$ is a tree and its right vertex set $R(\theta_v) = \hat\outedge(v)$.
Since $\gamma_v$ is a tree, then it has $|L(\gamma_v)|+|R(\gamma_v)|-1$ edges.
But each vertex in $L(\gamma_v)$ is incident to at least one vertex in $R(\gamma_v)$, hence there are $|R(\gamma_v)|-1$ splits in $\gamma_v$.
By definition, there are precisely $|\hat\outedge(v)|-1$ splits in $\Splits(\hatG)$ whose tail is at $v$. Therefore, $\Gamma$ is a saturated grove implies 
\[|\Splits(\Gamma)| = \sum_v  (|R(\gamma_v)|-1)
= \sum_v (|\hat\outedge(v)|-1) = |\Splits(\hatG)|,
\]
so the inclusion $\Splits(\Gamma) \rightarrow \Splits(\hatG): (P,e) \mapsto e$ is a bijection.

Conversely, if the injection $\Splits(\Gamma) \rightarrow \Splits(\hatG): (P,e) \mapsto e$ is a bijection, then by our discussion above, this forces each noncrossing bipartite forest $\gamma_v$ to be a tree with $R(\gamma_v) = \hat\outedge(v)$, and $\Gamma$ is a saturated grove.

Lastly, given a grove shuffle $(\Gamma, \sigma_\Gamma)$, the surjective function $\sigma_\Gamma: \Splits(\Gamma) \rightarrow [d]$ is a bijection if and only if 
$|\Splits(\Gamma)| = |\Splits(\hatG)| = d$, which holds if and only if $(\Gamma,\sigma_\Gamma)$ is a saturated grove shuffle.
\end{proof}

\begin{corollary}
\label{cor:dimension_d}
Saturated cliques are all of cardinality $d+1$.
As a consequence, the simplicial complex $\Cliques(\hatG,\hatF)$ is pure of dimension $d$. 
\end{corollary}

\subsubsection{Natural labeling  of a grove}
\label{subsec.natlabel_grove}
\phantom{W}

The isomorphism of Theorem~\ref{thm:VtoG} shows that the notion of a natural labeling carries over from vineyards to groves as well. We can define $\lambda(\cdot)$ on both prefixes $P$ and (grove) edges $(P,e)$ using labels from $X\cup \Splits(\supp(\Gamma))$. This labeling coincides with $\lambda$ from vineyards in prefixes and can be easily extended to edges by defining
\begin{itemize}
    \item $\lambda(x)=x$ for every $x\in X$,
    \item $\lambda((P,e)) := \lambda(Pe) := \begin{cases}
        e, & \hbox{if } (P,e)\in \Splits(\Gamma)\\
        \lambda(P), & \hbox{if } (P,e)\not \in \Splits(\Gamma).
    \end{cases}$   
\end{itemize}

Using the natural labeling $\lambda$ we can encode all the information from $\Gamma$ into a seemingly simplified object $\lambda(\Gamma)=(\lambda(\gamma_0),\dots, \lambda(\gamma_n))$, which we call a \defn{labeled grove}, where $\lambda(\gamma_v)$ is a  noncrossing bipartite labeled forest, defined with the same conditions as in Definition \ref{def:noncrossing_bipartite_tree} except that  $$L(\lambda(\gamma_v))\subseteq \{e\in X\cup E\mid \exists P\in \Prefixes(v)\text{ such that }e\in P\},$$
whose right vertex set $R(\lambda(\gamma_n))=R(\gamma_v)$, left vertex set $$L(\lambda(\gamma_v))=\{\lambda(P)\mid P \in L(\gamma_v)\}$$ where the labels come with the same relative order as its corresponding prefixes, that is, $\lambda(P)<_{L(\lambda(\gamma_v))}\lambda(Q)$ whenever $P<_{L(\gamma_v)}Q$, and with edge set 
$$E(\lambda(\gamma_v))=\{(\lambda(P),e)\mid (P,e)\in \gamma_v\}.$$
In other words, as a forest $\lambda(\Gamma)$, is isomorphic  to $\Gamma$ except that the labels of the vertices $L(\gamma_v)$ have been replaced by $L(\lambda(\gamma_v))$ for all $v$. See Figures~\ref{fig:partial_grove_shuffle} and \ref{fig:max_grove_shuffle} for the examples of $\lambda(\Gamma)$ and $\lambda(\Theta)$ for $\Gamma$ and $\Theta$ of Figures \ref{fig:partial_grove_definition} and \ref{fig:grove_definition} respectively.

\begin{figure}[htb!]
    \input{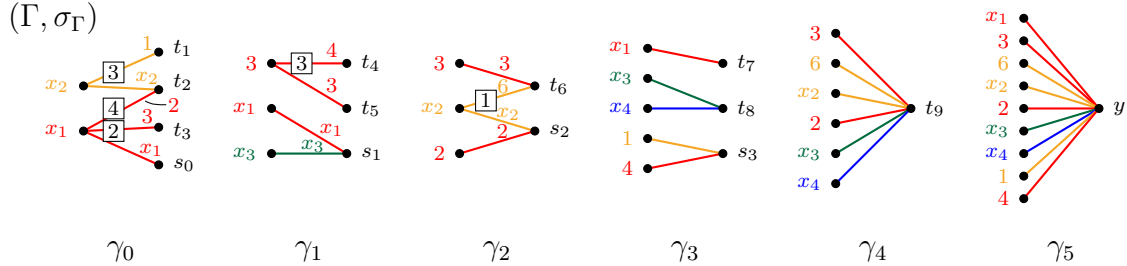}
    \caption{The grove shuffle $(\Gamma,\sigma_\Gamma)$ corresponding to the clique $\calC$ in Example~\ref{fig:cliques_and_multicliques}. The numbers in the boxes define $\sigma_\Gamma$ on $\Splits(\Gamma)$. Left vertices and edges in $\Gamma$ are labeled by the natural labeling $\lambda$. For legibility, labels on edges $(P,e)$ are suppressed in $\gamma_3$, $\gamma_4$, and $\gamma_5$; in those cases, $\lambda((P,e))=\lambda(P)$.}
\label{fig:partial_grove_shuffle}
\end{figure}

\begin{figure}[htb!]
    \input{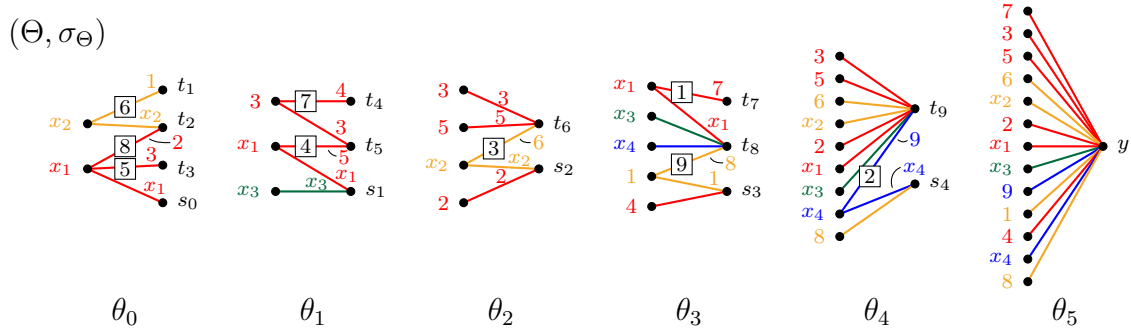}
    \caption{The saturated grove shuffle $(\Theta,\sigma_\Theta)$ corresponding to the saturated clique~$\calD$ in Example~\ref{fig:SatCliques_and_multicliques}. The numbers in the boxes define $\sigma_\Theta$ on $\Splits(\Theta)$. For legibility, labels on edges $(P,e)$ are suppressed in $\theta_3$, $\theta_4$, and $\theta_5$ when the incident left vertex $P$ has degree one; in those cases, $\lambda((P,e))=\lambda(P)$.}
\label{fig:max_grove_shuffle}
\end{figure}

\begin{lemma}\label{lemma:labeled_grove} Given $\Gamma \in \Groves(\hatG,\hatF)$, $\lambda(\Gamma)=(\lambda(\gamma_0),\dots, \lambda(\gamma_n))$ is a set  of noncrossing bipartite labeled forests, that when are not all empty, satisfy that $e \in L(\gamma_w)$ if and only if either 
    \begin{itemize}
        \item $e=x$ and $w=v_x$ for $x\in X$, or
        \item $e=(v,w)$ and there is a split edge of the form $(e',e)\in E(\lambda(\gamma_v))$, or
        \item there is an edge $e'=(v,w)$ and a split edge of the form $(e,e') \in E(\lambda(\gamma_v))$.
    \end{itemize}
\end{lemma}

It is not difficult to verify that the conditions in Lemma \ref{lemma:labeled_grove} completely characterize the image $\lambda(\Groves(\hatG,\hatF))$. Indeed, if $\widetilde\Gamma=(\tilde\gamma_0,\dots,\tilde\gamma_n)$ is in the image, for every $e\in L(\tilde\gamma_v)$ we can recover its associated sets of prefixes using the function $$P:\{(e,w)\mid e \in L(\tilde\gamma_w)\text{ for }w\in [0,n]\}\rightarrow \Prefixes(\hatG)$$
defined recursively as
\begin{equation}\label{equation:labels_to_prefixes_groves}
    P(e,w)=\begin{cases}
        x&\text{ if }e=x\in X\\
        P(e',v)e&\text{ if } (e',e) \text{ is a split in } \tilde \gamma_{v=\tail(e)}\\
        P(e,v)e'&\text{ if } (e,e') \text{ is a minimal continuation in } \tilde \gamma_{v=\tail(e')}.
    \end{cases}
\end{equation}

We can verify inductively on $v=0,\dots,n$ the following lemma.

\begin{lemma}\label{lemma:labels_to_prefixes}
    For every $\widetilde \Gamma \in \lambda(\Groves(\hatG,\hatF))$ and  $e\in L(\tilde \gamma_v)$ we have $\lambda(P(e,w))=e$. On the other hand for every $\Gamma \in \Groves(\hatG,\hatF)$ and $P\in L(\gamma_v)$ we have $P(\lambda(P),\tail(\terminal(P)))$.
\end{lemma}

So Lemma \ref{lemma:labels_to_prefixes} induces a bijection between the set of groves in $\Groves(\hatG,\hatF)$ and the set of labeled groves of the form $\lambda(\Groves(\hatG,\hatF))$. Hence, the elements in $\lambda(\Groves(\hatG,\hatF))$ are in practice a more convenient way to construct and store a grove. With the added caveat that the process of recovering a prefix in an element in $\Groves(\hatG,\hatF)$  is by a direct retrieval while from $\lambda(\Groves(\hatG,\hatF))$ we need to recursively construct it using Equation \eqref{equation:labels_to_prefixes_groves}.

In particular, the set of routes $\Routes(\Gamma)$ can be obtained from $\lambda(\Gamma)$ as
$$\Routes(\lambda(\Gamma))=\{P(e,n)y\mid e\in L(\lambda(\gamma_n))\}.$$

\subsection{Permutation flow shuffles}
\label{sec:permutation_flow_shuffles}
\phantom{W}

The natural labeling $\lambda$ (of a grove (Section~\ref{subsec.natlabel_grove}) or equivalently of a vineyard (Definition~\ref{def:natural_labeling_vineyard})) leads to the definition of the final combinatorial object we study.  This object acts as a bookkeeper of the pertinent features (routes, splits, and multiplicities) we have discussed so far, and is the easiest to construct while all the previous information can be extracted by performing simple recursive calculations.

The set of \defn{partial permutations} of a set $A$ is $\perm_A:=\bigcup_{B\subseteq A}\sym_B$. Note that $\perm_A$ also includes the empty partial permutation denoted by $\emptyset$. For a function $\pi:\hat{E}\rightarrow \perm_{X \cup \Splits(\hat{G})}$  we define its  \defn{support} to be the subgraph $\supp(\pi)$ for which $\pi(e) \neq \emptyset$.

At every vertex $v$ of $\hatG$, with ordered lists of incoming edges $\hat\inedge(v)=(e_0, \hdots, e_i)$ and outgoing edges $\hat\outedge(v)=(e_0', \hdots, e_j')$, we define the \defn{incoming} and \defn{outgoing} \defn{summary permutations} at $v$ to be the concatenations  
\[\InPerm(v)=\pi({e_0})\cdots \pi({e_i})\text{ and }\OutPerm(v)=\pi(e_0')\cdots \pi(e_j'),
\] respectively.

\begin{definition}[Permutation flow]
\label{def:permuflow_a}
A \defn{permutation flow} on  $(\hatG,\hatF)$ is a function \[\pi:\hat{E}\rightarrow
\perm_{X \cup \Splits(\hat{G})}\]
that, if it is nonempty, satisfies conditions (i) through (iii) below:
\begin{itemize}
    \item[(i)]  For each $x\in X$, $\pi(x)=x$.
    \item[(ii)] At an inner vertex $v$ with ordered lists of incoming edges $\hat\inedge(v) = (e_0,\ldots,e_i)$ and outgoing edges $\hat\outedge(v) = (e_0',e_1',\ldots, e_j')$, $\OutPerm(v)$ is an unshuffle\footnote[2]{Remark~\ref{remark:shuffle} discusses this language.} of $\InPerm(v)$  and a (possibly empty) subword of $e_1' \cdots e_j'$. 
    \item[(iii)] 
    If $e_h' \in \OutPerm(v)$ for some $1\leq h \leq j$, then $e_h'$ is the first letter of $\pi(e_h')$ and $e_h'$ is a split in $\supp(\pi)$.
\end{itemize}
We denote by $\PermutationFlows(\hatG,\hatF)$ the set of permutation flows on $(\hatG,\hatF)$.
\end{definition}

Note that edges $e \in \hatE$ play two different roles in this definition: as edges in $\hatE$ and as letters in the words of $\perm_{X \cup \Splits(\hat{G})}$. 
In order to clarify this distinction, whenever $e\in\pi(e')$, we will refer to the edge $e'$ as the \defn{carrier} of the \defn{letter} $e$. For a carrier $e'\in X\sqcup E$ and a letter~$e$ in $\pi(e')$, we denote  $e^*=\next_\pi(e,e')$ the carrier of $e$ such that $\head(e')=\tail(e^*)$, we call it the \defn{next carrier} of $e$ after $e'$. 

Figure~\ref{fig:FSM_subpermuflo} shows an example of a permutation flow $\pi$.
For each edge $e\in \hatE$, $\pi(e)$ is pictured on the corresponding edge.
In this example, the only edge with empty $\pi(e)$ is $s_4$, so $\supp(\pi) = \hatE \backslash \{s_4\}$.

\begin{figure}[ht!]
    \centering\vspace{-.3in}
\begin{tikzpicture}
\begin{scope}[scale=1.5, yscale=1.0]
\node() at (-1,.75) {$\pi$};

\vertex[fill=orange, minimum size=4pt, label=below:{\tiny\textcolor{orange}{$0$}}](v0) at (0,0) {};
\vertex[fill=orange, minimum size=4pt, label=below:{\tiny\textcolor{orange}{$1$}}](v1) at (1,0) {};
\vertex[fill=orange, minimum size=4pt, label=below:{\tiny\textcolor{orange}{$2$}}](v2) at (2,0) {};
\vertex[fill=orange, minimum size=4pt, label=below:{\tiny\textcolor{orange}{$3$}}](v3) at (3,0) {};
\vertex[fill=orange, minimum size=4pt, label=below:{\tiny\textcolor{orange}{$4$}}](v4) at (4,0) {};
\vertex[fill=orange, minimum size=4pt, label=below:{\tiny\textcolor{orange}{$5$}}](v5) at (5,0) {};		

\draw[-stealth, thick, color=black!30, color=black!30] (v0) .. controls (1.2, 1.6) and (2.5, -0.3) .. (v3);
\draw[-stealth, thick, color=black!30] (v0) .. controls (0.9, 1.0) and (1.5, -0.7) .. (v2);
\draw[-stealth, thick, color=black!30] (v0) to [out=30,in=150] (v1);
\draw[-stealth, thick, color=black!30] (v0) to [out=-30,in=-150] (v1);

\draw[-stealth, thick, color=black!30] (v1) .. controls (1.9, 1.0) and (2.5, -0.7) .. (v3);
\draw[-stealth, thick, color=black!30] (v1) to [out=30,in=150] (v2);
\draw[-stealth, thick, color=black!30] (v1) .. controls (2.0, -1.2) and (2.5, 0.7) .. (v3);	

\draw[-stealth, thick, color=black!30] (v2) to [out=45,in=135] (v4);
\draw[-stealth, thick, color=black!30] (v2) .. controls (3.0, -1.0) and (3.1, 0.3) .. (v4);	

\draw[-stealth, thick, color=black!30] (v3) to [out=60,in=120] (v5);
\draw[-stealth, thick, color=black!30] (v3) to [out=-30,in=-150] (v4);
\draw[-stealth, thick, color=black!30] (v3) .. controls (4.0, -1.0) and (4.5, 0.0) .. (v5);

\draw[-stealth, thick, color=black!30] (v4) to [out=30,in=150] (v5);
\draw[-stealth, thick, color=black!30] (v4) to [out=-30,in=-150] (v5);

\draw[-stealth, thick, color=black!30] (-0.5,0.2) .. controls (-0.4, 0.2) and (-0.15, 0.1) .. (v0);
\draw[-stealth, thick, color=black!30] (-0.5, -.4) .. controls (-0.4, -.4) and (-0.25, -.3) .. (v0);
\draw[-stealth, thick, color=black!30] (0.5, -.4) .. controls (0.6, -.4) and (0.75, -.3) .. (v1);
\draw[-stealth, thick, color=black!30] (2.5, -.5) .. controls (2.6, 0) and (2.7, 0.1) .. (v3);
\draw[-stealth, thick, color=black!30] (v5) to [out=20, in=160] (5.5, 0);

\node[] at (-0.6, -0.4){\scriptsize\textcolor{red}{$x_1$}};
\node[] at (-0.6, 0.2){\scriptsize\textcolor{darkyellow}{$x_2$}};
\node[] at (0.4, -0.4){\scriptsize\textcolor{cadmiumgreen}{$x_3$}};
\node[] at (2.5, -0.6){\scriptsize\textcolor{blue}{$x_4$}};
\node[] at (.6, .6){\scriptsize\textcolor{black}{$1$}};
\node[] at (.55, .35){\scriptsize\textcolor{black}{$2x_2$}};
\node[] at (.4, .15){\scriptsize\textcolor{black}{$3$}};
\node[] at (1.5, .4){\scriptsize\textcolor{black}{$4$}};
\node[] at (1.5, .15){\scriptsize\textcolor{black}{$3$}};
\node[] at (2.8, .45){\scriptsize\textcolor{black}{$63$}};
\node[] at (3.7, .55){\scriptsize\textcolor{black}{$x_1$}};
\node[] at (3.6, -0.2){\scriptsize\textcolor{black}{$x_4x_3$}};
\node[] at (4.48, .2){\scriptsize\textcolor{black}{$x_4x_32x_263$}};
\node[] at (.3, -.15){\scriptsize\textcolor{black}{$x_1$}};
\node[] at (1.5, -.43){\scriptsize\textcolor{black}{$x_3x_1$}};
\node[] at (2.8, -.43){\scriptsize\textcolor{black}{$2x_2$}};
\node[] at (3.7, -.45){\scriptsize\textcolor{black}{$41$}};
\node[] at (6.2, -.1){\scriptsize\textcolor{black}{$41x_4x_32x_263x_1$}};
\end{scope}
\end{tikzpicture}
    \vspace{-.3in}
    \caption{A permutation flow $\pi$ corresponding to the grove $\Gamma$ in Figure~\ref{fig:partial_grove_definition}.
    $\pi(e)$ is pictured on each edge $e$.
     }
    \label{fig:FSM_subpermuflo}
\end{figure}
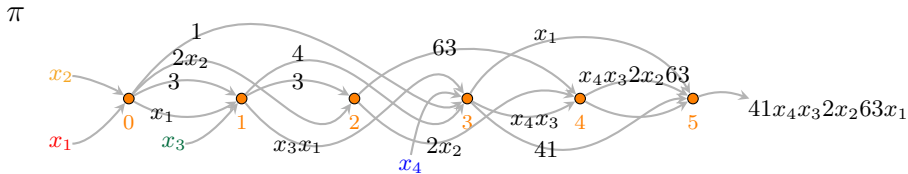

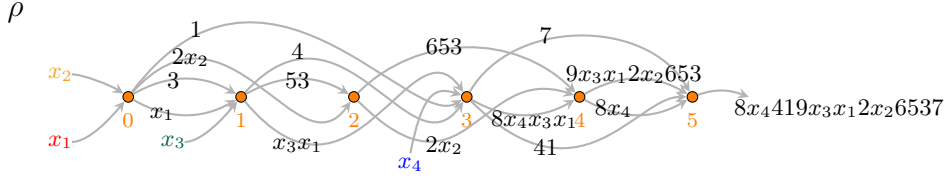
\begin{figure}[ht!]
    \centering\vspace{-.5in}
\begin{tikzpicture}
\begin{scope}[scale=1.5, yscale=1.0]
\node() at (-1,.75) {$\rho$};

\vertex[fill=orange, minimum size=4pt, label=below:{\tiny\textcolor{orange}{$0$}}](v0) at (0,0) {};
\vertex[fill=orange, minimum size=4pt, label=below:{\tiny\textcolor{orange}{$1$}}](v1) at (1,0) {};
\vertex[fill=orange, minimum size=4pt, label=below:{\tiny\textcolor{orange}{$2$}}](v2) at (2,0) {};
\vertex[fill=orange, minimum size=4pt, label=below:{\tiny\textcolor{orange}{$3$}}](v3) at (3,0) {};
\vertex[fill=orange, minimum size=4pt, label=below:{\tiny\textcolor{orange}{$4$}}](v4) at (4,0) {};
\vertex[fill=orange, minimum size=4pt, label=below:{\tiny\textcolor{orange}{$5$}}](v5) at (5,0) {};		

\draw[-stealth, thick, color=black!30, color=black!30] (v0) .. controls (1.2, 1.6) and (2.5, -0.3) .. (v3);
\draw[-stealth, thick, color=black!30] (v0) .. controls (0.9, 1.0) and (1.5, -0.7) .. (v2);
\draw[-stealth, thick, color=black!30] (v0) to [out=30,in=150] (v1);
\draw[-stealth, thick, color=black!30] (v0) to [out=-30,in=-150] (v1);

\draw[-stealth, thick, color=black!30] (v1) .. controls (1.9, 1.0) and (2.5, -0.7) .. (v3);
\draw[-stealth, thick, color=black!30] (v1) to [out=30,in=150] (v2);
\draw[-stealth, thick, color=black!30] (v1) .. controls (2.0, -1.2) and (2.5, 0.7) .. (v3);	

\draw[-stealth, thick, color=black!30] (v2) to [out=45,in=135] (v4);
\draw[-stealth, thick, color=black!30] (v2) .. controls (3.0, -1.0) and (3.1, 0.3) .. (v4);	

\draw[-stealth, thick, color=black!30] (v3) to [out=60,in=120] (v5);
\draw[-stealth, thick, color=black!30] (v3) to [out=-30,in=-150] (v4);
\draw[-stealth, thick, color=black!30] (v3) .. controls (4.0, -1.0) and (4.5, 0.0) .. (v5);

\draw[-stealth, thick, color=black!30] (v4) to [out=30,in=150] (v5);
\draw[-stealth, thick, color=black!30] (v4) to [out=-30,in=-150] (v5);

\draw[-stealth, thick, color=black!30] (-0.5,0.2) .. controls (-0.4, 0.2) and (-0.15, 0.1) .. (v0);
\draw[-stealth, thick, color=black!30] (-0.5, -.4) .. controls (-0.4, -.4) and (-0.25, -.3) .. (v0);
\draw[-stealth, thick, color=black!30] (0.5, -.4) .. controls (0.6, -.4) and (0.75, -.3) .. (v1);
\draw[-stealth, thick, color=black!30] (2.5, -.5) .. controls (2.6, 0) and (2.7, 0.1) .. (v3);
\draw[-stealth, thick, color=black!30] (v5) to [out=20, in=160] (5.5, 0);

\node[] at (-0.6, -0.4){\scriptsize\textcolor{red}{$x_1$}};
\node[] at (-0.6, 0.2){\scriptsize\textcolor{darkyellow}{$x_2$}};
\node[] at (0.4, -0.4){\scriptsize\textcolor{cadmiumgreen}{$x_3$}};
\node[] at (2.5, -0.6){\scriptsize\textcolor{blue}{$x_4$}};
\node[] at (.6, .6){\scriptsize\textcolor{black}{$1$}};
\node[] at (.55, .35){\scriptsize\textcolor{black}{$2x_2$}};
\node[] at (.4, .15){\scriptsize\textcolor{black}{$3$}};
\node[] at (1.5, .4){\scriptsize\textcolor{black}{$4$}};
\node[] at (1.5, .15){\scriptsize\textcolor{black}{$53$}};
\node[] at (2.8, .45){\scriptsize\textcolor{black}{$653$}};
\node[] at (3.7, .55){\scriptsize\textcolor{black}{$7$}};
\node[] at (3.6, -0.2){\scriptsize\textcolor{black}{$8x_4x_3x_1$}};
\node[] at (4.48, .2){\scriptsize\textcolor{black}{$9x_3x_12x_2653$}};
\node[] at (.3, -.15){\scriptsize\textcolor{black}{$x_1$}};
\node[] at (1.5, -.43){\scriptsize\textcolor{black}{$x_3x_1$}};
\node[] at (2.8, -.43){\scriptsize\textcolor{black}{$2x_2$}};
\node[] at (3.7, -.45){\scriptsize\textcolor{black}{$41$}};
\node[] at (4.3, -.1){\scriptsize\textcolor{black}{$8x_4$}};
\node[] at (6.3, -.1){\scriptsize\textcolor{black}{$8x_4419x_3x_12x_26537$}};
\end{scope}
\end{tikzpicture}
    \vspace{-.3in}
    \caption{A maximal permutation flow $\rho$ corresponding to the grove $\Theta$ in Figure~\ref{fig:grove_definition}.
    This maximal permutation flow contains $\pi$ of Figure~\ref{fig:FSM_subpermuflo} as a subpermutation flow.
    }
    \label{fig:FSM_maxpermuflo}
\end{figure}

The notion of splits carried over from vineyards to groves, and now to permutation flows.
Splits are now represented as edges in $E$.

\begin{definition}[Permutation flow split]
\label{def:permuflow_splits}
Let $\pi$ be a permutation flow of $(\hatG,\hatF)$. 
An edge $t\in E$ is a \defn{split} in $\pi$ if $t$ is the first letter of $\pi(t)$. 
We denote by \[\Splits(\pi)=\{t\in E\mid \textup{$t$ is a split in $\pi$}\}\] the \defn{set of splits} in $\pi$.

If $t\in \Splits(\pi)$ with  $v=\tail(t)$, let $e^*\in\hat\outedge(v)$ be the largest element satisfying $e^* \prec_{\hat{\outedge}(v)} t$, $\pi(e^*)\neq e^*$, and $\pi(e^*) \neq \emptyset$. Let $e$ be the last letter of $\pi(e^*)$.  In this situation, we say that \defn{$t$ is a direct split of $e$ at $v$}.
\end{definition}

\begin{definition}[Sets of splits]
Given a letter $e\in \pi(e')$, suppose $\{t_1,\ldots, t_l\}$ is the set of direct splits of the letter $e$ at $v=\head(e')$. 
We define the \defn{set of splits of $e$ at $e'$ in $\pi$}, denoted $\Splits(\pi; e,e')$, recursively as follows.

First, define $\Splits(\pi;e,y)=\emptyset$ for all $e\in \hat{E}$. Then define 
\begin{equation}\label{def:splitsofe} 
\Splits(\pi;e,e')=\Splits(\pi;e,\next_{\pi}(e,e'))\,\sqcup \bigsqcup_{i\in[l]}\Big(\{t_i\}\sqcup\Splits(\pi;t_i,t_i)\Big).
\end{equation}
Moreover, for convenience we also denote $\Splits(\pi;x):=\Splits(\pi;x,x)$ for $x\in X$ and note that $\Splits(\pi)=\bigsqcup_{x\in X}\Splits(\pi;x)$. The sets $\Splits(\pi;x)$ come with a total order $\prec_x$ that respects the framing $\hatF$ and the following conditions using the notation from Equation~\eqref{def:splitsofe}. For all $a\in\Splits(\pi,e,\next_{\pi}(e,e'))$ and for all $b_i\in\Splits(\pi,t_i,t_i)$, 
\[a\prec_x t_1\prec_x b_1 \prec_x t_2 \prec_x b_2 \prec_x \cdots \prec_x t_l \prec_x b_l.\]
\end{definition}

Note that for every $e\in\pi(e')$, neither $e$ nor $e'$ are elements of $\Splits(\pi;e,e')$. 

\begin{example}\label{eg.directsplit}
Referring to Figure~\ref{fig:direct_split_example}, suppose that $\inedge(v) = \{e_0,e_1 \} = \{ t_2,t_1\}$ and $\outedge(v) = \{e_0', \ldots, e_4'\} = \{s_v, t_6,\ldots, t_3 \}$.
The values $\pi(e)$ are pictured on the corresponding edge.
We have the incoming summary permutation $\InPerm(v) = abcdef$, the outgoing summary permutation $abct_2t_4def$, which is an unshuffle of $abcdef$ with $t_5t_3$, and the splits of $\pi$ whose tail is $v$ are $t_5$ and $t_3$.
For example, we see that the edge $t_5$ is a direct split of $c$ at the carrier $e_0=t_2$, and $t_3$ is also a direct split of $c$ at the carrier $e_0=t_2$.
Therefore, 
\[\Splits(\pi; c, t_2) =\{ t_5, t_3\}\sqcup \Splits(\pi; c, t_6) \sqcup \Splits(\pi; c, t_5) \sqcup \Splits(\pi; c, t_3).
\]
\end{example}
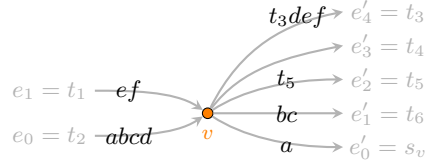
\begin{figure}[ht!]
\begin{tikzpicture}
\begin{scope}[scale=1.5, yscale=1.0]

\vertex[fill=orange, minimum size=4pt, label=below:{\tiny\textcolor{orange}{$v$}}](v0) at (0,0) {};
	
\draw[-stealth, thick, color=black!30] (-1,.2) .. controls (-.7, .2) and (-0.3,.2) .. (v0);
\draw[-stealth, thick, color=black!30] (-1, -.2) .. controls (-.7, -.2) and (-0.3, -.2) .. (v0);
\draw[-stealth, thick, color=black!30, color=black!30] (v0) .. controls (.3, .6) and (.7, .8) .. (1.2,.9);
\draw[-stealth, thick, color=black!30, color=black!30] (v0) .. controls (.3, .4) and (.7, .5) .. (1.2,.6);
\draw[-stealth, thick, color=black!30] (v0) .. controls (0.3, .2) and (.7, .3) .. (1.2,.3);
\draw[-stealth, thick, color=black!30] (v0) to (1.2,0);
\draw[-stealth, thick, color=black!30] (v0) .. controls (0.3,-.2) and (.7,-.3) .. (1.2,-.3);

\node[] at (-1.4,.2) {\scriptsize \textcolor{black!30}{$e_1=t_1$}};
\node[] at (-1.4,-.2) {\scriptsize \textcolor{black!30}{$e_0=t_2$}};
\node[] at (1.6,.9) {\scriptsize \textcolor{black!30}{$e_4'=t_3$}};
\node[] at (1.6,.6) {\scriptsize \textcolor{black!30}{$e_3'=t_4$}};
\node[] at (1.6,.3) {\scriptsize \textcolor{black!30}{$e_2'=t_5$}};
\node[] at (1.6,0) {\scriptsize \textcolor{black!30}{$e_1'=t_6$}};
\node[] at (1.6,-.3) {\scriptsize \textcolor{black!30}{$e_0'=s_v$}};

\node[] at (-.7, .2){\scriptsize$ef$};
\node[] at (-.7, -.2){\scriptsize$abcd$};
\node[] at (.8, .85){\scriptsize$t_3def$};
\node[] at (.7, .3){\scriptsize$t_5$};
\node[] at (.7, 0){\scriptsize$bc$};
\node[] at (.7, -.3){\scriptsize$a$};

\end{scope}
\end{tikzpicture}
    \caption{Illustrating direct splits and sets of splits in a permutation flow. See Example~\ref{eg.directsplit}.
    }
    \label{fig:direct_split_example}
\end{figure}

We can construct permutation flows from groves.
Given a grove $\Gamma$, define the permutation flow $\pi(\Gamma)$ as follows.
For each vertex $v\in [0,n]$, and $e\in \hat\outedge(v)$, let
\[
(\pi(\Gamma))(e) = \lambda((P_0,e))\cdots \lambda((P_l,e)),
\]
where $I(e)=\{P\in L(\gamma_v)\mid (P,e) \in E(\gamma_v)\}=\{P_0,\ldots, P_l\}$ is the ordered set of edges in the bipartite noncrossing tree $\gamma_v$ which are incident to $e\in R(\gamma_v)$, and $\lambda$ is the natural labeling of the grove.

It is easy to check inductively on $v=0,1,\dots,n$ that this map indeed gives a valid permutation flow due to the fact that every tree $\gamma_v$ is non-crossing and that the only possible grove edge in $I(e)$ which can be a split is $(P_0,e)$, which is the only case where we have that $\lambda(P_0,e)=e$.

Using the map given by $\pi(\Gamma)$ we can also translate the order relation $\Gamma \subseteq \Theta$ in $\Groves(\hatG,\hatF)$ as follows.
We say that $t\in X\cup\Splits(\pi)$ is a \defn{split ancestor} of $t'$ whenever $\emptyset \neq \Splits(\pi;t')\subset \Splits(\pi,t)$. We define for every $t \in \Splits(\pi)$,
$$\SplitAncestors(\pi,t):=\{e \mid e \text{ a split ancestor of }t\},$$
the \defn{set of split ancestors} of $t$. Note that $\SplitAncestors(\pi,x)=\emptyset$ for every $x\in X$.

\begin{definition}[Split reduction]\label{definition:split_reduction} Let $\pi,\rho \in \PermutationFlows(\hatG,\hatF)$. For $e\in \hatE$, we say that the word $\pi(e)$ is \defn{obtained by split reduction} from the word $\rho(e)$, whenever $\pi(e)$ is obtained from $\rho(e)$ by performing in order to every letter $t$ in $\rho(e)$ exactly one of the following actions: 
\begin{itemize}
    \item keeping $t$, or
    \item deleting $t$, or
    \item replacing $t$ by the letter $\max \left( \SplitAncestors(\rho,t)\cap\Splits(\pi)\right)$.
\end{itemize}
we say that $\pi$ is a \defn{split reduction} of $\rho$ if for every $e\in E$ the word $\pi(e)$ is obtained by split reduction from $\rho(e)$.
\end{definition}

\begin{lemma} We have that $\Gamma \subseteq \Theta$ in $\Groves(\hatG,\hatF)$ if and only if for every $e\in\hatE$,
the word $\pi(\Gamma)(e)$ is obtained from $\pi(\Theta)(e)$ by  split reduction.
\end{lemma}

We define an order on $\PermutationFlows(\hatG,\hatF)$ by saying that $\pi\subseteq \rho$ if $\pi$ is obtained from $\rho$ by split reduction. We denote by $\SatPermutationFlows(\hatG,\hatF)$ the set of permutation flows that are maximal with respect to this order. We then have the following theorem.

\begin{proposition}\label{prop.GtoPF}
The map
\[\pi : \Groves( \hatG, \hatF) \rightarrow \PermutationFlows( \hatG, \hatF )\]
is a poset isomorphism.    
\end{proposition}
\begin{proof}
The inverse map is given by the function
\[
\Gamma: \PermutationFlows(\hatG,\hatF) \rightarrow \Groves(\hatG,\hatF)
\]
where $\Gamma(\pi)=(\gamma_v)_{v\in [0,n]}$ is the corresponding grove associated to the labeled grove with noncrossing bipartite labeled forests such that
\begin{align*}
L(\tilde \gamma_v) 
    &= \{e \mid e\in\pi(e') \hbox{ for some } e'\in X\cup E\text{ with }  \head(e') =v\},\\
R(\tilde \gamma_v)
    &= \{e'' \mid e''\in \hat\outedge(v) \hbox{ and } \pi(e'') \neq \emptyset\},
\end{align*}
and $(e,e'')$ is an edge in $\gamma_v$ if $e\in \pi(e'')$ or $e''$ is a direct split of $e$ at $v$ in $\pi$. 

The conditions for $\widetilde \Gamma$ for being a valid labeled grove are verified inductively on $v=0,1,\dots,n$.
\end{proof}

We can use Theorem \ref{prop.GtoPF} together with the prefix function of Equation \eqref{equation:labels_to_prefixes_groves} to recursively define a function $$P:\{(e,w)\mid e \in \InPerm(w)\text{ for }w\in [0,n]\}\rightarrow \Prefixes(\hatG)$$
by
\begin{equation}\label{equation:labels_to_prefixes_permutation_flows}
    P(e,w)=\begin{cases}
        x&\text{ if }e=x\in X\\
        P(e',v)e&\text{ if } e=(v,w) \text{ is a split of } e' \text{ at }v\\
        P(e,v)e'&\text{ if } e\neq e'=(v,w) \text{ and } e\in \pi(e').
    \end{cases}
\end{equation}
This function can be extended to give the associated set of routes
$$\Routes(\pi):=\{P(e,n)y\mid e\in \pi(y)\}.$$

\begin{remark}
In practice the prefixes of $\pi\in \PermutationFlows(\hatG,\hatF)$ can be constructed using the recursion of Equation \eqref{equation:labels_to_prefixes_permutation_flows} as follows:

For $v\in [0,n]$ and for each letter $l\in \InPerm(v)$, the prefix of $l$ at $v$ is constructed as follows.
First form the path $P_l$ by tracing backward from $v$ and selecting all edges $e$ such that $l\in \pi(e)$.
If $P_l$ is a prefix of $\hatG$, then $P_l$ is the prefix of $l$ at $v$.
Otherwise, suppose the path $P_l$ has the initial edge $e_l$ with $\tail(e_l)=u$.  
Note that $e_l$ is a split in $\pi$ (for otherwise, $l$ would be in the incoming summary permutation at $u$ and $P_l$ could have been further extended backward), so suppose $l$ is a direct split of $l'$ (at its carrier $e'$).  
Recursively, we define the prefix of $l$ at $v$ to be the concatenation of the prefix of $l'$ at $u$ with the path $P_l$.
We also define for $x\in X$, the prefix of $x$ at $v$ to be $x$.
See Example~\ref{eg.recovering_prefixes}.
\end{remark}

\begin{example}\label{eg.recovering_prefixes}
Let $\pi$ be the permutation flow in Figure~\ref{fig:FSM_subpermuflo}.
We will construct the prefix of the letter $4$ at the vertex $3$.
First, tracing the letter $4$ backward from vertex $3$ gives the path $P_4= t_4$.
Its initial edge $t_4$ is a direct split of the letter $3$ at the vertex $1$, so tracing the letter $3$ backward from vertex $1$ in $\pi$ yields the concatenated path $t_3t_4$.
Its initial edge $t_3$ is a direct split of $x_1$ at the vertex $0$, so finally tracing $x_1$ backwards from vertex $0$ yields the concatenated path $x_1t_3t_4$, which is the prefix of the letter $4$ at vertex $3$.
See Figure~\ref{fig:partial_permuflow_route} where the prefix just constructed is highlighted in black.

The reader can verify that $\Routes(\pi) = \Routes(\calV)$, where $\calV$ is the vineyard in Figure~\ref{fig:FSM_subvines}.
It can also be verified that $\Routes(\pi) \subset \Routes(\rho)$ where $\rho$ is the permutation flow in Figure~\ref{fig:FSM_maxpermuflo}.
\end{example}

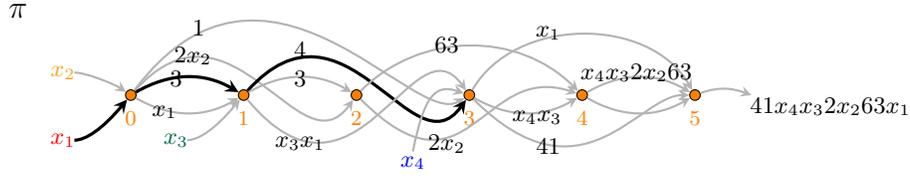
\begin{figure}[ht!]
    \centering\vspace{-.3in}
\begin{tikzpicture}
\begin{scope}[scale=1.5, yscale=1.0]
\node() at (-1,.75) {$\pi$};

\vertex[fill=orange, minimum size=4pt, label=below:{\tiny\textcolor{orange}{$0$}}](v0) at (0,0) {};
\vertex[fill=orange, minimum size=4pt, label=below:{\tiny\textcolor{orange}{$1$}}](v1) at (1,0) {};
\vertex[fill=orange, minimum size=4pt, label=below:{\tiny\textcolor{orange}{$2$}}](v2) at (2,0) {};
\vertex[fill=orange, minimum size=4pt, label=below:{\tiny\textcolor{orange}{$3$}}](v3) at (3,0) {};
\vertex[fill=orange, minimum size=4pt, label=below:{\tiny\textcolor{orange}{$4$}}](v4) at (4,0) {};
\vertex[fill=orange, minimum size=4pt, label=below:{\tiny\textcolor{orange}{$5$}}](v5) at (5,0) {};		

\draw[-stealth, thick, color=black!30, color=black!30] (v0) .. controls (1.2, 1.6) and (2.5, -0.3) .. (v3);
\draw[-stealth, thick, color=black!30] (v0) .. controls (0.9, 1.0) and (1.5, -0.7) .. (v2);
\draw[-stealth, very thick, color=black] (v0) to [out=30,in=150] (v1);
\draw[-stealth, thick, color=black!30] (v0) to [out=-30,in=-150] (v1);

\draw[-stealth, very thick, color=black] (v1) .. controls (1.9, 1.0) and (2.5, -0.7) .. (v3);
\draw[-stealth, thick, color=black!30] (v1) to [out=30,in=150] (v2);
\draw[-stealth, thick, color=black!30] (v1) .. controls (2.0, -1.2) and (2.5, 0.7) .. (v3);	

\draw[-stealth, thick, color=black!30] (v2) to [out=45,in=135] (v4);
\draw[-stealth, thick, color=black!30] (v2) .. controls (3.0, -1.0) and (3.1, 0.3) .. (v4);	

\draw[-stealth, thick, color=black!30] (v3) to [out=60,in=120] (v5);
\draw[-stealth, thick, color=black!30] (v3) to [out=-30,in=-150] (v4);
\draw[-stealth, thick, color=black!30] (v3) .. controls (4.0, -1.0) and (4.5, 0.0) .. (v5);

\draw[-stealth, thick, color=black!30] (v4) to [out=30,in=150] (v5);
\draw[-stealth, thick, color=black!30] (v4) to [out=-30,in=-150] (v5);

\draw[-stealth, thick, color=black!30] (-0.5,0.2) .. controls (-0.4, 0.2) and (-0.15, 0.1) .. (v0);
\draw[-stealth, very thick, color=black] (-0.5, -.4) .. controls (-0.4, -.4) and (-0.25, -.3) .. (v0);
\draw[-stealth, thick, color=black!30] (0.5, -.4) .. controls (0.6, -.4) and (0.75, -.3) .. (v1);
\draw[-stealth, thick, color=black!30] (2.5, -.5) .. controls (2.6, 0) and (2.7, 0.1) .. (v3);
\draw[-stealth, thick, color=black!30] (v5) to [out=20, in=160] (5.5, 0);

\node[] at (-0.6, -0.4){\scriptsize\textcolor{red}{$x_1$}};
\node[] at (-0.6, 0.2){\scriptsize\textcolor{darkyellow}{$x_2$}};
\node[] at (0.4, -0.4){\scriptsize\textcolor{cadmiumgreen}{$x_3$}};
\node[] at (2.5, -0.6){\scriptsize\textcolor{blue}{$x_4$}};
\node[] at (.6, .6){\scriptsize\textcolor{black}{$1$}};
\node[] at (.55, .35){\scriptsize\textcolor{black}{$2x_2$}};
\node[] at (.4, .15){\scriptsize\textcolor{black}{$3$}};
\node[] at (1.5, .4){\scriptsize\textcolor{black}{$4$}};
\node[] at (1.5, .15){\scriptsize\textcolor{black}{$3$}};
\node[] at (2.8, .45){\scriptsize\textcolor{black}{$63$}};
\node[] at (3.7, .55){\scriptsize\textcolor{black}{$x_1$}};
\node[] at (3.6, -0.2){\scriptsize\textcolor{black}{$x_4x_3$}};
\node[] at (4.48, .2){\scriptsize\textcolor{black}{$x_4x_32x_263$}};
\node[] at (.3, -.15){\scriptsize\textcolor{black}{$x_1$}};
\node[] at (1.5, -.43){\scriptsize\textcolor{black}{$x_3x_1$}};
\node[] at (2.8, -.43){\scriptsize\textcolor{black}{$2x_2$}};
\node[] at (3.7, -.45){\scriptsize\textcolor{black}{$41$}};
\node[] at (6.2, -.1){\scriptsize\textcolor{black}{$41x_4x_32x_263x_1$}};
\end{scope}
\end{tikzpicture}
    \vspace{-.3in}
    \caption{Recovering a prefix from a permutation flow.  See Example~\ref{eg.recovering_prefixes}.
     }
    \label{fig:partial_permuflow_route}
\end{figure}

Saturated permutation flows have a number of nice properties. For $\rho\in \SatPermutationFlows$, the first letter of $\pi(t)$ is $t$ for every split $t\in \Splits(\hatG)$.
Moreover, since the first letter of $\OutPerm(v)$ is the first letter of $\InPerm(v)$ for every $v\in [0,n]$, then $\pi(s_v)\neq \emptyset$ for every $v\in [0, n-1]$, and therefore $\supp(\rho) = \hatG$.

Furthermore, the outgoing summary permutation at $n$, which we call the \defn{final summary}~$\zeta$, always has $|X|+d$ letters and is unique for every $\rho\in\SatPermutationFlows(\hatG,\hatF)$, as we prove below. Define $\FinalSummaries(\hatG,\hatF)$ to be the set of these final summaries. 

\begin{proof}[Proof of Theorem \ref{thm:finalsummary}]
It is possible to recover $\rho$ from $\zeta$ by making use of the ordered sets $\inedge(v)$ and $\outedge(v)$. The first step is to identify the first letter of $\rho(e)$ for all $e\in\hatE$ using Definition~\ref{def:permuflow_a}. Condition~(i) implies that $\rho(x)=x$ for all $x\in X$. Condition~(ii) implies that the first letter of $\pi(t)$ is $t$ for all $t\in\Splits(\hatG)$. For every slack $s_v$ where $\inedge(v)=(e_1,\ldots,e_h)$, condition~(iii) implies that the first letter of $\rho(s_v)$ is the first letter of $\rho(e_1)$. Applying this recursive definition from left to right completes the identification of first letters of $\rho$ in the entire graph.

From this you can determine $\rho(e)$ for all $e\in\hatE$ recursively from right to left. From $\OutPerm(v)$, determine the words $\rho(e)$ for $e\in\inedge(v)$ by partitioning the summary $\zeta_v$ into subwords that start with each $e$. 
\end{proof}

\subsubsection{Permutation flow shuffles from grove shuffles}
\label{sec:permuflow_shuffles_from_grove_shuffles}
\phantom{W}

\begin{definition}[Permutation flow shuffle]
A \defn{permutation flow shuffle} of rank $k$ is a pair $(\pi,\sigma_{\pi})$ where $\pi$ is a permutation flow and $\sigma_{\pi}:\Splits(\pi)\rightarrow [k]$ is a surjective function satisfying 
\begin{equation}
\textup{$\sigma_{\pi}(e)<\sigma_{\pi}(e')$ for every pair $e\prec_x e'$ in $\Splits(\pi;x)$ for some $x\in X$.}
\end{equation}
We denote by $\PermutationFlowShuffles(\hatG,\hatF)$ the set of permutation flow shuffles of $(\hatG,\hatF)$. 
\end{definition}

See Figures \ref{fig:FSM_subpermuflo_shuffle} and \ref{fig:FSM_maxpermuflo_shuffle} for two example of permutation flow shuffles.

\begin{remark}
When $\ba=\be_0-\be_n$, there is only one permutation flow shuffle associated to an underlying permutation flow because $\Splits(\pi)$ is totally ordered.    
\end{remark}

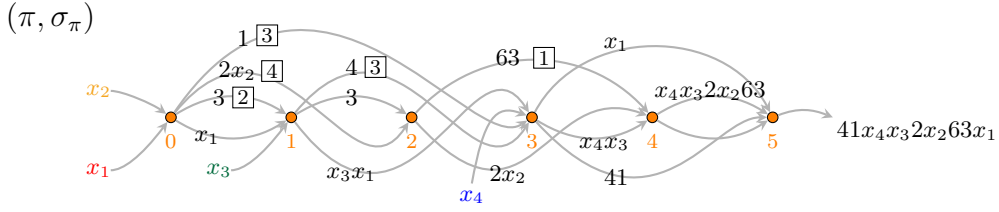
\begin{figure}[ht!]
    \centering\vspace{-.3in}
\begin{tikzpicture}
\begin{scope}[scale=1.6, yscale=1.1]
\node() at (-1,.75) {$(\pi,\sigma_\pi)$};

\vertex[fill=orange, minimum size=4pt, label=below:{\tiny\textcolor{orange}{$0$}}](v0) at (0,0) {};
\vertex[fill=orange, minimum size=4pt, label=below:{\tiny\textcolor{orange}{$1$}}](v1) at (1,0) {};
\vertex[fill=orange, minimum size=4pt, label=below:{\tiny\textcolor{orange}{$2$}}](v2) at (2,0) {};
\vertex[fill=orange, minimum size=4pt, label=below:{\tiny\textcolor{orange}{$3$}}](v3) at (3,0) {};
\vertex[fill=orange, minimum size=4pt, label=below:{\tiny\textcolor{orange}{$4$}}](v4) at (4,0) {};
\vertex[fill=orange, minimum size=4pt, label=below:{\tiny\textcolor{orange}{$5$}}](v5) at (5,0) {};		

\draw[-stealth, thick, color=black!30, color=black!30] (v0) .. controls (1.2, 1.6) and (2.5, -0.3) .. (v3);
\draw[-stealth, thick, color=black!30] (v0) .. controls (0.9, 1.0) and (1.5, -0.7) .. (v2);
\draw[-stealth, thick, color=black!30] (v0) to [out=30,in=150] (v1);
\draw[-stealth, thick, color=black!30] (v0) to [out=-30,in=-150] (v1);

\draw[-stealth, thick, color=black!30] (v1) .. controls (1.9, 1.0) and (2.5, -0.7) .. (v3);
\draw[-stealth, thick, color=black!30] (v1) to [out=30,in=150] (v2);
\draw[-stealth, thick, color=black!30] (v1) .. controls (2.0, -1.2) and (2.5, 0.7) .. (v3);	

\draw[-stealth, thick, color=black!30] (v2) to [out=45,in=135] (v4);
\draw[-stealth, thick, color=black!30] (v2) .. controls (3.0, -1.0) and (3.1, 0.3) .. (v4);	

\draw[-stealth, thick, color=black!30] (v3) to [out=60,in=120] (v5);
\draw[-stealth, thick, color=black!30] (v3) to [out=-30,in=-150] (v4);
\draw[-stealth, thick, color=black!30] (v3) .. controls (4.0, -1.0) and (4.5, 0.0) .. (v5);

\draw[-stealth, thick, color=black!30] (v4) to [out=30,in=150] (v5);
\draw[-stealth, thick, color=black!30] (v4) to [out=-30,in=-150] (v5);

\draw[-stealth, thick, color=black!30] (-0.5,0.2) .. controls (-0.4, 0.2) and (-0.15, 0.1) .. (v0);
\draw[-stealth, thick, color=black!30] (-0.5, -.4) .. controls (-0.4, -.4) and (-0.25, -.3) .. (v0);
\draw[-stealth, thick, color=black!30] (0.5, -.4) .. controls (0.6, -.4) and (0.75, -.3) .. (v1);
\draw[-stealth, thick, color=black!30] (2.5, -.5) .. controls (2.6, 0) and (2.7, 0.1) .. (v3);
\draw[-stealth, thick, color=black!30] (v5) to [out=20, in=160] (5.5, 0);

\node[] at (-0.6, -0.4){\scriptsize\textcolor{red}{$x_1$}};
\node[] at (-0.6, 0.2){\scriptsize\textcolor{darkyellow}{$x_2$}};
\node[] at (0.4, -0.4){\scriptsize\textcolor{cadmiumgreen}{$x_3$}};
\node[] at (2.5, -0.6){\scriptsize\textcolor{blue}{$x_4$}};
\node[] at (.6, .6){\scriptsize\textcolor{black}{$1$}};
    \node[rectangle, fill=white, draw, label=center:{\tiny$3$}] at ($(0,0.62)!0.8!(1,0.62)$){};
\node[] at (.55, .35){\scriptsize\textcolor{black}{$2x_2$}};
    \node[rectangle, fill=white, draw, label=center:{\tiny$4$}] at ($(0,0.35)!0.42!(2,0.35)$){};
\node[] at (.4, .15){\scriptsize\textcolor{black}{$3$}};
    \node[rectangle, fill=white, draw, label=center:{\tiny$2$}] at ($(0,0.15)!0.6!(1,0.15)$){};
\node[] at (1.5, .38){\scriptsize\textcolor{black}{$4$}};
    \node[rectangle, fill=white, draw, label=center:{\tiny$3$}] at ($(1,0.38)!0.7!(2,0.38)$){};
\node[] at (1.5, .15){\scriptsize\textcolor{black}{$3$}};
\node[] at (2.8, .45){\scriptsize\textcolor{black}{$63$}};
    \node[rectangle, fill=white, draw, label=center:{\tiny$1$}] at ($(2,0.45)!0.55!(4,0.45)$){};
\node[] at (3.7, .55){\scriptsize\textcolor{black}{$x_1$}};
\node[] at (3.6, -0.2){\scriptsize\textcolor{black}{$x_4x_3$}};
\node[] at (4.48, .2){\scriptsize\textcolor{black}{$x_4x_32x_263$}};
\node[] at (.3, -.15){\scriptsize\textcolor{black}{$x_1$}};
\node[] at (1.5, -.43){\scriptsize\textcolor{black}{$x_3x_1$}};
\node[] at (2.8, -.43){\scriptsize\textcolor{black}{$2x_2$}};
\node[] at (3.7, -.45){\scriptsize\textcolor{black}{$41$}};
\node[] at (6.2, -.1){\scriptsize\textcolor{black}{$41x_4x_32x_263x_1$}};
\end{scope}
\end{tikzpicture}
    \vspace{-.3in}
    \caption{A permutation flow shuffle $(\pi,\sigma_\pi)$ corresponding with the grove shuffle $(\Gamma, \sigma_\Gamma)$ in Figure~\ref{fig:partial_grove_shuffle}.
The numbers in the boxes define $\sigma_\pi$ on $\Splits(\pi)$.
     }
    \label{fig:FSM_subpermuflo_shuffle}
\end{figure}

\begin{figure}[ht!]
    \centering\vspace{-.3in}
\begin{tikzpicture}
\begin{scope}[scale=1.6, yscale=1.1]
\node() at (-1,.75) {$(\rho,\sigma_\rho)$};

\vertex[fill=orange, minimum size=4pt, label=below:{\tiny\textcolor{orange}{$0$}}](v0) at (0,0) {};
\vertex[fill=orange, minimum size=4pt, label=below:{\tiny\textcolor{orange}{$1$}}](v1) at (1,0) {};
\vertex[fill=orange, minimum size=4pt, label=below:{\tiny\textcolor{orange}{$2$}}](v2) at (2,0) {};
\vertex[fill=orange, minimum size=4pt, label=below:{\tiny\textcolor{orange}{$3$}}](v3) at (3,0) {};
\vertex[fill=orange, minimum size=4pt, label=below:{\tiny\textcolor{orange}{$4$}}](v4) at (4,0) {};
\vertex[fill=orange, minimum size=4pt, label=below:{\tiny\textcolor{orange}{$5$}}](v5) at (5,0) {};		

\draw[-stealth, thick, color=black!30, color=black!30] (v0) .. controls (1.2, 1.6) and (2.5, -0.3) .. (v3);
\draw[-stealth, thick, color=black!30] (v0) .. controls (0.9, 1.0) and (1.5, -0.7) .. (v2);
\draw[-stealth, thick, color=black!30] (v0) to [out=30,in=150] (v1);
\draw[-stealth, thick, color=black!30] (v0) to [out=-30,in=-150] (v1);

\draw[-stealth, thick, color=black!30] (v1) .. controls (1.9, 1.0) and (2.5, -0.7) .. (v3);
\draw[-stealth, thick, color=black!30] (v1) to [out=30,in=150] (v2);
\draw[-stealth, thick, color=black!30] (v1) .. controls (2.0, -1.2) and (2.5, 0.7) .. (v3);	

\draw[-stealth, thick, color=black!30] (v2) to [out=45,in=135] (v4);
\draw[-stealth, thick, color=black!30] (v2) .. controls (3.0, -1.0) and (3.1, 0.3) .. (v4);	

\draw[-stealth, thick, color=black!30] (v3) to [out=60,in=120] (v5);
\draw[-stealth, thick, color=black!30] (v3) to [out=-30,in=-150] (v4);
\draw[-stealth, thick, color=black!30] (v3) .. controls (4.0, -1.0) and (4.5, 0.0) .. (v5);

\draw[-stealth, thick, color=black!30] (v4) to [out=30,in=150] (v5);
\draw[-stealth, thick, color=black!30] (v4) to [out=-30,in=-150] (v5);

\draw[-stealth, thick, color=black!30] (-0.5,0.2) .. controls (-0.4, 0.2) and (-0.15, 0.1) .. (v0);
\draw[-stealth, thick, color=black!30] (-0.5, -.4) .. controls (-0.4, -.4) and (-0.25, -.3) .. (v0);
\draw[-stealth, thick, color=black!30] (0.5, -.4) .. controls (0.6, -.4) and (0.75, -.3) .. (v1);
\draw[-stealth, thick, color=black!30] (2.5, -.5) .. controls (2.6, 0) and (2.7, 0.1) .. (v3);
\draw[-stealth, thick, color=black!30] (v5) to [out=20, in=160] (5.5, 0);

\node[] at (-0.6, -0.4){\scriptsize\textcolor{red}{$x_1$}};
\node[] at (-0.6, 0.2){\scriptsize\textcolor{darkyellow}{$x_2$}};
\node[] at (0.4, -0.4){\scriptsize\textcolor{cadmiumgreen}{$x_3$}};
\node[] at (2.5, -0.6){\scriptsize\textcolor{blue}{$x_4$}};
\node[] at (.6, .6){\scriptsize\textcolor{black}{$1$}};
    \node[rectangle, fill=white, draw, label=center:{\tiny$6$}] at ($(0,0.62)!0.8!(1,0.62)$){};
\node[] at (.55, .35){\scriptsize\textcolor{black}{$2x_2$}};
    \node[rectangle, fill=white, draw, label=center:{\tiny$8$}] at ($(0,0.35)!0.42!(2,0.35)$){};
\node[] at (.4, .15){\scriptsize\textcolor{black}{$3$}};
    \node[rectangle, fill=white, draw, label=center:{\tiny$5$}] at ($(0,0.15)!0.6!(1,0.15)$){};
\node[] at (1.5, .4){\scriptsize\textcolor{black}{$4$}};
    \node[rectangle, fill=white, draw, label=center:{\tiny$7$}] at ($(1,0.38)!0.7!(2,0.38)$){};
\node[] at (1.5, .15){\scriptsize\textcolor{black}{$53$}};
    \node[rectangle, fill=white, draw, label=center:{\tiny$4$}] at ($(1,0.15)!0.75!(2,0.15)$){};
\node[] at (2.8, .45){\scriptsize\textcolor{black}{$653$}};
    \node[rectangle, fill=white, draw, label=center:{\tiny$3$}] at ($(2,0.45)!0.55!(4,0.45)$){};
\node[] at (3.7, .55){\scriptsize\textcolor{black}{$7$}};
    \node[rectangle, fill=white, draw, label=center:{\tiny$1$}] at ($(3,0.55)!0.45!(5,0.55)$){};
\node[] at (3.42, -0.2){\scriptsize\textcolor{black}{$8x_4x_3x_1$}};
    \node[rectangle, fill=white, draw, label=center:{\tiny$9$}] at ($(3,-.2)!0.85!(4,-.2)$){};
\node[] at (4.48, .2){\scriptsize\textcolor{black}{$9x_3x_12x_2653$}};
    \node[rectangle, fill=white, draw, label=center:{\tiny$2$}] at ($(5,0.22)!0.33!(5.5,0.22)$){};
\node[] at (.3, -.15){\scriptsize\textcolor{black}{$x_1$}};
\node[] at (1.5, -.43){\scriptsize\textcolor{black}{$x_3x_1$}};
\node[] at (2.8, -.43){\scriptsize\textcolor{black}{$2x_2$}};
\node[] at (3.7, -.45){\scriptsize\textcolor{black}{$41$}};
\node[] at (4.3, -.1){\scriptsize\textcolor{black}{$8x_4$}};
\node[] at (6.3, -.1){\scriptsize\textcolor{black}{$8x_4419x_3x_12x_26537$}};
\end{scope}
\end{tikzpicture}
    \vspace{-.3in}
    \caption{A permutation flow shuffle $(\rho,\sigma_\rho)$ corresponding with the grove shuffle $(\Theta, \sigma_\Theta)$ in Figure~\ref{fig:max_grove_shuffle}.
The numbers in the boxes define $\sigma_\rho$ on $\Splits(\rho)$.
     }
    \label{fig:FSM_maxpermuflo_shuffle}
\end{figure}
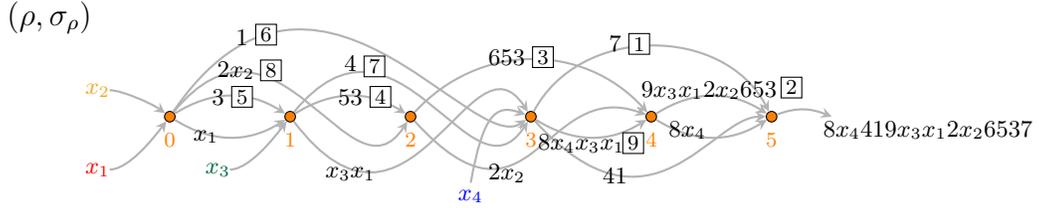

For $S\in \Splits(\pi;x)$, let $\succSplit_\pi(S)$ be the succesor $S^{i+1}$ of a split $S^i\in \Splits(\pi;x)$, with the understanding that if $S$ is the last split in $\Splits(\pi;x)$, then $\sigma_\pi(\succSplit_{\pi}(S)) := \rank(\pi)+1$ for notational convenience.
\begin{definition}
We define $(\pi, \sigma_\pi) \leq (\rho, \sigma_\rho)$ on $\PermutationFlowShuffles(\hatG,\hatF)$ whenever
\begin{itemize}
\item $\pi \subseteq \rho$ and
\item  for splits $S \in \Splits(\pi;x)$ and $S' \in \Splits(\rho;x')$, 
\[\hbox{if } \sigma_\pi(S) \leq 
\sigma_\pi(S') \leq \sigma_\pi(\succSplit_\pi(S))-1, \hbox{ then }
\sigma_\rho(S) \leq \sigma_\rho(S') 
\leq \sigma_\rho(\succSplit_\rho(S))-1.
\]
\end{itemize}
\end{definition}

We denote by $\SatPermutationFlowShuffles(\hatG,\hatF)$ the set of saturated permutation flow shuffles with respect to this partial order.

We can encode a grove shuffle $(\Gamma,\sigma_{\Gamma})$ using a a permutation flow shuffle $(\pi(\Gamma),\sigma_{\pi(\Gamma)})$ where  $\sigma_{\pi(\Gamma)}(t)=\sigma_{\Gamma}((P,t))$ for every $(P,t)\in \Splits(\Gamma)$.

\begin{theorem}
The map
\[(\pi,\textup{id}) : \GroveShuffles( \hatG, \hatF) \rightarrow \PermutationFlowShuffles( \hatG, \hatF )\]
is a poset isomorphism.    
\end{theorem}

In Figures \ref{fig:FSM_subpermuflo_shuffle} and \ref{fig:FSM_maxpermuflo_shuffle} we have the permutation flow shuffles which are images of the grove shuffles in Figures \ref{fig:partial_grove_shuffle} and \ref{fig:max_grove_shuffle}.

In Figure \ref{fig:permutation_flow_shuffles_PS_111} we list the $16$ permutation flow shuffles corresponding to $PS_3(1,1,1,-3)$ with a the given planar framing.

\begin{figure}
    \centering
\includegraphics{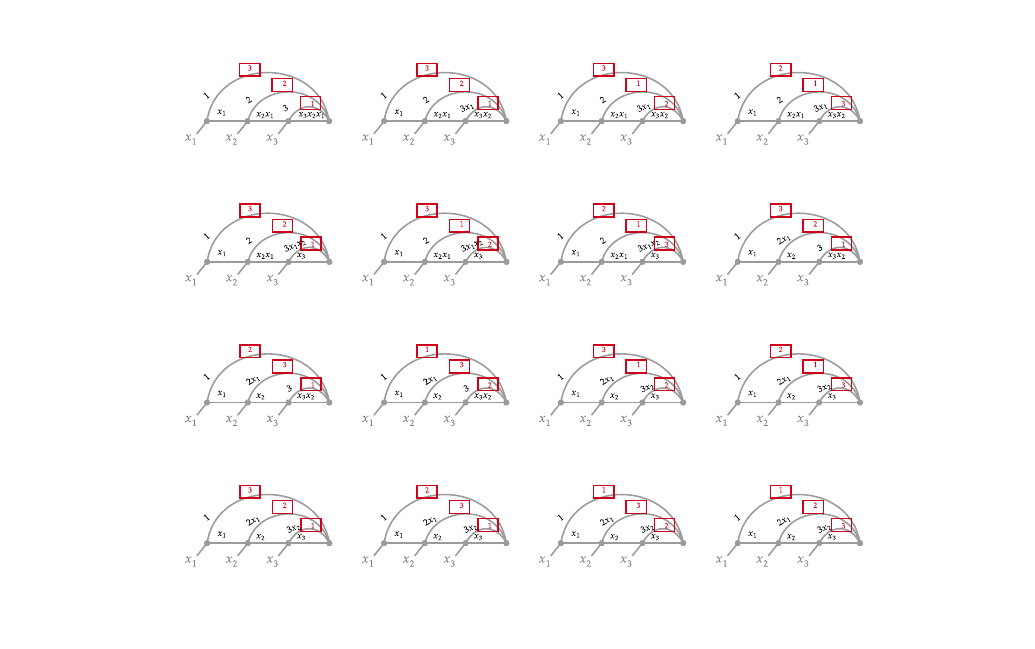}   
    \caption{Permutation flow shuffles for $PS_3$ with $\ba=(1,1,1,-3)$ and the given planar framing}
    \label{fig:permutation_flow_shuffles_PS_111}
\end{figure}

\clearpage 
\section{Permutation flow triangulations on a positive netflow}
\label{sec:triangulation_flow_polytopes}

In this section we prove Theorem \ref{theorem:triangulation} which gives a unimodular triangulation of the flow polytope $\calF_G(\ba)$.
This triangulation generalizes the triangulation of Danilov, Karzanov, and Koshevoy in \cite{DanilovKarzanovKoshevoy2012} for the special case with netflow $\ba=\be_0- \be_n$, and hence as a special case we provide a new proof of the triangulation in \cite{DanilovKarzanovKoshevoy2012}. 

\subsection{The triangulation}
Recall from Proposition \ref{proposition:matchings_are_integer_flows} 
that the vertices of $\calF_G(\ba)$ are in correspondence with route matchings. Recall that for a route matching $\calP \in \RouteMatchings(\hatG,\hatF)$ we denote by $\bz(\calP)\in \calF_G(\ba)\cap \bbZ^E$ its indicator flow, which provides the desired bijection.

For a set $\calC \subseteq \RouteMatchings(\hatG,\hatF)$ denote by $\triangle_{\calC}$ the convex hull
\[\triangle_{\calC}:=\conv\{\bz(\calP)\mid \calP \in \calC\}.
\]

We will show that the set formed by the $\triangle_\calC$ for $\calC\in\SatCliques(\hatG,\hatF)$ defines a triangulation of $\calF_G(\ba)$. The proof will show that every point $\bw\in \calF_G(\ba)$ can be represented uniquely as a positive convex combination of indicator flows
\begin{equation}\label{eq:convex_combination_clique}
\bw = \alpha_0 \bz(\calP^0)+\dots + \alpha_k\bz(\calP^k),\end{equation}
for a clique $\calC=\{\calP^0,\dots,\calP^k\}$ and positive coefficients satisfying $\alpha_0+\dots + \alpha_k=1$.  

Suppose that $\bw$, $\calC$, and $\alpha_i$ are as in Equation~\eqref{eq:convex_combination_clique} and let $\Gamma$ be the grove that corresponds to $\calC$. Viewing $\bw$ as a function from edges $e\in E(G)$ to $\bbR$, we now extend $\bw$ to a function from prefixes $P\in\Prefixes(\hatG)$ to $\bbR$ by defining it on vertices and edges of $\Gamma$ as follows.

In every forest $\gamma_v\in\Gamma$, we will define $\bw(P)$, $\bw(e)$, and $\bw((P,e))$. For right vertices $e\ne y$, $\bw(e)$ is inherited directly from $\bw$ and we define $\bw(y)=|a_n|$. For a prefix $P$, define \[\bw(P):=\sum_{i=0}^k \alpha_i\delta ( \calP^i\cap\Routes(\Gamma;P)\neq \emptyset),
\]
where $\Routes(\Gamma;P)$ is the set of routes in $\Gamma$ which contain the prefix $P$.
Lastly for a grove edge $(P,e)$, define $\bw((P,e))=\bw(Pe)$. 

Setting $e=y$, we obtain the following definition of $\bw(R)$ whenever $R$ is 
a route in any of the route matchings $\calP^i$ for $i\in [0,k]$,
\begin{align}\label{eqn:bw_on_routes}
\bw(R)=\sum_{i=0}^{k}\alpha_i\delta(R\in \calP^i).   
\end{align}

Recall that given a grove $\Gamma=\{\gamma_v\}_{v\in [0,d]}$, if $P$ is a left vertex and $e$ is a right vertex of a bipartite noncrossing tree $\gamma_v$, then $I(P)$ is the set of edges in $\gamma_v$ incident with $P$ and $I(e)$ is the set of edges in $\gamma_v$ incident with $e$.
We have the following two lemmas.

\begin{lemma}\label{lemma:defining_equations}
The extension of $\bw$ to $\Gamma$ satisfies the defining equations
\begin{align}
    \bw(x)
        &=1 
        &\text{for all } x\in X \label{eqn:defining1}\\
    \bw(P)
        &=\sum_{(P,e)\in I(P)}\bw(P,e) 
        &\text{for all }P\in \Prefixes(\Gamma)\label{eqn:defining2} \\
    \bw(e)
        &=\sum_{(P,e)\in I(e)}\bw(P,e) 
        &\text{for all }e\in E(\Gamma)\cup Y\label{eqn:defining3}\\
    \bw((P,e))
        &=\bw(Pe) 
        &\text{for all }(P,e) \in E(\Gamma)\textup{ with }e\neq y.\label{eqn:defining4}
\end{align}
\end{lemma}

\begin{proof}
For all $x\in X$, there exists exactly one route in a route matching that extends $x$, so
\[\bw(x)=\sum_{i=0}^k \alpha_i\delta ( \calP^i\cap\Extensions(x)\neq \emptyset)=\sum_{i=0}^k \alpha_i=1.\]
Furthermore, for all prefixes $P$, 
\begin{align*}
    \bw(P) &=\sum_{i=0}^k \alpha_i\delta ( \calP^i\cap\Routes(\Gamma;P)\neq \emptyset)\\    
    &=\sum_{i=0}^k \alpha_i\delta\left( \calP^i\cap\bigsqcup_{\substack{(P,e)\in I(P)}}\Extensions(Pe)\right)\\    
    &=\sum_{i=0}^k\alpha_i\sum_{(P,e)\in I(P)} \delta(\calP^i \cap \Extensions(Pe)) \\
    &=\sum_{(P,e)\in I(P)}\sum_{i=0}^k \alpha_i\delta(\calP^i \cap \Extensions(Pe)) \\
    &=\sum_{(P,e)\in I(P)}\bw(P,e).
\end{align*}
Finally, from Equation~\eqref{eq:convex_combination_clique},
\begin{align*}
    \bw(e)&=\sum_{i=0}^k\alpha_i|\{R\in \calP^i\mid e \in R\}|\\
    &=\sum_{i=0}^k\alpha_i|\{Q\in \Prefixes(\Gamma) \mid \terminal(Q)=e \text{ and } R\in \calP^i \cap \Extensions(Q) \}|\\
    &=\sum_{Pe \in \Prefixes(\Gamma)}\sum_{i=0}^k\alpha_i\delta(\calP^i\cap \Extensions(Pe)\neq \emptyset)\\
    &=\sum_{(P,e)\in I(e)}\bw(Pe)\\
     &=\sum_{(P,e)\in I(e)}\bw(P,e).
\end{align*}
That these equations are defining follows from using this set of equations inductively from vertex $v=0$ to $v=n$ of $\hatG$. 
\end{proof}

\begin{lemma} \label{lemma:coefficients}
We have that 
\begin{gather*}
\alpha_0=\min\left(\left\{\bw\left(P_x^0\right) \mid P_x^0 \in \calP^0\right\}\right)\label{eqn:defining5}\\
P^i_x\neq P_x^ {i+1} \text{ if and only if } \sum_{\substack{R_x\in \calP^j_x\\ \text{ for some }j\le i}}\bw(R_x)=\sum_{j=0}^{i}\alpha_j \label{eqn:defining6}
\end{gather*}
\end{lemma}
\begin{proof}
    These follow directly from Equations \eqref{eq:convex_combination_clique} and \eqref{eqn:bw_on_routes}.
\end{proof}

\begin{proposition}\label{proposition:triangulation}
	Let $(\hatG,\hatF)$ be a framed $\ba$-augmented graph of $G$. The set
    $$\triangle(\hatG,\hatF)=\{\triangle_{\calC}\mid \calC \in \SatCliques(\hatG,\hatF)\}$$
    is a triangulation of $\calF_G(\ba)$.
\end{proposition} 

\begin{proof}
Starting with a point $\bw=(\bw(e))_{e\in E}\in \calF_G(\ba)$ we build a sequence $\Gamma=(\gamma_0,\hdots,\gamma_n)$ of non-crossing bipartite forests. In each forest $\gamma_i$, all left vertices $P$, right vertices $e$, and edges $(P,e)$ will be given a weight $\bw(P)$, $\bw(e)$, or $\bw(P,e)$ satisfying the defining equations \eqref{eqn:defining1}, \eqref{eqn:defining2}, \eqref{eqn:defining3}, and \eqref{eqn:defining4} of Lemma \ref{lemma:defining_equations}. 
    
Start at vertex $v=0$; suppose there are $a_0$ inflow half-edges to $v$, $\hat{\inedge}(v)=(x_1,\ldots,x_{a_0})$, and $l$ outgoing edges from $v$, $\hat{\outedge}(v)=\{e_1',\ldots,e_l'\}$, both ordered by $\hatF$. We will create a non-crossing bipartite forest $\gamma_0$ assuming that to each $x_i$ we have assigned the weight $\bw(x_i)=1$ according to \eqref{eqn:defining1} and for each $e_j'$ the weight $\bw(e_j')$ coming from the original point~$\bw$.     Equation~\eqref{eqn:defining_equations} ensures that the sums of the weights of the left vertices and of the right vertices are equal. To determine the set of edges of $\gamma_0$ which makes it possible to satisfy Equations \eqref{eqn:defining2} and \eqref{eqn:defining3}, we have to consider the list of \defn{left intervals} $I_1$ through $I_{a_0}$ \[\left(\big[0,\bw(x_1)\big],\big[\bw(x_1),\bw(x_1)+\bw(x_2)\big],\ldots,\left[\sum_{i=1}^{a_0-1}\bw(x_i),\sum_{i=1}^{a_0}\bw(x_i)\right]\right)\] and the list of \defn{right intervals} $J_1$ through $J_l$
\[\left(\big[0,\bw(e_1')\big],\big[\bw(e_1'),\bw(e_1')+\bw(e_2')\big],\ldots,\left[\sum_{j=1}^{l-1}\bw(e_j'),\sum_{j=1}^{l}\bw(e_j')\right]\right),\]
determined respectively by the contributions of the weights of the left and right vertices of $\gamma_0$ in their respective orders.

There is a unique way to satisfy \eqref{eqn:defining2} and \eqref{eqn:defining3} while preserving the noncrossing condition on $\gamma_0$: we have to connect vertex $x_i$ to vertex $e_j'$ by the edge $(x_i,e_j')$ when the sets $I_i$ and $J_j$ have an intersection of positive measure and give $(x_i,e_j')$ the weight $\bw((x_i,e_j'))$ equal to the measure of this intersection. We then, according to Definition \ref{def:grove}, create a new prefix $x_ie_j'$ as a left vertex in $\gamma_{\head(e_j')}$ with $\bw(x_ie_j')=\bw((x_i,e_j'))$ by \eqref{eqn:defining4}. Every left vertex will connect to at least one right vertex. Remove any isolated right vertices, which correspond to edges~$e$ such that $\bw(e)=0$.

At a vertex $v>0$, build the non-crossing forest $\gamma_v$ as follows. Suppose $v$ has incoming incidences $\hat\inedge(v)=\{e_1,\dots,e_r\}$, incoming prefixes $\Prefixes(\Gamma, v)=(P_1,\ldots,P_h)$ and outgoing incidences $\hat{\outedge}(v)=\{e_1',\ldots,e_l'\}$. (At vertex $v=n$, we have $\hat\outedge(v)=\{y\}$ and $\bw(y)=|a_n|$) Equations \eqref{eqn:defining_equations} and  \eqref{eqn:defining3} imply
    \begin{align*}
\sum_{i=1}^h\bw(P_i)&=\sum_{s=1}^r\sum_{\substack{i\in [h]\\\terminal(P_i)=e_s}}\bw(P_i)\\
    &=\sum_{s=1}^r\bw(e_s)\\
    &=\sum_{j=1}^l\bw(e_j')
    \end{align*}

The same procedure as above involving intervals is followed to determine that the edges of~$\gamma_v$ are of the form $(P_i,e_j')$  when the sets $I_i$ and $J_j$ have an intersection of positive measure, giving $(P_i,e_j')$ the weight $\bw((P_i,e_j'))$ equal to the measure of the intersection and creating a new prefix $P_ie_j'$ as a left vertex in $\gamma_{\head(e_j')}$ with $\bw(x_ie_j')=\bw((x_i,e_j'))$.

Again, here it is important to notice that the noncrossing condition on $\gamma_v$ implies that the set of edges $(P_i,e_j')$ and weights $\bw((P_i,e_j'))$ are the unique configuration that will satisfy \eqref{eqn:defining2} and \eqref{eqn:defining3}.
    
From $\Gamma$, we then consider the set of routes  $\calR=\{Py\mid P\in L(\gamma_n)\}=\Routes(\calV(\Gamma))$ and define $\bw(Py):=\bw((P,y))$ for $Py \in \calR$. By the bijective correspondence $\Groves(\hatG,\hatF)\cong\Vineyards(\hatG,\hatF)$ of Proposition \ref{proposition:VineyardstoGroves} we have that $\calR$ is coherent. We further partition $\calR$ into the totally ordered sets $\calR_x$ depending on the inflow half-edge $x$ where each route $R$ starts. This construction also ensures inductively that $\sum_{R\in\calR_x}\bw(R)=1$ and therefore $|\calR_x|\geq 1$ for all $x \in X$.
    
From $\calR$ we will define a sequence of route matchings $\calC=\{\calP^0, \calP^1, \ldots, \calP^k\}$, and a weight~$\alpha_i$ assigned to each one. Define $\calP^0=\{P_x^0 \mid x \in X\}$ where $P_x^0$ is the lowest route in $\calR_x$ for all $x \in X$. Note that $|\calP^0|=|X|$. Find the minimum value $\alpha^0=\min({\bw(P_x^0) \mid P_x^0\in\calP^0})$ and assign $\alpha_0=\alpha^0$. 
    
Create the next route matching $\calP^1$ by removing from $\calP^0$ every route $P^0_x$ where $\bw(P^0_x)-\alpha^0=0$ and replacing it by the next higher route in $\calR_x$. Continuing in this fashion, having created the route matching $\calP^i$ we create $\calP^{i+1}$ by removing from $\calP^i$ every route $P^i_x$ where $\sum_{\substack{R_x\in \calR_x\\R_x\preceq_x P^i_x}}\bw(R_x)-\sum_{j=0}^{i}\alpha^j=0$ and replacing it by the next higher route in $\calR_x$. The process finalizes then with a sequence of coherent route matchings $\calC=\{\calP^0, \calP^1, \ldots, \calP^k\}$ such that by construction, $\sum_{i=1}^k \alpha_i\bz(\calP^i)=\bw$, as desired. 
    
We remark here that the maximum number of routes that appear in $\calR$ is when each non-crossing forest $\gamma_v$ is a tree that spans all left vertices and all right vertices. In this situation, the number of distinct routes in $\calR$ is exactly $m-n+|X|$ because there are $|X|$ total inflow half-edges and the number of prefixes at vertex $v$ increases by the number of outgoing edges from $v$ minus one at every vertex $v$. 

Assume there is another $\calD=\{\calQ^0,\calQ^1,\ldots,\calQ^l\}$ and coefficients $\beta_j$ such that $\sum_{j=1}^l \beta_j\bz(\calQ^j)=\bw$. By Lemma \ref{lemma:defining_equations} we have that the extension of $\bw$ satisfies the Equations~\eqref{eqn:defining1}--\eqref{eqn:defining4}.
Since the procedure outlined above shows that at every step there is a unique grove $\Gamma$ that can be build to satisfy these equations, we conclude that $\calD=\calC$. Then
Lemma \ref{lemma:coefficients} and the procedure used to construct the values  $\alpha_i$ imply that $\alpha_i=\beta(\calP^i)$ for $i=0,\dots,k$.
\end{proof}

\begin{lemma}
\label{lem:projection_integral_equiv}
    The natural projection $\bbR^E\rightarrow \bbR^{\Splits(\hatG)}$ is an integral equivalence between $\calF_G(\ba)$ and its image.
\end{lemma}

\begin{proposition}\label{proposition:unimodularity}
	For every $\calC \in \SatCliques(\hatG,\hatF)$, $\triangle_{\calC}$ 
    is a unimodular simplex.
\end{proposition} 
\begin{proof}
Using Lemma~\ref{lem:projection_integral_equiv}, we will show that for every simplex $\triangle_{\calC}$, the direction vectors $\bz(\calP^i)-\bz(\calP^0)$ projected onto $\bbR^{\Splits(\hatG)}$ form an integral basis of $\bbZ^{\Splits(\hatG)}$. For this, create the $d\times d$ matrix whose rows consist of the difference of indicator flows $\bz(\calP^i)-\bz(\calP^0)$ for $i\in[d]$ and whose columns correspond in order to the edges $\{t_1,\ldots,t_d\}$. 

Apply the matrix transformation to $A$ that replaces row $i$ by row $i$ minus row $i-1$ for $i\in[2,d]$ to form the matrix $A'$. The determinants $\det A$ and $\det A'$ are equal and the rows of $A'$ consist in the difference of indicator flows $\bz(\calP^i)-\bz(\calP^{i-1})$ for $i\in[d]$.

Since $\calC$ is a saturated clique, two consecutive route matchings $\calP^{i-1}$ and $\calP^i$ differ by exactly one route ($P\in\calP^{i-1}$ and $P'\in\calP^i$) that share a maximal prefix. At the first place where they diverge, $P$ follows an edge $e$  and $P'$ follows an edge $e'$. The edge $e'$ is a split (and therefore of the form $t_k$ for some $k$) and by the canonical edge ordering in $G$, the multiplicity of every edge $t_j$ for $j<k$ is the same in $\calP^{i-1}$ as it is in $\calP^i$. Therefore row $i$ of $A'$ consists of zeroes in columns $1$ through $k-1$, a $\pm1$ in column $k$, and some collection of integers in later columns.

Lemma~\ref{lemma:max_grove_bijection} establishes that each split occurs exactly once in $\Gamma(\calC)$. As a consequence, the rows of $A'$ can be permuted to form an upper triangular matrix with all diagonal entries equal to $\pm 1$. The determinant of this matrix is $\pm 1$, completing the proof. 
\end{proof}

Notice that this triangulation depends on the particular framing $\hatF$ of the augmented graph $\hatG$ of $G$. As such, this construction defines many different unimodular triangulations of $\calF_G(\ba)$.

\begin{example}
We illustrate Theorem \ref{proposition:triangulation} with an example of the triangulation of Figure~\ref{fig:PS_polytope_with_route_matchings_versionM} with the planar framing that is illustrated. For example the simplex of corresponding to the permutation flow shuffle in Figure \ref{fig:simplex_triangulation_PS_111} is constructed as follows: we start with the route matching $$\calP^0=\{P_{x_1}, P_{x_2}, P_{x_3}\}=\{x_1s_0s_1t_3, x_2s_1t_3, x_3s_2\}$$ indexed by $x_1$, $x_2$, and $x_3$. Following the instructions in the shuffle the next vertex changes $P_{x_1}$ to $P_{t_2}$, so $$\calP^1=\{P_{t_2}, P_{x_2}, P_{x_3}\}=\{x_1s_0t_2, x_2s_1t_3, x_3s_2\},$$then  $P_{x_3}$ switches to $P_{t_3}$ to get $$\calP^2=\{P_{t_2}, P_{x_2}, P_{t_3}\}=\{x_1s_0t_2, x_2s_1t_3, x_3t_3\},$$and finally $P_{t_2}$ switches to $P_{t_1}$ to get $$\calP^3=\{P_{t_1}, P_{x_2}, P_{t_3}\}=\{x_1t_1, x_2s_1t_3, x_3t_3\}.$$
Following the same procedure with the $16$ permutation flow shuffles that were given in Figure~\ref{fig:permutation_flow_shuffles_PS_111} we obtain the triangulation in Figure \ref{fig:triangulation_PS_111} whose dual graph is illustrated.
\end{example}

\begin{figure}
    \centering
\resizebox{.6\textwidth}{!}{  
\includegraphics{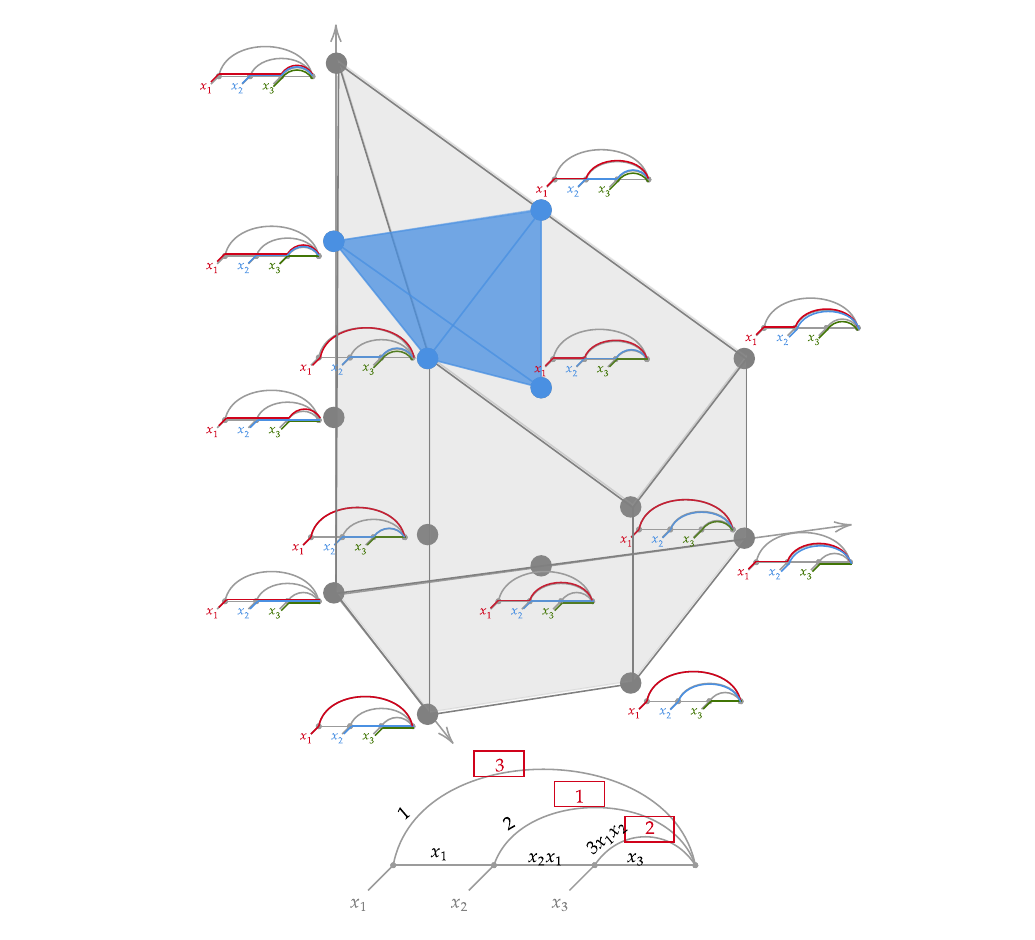}}
    \caption{Permutation flow shuffle and its associated simplex in the triangulation of $\calF_{\PS_3}(1,1,1,-3)$ with the given framing.}
    \label{fig:simplex_triangulation_PS_111}
\end{figure}

\begin{figure}
    \centering
    
\resizebox{\textwidth}{!}{  
\includegraphics{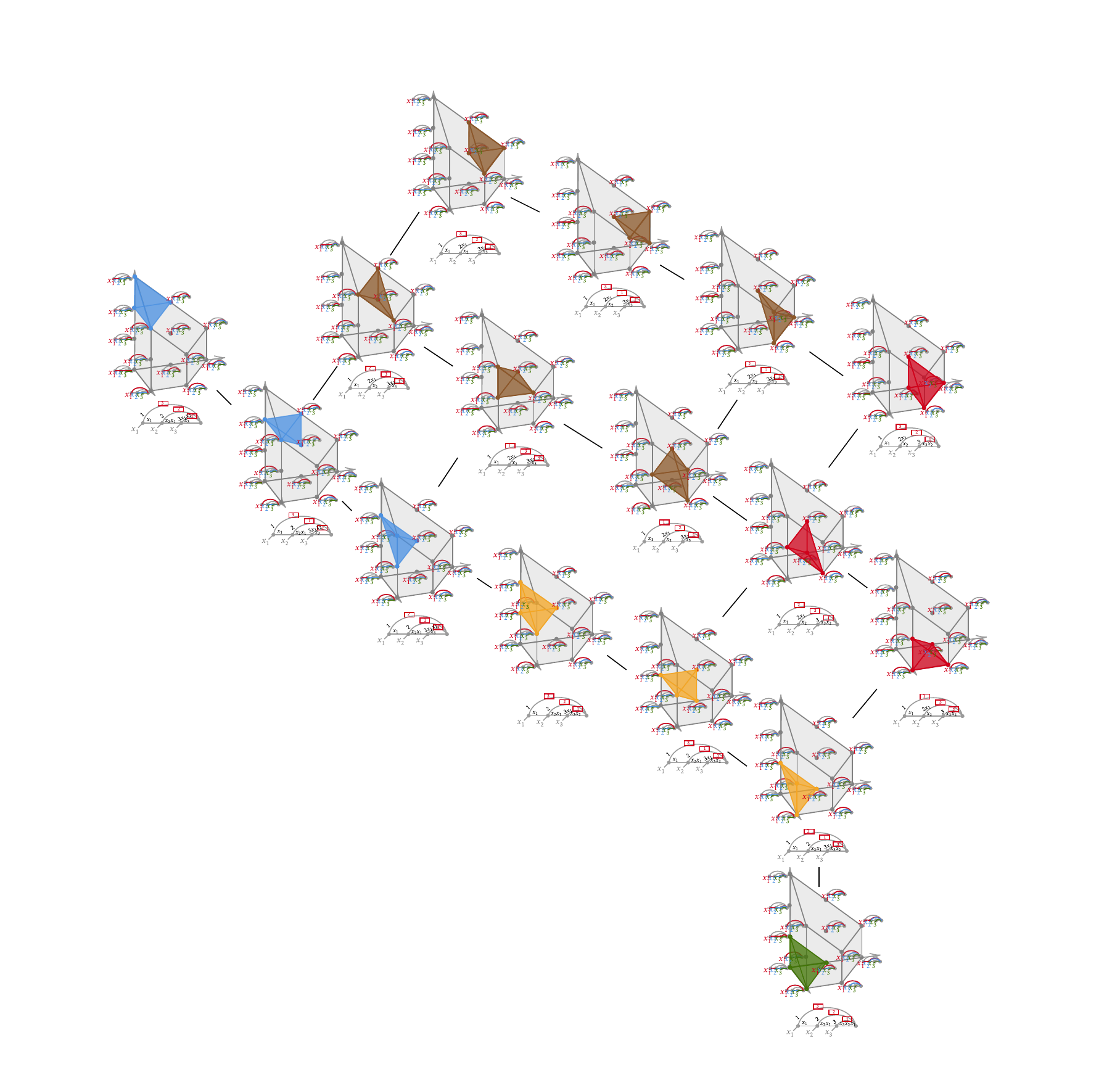}}
    \caption{Permutation flow triangulation of $\calF_{\PS_3}(1,1,1,-3)$ with the given framing. Simplices of the same color have the same underlying permutation flow.}
    \label{fig:triangulation_PS_111}
\end{figure}

\clearpage
\section{A new proof of a Lidskii volume formula of Baldoni-Vergne}
\label{sec:Lidskii}

Baldoni and Vergne \cite[Theorem 38]{BaldoniVergne2008} prove a generalization of Lidskii's volume formula~\cite{Lidskii1984} for flow polytopes (Theorem \ref{thm.genlidskii}), which is calculated explicitly when $\ba$ has a unique negative entry $a_n<0$. The technique used in \cite{BaldoniVergne2008} involves residue computations that hid the combinatorial nature of the resulting formula. A proof of the formula was given by M\'esz\'aros and Morales in \cite{MeszarosMorales2019} using a geometric decomposition which generalizes an (unpublished) decomposition method by Postnikov and Stanley. Postnikov and Stanley also observed that a proof of the generalized Lidskii formula can be obtained using the Elliott–MacMahon algorithm.

As an application of Theorem \ref{theorem:triangulation} we provide a new independent proof of the Lidskii volume formula. Our theorem gives an intermediate restatement of Theorem \ref{thm.genlidskii} which appeals to the structure of the combinatorial objects involved. This new proof relies on a triangulation that refines the subdivision used in \cite{MeszarosMorales2019} and hence provides an interpretation of each term of the formula in terms of an enumeration of a set of simplices.

\subsection{The outdegree formula}
\label{sec:outdegree_formula}
\phantom{W}

Given $\Gamma \in  \Groves(\hatG,\hatF)$ recall, from Section \ref{sec:grove_shuffles}, that there is a set valued function \[\Splits:\Prefixes(\Gamma)\cup E(\Gamma)\rightarrow 2^{\Splits(\Gamma)}.\] We can associate the integer valued function 
\[\bj(\Gamma):\Prefixes(\Gamma)\cup E(\Gamma) \rightarrow \bbZ_{\ge 0}\]
defined for every $P\in \Prefixes(\Gamma)$ and every $(P,e)\in E(\Gamma)$ as 
$$\bj(\Gamma)(P):=|\Splits(\Gamma;P)| \text{ and }\bj(\Gamma)((P,e)):=|\Splits(\Gamma;(P,e))|.$$
In particular, if we only consider the elements $x\in X$ we obtain a weak composition 
$$\bj(\Gamma)=\big(\bj(\Gamma)(x)\big)_{x\in X}$$ of $d$ into $|X|$ parts.
In addition, to every weak composition $\bj:X\rightarrow \bbZ_{\ge 0}$ we will associate its coarsening $\hat \bj:[0,n]\rightarrow \bbZ_{\ge 0}$ defined by
$$\hat \bj_v=\sum_{\substack{x\in X\\v_x=v}}\bj(x),$$
which is then a weak composition of $d$ into $n$ parts.

We aim to determine the cardinality of $\SatGroveShuffles(\hatG,\hatF)$. First consider a fixed $\Gamma \in \SatGroves(\hatG,\hatF)$ and note that there are $\binom{d}{\bj(\Gamma)}$ shuffles $\sigma_{\Gamma}$ of type $\bj(\Gamma)$ such that $(\Gamma,\sigma_{\Gamma}) \in  \SatGroveShuffles(\hatG,\hatF)$. Now for a fixed weak composition~$\bj$ of $d$ with parts indexed by $X$, we need to determine the number of $\Gamma \in \SatGroves(\hatG,\hatF)$ such that $\bj(\Gamma)=\bj$.  

For this we extend the function $\Splits(\Gamma)$ to the set $E\cup \{y\}$ of right vertices of $\Gamma$ as follows. Let $v\in[0,n]$ and $e\in R(\gamma_v)$ where $I(e)$ is the set of incident grove edges to $e$,
\begin{align*}
 I(e):&=\{(P,e)\in E(\gamma_{\tail(e)})\mid P\in L(\gamma_{\tail(e)})\}\\
 &=\{(P_0,e),(P_1,e),\dots,(P_l,e)\}.
\end{align*}
We define $\bj(\Gamma)(e):=\lvert\Splits(\Gamma;e)\rvert$ where
$$\Splits(\Gamma;e):=\bigsqcup_{i=0}^l\Splits(\Gamma;(P_i,e)).$$

\begin{lemma}\label{lemma:bijection}
For every fixed weak composition $\bj$ of $d$ with parts in $X$,
the map 
\begin{equation}\label{equation:map_j_on_SatGroves}
   \bj:\left\{\Gamma \in \SatGroves(\hatG,\hatF)\mid\bj(\Gamma)(x)=\bj(x) \textup{ for all } x\in X \right\}\rightarrow \calF_G^{\bbZ}(\hat \bj-\bo) 
\end{equation}
that associates to every $\Gamma \in \SatGroves(\hatG,\hatF)$ the integer flow $\big(\bj(\Gamma)(e)\big)_{e\in E(G)}$ is a bijection.
\end{lemma}
\begin{proof}
We first show that $\bj(\Gamma)\in \calF_G^{\bbZ}(\hat \bj(\Gamma)-\bo)$ when $\bj(\Gamma)$ is restricted to $E$. Indeed, for every vertex $v\in [0,n]$ we have
that 
\begin{align*}
\sum_{e'\in  \outedge(v)} \bj(\Gamma)(e')-\sum_{e\in  \inedge(v)} \bj(\Gamma)(e)&=\sum_{e'\in  \outedge(v)} \bj(\Gamma)(e')-\sum_{e\in  \inedge(v)}\sum_{(P,e)\in  E(\gamma_{\tail(e)})} \bj(\Gamma)((P,e))\\
&=\sum_{e'\in  \outedge(v)} \bj(\Gamma)(e')-\sum_{e\in  \inedge(v)} \sum_{(P,e)\in  E(\gamma_{\tail(e)}) }\bj(\Gamma)(Pe)\\
\intertext{The rightmost term is the sum over all prefixes that are not elements of $X$, so }
&=\sum_{e'\in  \outedge(v)} \bj(\Gamma)(e')-\sum_{P'\in L(\gamma_{v})\setminus X}\bj(\Gamma)(P')\\
\intertext{Because the set of splits of all incoming prefixes and the splits of all incoming half-edges is the same as the set of splits of all outgoing edges and the set of splits that occur at $v$,}
&=\sum_{\substack{x\in X\\v_{x}=v}}|\Splits(\Gamma;x)|-|\{j \in [m-n] \mid \tail(t_j)=v \}|   \\
&=\sum_{\substack{x\in X\\v_{x}=v}}\bj(x) - o_v\\
&=\hat \bj_v-o_v.
\end{align*}

The map $\bj(\Gamma)$ of Equation \eqref{equation:map_j_on_SatGroves} has an inverse function that to every $\by \in \calF_G^{\bbZ}(\hat \bj-\bo)$ associates a graph $\Gamma(\by)$ such that $\bj(\Gamma(\by))=\bj$ on $X$ and $\bj(\Gamma(\by))=\by$ on $E$.
The proof of this follows a similar argument to the one in the proof of Proposition \ref{proposition:triangulation} showing, inductively on vertices $v$ from $0$ to $n$, that there is a unique $\Gamma$ satisfying both of these conditions.
\end{proof}

\begin{remark}
  The special case of Lemma \ref{lemma:bijection} when $\ba=\be_0- \be_n$ gives a new bijection,  $\bj\circ \Gamma:\DKK(G,F)\rightarrow \calF_G^{\bbZ}(\tilde\bd)$. This bijection is different from the one in \cite{MeszarosMorales2019}.
\end{remark}

\begin{proof}[Proof of Theorem \ref{theorem:reformulated_lidskii_volume}]
By Theorem \ref{theorem:triangulation} we have that
\begin{align*}
    \vol \calF_G(a)&=|\SatGroveShuffles(\hatG,\hatF)|\\
    &=\sum_{(\Gamma,\sigma)\in \SatGroveShuffles(\hatG,\hatF)}1\\
     &=\sum_{\bj}\sum_{ \substack{(\Gamma,\sigma)\in\SatGroveShuffles(\hatG,\hatF)\\\bj(\Gamma)=\bj}}1\\
     &=\sum_{\bj}\sum_{ \substack{\Gamma\in\SatGroves(\hatG,\hatF)\\\bj(\Gamma)=\bj}}\binom{d}{\bj(\Gamma)}\\
     &=\sum_{\bj}\binom{d}{\bj}K_G(\hat \bj-\bo). \qedhere
\end{align*}
\end{proof}

Theorem \ref{thm.genlidskii} is then a direct corollary of Theorem \ref{theorem:reformulated_lidskii_volume}.
\begin{proof}[Proof of Theorem \ref{thm.genlidskii}]
Let $\bs$ be a fixed weak composition of $d$ in $n$ parts and $\bj$ indicate weak compositions of $d$ in $|X|$ parts. We establish that
$$\sum_{\substack{\bj\\\hat \bj=\bs}}\binom{d}{\bj}=\binom{d}{\bs}\ba^{\bs}.$$
Indeed, on the left we enumerate  $\bj$-shuffles satisfying  $\hat \bj=\bs$. We can also obtain such a shuffle by first selecting a $\bs$-shuffle and then deciding to associate each of the $s_v$ numbers to an element of the set $\{i \mid v_{x_i}=v\}$ for every $v\in [0,n-1]$, which is enumerated by the right hand side of the equation.
\end{proof}

\subsection{Enumerating unshuffled objects}
\label{enumerating_unshuffled_objects}
\phantom{W}

We are going to determine the cardinality of saturated unshuffled objects in $(\hatG,\hatF)$ by applying a counting argument in two different ways.

\begin{corollary}\label{corollary:number_of_permutation_flows} The number of saturated (unshuffled) objects determined by $(\hatG,\hatF)$, i.e., any of the cardinalities $$|\SatVineyards(\hatG,\hatF)|=|\SatGroves(\hatG,\hatF)|=|\SatPermutationFlows(\hatG,\hatF)|,$$ is given by any one of the following
\begin{enumerate}[label=(\alph*)]
    \item $\displaystyle\sum_{\substack{\bj\\\hat\bj \triangleright \bo}}K_G(\hat \bj-\bo).$
    \item $K_{G^*}(d,-o_0,-o_1,\dots,-o_{n-1},0)$.
    \item $\displaystyle\vol \calF_{G^{\star}}(\be_0- \be_n),$
\end{enumerate}
where $G^{\star}$ is obtained from $G$ by adding a new source vertex $-1$ and adding $a_v$ edges of the form $(-1,v)$ for $v\in [0,n-1]$.
\end{corollary}

\begin{proof}
Note first that (a) is a direct corollary of the proof of Theorem \ref{theorem:reformulated_lidskii_volume}.

    To show that (a) and (b) are equal note that part (b) is the number of integer flows on $G^*$ with netflow $(d,-o_0,-o_1,\dots,-o_{n-1},0)$. We can enumerate these in a different way. Consider a labeling of $G^*$ where the edges of the form $(-1,v)$ are labeled $x_1,x_2,\dots,x_{|a_n|}$ consistent with the order of the vertices $v\in [0,n-1]$. The distribution of the netflow $d$ on these edges gives a weak composition $\bj=(j(x))_{x\in X}$ of $d$ in $|X|$ parts. When we vary these assignment in all possible $\bj$ we obtain the formula in (a).
    
    That (b) and (c) are equal is an application of Theorem \ref{thm.genlidskii} in the case $\ba=\be_0- \be_n$. 
\end{proof}

A simpler formula for the number of unshuffled objects can be computed from the bijection between permutation flows on $(\hatG,\hatF)$ with netflow $\ba$ and integer flows on $G$ with netflow \[\hat{\bd}:=\bd+\ba-(\be_0-\be_n).\]

\begin{theorem}
\label{thm:integer_flows_a}
    Saturated permutation flows on $(\hatG,\hatF)$ are in bijection with integer $\hat\bd$-flows on $G$. As a consequence, 
    $|\SatPermutationFlows(\hatG,\hatF)|=|\calF_G^{\bbZ}(\hat\bd)|=K_G(\hat\bd)$.
\end{theorem}
\begin{proof}
Let $\rho\in\SatPermutationFlows(\hatG,\hatF)$. We show that the flow $\psi$ defined by $\psi(\rho)(e)=|\rho(e)|-1$ satisfies Equation~\eqref{eqn:defining_equations} with netflow vector $\hat{\bd}=(\hat d_0,\ldots,\hat d_n)$.
    
At $v=0$, $a_0$ is the number of inflow half-edges; each one except the first contributes $1$ to $\psi(e)$ for all $e\in\outedge(0)$, so $\hat d_0=a_0-1$.

At an inner vertex $v$, conservation of flow of $\psi$ implies 
\begin{align*}
\sum_{e\in\outedge(v)}\psi(e)  &=\sum_{e\in\inedge(v)}\psi(e) + \hat{d}_v \\
&=\sum_{e\in\inedge(v)}(|\rho(e)|-1)+\hat{d}_v\\ 
&=|\InPerm(v)|-a_v-|\inedge(v)|+ \hat{d}_v\\ 
&=|\OutPerm(v)|-(|\outedge(v)|-1)-a_v-|\inedge(v)|+ \hat{d}_v\\
&=\sum_{e\in\outedge(v)}(|\rho(e)|-1)-a_v-(|\inedge(v)|-1)+\hat{d}_v\\
&=\sum_{e\in\outedge(v)}\psi(e)-a_v-d_v+\hat{d}_v,
\end{align*}
solve for  $\hat{d}_v$ to find $\hat{d}_v=a_v+d_v$.

Conservation of flow must also be satisfied at vertex $n$, so $\hat{d}_n=-\sum_{v<n}\hat{d}_v$, completing the proof.
\end{proof}

\begin{example}
    For the graph $G$ in Figure~\ref{fig:auggraph} and the netflow vector $\ba=(2,1,0,1,0,-4)$, we have \[\hat\bd=(0,1,1,2,2,-6)+(2,1,0,1,0,-4)-(1,0,0,0,0,-1)=(1,2,1,3,2,-9)\] The integer flow $\hat{\psi}$ in Figure~\ref{fig:FSM_integerflow} corresponds to the saturated permutation flow $\rho$ in Figure~\ref{fig:FSM_maxpermuflo}.
\end{example}

\begin{figure}[h!]
    \centering\vspace{-.5in}
\begin{tikzpicture}
\begin{scope}[scale=1.5, yscale=1.0]
\node() at (-.5,.5) {\textcolor{purple}{$\hat\psi$}};

\vertex[fill=orange, minimum size=4pt, label=below:{\tiny\textcolor{orange}{$0$}}](v0) at (0,0) {};
\vertex[fill=orange, minimum size=4pt, label=below:{\tiny\textcolor{orange}{$1$}}](v1) at (1,0) {};
\vertex[fill=orange, minimum size=4pt, label=below:{\tiny\textcolor{orange}{$2$}}](v2) at (2,0) {};
\vertex[fill=orange, minimum size=4pt, label=below:{\tiny\textcolor{orange}{$3$}}](v3) at (3,0) {};
\vertex[fill=orange, minimum size=4pt, label=below:{\tiny\textcolor{orange}{$4$}}](v4) at (4,0) {};
\vertex[fill=orange, minimum size=4pt, label=below:{\tiny\textcolor{orange}{$5$}}](v5) at (5,0) {};		

\draw[-stealth, thick, color=black!30, color=black!30] (v0) .. controls (1.2, 1.6) and (2.5, -0.3) .. (v3);
\draw[-stealth, thick, color=black!30] (v0) .. controls (0.9, 1.0) and (1.5, -0.7) .. (v2);
\draw[-stealth, thick, color=black!30] (v0) to [out=30,in=150] (v1);
\draw[-stealth, thick, color=black!30] (v0) to [out=-30,in=-150] (v1);

\draw[-stealth, thick, color=black!30] (v1) .. controls (1.9, 1.0) and (2.5, -0.7) .. (v3);
\draw[-stealth, thick, color=black!30] (v1) to [out=30,in=150] (v2);
\draw[-stealth, thick, color=black!30] (v1) .. controls (2.0, -1.2) and (2.5, 0.7) .. (v3);	

\draw[-stealth, thick, color=black!30] (v2) to [out=45,in=135] (v4);
\draw[-stealth, thick, color=black!30] (v2) .. controls (3.0, -1.0) and (3.1, 0.3) .. (v4);	

\draw[-stealth, thick, color=black!30] (v3) to [out=60,in=120] (v5);
\draw[-stealth, thick, color=black!30] (v3) to [out=-30,in=-150] (v4);
\draw[-stealth, thick, color=black!30] (v3) .. controls (4.0, -1.0) and (4.5, 0.0) .. (v5);

\draw[-stealth, thick, color=black!30] (v4) to [out=30,in=150] (v5);
\draw[-stealth, thick, color=black!30] (v4) to [out=-30,in=-150] (v5);


\node[] at (.6, .6){\scriptsize\textcolor{purple}{$0$}};
\node[] at (.55, .35){\scriptsize\textcolor{purple}{$1$}};
\node[] at (.4, .15){\scriptsize\textcolor{purple}{$0$}};
\node[] at (1.5, .4){\scriptsize\textcolor{purple}{$0$}};
\node[] at (1.5, .15){\scriptsize\textcolor{purple}{$1$}};
\node[] at (2.8, .45){\scriptsize\textcolor{purple}{$2$}};
\node[] at (3.7, .55){\scriptsize\textcolor{purple}{$0$}};
\node[] at (3.6, -0.2){\scriptsize\textcolor{purple}{$3$}};
\node[] at (4.48, .2){\scriptsize\textcolor{purple}{$7$}};
\node[] at (.3, -.15){\scriptsize\textcolor{purple}{$0$}};
\node[] at (1.5, -.43){\scriptsize\textcolor{purple}{$1$}};
\node[] at (2.8, -.43){\scriptsize\textcolor{purple}{$1$}};
\node[] at (3.7, -.45){\scriptsize\textcolor{purple}{$1$}};
\node[] at (4.3, -.1){\scriptsize\textcolor{purple}{$1$}};
\end{scope}
\end{tikzpicture}
    \vspace{-.3in}
    \caption{The integer $\hat\bd$-flow $\hat\psi$ that corresponds to the saturated permutation flow $\rho$ in Figure~\ref{fig:FSM_maxpermuflo}.
    }
    \label{fig:FSM_integerflow}
\end{figure}

\section*{Acknowledgments}
We wholeheartedly acknowledge the fundamental contribution of our collaborator Alejandro Morales. This project has been pushed forward by his constant encouragement, support, and the many facets 
of flow polytopes that we have learned directly from him and his work. We also thank very specially our wonderful collaborators Carolina Benedetti and Pamela Harris whose conversations and amazing work have also informed and contributed to this project.

We thank the Institute for Computational and Experimental Research in Mathematics (ICERM) for funding two Collaborate grants in 2023 and 2024, that advanced this project greatly. This project has grown out of a ten year collaboration that started at the American Institute of Mathematics, for which we are also grateful. 

We also want to thank Matias von Bell, Benjamin Braun, Rachel Brunner, Cesar Ceballos, Luis Alfonso Contreras Ordo\~nez, William Dugan, Aaron Lauve, Eva Philippe, Vincent Pilaud, Viviane Pons, Maria Ronco, Khrystyna Serhiyenko, Daniel Tamayo Jim\'enez, and Yannic Vargas for various helpful conversations and insights about flow polytopes and the objects related to this project. 

R.\ S.\ Gonz\'alez D'Le\'on was partially supported by an AMS-Simons Research Enhancement Grant for Primarily Undergraduate Institution Faculty. C.\ R.\ H.\ Hanusa acknowledges the support of a PSC-CUNY grant and travel funding from the Department of Mathematics at Queens College. M.\ Yip was partially supported by Simons Collaboration Grant 964456.
   
\appendix

\section{Dictionary of Symbols}

\noindent
$\leq$, $<$,  $\lessdot$ - weak order relation \newline
$\preceq$, $\prec$, $\precdot$ - path order relation or framing relation \newline
$\subseteq$, $\subset$, $\subsetdot$ - containment order relation \newline
$\unlhd$, $\lhd$ - dominance order relation


\noindent
$a_v$ - netflow vector entry \newline
$d_v, \hat{d}_v$ - netflow vector entry \newline
$d=m-n$ - dimension of flow polytope \newline
$e, e', e''$ - an edge in $G$ \newline
$f$ - injective function \newline
$h^*$ - polynomial \newline
$i_v$ - indegree netflow vector entry \newline
$h,i,j,k,l$ - indices \newline
$k$ - rank of a combinatorial object \newline
$l$ - letter in a word \newline
$m$ - number of edges \newline
$n$ - number of vertices minus 1. \newline
$o_v$ - outdegree netflow vector entry \newline
$s_v$ - slack (edge in $G$) \newline
$t_j$ - split (edge in $G$) \newline
$u, v, w$ - vertex \newline
$v_x$ - the unique vertex associated to the inflow half-edge $x$ \newline
$x, x_i$ - inflow half-edge \newline
$x_0, y_0$ - paths of length $0$, at vertices $0$ and $n$, respectively\newline
$y$ - outflow half-edge 

\noindent
$A, A'$ - integer matrix \newline
$E, \hatE$ - edge set / augmented edge set \newline
$F, \hatF$ - framing of a graph / augmented graph  \newline
$G, \hatG$ - graph / augmented graph \newline
$H, \hatH$ - subgraph, support subgraph \newline
$I_w, I_i,J_j$ - intervals \newline
$I(P), I(e)$ - incidences to grove vertices  \newline
$K_G(\ba)$ - Kostant partition function for $\ba$ \newline
$L_t$ - lowering operator \newline
$L(\gamma_v)$ - set of left vertices (prefixes) of a grove forest \newline
$M, N$ - routes in a multiclique \newline
$O, P, Q, R$ - paths, routes, prefixes, suffixes \newline
$(P,e)$ - a grove edge \newline
$R_t$ - raising operator \newline
$R(\gamma_v)$ - set of right vertices (graph edges) of a grove forest \newline
$S_x^i$ - the $i$-th split of $x$ in $\calV$ \newline
$V$ - vertex set \newline
$V_x$ $V_x^i$ - route in a vine \newline
$W(G,F)$ - weak order \newline 
$X$ - set of inflow half-edges \newline
$Y$ - set of outflow half-edges 

\noindent
$\ba$ - netflow vector \newline
$\bd, \tilde{\bd}, \hat{\bd}$ - in-degree vector \newline
$\be_i$ - $i$-th basis unit vector \newline 
$\bi$ - in-degree netflow vector \newline 
$\bj$ - integer flow corresponding to splits. \newline 
$\bj(\Gamma)$ - weak composition \newline 
$\hat\bj$ - coarsening of $\bj$ \newline 
$\bo$ - out-degree netflow vector \newline 
$\bw=(\bw(e))_{e\in E}$ - a (real-valued) point in $\calF$ \newline
$\bw(e)$ - coordinate of a point in $\calF$ \newline
$\bw(P), \bw(e), \bw(R)$ - weight function derived from $\bw$ \newline
$\bz$ - integer vertex of $\calF$ \newline
$\bz(P)$, $\bz(\calP)$ - indicator flow of $P$, $\calP$ 

\noindent
$\alpha_i, \beta_j$ - positive coefficients for the convex combination giving $\bw$ \newline
$\gamma$, $\gamma_v$, $\theta$, $\theta_v$ - a non-crossing bipartite forest in a grove \newline
$\delta$ - indicator function---equals 1 or 0\newline
$\zeta_v, \zeta$ - $v$-th summary, final summary \newline
$\lambda$ - natural labeling of prefixes and groves \newline
$\pi$, $\rho$ - permutation flow \newline
$\pi^0$, $\pi^1$ - bottom and top saturated permutation flow \newline
$(\pi,\sigma_\pi), (\rho,\sigma_\rho)$ - permutation flow shuffle \newline
$\sigma, \tau$ - shuffle \newline
$\phi$ - flow \newline
$\psi$ - integer flow 

\noindent
$\Gamma, \Theta$ - grove \newline
$\underline{\Gamma},\overline{\Gamma}$ - minimal and maximal saturation of $\Gamma$ \newline
$(\Gamma,\sigma_\Gamma), (\Theta,\sigma_\Theta)$ - grove shuffle  

\noindent
$\calC,\calD$ - clique of route matchings \newline
$\calF_G,\calF_G(\ba)$ - flow polytope with netflow vector $\be_0- \be_n$ or  $\ba$ \newline
$\calF_G^\bbZ,\calF_G^\bbZ(\ba)$,  - integer lattice points of a flow polytope \newline
$\calM, \calN$ - multiclique of routes \newline
$\calP, \calQ$ - route matchings \newline
$\calR$ - coherent collection of routes \newline
$\calV_x, \calW_x$ - a vine in a vineyard \newline
$\calV, \calW$ - a vineyard (or a vine when $\ba=\be_0- \be_n$)\newline
$(\calV,\sigma_\calV), (\calW,\sigma_\calW)$ - a vineyard shuffle 

\noindent
$\perm_A$ - partial permutations of $A$ \newline
$\sym_A$ - permutations of $A$ \newline
$\triangle_\calC$ - simplex corresponding to a clique $\calC$ 

\printbibliography

\end{document}